\documentclass{article}
\usepackage[utf8]{inputenc}
\usepackage[all,cmtip]{xy}
\usepackage{amssymb,amsmath,amsthm}
\usepackage{mathtools, mathdots}
\usepackage{fancyhdr} 
\usepackage{mathrsfs, calligra, bbm}
\usepackage[top=1.3in, bottom=1.3in, left=1in, right=1in]{geometry}
\usepackage{hyperref}
\usepackage{tikz-cd}
\usetikzlibrary{decorations.pathmorphing}
\usepackage{tikz}
\usepackage{enumitem}
\usepackage[T1]{fontenc}
\usepackage{graphicx}
\usepackage{subcaption}
\usepackage{array}
\usepackage[toc,page]{appendix}
\usepackage{stfloats}
\usepackage[strict=true, style=british]{csquotes}
\usepackage[british]{babel}
\usepackage[backend=biber,style=alphabetic,sorting=anyt, maxbibnames=99,maxcitenames=99,maxalphanames=99,
]{biblatex}
\let\emptyset\varnothing

\newtheorem{Definition}{Definition}[subsection]
\newtheorem{Theorem}[Definition]{Theorem}
\newtheorem{Lemma}[Definition]{Lemma}
\newtheorem{Proposition}[Definition]{Proposition}
\newtheorem{Corollary}[Definition]{Corollary}

\newtheorem{Hypothesis}[]{Hypothesis}

\newtheorem{Conjecture}[]{Conjecture}
\newtheorem{Assumption}[Definition]{Assumption}

\newtheorem*{Question*}{Question}
\newtheorem*{Theorem*}{Theorem}
\newtheorem*{Conjecture*}{Conjecture}

\theoremstyle{remark}
\newtheorem{Example}[Definition]{Example}
\newtheorem{Remark}[Definition]{Remark}

\newcommand\mapsfrom{\mathrel{\reflectbox{\ensuremath{\mapsto}}}}
\newcommand\scalemath[2]{\scalebox{#1}{\mbox{\ensuremath{\displaystyle #2}}}}

\DeclareMathOperator\A{\mathbf{A}}

\DeclareMathOperator\C{\mathbf{C}}
\DeclareMathOperator\D{\mathbf{D}}

\DeclareMathOperator\F{\mathbf{F}}

\DeclareMathOperator\Q{\mathbf{Q}}
\DeclareMathOperator\R{\mathbf{R}}

\DeclareMathOperator\Z{\mathbf{Z}}

\DeclareMathOperator\bft{\mathbf{t}}

\DeclareMathOperator\bfitu{\textbf{\textit{u}}}

\DeclareMathOperator\bbA{\mathbb{A}}

\DeclareMathOperator\bbG{\mathbb{G}}
\DeclareMathOperator\bbH{\mathbb{H}}

\DeclareMathOperator\bbT{\mathbb{T}}
\DeclareMathOperator\bbX{\mathbb{X}}

\DeclareMathOperator\calD{\mathcal{D}}
\DeclareMathOperator\calE{\mathcal{E}}

\DeclareMathOperator\calG{\mathcal{G}}

\DeclareMathOperator\calL{\mathcal{L}}

\DeclareMathOperator\calN{\mathcal{N}}
\DeclareMathOperator\calO{\mathcal{O}}
\DeclareMathOperator\calP{\mathcal{P}}

\DeclareMathOperator\calR{\mathcal{R}}

\DeclareMathOperator\calT{\mathcal{T}}
\DeclareMathOperator\calU{\mathcal{U}}
\DeclareMathOperator\calV{\mathcal{V}}
\DeclareMathOperator\calW{\mathcal{W}}
\DeclareMathOperator\calX{\mathcal{X}}
\DeclareMathOperator\calY{\mathcal{Y}}
\DeclareMathOperator\calZ{\mathcal{Z}}

\DeclareMathOperator\scrF{\mathscr{F}}
\DeclareMathOperator\scrG{\mathscr{G}}

\DeclareMathOperator\scrM{\mathscr{M}}

\DeclareMathOperator\scrO{\mathscr{O}}

\DeclareMathOperator\scrS{\mathscr{S}}

\DeclareMathOperator\frakG{\mathfrak{G}}

\DeclareMathOperator\frakI{\mathfrak{I}}
\DeclareMathOperator\frakJ{\mathfrak{J}}

\DeclareMathOperator\frakL{\mathfrak{L}}

\DeclareMathOperator\frakT{\mathfrak{T}}

\DeclareMathOperator\frakX{\mathfrak{X}}

\DeclareMathOperator\frakk{\mathfrak{k}}

\DeclareMathOperator\frakm{\mathfrak{m}}

\DeclareMathOperator\bfalpha{\boldsymbol{\alpha}}

\DeclareMathOperator\bfgamma{\boldsymbol{\gamma}}
\DeclareMathOperator\bfdelta{\boldsymbol{\delta}}
\DeclareMathOperator\bfepsilon{\boldsymbol{\varepsilon}}

\DeclareMathOperator\bfpi{\boldsymbol{\pi}}

\DeclareMathOperator\bftau{\boldsymbol{\tau}}

\DeclareMathOperator\ad{ad}
\DeclareMathOperator\an{an}

\DeclareMathOperator\Aut{Aut}
\DeclareMathOperator\CH{CH}
\DeclareMathOperator\charpoly{char}
\DeclareMathOperator\cl{cl}

\DeclareMathOperator\cris{cris}
\DeclareMathOperator\cusp{cusp}
\DeclareMathOperator\cyc{cyc}
\DeclareMathOperator\Det{Det}
\DeclareMathOperator\Def{\mathbf{Def}}
\DeclareMathOperator\diag{diag}

\DeclareMathOperator\dR{dR}
\DeclareMathOperator\End{End}
\DeclareMathOperator\et{\mathrm{\acute{e}t}}

\DeclareMathOperator\Ext{Ext}
\DeclareMathOperator\Fil{Fil}
\DeclareMathOperator\Frob{Frob}
\DeclareMathOperator\fs{fs}
\DeclareMathOperator\Gal{Gal}
\DeclareMathOperator\Gr{Gr}
\DeclareMathOperator\Hom{Hom}
\DeclareMathOperator\HT{HT}
\DeclareMathOperator\HTS{HTS}
\DeclareMathOperator\id{id}
\DeclareMathOperator\image{image}
\DeclareMathOperator\Iw{Iw}
\DeclareMathOperator\nc{nc}
\DeclareMathOperator\one{\mathbbm{1}}
\DeclareMathOperator\oneanti{\breve{\one}}
\DeclareMathOperator\opp{opp}
\DeclareMathOperator\Pf{Pf}
\DeclareMathOperator\pr{pr}
\DeclareMathOperator\pst{pst}
\DeclareMathOperator\rank{rank}
\DeclareMathOperator\red{red}

\DeclareMathOperator\rig{rig}
\DeclareMathOperator\Sen{Sen}

\DeclareMathOperator\SK{SK}

\DeclareMathOperator\Spa{Spa}
\DeclareMathOperator\Spec{Spec}

\DeclareMathOperator\st{st}
\DeclareMathOperator\tor{tor}
\DeclareMathOperator\trans{^{\texttt{t}}}

\DeclareMathOperator\univ{univ}
\DeclareMathOperator\unr{unr}
\DeclareMathOperator\wt{wt}
\DeclareMathOperator\Yos{Yos}

\DeclareMathOperator\GL{GL}
\DeclareMathOperator\SL{SL}
\DeclareMathOperator\GSp{GSp}
\DeclareMathOperator\Sp{Sp}

\DeclareMathOperator\Alg{{\mathsf{Alg}}}
\DeclareMathOperator\Ar{{\mathsf{Ar}}}

\DeclareMathOperator\CAlg{{\mathsf{CAlg}}}

\DeclareMathOperator\Group{{\mathsf{Group}}}

\DeclareMathOperator\Mod{{\mathsf{Mod}}}

\DeclareMathOperator\Rep{{\mathsf{Rep}}}
\DeclareMathOperator\Set{{\mathsf{Set}}}

\DeclareMathOperator\VB{{\mathsf{VB}}}

\makeatletter
\renewcommand{\maketitle}{\bgroup\setlength{\parindent}{0pt}
\begin{flushleft}
  \LARGE{\textbf{\@title}}
  
  \vspace{4mm}
  
  \large{\textsc{\@author}}
  
  \vspace{4mm}
\end{flushleft}\egroup
}
\makeatother

\title{Families of symplectic Galois representations over small parabolic eigenvarieties for Siegel cuspforms of genus $2$}
\author{Muhammad Manji, Frederick E. Thøgersen, and Ju-Feng Wu}
\date{}

\begin{document}

\maketitle

{\footnotesize
\paragraph{Abstract.} We construct small parabolic eigenvarieties for holomorphic Siegel cuspforms of genus $2$ and study families of Galois representations attached to them in the spirit of Bellaïche--Chenevier. In the course, we introduce the notion of $(\varphi, \Gamma)$-modules with $G$-structures and the notion of refined families of symplectic Galois representations by implementing the theory of symplectic Galois determinant d'après Moakher--Quast. Such families of symplectic Galois representations provide two applications: In the first application, we show that the small parabolic eigenvarieties are smooth at non-critical points by proving an infinitesimal $R=\mathbb{T}$ theorem. In the second application, we study the relationship between the geometry of the small parabolic eigenvarieties at the Saito--Kurokawa lifts for cuspidal eigenforms (both finite- and infinite-slope) and the Bloch--Kato Selmer groups of those eigenforms. 
}

\tableofcontents

\section{Introduction}\label{section: Intro}

\subsection{Background}
The theory of $p$-adic families of modular forms has its origin in J.-P. Serre's famous Antwerp paper (\cite{Serre-Antwerp}), where he studied the congruences among Eisenstein series, introduced the notion of $p$-adic modular forms, and used it as an application to construct $p$-adic $L$-functions for totally real number fields. The first turning point of this theory is due to H. Hida. In \cite{Hida-IwasawaCongruence, Hida-bigGalois}, he constructed and studied ordinary families of cuspidal eigenforms varying continuously over the $p$-adic weight space and their corresponding $p$-adic families of Galois representations. In the late $20^{\text{th}}$ century, R. Coleman introduced the notion of finite-slope families of modular forms, which generalised Hida's theory (\cite{Coleman-pAdicFamilies}). Together with B. Mazur, they showed that finite-slope families of modular forms can be patched into a geometric object, now known as the \emph{eigencurve} (\cite{Coleman_Mazur}).

The construction of the eigencurve admits generalisations to general automorphic forms. For example, in \cite{Ash-Stevens, Urban-2011, Hansen-PhD}, the authors studied the so-called overconvergent cohomology groups of locally symmetric spaces to construct eigenvarieties. On the other hand, by further developing Hida's idea, M. Emerton gave another construction in \cite{Emerton} by studying the completed cohomology of locally symmetric spaces.

The approach of Andreatta--Iovita--Pilloni (\cite{AIP-2015}, see also \cite{AIS-2014} and \cite{Pilloni-GL2}) is different. They constructed big sheaves over overconvergent neighbourhoods of the ordinary loci of Siegel modular varieties, interpolating the classical automorphic sheaves. They showed that the global sections of these big sheaves $p$-adically interpolate the space of classical Siegel modular forms and used these larger spaces to construct eigenvarieties for Siegel cuspforms. Their eigenvarieties parametrise holomorphic cuspidal Siegel eigenforms. This approach later led Boxer--Pilloni to develop their higher Coleman theory (\cite{BP-HigherColeman}), which yields a construction of eigenvarieties for coherent cohomology in all degrees.

All these different constructions of eigenvarieties are expected to encode rich arithmetic properties of automorphic representations. For example, in their seminal book \cite{BC}, J. Bellaïche and G. Chenevier developed a method to understand the relationships between the local geometry of eigenvarieties and the arithmetic of Galois representations attached to automorphic forms. They developed a strategy to prove an \emph{infinitesimal $R=\bbT$ theorem} which links the local geometry with the arithmetic of Galois representations. As an application, Bellaïche--Chenevier proved a non-vanishing result for the Bloch--Kato Selmer groups attached to certain unitary automorphic representations.

Although these (rather classical) eigenvarieties provide a powerful framework for studying the arithmetic of automorphic forms via $p$-adic methods, a fundamental feature is that they parametrise automorphic forms that are of finite slope with respect to all Hecke operators at $p$. On the other hand, from an arithmetic point of view, automorphic forms with weaker slope conditions at $p$ are also of interest to mathematicians. For example, in \cite{Kato-2004}, K. Kato proved vanishing results for certain Bloch--Kato Selmer groups attached to cuspidal eigenforms without any slope condition at $p$. This leads to a natural question: can we study the arithmetic of automorphic forms with weaker slope condition at $p$ using $p$-adic methods? 

In search of the natural home for Euler systems, Iwasawa theory and $p$-adic $L$-functions, D. Loeffler conjectured the existence of \emph{small} and \emph{big parabolic eigenvarieties} by studying ordinary Galois deformation problems (\cite{Loeffler-UnivDeform}). The construction of small parabolic eigenvarieties has been carried out by Hill--Loeffler in \cite{Hill-Loeffler} and Breuil--Ding (\cite{Breuil--Ding-BernsteinEigen}) (using completed cohomology) and Barrera Salazar--Williams in \cite{BSW-ParabolicEigen} (using overconvergent cohomology). Although this type of eigenvarieties has more restrictive $p$-adic weight spaces, it requires weaker slope conditions at $p$. More precisely, small parabolic eigenvarieties only impose finite-slope conditions for certain Hecke operators at $p$, depending on the chosen parabolic subgroup.\footnote{ We shall see later that, in the present paper, we give a construction of small parabolic eigenvarieties for $\GSp_4$ for holomorphic Siegel cuspforms (\emph{i.e.}, the Andreatta--Iovita--Pilloni style). } Moreover, this type of eigenvarieties also have fruitful arithmetic applications; for example, in \cite{BSDW26}, Barrera Salazar--Dimitrov--Williams constructed a $p$-adic $L$-function living over a small parabolic eigenvariety and used their $p$-adic $L$-function to understand the local geometry of the small parabolic eigenvariety.

However, to our knowledge, although the Galois information carried by small parabolic eigenvarieties is conjectured in \cite[Sect. 6.3.3]{Loeffler-UnivDeform} (similar to the one in \cite{BC}), little is known except for the cases studied in \cite{Breuil--Ding-BernsteinEigen}. The aim of this paper is to bridge this gap by studying the small parabolic eigenvarieties for holomorphic Siegel cuspforms of genus $2$, and by relating their geometry to the Galois information they carry. Moreover, since the slope conditions at $p$ are relaxed, this allows us to obtain applications towards new cases of the Bloch--Kato conjecture for both finite-slope and, more interestingly, infinite-slope cuspidal eigenforms by looking at the Saito--Kurokawa lifts of these eigenforms. 

We should point out that infinite-slope modular forms do not live in an interesting family when looking at small parabolic eigenvarieties for holomorphic modular forms since the corresponding weight space is just one single point. The key idea in the present paper is that, once we consider the corresponding Saito--Kurokawa lifts, these Siegel modular forms then live in an $1$-dimensional small parabolic eigenvariety for holomorphic Siegel modular forms of genus $2$. We hope that this strategy will open future avenues of studying infinite-slope cuspidal eigenforms using the eigenvariety machinery.  More detailed discussions will be provided in the following subsections.

\subsection{Small parabolic eigenvarieties and families of Galois representations}

Consider the algebraic group $\GSp_4$ defined as in Sect. \ref{subsection: GSp} and let $B$ be the upper-triangular Borel subgroup. The proper parabolic subgroups of $\GSp_4$ containing $B$ are $B$, $P_{\mathrm{Si}}$, and $P_{\mathrm{Kl}}$, where \begin{align*}
    P_{\mathrm{Si}} & = \begin{pmatrix} * & * & * &*\\ * & * & * & *\\ &  & * & * \\ & & * & *\end{pmatrix} \cap \GSp_4 \text{: the Siegel parabolic subgroup,}\\
    P_{\mathrm{Kl}} & = \begin{pmatrix} * & * & * & * \\  & * & * & *\\ & *  & * & * \\ & &  & *\end{pmatrix} \cap \GSp_4 \text{: the Klingen parabolic subgroup.}
\end{align*}
Let $P$ be one of these parabolic subgroups. Since the classical theory of eigenvarieties corresponds to the choice $P = B$, we focus only on the cases $P = P_{\mathrm{Si}}, P_{\mathrm{Kl}}$ in what follows.

Let $p$ be an odd prime. To construct small parabolic eigenvarieties for $P$, we first need to know the corresponding $p$-adic weight space, denoted by $\calW_{P, k_{\heartsuit}}$, depending on the parabolic subgroup $P$ and a fixed classical weight $k_{\heartsuit} = (k_{\heartsuit, 1}, k_{\heartsuit, 2})\in \Z^2$. This $p$-adic weight space is a one-dimensional subspace in the usual two-dimensional $p$-adic weight space for Siegel modular forms (\emph{e.g.}, \cite[Sect. 2.2]{AIP-2015}). More precisely, $\calW_{P, k_{\heartsuit}}$ is the one-dimensional subspace varying in the parallel direction (resp., the $k_{1}$-direction) if $P = P_{\mathrm{Si}}$ (resp., $P  = P_{\mathrm{Kl}}$) passing through $k_{\heartsuit}$.

Although there are existing constructions of small parabolic eigenvarieties (\emph{e.g.}, \cite{Hill-Loeffler, BSW-ParabolicEigen, Breuil--Ding-BernsteinEigen}), it turns out that, for the purpose of the applications, it would be more convenient for us to have a version for the degree-$0$ coherent cohomology, which we could not find in the literature. In Sect. \ref{section: small par. eigenvar. for H0}, we take this opportunity to include a construction of small parabolic eigenvarieties for holomorphic Siegel cuspforms of genus $2$. Our approach is highly inspired by \cite{AIP-2015, Pilloni-higherHidaColemanGSp4}. In particular, we construct big sheaves over overconvergent neighbourhoods of ordinary loci that interpolate the classical automorphic sheaves. By considering the corresponding controlling operators $U_P$, we apply the general eigenvariety machinery to construct small parabolic eigenvarieties for holomorphic Siegel cuspforms of genus $2$, denoted by $\calE_{P, k_{\heartsuit}}^{\cusp}$.\footnote{ We do not claim full credit. In \cite{Pilloni-higherHidaColemanGSp4}, Pilloni provided almost all the ingredients for constructing the small parabolic eigenvariety when $P = P_{\mathrm{Kl}}$, except that he did not really write down such a construction since his focus was on something else. Our contribution in this case is simply to spell out how to use his construction to construct small parabolic eigenvarieties when $P = P_{\mathrm{Kl}}$. When $P = P_{\mathrm{Si}}$, see the related work of Brasca--Rosso (\cite{Brasca--Rosso}) We should also point out that similar construction for Hilbert modular forms can be found in \cite{Dimitrov--Hsu}.} As in the classical theory, it admits a natural map \[
    \wt: \calE_{P, k_{\heartsuit}}^{\cusp} \rightarrow \calW_{P, k_{\heartsuit}},
\]
called the weight map. Both $\calE_{P, k_{\heartsuit}}^{\cusp}$ and $\calW_{P, k_{\heartsuit}}$ are equidimensional of dimension $1$. Roughly speaking, $\calE_{P, k_{\heartsuit}}^{\cusp}$ can be viewed as the moduli space (over $\calW_{P, k_{\heartsuit}}$) of $U_P$-finite-slope cuspidal Siegel eigenforms of genus $2$.

Thanks to the work of many mathematicians (see, for example, \cite{Taylor-Siegel, Laumon, Weissauer, Urban-GSp4, Sorensen-HilbertSiegel, Jorza-GSp, Mok-GL2CM}), if $f$ is a cuspidal Siegel eigenform of genus $2$, then there is a corresponding Galois representation \[
    \rho_f: \Gal_{\Q} \rightarrow \GSp_4(\overline{\Q}_p)
\]
which satisfies nice properties (summarised in Theorem \ref{Theorem: Galois representation attached to Siegel forms}). If $f$ is furthermore $U_P$-finite-slope, then the $(\varphi, \Gamma)$-module $\D_{\rig}^{\dagger}(\rho_f)$ is expected to admit a $(\varphi, \Gamma)$-stable flag with stabiliser $P^{\vee}$ (Hypothesis \ref{Hypothesis: explicit local-global compatibility at p}), where $P^{\vee} = P_{\mathrm{Kl}}$ if $P = P_{\mathrm{Si}}$, $P^{\vee} = P_{\mathrm{Si}}$ if $P = P_{\mathrm{Kl}}$. Using this information, we construct a $P^{\vee}$-refined family of symplectic $\Gal_{\Q}$-representations of dimension $4$ over $\calE_{P, k_{\heartsuit}}^{\cusp}$ (Proposition \ref{Proposition: global symplectic determinant over small parabolic eigenvarieties} and Corollary \ref{Corollary: refined families over small parabolic eigenvarieties}).    

We would like to understand how the refined family of Galois representations interacts with the geometry of $\calE_{P, k_{\heartsuit}}^{\cusp}$.  To this end, we fix a $U_P$-finite slope cuspidal Siegel eigenform $f$ of weight $k_{\heartsuit} = (k_{\heartsuit, 1}, k_{\heartsuit, 2})$ with $k_{\heartsuit, 1}\geq k_{\heartsuit, 2}>3$ and let $x$ be the corresponding point in $\calE_{P, k_{\heartsuit}}^{\cusp}$. We then prove the following theorem by proving an `infinitesimal $R = \bbT$ theorem'. 

\begin{Theorem}[\text{Theorem \ref{Theorem: infinitesimal R=T}}]
    Keep the notations as above and assume the following hold. \begin{itemize}
        \item The Galois representation $\rho_f$ is absolutely irreducible. 
        \item The $(\varphi, \Gamma)$-stable $P^{\vee}$-flag $\Fil_{\bullet}\D_{\rig}^{\dagger}(\rho_f)$ satisfies ($P^{\vee}$-REG) and ($P^{\vee}$-NCR)  (see Sect. \ref{subsection: deformation}).
        \item For any bad prime $\ell \neq p$, $p\nmid \ell^{12}-1$ and $\rho_f|_{I_{\ell}}$ is irreducible, where $I_{\ell}$ is the inertia subgroup at $\ell$. 
        \item The geometric adjoint Bloch--Kato Selmer group for $f$ vanishes (Conjecture \ref{Conjecture: geometric adjoint BK conjecture}).
    \end{itemize}
    Then, $\calE_{P, k_{\heartsuit}}^{\cusp}$ is smooth at $x$.
\end{Theorem}

\subsection{Saito--Kurokawa points and non-vanishing of Bloch--Kato Selmer groups}

As an application of having families of Galois representations over small parabolic eigenvarieties, we take inspiration from \cite{Skinner--Urban, BC, BB22} and study the arithmetic of points coming from Saito--Kurokawa lifts. It turns out, working with small parabolic eigenvarieties offers two advantages: on the one hand, it is more convenient for our purposes when dealing with Saito--Kurokawa lifts of finite-slope elliptic cuspforms; on the other hand, it allows access to Saito--Kurokawa lifts of infinite-slope elliptic cuspforms.

To be more precise, let $N$ and $k$ be positive integers and $k>2$. Let $f$ be an elliptic cuspidal newform of level $\Gamma_0(N)$ and weight $2k-2$. We further assume $f$ has sign $-1$. In this situation, $f$ admits a lift to a holomorphic cuspidal Siegel eigenform of paramodular-$N$ level and weight $(k,k)$, \emph{i.e.}, the so-called Saito--Kurokawa lift, denoted by $\SK(f)$. The Galois representation attached to $\SK(f)$ is of the form \[
    \rho_{\SK(f)} \sim \begin{pmatrix} \chi_{\cyc}^{2-k} \\ & \rho_f \\ && \chi_{\cyc}^{1-k} \end{pmatrix},
\]
where $\chi_{\cyc}$ is the $p$-adic cyclotomic character and $\rho_f$ is the two-dimensional Galois representation attached to $f$.

\paragraph{The finite-slope case.} Suppose $p\nmid N$ and let $\alpha, \beta$ be the two roots of the Hecke polynomial of $f$ at $p$. We further assume that the following conditions hold (Assumption \ref{Assumption: nice properties for finite-slope cuspidal newform}): \begin{itemize}
    \item (REG) We have $\alpha \neq \beta$ and $\alpha, \beta\neq p^{k-1}$. 
    \item (ST) For $\ell \mid N$, the corresponding local representation at $\ell$ is the Steinberg representation twisted by the unramified character $\xi$ with $\xi(\ell) = -1$. 
\end{itemize}
In this case, $\SK(f)$ has four $p$-stabilisations that are both $U_{\mathrm{Si}}$- and $U_{\mathrm{Kl}}$-finite-slope. We work with the $p$-stabilisation that corresponds to the ordered $\varphi$-eigenvalues $(p^{k-1}, \alpha, \beta, p^{k-2})$, and it defines a point $x$ in the small Klingen parabolic eigenvariety $\calE_{\mathrm{Kl}, (k,k)}^{\cusp}$.\footnote{ Of course, it also defines a point in the small Siegel parabolic eigenvariety, but we deliberately choose to work with $\calE_{\mathrm{Kl}, (k,k)}^{\cusp}$. } Note that this is a different choice of $p$-stabilisation compared with \cite{Skinner--Urban, BB22}.

Let $\bbT_x = \scrO_{\calE_{\mathrm{Kl}, (k,k)}^{\cusp}, x}$ be the local ring at $x$. The families of Galois representations over $\calE_{\mathrm{Kl}, (k,k)}^{\cusp}$ yields a four-dimensional Galois determinant \[
    \Det_x: \bbT_x[\Gal_{\Q}] \rightarrow \bbT_x.
\]
Since we are working over $\calE_{\mathrm{Kl}, (k,k)}^{\cusp}$ (in particular, the weight space only varies in the $k_1$-direction), one can show that the reducibility ideal of $\Det_{x}$ is exactly the maximal ideal of $\bbT_x$ (Proposition \ref{Proposition: reducibility idea = maximal ideal; fs}). In other words, $\Det_x$ is as irreducible as possible.

Using such information, together with some computations with \emph{generalised matrix algebras} (GMA) inspired by \cite{BC, BB22}, we deduce the following inequality, relating the dimension of the tangent space of $\calE_{\mathrm{Kl}, (k,k)}^{\cusp}$ and the dimension of the Bloch--Kato Selmer group $H^1_f(\Q, \rho_f(k-1))$.

\begin{Theorem}[$\text{Theorem \ref{Theorem: bound of tangent space; fs} and Corollary \ref{Corollary: non-vanishing of BK Selmer group; fs}}$]\label{Theorem: non-vanishing; fs, intro}
    Let $t$ be the dimension of the tangent space of $\calE_{\mathrm{Kl}, (k,k)}^{\cusp}$ at $x$ and $d$ be the dimension of $H^1_f(\Q, \rho_f(k-1))$. Under {\normalfont (REG)} and {\normalfont (ST)}, we have \[
        t \leq d(d+1).
    \]
    In particular, $H^1_f(\Q, \rho_f(k-1))$ does not vanish. 
\end{Theorem}

\paragraph{The infinite-slope case.} In this case, we apply a similar strategy to prove a similar theorem as Theorem \ref{Theorem: non-vanishing; fs, intro}. However, the infinite-slope case has a very different nature as in the finite-slope case. 

To be more precise, we assume $N = p^r$ for some $r>1$ and $f$ is an infinite-slope cuspidal newform of level $\Gamma_0(p^r)$. In this situation, under Hypothesis \ref{Hypothesis: explicit local-global compatibility at p}, $\SK(f)$ has at most two $p$-stabilisations $\SK(f)_1$ and $\SK(f)_2$ that are $U_{\mathrm{Si}}$-finite-slope but $U_{\mathrm{Kl}}$-infinite-slope (Proposition \ref{Proposition: p-stabilisations of infinite-slope SK(f)}). As a result, these $p$-stabilisations only define points in the small Siegel parabolic eigenvariety $\calE_{\mathrm{Si}, (0,0)}^{\cusp}$. Let $x$ be the point correspond to $\SK(f)_1$ (Sect. \ref{subsection: BK Selmer groups; infinite slope}).

Let $\bbT_x = \scrO_{\calE_{\mathrm{Si}, (0,0)}^{\cusp}, x}$ be the local ring at $x$. We again have a four-dimensional Galois determinant \[
    \Det_x: \bbT_x[\Gal_{\Q}] \rightarrow \bbT_x
\]
coming from the families of Galois representations over $\calE_{\mathrm{Si}, (0,0)}^{\cusp}$. In Proposition \ref{Proposition: a generic family passing through SK lifts of infinite-slope forms} and Corollary \ref{Corollary: reducibility idea = maximal ideal; infinite slope}, we show that, under some reasonable conditions (see Theorem \ref{Theorem: non-vanishing; infs, intro} below), the reducibility ideal of $\Det_x$ again agrees with the maximal ideal in $\bbT_x$. We should point out that, although this is a similar statement as in the finite-slope case, the nature of the proof is very different. One of the conditions we imposed therein is highly related to a question posted by Coleman--Stein in \cite{Coleman--Stein}, where they asked when an infinite-slope form can be approximated by finite-slope forms.

Finally, by further generalising a result of Bellaïche--Chenevier in Appendix \ref{section: pstperiod} and combining with some GMA computations (again), we prove the following inequality.

\begin{Theorem}[$\text{Theorem \ref{Theorem: lower bound theorem for infinite forms} and Corollary \ref{Corollary: non-vanishing}}$]\label{Theorem: non-vanishing; infs, intro}
    Suppose the elliptic cuspidal eigenform $f$ satisfies the following properties: \begin{itemize}
        \item The residual representation $\overline{\rho}_f$ is absolutely irreducible. 
        \item It is not approximable by finite-slope modular forms (see Remark \ref{Remark: infinite-slope forms not approximable by finite-slope forms}). 
    \end{itemize}
    Let $t$ be the dimension of the tangent space of $\calE_{\mathrm{Si}, (0,0)}^{\cusp}$ at $x$ and $d$ be the dimension of $H^1_f(\Q, \rho_f(k-1))$. Then, we have \[
        t \leq \frac{(d+1)(3d+2)}{2}+1.
    \]
    In particular, if $t>2$, $H^1_f(\Q, \rho_f(k-1))$ does not vanish. 
\end{Theorem}

\begin{Remark}
    If $f$ is an elliptic cuspidal eigenform of weight $2k-2$, the Bloch--Kato conjecture predicts that \[
        \mathrm{ord}_{s=k-1}L(f, s) = \dim H^1_f(\Q, \rho_f(k-1)).
    \]
    In particular, if $f$ has sign $-1$, then $H^1_f(\Q, \rho_f(k-1))$ is expected to be non-vanishing. Therefore, Theorem \ref{Theorem: non-vanishing; fs, intro} and Theorem \ref{Theorem: non-vanishing; infs, intro} can be viewed as partial results towards the Bloch--Kato conjecture. To our knowledge, previous known results in this direction require $f$ to be ordinary (see \cite{Skinner--Urban} and \cite{BB22}), while the similar result in 
    \cite{BC} for unitary automorphic forms requires a finite-slope assumption. On the other hand, the approach using Gross--Zagier formula and Kolyvagin method requires knowing the order of vanishing of the $L$-function to be exactly $1$; however, under such a hypothesis, one usually can show that Bloch--Kato Selmer group has dimension exactly 1. 
\end{Remark}

\begin{Remark}
    Our study raises two natural questions for future investigation:
    \begin{enumerate}
        \item[(i)] The Bloch--Kato conjecture, together with Theorem \ref{Theorem: non-vanishing; fs, intro} and Theorem \ref{Theorem: non-vanishing; infs, intro}, suggests that the local geometry of small parabolic eigenvarieties at these Saito--Kurokawa points is an interesting subject for further study.
        \item[(ii)] As remarked in \cite[Remark 9.1.5]{BC}, it would be very interesting to construct $p$-adic $L$-functions and study their behaviour around these Saito--Kurokawa points. As suggested by Iwasawa theory, these $p$-adic $L$-functions could hopefully shed some light on understanding the upper bound of $\dim H^1_f(\Q, \rho_f(k-1))$.
    \end{enumerate}
    We hope to return to these in future work.
\end{Remark}

\subsection{Organisation of the paper}
The present paper consists of the following three parts: 

\paragraph{Part I: Preliminaries (Sect. \ref{section: (phi, Gamma)} and Sect. \ref{section: families of Galois reps}).}  ~\\
\indent The purpose of Sect. \ref{section: (phi, Gamma)} is to set up notations and conventions in the theory of $(\varphi, \Gamma)$-modules. We begin by reviewing the classical theory of $(\varphi, \Gamma)$-modules in Sect. \ref{subsection: p-adic Hodge preparation}. The classical theory of $(\varphi,\Gamma)$-modules provides a powerful tool for studying Galois representations by translating them into linear-algebraic objects, which are often more amenable to explicit manipulation. However, this approach has a drawback when one considers Galois representations valued in an algebraic group $G$. In order to apply the theory, one must first choose a faithful representation $G \rightarrow \GL_n$, and this procedure often forgets the underlying $G$-structure. A natural way to retain this structure is to work instead with $G$-torsors and to use the language of Tannakian formalism. Motivated by this perspective, we review the theory of $G$-torsors and Tannakian formalism in Sect. \ref{subsection: review of G-torsors}. In Sect. \ref{subsection: (phi, Gamma)-modules with G-structure}, we introduce the notion of $(\varphi,\Gamma)$-modules with $G$-structure by combining the classical theory of $(\varphi,\Gamma)$-modules with the Tannakian formalism. When $P$ is a parabolic subgroup of an algebraic group $G$, we introduce the notion of $P$-flags on $(\varphi, \Gamma)$-modules with $G$-structure (resp., $P$-refinements for $G$-valued Galois representations) in Sect. \ref{subsection: P-flag} (resp., Sect. \ref{subsection: refinements}). These notions generalise the classical notions of \emph{triangulation} and \emph{refinements} on $(\varphi, \Gamma)$-modules.

Since we are interested in families of Galois representations attached to small parabolic eigenvarieties for holomorphic Siegel modular forms of genus $2$, the main theme of Sect. \ref{section: families of Galois reps} is families of symplectic Galois representations. In Sect. \ref{subsection: GSp}, we review some basic facts about the general symplectic group $\GSp_{2n}$; in Sect. \ref{subsection: symp. det}, we review the theory of symplectic determinants by following \cite{MQ-SympDet}. In Sect. \ref{subsection: families of Galois reps}, we introduce notion of \emph{refined families of symplectic Galois representations} by combining our studies in Sect. \ref{subsection: P-flag}, Sect. \ref{subsection: refinements}, and the theory of symplectic determinants. These notions generalise existing notions of \emph{refined families of Galois representations} in the literature (see, for example, \cite{BC} and \cite{Bergdall-paraboline}). Finally, we study certain deformation problems in Sect. \ref{subsection: deformation}.

\paragraph{Part II: Small parabolic eigenvarieties and families of symplectic Galois representations (Sect. \ref{section: small par. eigenvar. for H0} and Sect. \ref{section: R=T}).} ~\\
\indent Sect. \ref{section: small par. eigenvar. for H0} is dedicated to constructing small parabolic eigenvarieties for Siegel cuspforms of genus $2$. We start by reviewing Siegel threefolds and (holomorphic) algebraic Siegel modular forms in Sect. \ref{subsection: classical Siegel modular forms}. Sect. \ref{subsection: finite-slope parts, H0} is a slight detour in which we study the \emph{finite-slope parts} of spaces of (holomorphic) Siegel modular forms. In particular, inspired by the discussion in \cite[Sect. 4.2]{BP-HigherColeman}, we show in Proposition \ref{Proposition: finite-slope part it independent to the level at p, H0} that the finite-slope parts of spaces of (holomorphic) Siegel modular forms are independent of the level at $p$. In Sect. \ref{subsection: p-adic weight spaces}, we introduce the relevant $p$-adic weight spaces by following \cite[Sect. 3.1]{BSW-ParabolicEigen}. Inspired by the work of Andreatta--Iovita--Pilloni (\cite{AIP-2015}) and Pilloni (\cite{Pilloni-higherHidaColemanGSp4}), we construct families of automorphic sheaves in Sect. \ref{subsection: families of automorphic sheaves}, and use them to define overconvergent Siegel modular forms over families of weights. Then, in Sect. \ref{subsection: small par. eigenvar.}, we construct small parabolic eigenvarieties for (holomorphic) Siegel cuspforms of genus $2$.

Inspired by \cite[Chapter 7]{BC}, the goal of Sect. \ref{section: R=T} is to construct refined families of Galois representations over the small parabolic eigenvarieties constructed in Sect. \ref{subsection: small par. eigenvar.} and prove an infinitesimal $R=\bbT$ theorem. We begin by reviewing the Galois representations attached to Siegel cuspforms of genus $2$ in Sect. \ref{subsection: Gal reps for GSp4}. We also state and justify the hypothesis on \emph{explicit local-global compatibility at $p$} (Hypothesis \ref{Hypothesis: explicit local-global compatibility at p}) therein. In Sect. \ref{subsection: families of Gal reps over small par. eigenvar}, we construct the refined families of Galois representations over the small parabolic eigenvarieties. Finally, in Sect. \ref{subsection: adj Selmer}, we discuss the relationship between the adjoint Selmer groups and the deformation problems studied in Sect. \ref{subsection: deformation}. Under the usual conjecture on the vanishing of the adjoint Selmer group, we prove the infinitesimal $R=\bbT$ theorem in Theorem \ref{Theorem: infinitesimal R=T}.

\paragraph{Part III: Applications (Sect. \ref{section: SK points; finite-slope} and Sect. \ref{section: SK points; infinite slope}).} ~\\
\indent As an application of having families of Galois representations over the small parabolic eigenvarieties, we prove non-vanishing of Bloch--Kato Selmer groups of elliptic cuspidal eigenforms by studying their Saito--Kurokawa lifts. Our approach is highly inspired by \cite[Chapter 9]{BC} and \cite{BB22}. We should mention that the eigenvarieties considered in this part are slightly different from the one considered in the previous part (Remark \ref{Remark: slightly different eigenvariety} and Remark \ref{Remark: slightly different eigenvariety; infintie slope}). We make such changes due to technical convenience for our applications.

In Sect. \ref{section: SK points; finite-slope}, we study the case when the cuspidal eigenform is of finite slope. We begin by reviewing the general theory of Saito--Kurokawa lifts in Sect. \ref{subsection: SK lifts}. Then, we focus on the case when the cuspidal eigenform is of finite slope in Sect. \ref{subsection: finite-slope SK lifts}. In particular, we specify our assumptions on the finite-slope cuspidal eigenforms in Assumption \ref{Assumption: nice properties for finite-slope cuspidal newform} and show that the local symplectic Galois determinant at our Saito--Kurokawa point in the small Klingen parabolic eigenvariety is as irreducible as it can be (Proposition \ref{Proposition: reducibility idea = maximal ideal; fs}). In Sect. \ref{subsection: BK Selmer; fs}, we study the extension classes coming from the local symplectic Galois determinant and prove the desired non-vanishing result of Bloch--Kato Selmer group in Theorem \ref{Theorem: bound of tangent space; fs} and Corollary \ref{Corollary: non-vanishing of BK Selmer group; fs}.

In Sect. \ref{section: SK points; infinite slope}, we switch our focus to infinite-slope cuspidal eigenforms. We explain in Sect. \ref{subsection: SK lifts for infinite-slope cuspforms} why their Saito--Kurokawa lifts contribute to points on the small Siegel parabolic eigenvariety (under Hypothesis \ref{Hypothesis: explicit local-global compatibility at p}). In Sect. \ref{subsection: heuristic of geometry}, we discuss what kind of classical Siegel cuspforms can appear in a sufficiently small affinoid open neighbourhood of our Saito--Kurokawa point. In particular, in Proposition \ref{Proposition: a generic family passing through SK lifts of infinite-slope forms}, we show that, under some mild assumptions and up to shrinking, except for the Saito--Kurokawa point, the other classical points in the neighbourhood are given by genuine Siegel cuspforms. Finally, we again study the extension classes coming from the local symplectic Galois determinant in Sect. \ref{subsection: BK Selmer groups; infinite slope} and prove the desired non-vanishing result of Bloch--Kato Selmer group in Theorem \ref{Theorem: lower bound theorem for infinite forms} and Corollary \ref{Corollary: non-vanishing}.

\subsection*{Conventions}

Throughout this paper, we fix the following: \begin{itemize}
    \item Unless specified, every ring in this paper will be a commutative ring with multiplicative identity.
    \item For any prime number $\ell$, we fix once and forever an algebraic closure $\overline{\Q}_{\ell}$ of $\Q_{\ell}$ and an algebraic isomorphism $\C_{\ell}\simeq \C$, where $\C_{\ell}$ is the $\ell$-adic completion of $\overline{\Q}_{\ell}$. For any finite extension $F$ over $\Q_{\ell}$, we write $\Gal_{F}$ for the absolute Galois group $\Gal(\overline{\Q}_{\ell}/F)$. We also fix the $\ell$-adic absolute value on $\C_{\ell}$ so that $|\ell|=\ell^{-1}$.
    \item We also fix an algebraic closure $\overline{\Q}$ of $\Q$ and embeddings $\overline{\Q}_{\ell} \hookleftarrow \overline{\Q} \hookrightarrow \C$, which is compatible with the chosen isomorphisms $\C_{\ell} \simeq \C$. We analogously write $\Gal_{\Q}$ for the absolute group $\Gal(\overline{\Q}/\Q)$ and identify $\Gal_{\Q_{\ell}}$ as a (decomposition) subgroup of $\Gal_{\Q}$. 
    \item We fix an odd prime number $p\in \Z_{> 0}$. 
    \item[$\bullet$] For $m, n\in \Z_{\geq 1}$ and any set $R$, we denote by $M_n(R)$ (resp., $M_{n\times m}(R)$) the set of $(n\times n)$-matrices (resp., ($n\times m$)-matrices) with coefficients in $R$.
    \item[$\bullet$] The transpose of a matrix $\bfalpha$ is denoted by $\trans\bfalpha$.
    \item For any $n\in \Z_{\geq 1}$, we denote by $\one_n$ the $n\times n$ identity matrix and denote by $\oneanti_n$ the $n\times n$ anti-diagonal matrix whose non-zero entries are $1$; \emph{i.e.,} \[\one_n=\begin{pmatrix} 1& & \\ & \ddots & \\ & &1\end{pmatrix}\quad\text{ and }\quad\oneanti_n=\begin{pmatrix} & & 1\\ & \iddots & \\ 1 & &\end{pmatrix}.\]
    \item For any rigid analytic space $\calX$ over $\Spa(\Q_p, \Z_p)$, we write $\calX(\overline{\Q}_p) = \bigcup_{[F:\Q_p]<\infty} \calX(F)$. 
\end{itemize}

\subsection*{Acknowledgement} 
It will be clear to the readers that this paper is highly inspired by the work of Joël Bellaïche and Gaëtan Chenevier. We thank Raúl Alonso Rodríguez, Daniel Barrera Salazar, Adel Betina, Kâz{\i}m Büyükboduk, Antonio Cauchi, Andrew Graham, Peter Neamti, Giovanni Rosso, and Chris Williams for helpful discussions and interesting conversations regarding this work. We especially thank Adel Betina and Giovanni Rosso for suggesting us to look at Saito--Kurokawa lifts of infinite-slope forms and for follow-up discussions. We also thank Mohamed Moakher for answering our questions regarding his work on symplectic determinants; and Charlotte Clare-Hunt and James Newton for pointing out mistakes in an early version. 

This work has been supported by the International Centre for Mathematical Sciences (ICMS), Edinburgh, under the Research in Groups programme. We thank the hospitality of the institute during our stay in Edinburgh in May 2025. We also thank ICMS and the organisers of the conference \emph{p-adic Families of Automorphic Forms: Theories and Applications} since the initial ideas of this project grew out when all of us were participating at that conference. Many interesting conversations we had were during the conference \emph{Iwasawa 2025}; we thus also thank the organisers and National Center for Theoretical Sciences (NCTS), Taipei, for such a nice event. 

While working on this project, M.M. was supported by the NSERC-FRQ NOVA program under the Grant titled NSERC Alliance Grants ALLRP 577144 - 22; F.E.T. was supported by the University of Nottingham School of Mathematical Sciences studentship; J.-F.W. was supported by Taighde \'{E}ireann -- Research Ireland under Grant number IRCLA/2023/849 (HighCritical).

\section{\texorpdfstring{$(\varphi, \Gamma)$}{(phi, Gamma)}-modules with \texorpdfstring{$G$}{G}-structure}\label{section: (phi, Gamma)}

In this section, we introduce the notion of \emph{$(\varphi, \Gamma)$-modules with $G$-structure}. Our approach is highly inspired by the geometric picture of perfect Robba rings in \cite[Lecture 12]{Scholze-Weinstein-Berkeley} and the Tannakian formalism \cite[Appendix to Lecture 19]{Scholze-Weinstein-Berkeley} and \cite[Appendix A]{Bellovin-GDeform}. 

Throughout this section, we fix an unramified finite field extension $F$ over $\Q_p$\footnote{ We choose to present the theory only for unramified field extensions over $\Q_p$. It is possible to have a more general theory. However, in our later applications, we will only look at the situation where the finite field extension of $\Q_p$ is unramified. We thus leave the more general case for interested readers. }; denote its ring of integers by $\calO_F$ and residue field by $\F_q$, where $q$ is the cardinality of the residue field. We also fix a compatible system of primitive $p$-power roots of unity $\zeta_{p^n}\in \overline{\Q}_p$ and let \[
    F_n \coloneq F(\zeta_{p^n}), \quad F_{\infty} \coloneq \bigcup_n F_n, \quad \text{ and } \quad F^{\cyc} \coloneq \text{$p$-adic completion of $F_{\infty}$}.
\]
We write $H = H_F = \Gal(\overline{\Q}_p/F_{\infty})$ and $\Gamma = \Gamma_F = \Gal_{F}/H_F$. Note that $\Gamma$ acts on $F_{\infty}$ continuously and so the action extends to $F^{\cyc}$.

\subsection{Robba rings and \texorpdfstring{$(\varphi, \Gamma)$}{(phi, Gamma)}-modules }\label{subsection: p-adic Hodge preparation}

Consider the following rings (see \cite{Fontaine-pRep1}): \[
    \begin{array}{cll}
        \calO_{\C_p}^{\flat} \coloneq \varprojlim_{x\mapsto x^p}\calO_{\C_p}, & A_{\inf} \coloneq W(\calO_{\C_p}^{\flat}), & B_{\inf} \coloneq A_{\inf}[1/p]\\
        \C_p^{\flat} \coloneq \varprojlim_{x\mapsto x^p} \C_p, & \widetilde{A}_{\inf} \coloneq W(\C_p^{\flat}), & \widetilde{B}_{\inf} \coloneq \widetilde{A}_{\inf}[1/p],
    \end{array}
\] 
where $W(-)$ is the ring of Witt vectors. We denote by $[-]: \calO_{\C_p}^{\flat} \rightarrow A_{\inf}$ the Teichmüller character, which is a multiplicative map. Note that the fixed compatible system of primitive $p$-power roots of unity $(\zeta_{p^n})_n$ defines an element $\bfepsilon\in \calO_{\C_p}^{\flat}$. We set $\bfpi:= [\bfepsilon]-1$. 

Let $A_{F} := \calO_{F}[\![\bfpi ]\!] \subset A_{\inf}$, equipped with the $(p, \bfpi)$-adic topology. We also consider \[
    \widetilde{A}_{F} \coloneq A_{F}[\bfpi^{-1}]^{\wedge}_p \quad \text{ and }\quad \widetilde{B}_{F} = \widetilde{A}_{F}[1/p],
\] 
where $\bullet^{\wedge}_p$ stands for the $p$-adic completion. These rings receive the following actions: \begin{itemize}
    \item the $\varphi$-action given by $\varphi \cdot \bfpi = (\bfpi +1)^p -1$; 
    \item the $\Gal_{\Q_p}$-action given by $\sigma \cdot \bfpi = (\bfpi +1)^{\chi_{\cyc}(\sigma)}-1$,
\end{itemize}
where $\chi_{\cyc}$ stands for the $p$-adic cyclotomic character.

By construction, one sees that both $\widetilde{B}_{F}\hookrightarrow \widetilde{B}_{\inf}$ is a field extension. We then consider \[
    \widetilde{B}^{\unr} := \text{ $p$-adic completion of the maximal unramified extension of $\widetilde{B}_{F}$ in $\widetilde{B}_{\inf}$}
\]
and set \[
    \widetilde{A}^{\unr} \coloneq \widetilde{B}^{\unr} \cap \widetilde{A}_{\inf} \quad \text{ and }\quad A^{\unr} = \widetilde{B}^{\unr} \cap A_{\inf}.
\]
According to \cite[Proposition 1.1]{Berger-pDifferential}, we have $\Aut(\widetilde{B}^{\unr}/\widetilde{B}_{F}) = H$ and so we have the identifications\[
    A_{F} = A^{\unr, H}, \quad \widetilde{A}_{F} = \widetilde{A}^{\unr, H}, \quad \text{ and }\quad \widetilde{B}_{F} = \widetilde{B}^{\unr, H}.
\]
We summarise the above rings in the following commutative diagram \[
    \begin{tikzcd}
        A_{\inf} \arrow[r] & \widetilde{A}_{\inf} \arrow[r] & \widetilde{B}_{\inf}\\
        A^{\unr} \arrow[r]\arrow[u] & \widetilde{A}^{\unr}\arrow[r]\arrow[u] & \widetilde{B}^{\unr}\arrow[u]\\
        A_{F}\arrow[r]\arrow[u] & \widetilde{A}_{F}\arrow[r]\arrow[u] & \widetilde{B}_{F}\arrow[u]
    \end{tikzcd}.
\]
Each arrow in this diagram is an inclusion.

Consider the adic space $\Spa(A_F, A_F)$.\footnote{ Since $\bfpi$ is transcendental over $\calO_F$, we have an isomorphism of topological rings \[
    \calO_F[\![T]\!] \rightarrow \calO_F[\![\bfpi]\!]= A_F, \quad T \mapsto \bfpi,
\] where $\calO_F[\![T]\!]$ is equipped with a $(p, T)$-adic topology; it is well-known that $\Spa(\calO_F[\![T]\!], \calO_F[\![T]\!])$ is an adic space (\cite[Theorem 2.2]{Huber-1994}).} It has a unique non-analytic point $x_{\F_q}$ given by $A_F \rightarrow \F_q$; we consider the analytic adic space \[
    \calY = \calY_F \coloneq \Spa(A_F, A_F) \smallsetminus \{x_{\F_q}\}.
\]

Similarly as in \cite[Proposition 4.2.6]{Scholze-Weinstein-Berkeley}, we have a surjective continuous map \[
    \vartheta: \calY \rightarrow [0,\infty], \quad x \mapsto \frac{\log |\bfpi(\widetilde{x})|}{\log |p(\widetilde{x})|},\footnote{The similar map in \cite{Scholze-Weinstein-Berkeley} is denoted by $\kappa$.}
\]
where $\widetilde{x}$ is the rank-1 generalisation of $x$ (\cite[Lemma 4.2.2]{Scholze-Weinstein-Berkeley}). One observes that \begin{equation}\label{eq: vartheta and varphi}
    \vartheta \circ \varphi = p \vartheta.
\end{equation} We can depict the discussion above as in Figure \ref{Fig: AF; unramified}. A similar figure for $\Spa(A_{\inf}, A_{\inf})$ can be found in \cite[Figure 12.1]{Scholze-Weinstein-Berkeley}.

\begin{figure}[ht]
\centering
\[
\begin{tikzpicture}
    \draw[thick, ->] (0,0) -- (0,4.5);
    \draw[thick, ->] (0,0) -- (4.5,0);
    \fill[black] (0,0) circle (0.07cm) node[anchor=north] {$x_{\F_q}$};
    \draw[dashed] (0.3, 0) arc (0:90:0.3);
    \shade[inner color=blue,outer color=white] (3.5,0) circle (0.5cm);
    \draw (3.5, -0.5) node[anchor=north] {$x_{\F_q(\!(\overline{\bfpi})\!)}$};
    \shade[inner color=red, outer color = white] (0, 3.5) circle (0.5cm);
    \draw (-1.3, 3.5) node[anchor=west] {$x_{F}$};
    \draw[thick] (6,0) arc (0:90:6);
    \draw[thick] (5.8,0) -- (6.2,0) node[anchor = north]{$0$};
    \draw[thick] (0,5.8) -- (0,6.2) node[anchor = east]{$\infty$};
    \draw[thick, ->] (3.3,3.3) -- (4, 4);
    \draw (3.5, 3.75) node{$\vartheta$};
    \draw[thick, ->] (3.46, 2) arc (30:60:4);
    \draw (3, 3) node{$\varphi$};
    \draw (8, 0) node {$\Spa(A_{F}, A_{F})$};
\end{tikzpicture}
\]
\caption{A depiction of $\Spa(A_{F}, A_{F})$. The point $x_F$ (resp., $x_{\F_q(\!(\overline{\bfpi})\!)}$) corresponds to $A_F \rightarrow \calO_{F} \rightarrow F$ (resp., $A_F \xrightarrow{\mod \varpi} \F_q[\![\overline{\bfpi}]\!] \rightarrow \F_q(\!(\overline{\bfpi})\!)$, where $\varpi$ is a uniformiser of $\calO_{F}$ and $\overline{\bfpi}$ is the image of $\bfpi$ in $\C_p^{\flat}$). }
    \label{Fig: AF; unramified}
\end{figure}

For any $[a, b]\subset (0, \infty)$, we define \[
    \calY_{[a, b]} \coloneq \text{the interior of the preimage }\vartheta^{-1}[a,b].
\]
For any reduced affinoid algebra $A$ over $\Q_p$ (à la Tate), denote by $A^\circ$ the subring of power-bounded elements. Let \[
    \calY_{[a, b], A} \coloneq \calY_{[a,b]} \times_{\Spa(\Q_p, \Z_p)} \Spa(A, A^{\circ}).
\]
Note that this fibred product exists since both $\calY_{[a, b]}$ and $\Spa(A, A^{\circ})$ are analytic adic spaces. We then define \[
    R_{[a, b], A} \coloneq H^0(\calY_{[a, b], A}, \scrO_{\calY_{[a, b], A}}) \quad \text{ and }\quad R_{[a, b], A}^+ \coloneq H^0(\calY_{[a, b], A}, \scrO_{\calY_{[a, b], A}}^+).
\]
By the construction, if $[a',b'] \subset [a, b]$, we have the natural restriction map of Huber pairs \begin{equation}\label{eq: restriction map [a,b] to [a',b']}
    (R_{[a, b], A}, R_{[a, b], A}^+) \rightarrow (R_{[a', b'], A}, R_{[a', b'], A}^+).
\end{equation}
We then define \[
    (R_{(0, b], A}, R_{(0, b], A}^+) := \varprojlim_{a\rightarrow 0} (R_{[a, b], A}, R_{[a, b], A}^+) \quad \text{ and }\quad (R_{\rig, A}, R_{\rig, A}^+) := \varinjlim_{b\rightarrow 0} (R_{(0, b], A}, R_{(0, b], A}^+).
\]
Geometrically, if we define $\calY_{(0, b], A} \coloneq \bigcup_{[a, b]\subset (0,b]}\calY_{[a, b],A}$, then $(R_{(0, b], A}, R_{(0, b], A}^+)$ is nothing but the global sections of the structure sheaves on $\calY_{(0, b], A}$.

By extending the $\varphi$-action trivially on $A$, \eqref{eq: vartheta and varphi} yields a morphism of Huber pairs \[
    \varphi: (R_{[a, b], A}, R_{[a, b], A}^+) \rightarrow (R_{[a/p, b/p], A}, R_{[a/p, b/p], A}^+).
\] 
Hence, we have \[
    \varphi: (R_{(0, b], A}, R_{(0, b], A}^+) \rightarrow (R_{(0, b/p], A}, R_{(0, b/p], A}^+)
\]
and $(R_{\rig, A}, R_{\rig, A}^+)$ is stable under $\varphi$.
We also extend the $\Gamma$-action by taking the trivial action on $A$.

\begin{Definition}\label{Definition: (phi, Gamma)-modules}
    Let $A$ be a reduced affinoid algebra over $\Q_p$ (à la Tate). \begin{enumerate}
        \item[(i)] Fix $b\in \Q_{>0}$ and $n\in \Z_{\geq 0}$. A \textbf{$\varphi$-module of rank $n$ over $R_{(0, b], A}$} is a finitely generated projective $R_{(0, b], A}$-module  that satisfies the following properties: \begin{itemize}
            \item it is locally free of rank $n$;
            \item it is equipped with an isomorphism \[
                \varphi_D:  D\otimes_{R_{(0, b], A}, \varphi}R_{(0, b/p], A} \rightarrow D\otimes_{R_{(0, b], A}}R_{(0, b/p], A}.
            \]
        \end{itemize}
        \item[(ii)] Fix $b\in \Q_{>0}$ and $n\in \Z_{\geq 0}$. A \textbf{$(\varphi, \Gamma)$-module of rank $n$ over $R_{(0, b], A}$} is a $\varphi$-module $D$ over $R_{(0, b], A}$ equipped with a semilinear action of $\Gamma$ which commutes with $\varphi_D$. 
        \item[(iii)] A \textbf{$(\varphi, \Gamma)$-module of rank $n$ over $A$} is a finitely generated projective $R_{\rig, A}$-module $D$ such that there exists $b\in \Q_{>0}$, a $(\varphi, \Gamma)$-module $D_{(0, b]}$ of rank $n$ over $(R_{(0, b], A}, R_{(0, b], A}^+)$, and an isomorphism \[
            D \cong D_{(0, b]}\otimes_{R_{(0,b], A}}R_{\rig, A}.
        \]
    \end{enumerate}
    We denote by $\Mod_{R_{(0, b], A}}^{\varphi}$ (resp., $\Mod_{R_{(0, b], A}}^{(\varphi, \Gamma)}$; resp., $\Mod_A^{(\varphi, \Gamma)}$) the category of $\varphi$-modules over $R_{(0, b], A}$ (resp., $(\varphi, \Gamma)$-modules over $R_{(0, b], A}$; resp., $(\varphi, \Gamma)$-modules over $A$).
\end{Definition}

\begin{Remark}\label{Remark: geometric meaning of (phi, Gamma)-modules}
    Geometrically, one thinks of a $\varphi$-module $D$ over $R_{(0, b], A}$ as a vector bundle over $\calY_{(0, b], A}$ equipped with an isomorphism \[
        \varphi_D: \varphi^*D \rightarrow D|_{\calY_{(0, b/p], A}}.
    \]
    Note that $\varphi^* D$ is a vector bundle on $\calY_{(0, b/p], A}$.
\end{Remark}

Let $b,b'\in \Q_{>0}$ with $b\geq b'$. Given a $(\varphi, \Gamma)$-module $D$ over $R_{(0, b], A}$, it gives rise to a $(\varphi, \Gamma)$-module over $R_{(0, b'], A}$ via base change. In other words, we have natural functors \[
    \Mod_{R_{(0, b], A}}^{(\varphi, \Gamma)} \rightarrow \Mod_{R_{(0, b'], A}}^{(\varphi, \Gamma)}.
\] 
Hence, by construction, the following lemma is immediate.

\begin{Lemma}\label{Lemma: colimit of categories}
    Let $A$ be a reduced affinoid algebra over $\Q_p$ (à la Tate). Then, the natural functor \[
        2-\varinjlim_{b\rightarrow 0}\Mod_{R_{(0, b], A}}^{(\varphi, \Gamma)} \rightarrow \Mod_{A}^{(\varphi, \Gamma)}
    \]
    is an equivalence of category. 
\end{Lemma}

The following proposition summarises the relationships between $\Mod_{A}^{(\varphi, \Gamma)}$ and other categories arising from $p$-adic Hodge theory.

\begin{Proposition}\label{Proposition: from (phi, Gamma) to other p-adic Hodge theory modules}
    Let $L$ be a (large enough) finite field extension of $\Q_p$ containing all the embeddings of $F$ in $\C_p$. We have the following $\otimes$-functors: 
    \begin{enumerate}
        \item[(i)] Let $\Rep_{\Gal_F}(L)$ be the category of continuous $\Gal_F$-representations with coefficients in $L$. There exists an exact functor \[
            \D_{\rig}^{\normalfont \dagger}: \Rep_{\Gal_F}(L) \rightarrow \Mod_L^{(\varphi, \Gamma)}.
        \]
        \item[(ii)] Let $\Mod_{F\otimes_{\Q_p}L}^{\Fil^{\underline{\bullet}}}$ be the category of (decreasingly) filtred finite-free $F\otimes_{\Q_p}L$-modules. There exists a functor \[
            \calD_{\dR}: \Mod_{L}^{(\varphi, \Gamma)} \rightarrow \Mod_{F\otimes_{\Q_p}L}^{\Fil^{\bullet}}
        \]
        such that for any $\rho\in \Rep_{\Gal_F}(L)$, $\calD_{\dR}(\D_{\rig}^{\normalfont \dagger}(\rho)) = \D_{\dR}(\rho)$, where $\D_{\dR}$ is Fontaine's de Rham functor. 
        \item[(iii)] Let $\Mod_{F\otimes_{\Q_p}L}^{(\varphi, N)}$ be the category of (decreasingly) filtred $(\varphi, N)$-modules over $F\otimes_{\Q_p}L$. There exists a functor \[
            \calD_{\st}: \Mod_{L}^{(\varphi, \Gamma)} \rightarrow \Mod_{F\otimes_{\Q_p}L}^{(\varphi, N)}
        \]
        such that for any $\rho\in \Rep_{\Gal_F}(L)$, $\calD_{\st}(\D_{\rig}^{\normalfont \dagger}(\rho)) = \D_{\st}(\rho)$, where $\D_{\st}$ is Fontaine's semistable functor.
        \item[(iv)] Let $\Mod_{F\otimes_{\Q_p}L}^{\varphi}$ be the category of $\varphi$-modules over $F\otimes_{\Q_p}L$. There exists a functor \[
            \calD_{\cris}: \Mod_{L}^{(\varphi, \Gamma)} \rightarrow \Mod_{F\otimes_{\Q_p}L}^{\varphi}
        \]
        such that for any $\rho\in \Rep_{\Gal_F}(L)$, $\calD_{\cris}(\D_{\rig}^{\normalfont \dagger}(\rho)) = \D_{\cris}(\rho)$, where $\D_{\cris}$ is Fontaine's crystalline functor.
        \item[(v)] Let $\Mod_{F^{\unr}\otimes_{\Q_p}L}^{(\varphi, N, \Gal_F)}$ be the category of (decreasingly) filtred $(\varphi, N, \Gal_F)$-modules over $F^{\unr}\otimes_{\Q_p}L$, where $F^{\unr}$ is the maximal unramified extension of $F$. There exists a functor \[
            \calD_{\pst}: \Mod_L^{(\varphi, \Gamma)} \rightarrow \Mod_{F^{\unr}\otimes_{\Q_p}L}^{(\varphi, N, \Gal_F)}
        \]
        such that for any $\rho\in \Rep_{\Gal_F}(L)$, $\calD_{\pst}(\D_{\rig}^{\dagger}(\rho)) = \D_{\pst}(\rho)$, where $\D_{\pst}(\rho)$ is the potentially semistable module attached to $\rho$. 
        \item[(vi)] Let $A$ be a reduced affinoid algebra over $\Q_p$ (à la Tate) and  $\Mod_{F^{\cyc}\otimes_{\Q_p}A}$ be the category of finite projective $F^{\cyc}\otimes_{\Q_p}A$-modules. There exists a functor \[
            \calD_{\Sen}: \Mod_A^{(\varphi, \Gamma)} \rightarrow \Mod_{F^{\cyc}\widehat{\otimes}_{\Q_p}A}
        \] such that if $A=L$ and $\rho\in \Rep_{\Gal_F}(L)$, $\calD_{\Sen}(\D_{\rig}^{\dagger}(\rho)) = \D_{\Sen}(\rho)$, where $\D_{\Sen}(\rho)$ is the Sen-module attached to $\rho$.
    \end{enumerate}
\end{Proposition}
\begin{proof}
    The functor $\D_{\rig}^{\dagger}$ is a fully faithful embedding, defining a $\otimes$-equivalence between $\Rep_{\Gal_F}(L)$ and the étale $(\varphi, \Gamma)$-modules over $L$ (\cite[Proposition 1.7]{Colmez-trianguline}, see also \cite[Proposition 2.2.6]{BC}). We remark that, by \cite[Proposition 3.17]{Berger-pDifferential}, there exists a topological $\Q_p$-algebra $R$ (denoted by $B_{\rig}^{\dagger}$ therein) such that $\D_{\rig}^{\dagger}$ can be defined as in Fontaine's style, \emph{i.e.}, for any $\rho \in \Rep_{\Gal_F}(L)$, \[
        \D_{\rig}^{\dagger}(\rho) = (\rho \otimes R)^{H_F}.
    \] 
    
    For the constructions of $\calD_{\dR}$ and $\calD_{\cris}$, we refer the readers to \cite[Sect. 2.2.7]{BC}; for the construction of $\calD_{\st}$ and $\calD_{\pst}$, we refer the readers to \cite[Sect. 1.2.3 \&  1.2.5]{Benois-LInvariant}; for the construction of $\calD_{\Sen}$, we refer the readers to \cite[Sect. 1.3]{Liu-triangulation}. We remark that, in \cite{Liu-triangulation}, the $\otimes$-functor $\calD_{\Sen}$ has target in $\Mod_{F_{\infty}\otimes_{\Q_p}A}$; we base change to $F^{\cyc}\widehat{\otimes}_{\Q_p}A$ (and extend the $\Gamma$-action continuously) for our later purpose (see Lemma \ref{Lemma: HTS weight is continuous} below). 
\end{proof}

\subsection{Review of \texorpdfstring{$G$}{G}-torsors}\label{subsection: review of G-torsors}
Let $G$ be a smooth quasi-split reductive linear algebraic group over $F$. We denote by $G^{\an}$ its analytification over $\Spa(F, \calO_F)$. We denote by $\Rep_{G}$ the category of algebraic $G$-representations over $F$. We have the following definitions of $G$-torsors.

\begin{Definition}\label{Definition: G-torsors}
    Let $\calX$ be a rigid analytic space over $\Spa(F, \calO_F)$.
    \begin{enumerate}
        \item[(i)] A \textbf{geometric $G$-torsor} over $\calX$ is a rigid analytic space $\calG \rightarrow \calX$ such that, étale locally on $\calX$, there is a $G^{\an}$-equivariant isomorphism $\calG \cong G^{\an}\times \calX$.
        \item[(ii)] A \textbf{cohomological $G$-torsor} over $\calX$ is an étale sheaf $\scrG$ on $\calX_{\et}$ with an action of $G^{\an}$ such that, étale locally on $\calX$, there is a $G^{\an}$-equivariant isomorphism $\scrG \cong G^{\an}$.
        \item[(iii)] A \textbf{Tannakian $G$-torsor} over $\calX$ is an exact $\otimes$-functor $\frakG: \Rep_G \rightarrow \VB_{\calX}$, where $\VB_{\calX}$ is the category of vector bundles on $\calX$.
    \end{enumerate}
\end{Definition}

\begin{Theorem}\label{Theorem: comparison of G-torsors}
    Keep the notations as in Definition \ref{Definition: G-torsors}. The three definitions of $G$-torsors therein are equivalent.  
\end{Theorem}
\begin{proof}
    The proof goes verbatim as in \cite[Theorem 19.5.2]{Scholze-Weinstein-Berkeley}. We leave the details to the reader. 
\end{proof}

Thanks to Theorem \ref{Theorem: comparison of G-torsors}, we shall now jump between the different definitions of $G$-torsors freely and use the same notation without distinguishing whichever viewpoint we are in.

The following proposition is similar to \cite[Proposition 2.8]{Deligne-Milne}.

\begin{Proposition}\label{Proposition: automorphisms get back to the algebraic group}
    Let $\calX$ be a rigid analytic space over $\Spa(F, \calO_F)$ and let $\mathsf{\acute{E}t}_{\calX}$ be the big étale site over $\calX$. Let $\calG$ be a (Tannakian) $G$-torsor. Define a presheaf $\underline{\Aut}^{\otimes}(\calG)$ on $\mathsf{\acute{E}t}_{\calX}$ as follows: For any morphism $f: \calZ \rightarrow \calX$, \[
        \underline{\Aut}^{\otimes}(\calG)(\calZ) = \Aut^{\otimes}(f^{-1}\circ \calG),
    \]
    where $\Aut^{\otimes}$ stands for automorphisms that preserve tensor products. 
    Then, $\underline{\Aut}^{\otimes}(\calG)$ is representable by a rigid analytic space which is a form of $G^{\an}_{\calX}$.
\end{Proposition}
\begin{proof}
    We first consider the case when $\calX = \Spa(A, A^+)$ is an affinoid. Similarly as remarked in the proof of \cite[Theorem 19.5.2]{Scholze-Weinstein-Berkeley}, the datum for a Tannakian $G$-torsor is equivalent to a Tannakian $G$-torsor over $\Spec A$. 
    Therefore, it is a consequence of \cite[Theorem 3.2]{Deligne-Milne} that $\underline{\Aut}^{\otimes}(\calG)$ is a form of $G^{\an}_{\calX}$ (\cite[Corollary A.10]{Bellovin-GDeform}).

    Since the construction and the (algebraic) Tannakian formalism is functorial, the desired result follows from glueing. 
\end{proof}

Our goal in the next subsection is to apply the discussions above to $p$-adic Hodge theory. To this end, we need more structure on $G$-torsors. Materials in what follows are inspired by \cite[Sect. A]{Bellovin-GDeform}.

Let $\calX$ be a rigid analytic space over $\Spa(F, \calO_F)$. Consider the following three categories:\begin{align*}
    \VB_{\calX}^{\Fil^{\bullet}} & := \text{ the category of vector bundles over $\calX$ equipped with decreasing filtrations,}\\
    \VB_{\calX}^{\Fil_{\bullet}} & := \text{ the category of vector bundles over $\calX$ equipped with increasing filtrations,}\\
    \VB_{\calX}^{\Gr} & := \text{ the category of graded vector bundles over $\calX$.}
\end{align*}
Let us describe the relationships between these categories: \begin{itemize}
    \item For $?\in \{\Fil^{\bullet}, \Fil_{\bullet}, \Gr\}$, there is a forgetful functor \begin{equation}\label{eq: forgetful functor for vector bundles}
        \VB_{\calX}^{?} \rightarrow \VB_{\calX}.
    \end{equation}
    \item Given a graded vector bundle, one can obtain a filtred vector bundle by \begin{equation}\label{eq: gradings to filtrations}
        \begin{array}{ll}
            \VB_{\calX}^{\Gr} \rightarrow \VB_{\calX}^{\Fil^{\bullet}}, & \scrM = \bigoplus_{i\in\Z} \scrM_i \mapsto \left((\scrM, \Fil^{\bullet}\scrM) \text{ with }\Fil^i \scrM = \bigoplus_{j\geq i}\scrM_j\right), \\
            \VB_{\calX}^{\Gr} \rightarrow \VB_{\calX}^{\Fil_{\bullet}}, & \scrM = \bigoplus_{i\in\Z} \scrM_i \mapsto \left((\scrM, \Fil_{\bullet}\scrM) \text{ with }\Fil^i \scrM = \bigoplus_{j\leq i}\scrM_j\right).
        \end{array}
    \end{equation}
    We say a filtration is \emph{splittable} if it arises in this way.
    Note that these two functors give rise to a commutative diagram \[
        \begin{tikzcd}
            \VB_{\calX}^{\Gr} \arrow[r] \arrow[d]\arrow[rd]  & \VB_{\calX}^{\Fil^{\bullet}}\arrow[d]\\
            \VB_{\calX}^{\Fil_{\bullet}} \arrow[r] & \VB_{\calX}
        \end{tikzcd},
    \]
    where the functors $\VB_{\calX}^? \rightarrow \VB_{\calX}$ are the forgetful functors as in \eqref{eq: forgetful functor for vector bundles}.
    \item There is a functor \[
        \VB_{\calX}^{\Fil^{\bullet}} \rightarrow \VB_{\calX}^{\Gr}, \quad (\scrM, \Fil^{\bullet}\scrM) \mapsto \bigoplus_{i\in \Z} \Gr^i \scrM.
    \]
\end{itemize}

\begin{Definition}\label{Definition: structrues on G-torsors}
    Let $\calX$ be a rigid analytic space. \begin{enumerate}
        \item[(i)] An \textbf{decreasingly filtred $G$-torsor} over $\calX$ is a $G$-torsor $\calG$ equipped with a factorisation \[
            \begin{tikzcd}
                \Rep_G \arrow[rr, "\calG"]\arrow[rd,"\Fil^{\bullet}\calG"'] && \VB_{\calX}\\
                & \VB_{\calX}^{\Fil^{\bullet}}\arrow[ru, "\text{forgetful}"']
            \end{tikzcd}.
        \] 
        For objects $\calG(\sigma)$ in the essential image of $\calG$, we write $\Fil^{\bullet}\calG(\sigma)$ for the corresponding filtration. 
        \item[(ii)] A \textbf{increasingly filtred $G$-torsor} over $\calX$ is a $G$-torsor $\calG$ equipped with a factorisation \[
            \begin{tikzcd}
                \Rep_G \arrow[rr, "\calG"]\arrow[rd, "\Fil_{\bullet}\calG"'] && \VB_{\calX}\\
                & \VB_{\calX}^{\Fil_{\bullet}}\arrow[ru, "\text{forgetful}"']
            \end{tikzcd}.
        \] 
        For objects $\calG(\sigma)$ in the essential image of $\calG$, we write $\Fil_{\bullet}\calG(\sigma)$ for the corresponding filtration.
        \item[(i)] A \textbf{graded $G$-torsor} over $\calX$ is a $G$-torsor $\calG$ equipped with a factorisation \[
            \begin{tikzcd}
                \Rep_G \arrow[rr, "\calG"]\arrow[rd, "\calG_{\bullet}"'] && \VB_{\calX}\\
                & \VB_{\calX}^{\Gr}\arrow[ru, "\text{forgetful}"']
            \end{tikzcd}.
        \] 
        For objects $\calG(\sigma)$ in the essential image of $\calG$, we write $\calG(\sigma)_{\bullet}$ for the corresponding grading.
    \end{enumerate}
\end{Definition}

Given a decreasing filtred $G$-torsor $\calG$ with filtration $\Fil^{\bullet}\calG$ over $\calX$, we define the following sub-presheaves of $\underline{\Aut}^{\otimes}(\calG)$: \begin{align*}
    \calP_{\Fil^{\bullet}} : \mathsf{\acute{E}t}_{\calX} \rightarrow & \Set\\
    (f: \calZ \rightarrow \calX) \mapsto & \left\{ \alpha\in \Aut^{\otimes}(f^{-1} \calG): \alpha\left( \Fil^n f^{-1}\calG(V)\right) \subset \Fil^nf^{-1}\calG(V) \text{ for all }V\in \Rep_G \text{ and }n\in \Z \right\},\\ \\
    \calN_{\Fil^{\bullet}}: \mathsf{\acute{E}t}_{\calX} \rightarrow & \Set\\ 
    (f: \calZ \rightarrow \calX) \mapsto & \left\{ \alpha\in \Aut^{\otimes}(f^{-1} \calG): (\alpha-\id)\left( \Fil^n f^{-1}\calG(V)\right) \subset \Fil^{n+1}f^{-1}\calG(V) \text{ for all }V\in \Rep_G \text{ and }n\in \Z \right\}.
\end{align*}
Similar construction applies to an increasing $G$-torsor $\calG$ with filtration $\Fil_{\bullet}\calG$. One then obtain the sub-preseaves $\calP_{\Fil_{\bullet}}$ and $\calN_{\Fil_{\bullet}}$.\footnote{ To define $\calN_{\Fil_{\bullet}}$, one requires $(\alpha - \id)(\Fil_n f^{-1}\calG(V)) \subset \Fil_{n-1}f^{-1}\calG(V)$.}

\begin{Proposition}\label{Proposition: parabolics and unipotent radical from filtrations}
    Keep the notations a above. \begin{enumerate}
        \item[(i)]  The sub-presheaves $\calP_{\Fil^{\bullet}}$ and $\calN_{\Fil^{\bullet}}$ are representable by closed adic subspaces of $\underline{\Aut}^{\otimes}(\calG)$. Moreover, $\calP_{\Fil^{\bullet}}$ is a parabolic subgroup of $\underline{\Aut}^{\otimes}(\calG)$ while $\calN_{\Fil^{\bullet}}$ is its unipotent radical. 
        \item[(ii)] The sub-presheaves $\calP_{\Fil_{\bullet}}$ and $\calN_{\Fil_{\bullet}}$ are representable by closed adic subspaces of $\underline{\Aut}^{\otimes}(\calG)$. Moreover, $\calP_{\Fil_{\bullet}}$ is a parabolic subgroup of $\underline{\Aut}^{\otimes}(\calG)$ while $\calN_{\Fil_{\bullet}}$ is its unipotent radical. 
    \end{enumerate}
\end{Proposition}
\begin{proof}
    As remarked in the proof of Proposition \ref{Proposition: automorphisms get back to the algebraic group}, the datum of a Tannakian $G$-torsor is affinoid locally equivalent to an (algebraic) Tannakian $G$-torsor. Therefore, the assertions follows from \cite[Chapter IV, 2.2.5]{SR-Tannakian} and analytification. Since the construction is functorial, we conclude by glueing. 
\end{proof}

\subsection{\texorpdfstring{$(\varphi, \Gamma)$}{(phi, Gamma)}-modules with \texorpdfstring{$G$}{G}-structure}\label{subsection: (phi, Gamma)-modules with G-structure}
We now apply the discussions in Sect. \ref{subsection: review of G-torsors} to introduce and study \emph{$(\varphi, \Gamma)$-modules with $G$-structures}. Throughout this subsection, we will further assume $G$ is split over $F$. 

\begin{Definition}\label{Definition: (phi, Gamma)-modules with G-structure}
    Let $A$ be a reduced affinoid algebra over $\Q_p$ (à la Tate). \begin{enumerate}
        \item[(i)] A (Tannakian) \textbf{$(\varphi, \Gamma)$-module with $G$-structure} over $\calY_{(0, b], A}$ is an exact $\otimes$-functor  \[
            D_{(0, b]}: \Rep_{G} \rightarrow \Mod_{\calY_{(0,b], A}}^{(\varphi, \Gamma)}.
        \]
        We denote by $G\Mod_{\calY_{(0, b], A}}^{(\varphi, \Gamma)}$ the category of $(\varphi, \Gamma)$-module with $G$-structure over $\calY_{(0, b], A}$.
        \item[(ii)] Define the \textbf{category of $(\varphi, \Gamma)$-modules with $G$-structure} to be \[
            G\Mod_{A}^{(\varphi, \Gamma)} := 2-\varinjlim_{b\rightarrow 0}G\Mod_{\calY_{(0, b], A}}^{(\varphi, \Gamma)}.
        \]
        By a \textbf{$(\varphi, \Gamma)$-module with $G$-structure} over $A$, we then mean an object in $G\Mod_{A}^{(\varphi, \Gamma)}$.
    \end{enumerate}
\end{Definition}

\begin{Remark}\label{Remark: Y_A is rigid analytic}
    For $0<a\leq b<\infty$, we view $\calY_{[a, b], A}$ as a rigid analytic space. Indeed, recall the isomorphism $\calO_F[\![T]\!] \cong \calO_F[\![\bfpi]\!]$, thus $\calY_{[a,b]}$ is the rational subset $\{|\bfpi|\leq |p|^a, |p|^b\leq |\bfpi|\}$ in the open unit ball (over $F$). More precisely, $\calY_{[a, b]}$ is the affinoid given by 
    \[
        \calY_{[a, b]} = \Spa\left( \calO_F[\![\bfpi]\!]\left\langle \frac{\bfpi^{1/a}}{p}, \frac{p}{\bfpi^{1/b}} \right\rangle[1/p],  \calO_F[\![\bfpi]\!]\left\langle \frac{\bfpi^{1/a}}{p}, \frac{p}{\bfpi^{1/b}} \right\rangle \right). 
    \]
    Since fibre products exist in the category of rigid analytic spaces, we conclude the claim. 
\end{Remark}

\begin{Remark}\label{Remark: 2-limit in the definition of (phi, Gamma)-mod with G-structure}
    Let us briefly explain the $2$-colimit in Definition \ref{Definition: (phi, Gamma)-modules with G-structure}. Given $b, b'\in \Q_{>0}$ with $b\geq b'$, recall we have a natural functor \[
        \Mod_{\calY_{(0, b], A}}^{(\varphi, \Gamma)} \rightarrow \Mod_{\calY_{(0, b'], A}}^{(\varphi, \Gamma)}
    \]
    given by base change. Hence, given a $(\varphi, \Gamma)$-module with $G$-structure $D_{(0, b]}$ over $\calY_{(0, b], A}$ one obtains a $(\varphi, \Gamma)$-module with $G$-structure over $\calY_{(0, b'],A}$ by the composition of functors \[
        \Rep_{G} \xrightarrow{D_{(0, b]}} \Mod_{\calY_{(0, b], A}}^{(\varphi, \Gamma)} \rightarrow \Mod_{\calY_{(0, b'], A}}^{(\varphi, \Gamma)}.
    \]
    Consequently, the category $G\Mod_{A}^{(\varphi, \Gamma)}$ is equivalent to the category of exact $\otimes$-functors \[
        D: \Rep_{G} \rightarrow \Mod_{A}^{(\varphi, \Gamma)}
    \]
    that factor as \[
        \begin{tikzcd}
            \Rep_G \arrow[rr, "D"] \arrow[rd, "D_{(0, b]}"'] && \Mod_{A}^{(\varphi, \Gamma)}\\
            & \Mod_{\calY_{(0, b], A}}^{(\varphi, \Gamma)}\arrow[ru]
        \end{tikzcd}
    \]
    for some $b>0$ and $D_{(0, b]}$.
\end{Remark}

An immediate question is the relationship between $G$-valued Galois representations and $(\varphi, \Gamma)$-modules with $G$-structure. To address this question, we begin with the following lemma. 

\begin{Lemma}\label{Lemma: Tannakian viewpoint of Galois representations}
    Let $L$ be a (large enough) finite field extension of $\Q_p$. The category of $G$-valued $\Gal_{F}$-representation with coefficients in $L$ is equivalent to the category of exact $\otimes$-functors \[
        \Rep_G \xrightarrow{\frakG} \Rep_{\Gal_{F}}(L)
    \] such that the diagram \[
        \begin{tikzcd}
            \Rep_G \arrow[rr, "(\sigma: G(L)\rightarrow \GL(V)\mapsto V"]\arrow[rd, "\frakG"'] && \mathsf{Vect}_{L}\\
            & \Rep_{\Gal_{F}}\arrow[ru, "\text{forgetful}"']
        \end{tikzcd}
    \] is commutative. 
\end{Lemma}
\begin{proof}
    Given a $G$-valued $\Gal_{F}$-representation with coefficients in $L$ \[
        \rho: \Gal_{F} \rightarrow G(L),
    \]
    we construct a $\otimes$-functor $\frakG_{\rho}$ by \[
        \frakG_{\rho}: \Rep_G \rightarrow \Rep_{\Gal_{F}}, \quad (\sigma: G(L) \rightarrow \GL(V)) \mapsto (\Gal_{F} \xrightarrow{\rho} G(L)\xrightarrow{\sigma} \GL(V)).
    \]
    The converse is exactly \cite[Sect. A.2.6]{Bellovin-GDeform}.
\end{proof}

\begin{Corollary}\label{Corollary: from Galois to (phi, Gamma)}
    Let $L$ be a (large enough) finite field extension of $\Q_p$. Let $\Rep_{\Gal_{F}}^G(L)$ be the category of $G$-valued $\Gal_{F}$-representation with coefficients in $L$. Then, there exists an exact $\otimes$-functor \[
        \D_{\rig}^{\dagger} : \Rep_{\Gal_{F}}^G(L) \rightarrow G\Mod_{L}^{(\varphi, \Gamma)}.
    \]
\end{Corollary}
\begin{proof}
    By \cite[Proposition 3.17]{Berger-pDifferential}, there is an exact $\otimes$-functor \[
        \D_{\rig}^{\dagger}: \Rep_{\Gal_{F}}(L) \rightarrow \Mod_{L}^{(\varphi, \Gamma)}.
    \]
    Now, by Lemma \ref{Lemma: Tannakian viewpoint of Galois representations}, we view a $G$-valued $\Gal_{F}$-representation $\rho$ with coefficients in $L$ as a exact $\otimes$-functor \[
        \rho: \Rep_{G} \rightarrow \Rep_{\Gal_{\Q_p}}(L).
    \]
    By composing with $\D_{\rig}^{\dagger}$, one obtains an exact $\otimes$-functor \[
        \D_{\rig}^{\dagger}(\rho): \Rep_{G}\xrightarrow{\rho} \Rep_{\Gal_{F}}\xrightarrow{\D_{\rig}^{\dagger}} \Mod_{L}^{(\varphi, \Gamma)},
    \]
    which is as desired. 
\end{proof}

Applying similar idea as before, we define the following $p$-adic-Hodge-theoretic modules attached to $(\varphi, \Gamma)$-modules with $G$-structure. 

\begin{Definition}\label{Definition: p-adic Hodge theory modules}
    Let $L$ be a (large enough) finite field extension of $\Q_p$. Let $D\in G\Mod_L^{(\varphi, \Gamma)}$. 
    \begin{enumerate}
        \item[(i)] The \textbf{de Rham module} attached to $D$ is the tensor functor \[
            \calD_{\dR}(D): \Rep_G \xrightarrow{D} \Mod_L^{(\varphi, \Gamma)} \xrightarrow{\calD_{\dR}} \Mod_{L\otimes_{\Q_p}F}^{\Fil^{\bullet}}.
        \]
        We say $D$ is \textbf{de Rham} if for any $\sigma\in \Rep_G$, $D(\sigma)$ is de Rham, i.e., $\rank_{R_{\rig, L}} D(\sigma) = \rank_{L\otimes_{\Q_p}F} \calD_{\dR}(D)(\sigma)$. 
        \item[(ii)] The \textbf{crystalline module} attached to $D$ is the tensor functor \[
            \calD_{\cris}(D): \Rep_G \xrightarrow{D} \Mod_L^{(\varphi, \Gamma)} \xrightarrow{\calD_{\cris}} \Mod_{L\otimes_{\Q_p}F}^{\varphi}.
        \]
        We say $D$ is \textbf{crystalline} if for any $\sigma\in \Rep_G$, $D(\sigma)$ is crystalline, i.e., $\rank_{R_{\rig, L}} D(\sigma) = \rank_{L\otimes_{\Q_p}F} \calD_{\cris}(D)(\sigma)$. 
        \item[(iii)] The \textbf{semistable module} attached to $D$ is the  tensor functor \[
            \calD_{\st}(D): \Rep_G \xrightarrow{D} \Mod_L^{(\varphi, \Gamma)} \xrightarrow{\calD_{\st}} \Mod_{L\otimes_{\Q_p}F}^{(\varphi, N)}.
        \]
        We say $D$ is \textbf{semistable} if for any $\sigma\in \Rep_G$, $D(\sigma)$ is semistable, i.e., $\rank_{R_{\rig, L}} D(\sigma) = \rank_{L\otimes_{\Q_p}F} \calD_{\st}(D)(\sigma)$.
        \item[(iv)] The \textbf{potentially semistable module} attached to $D$ is the tensor functor \[
            \calD_{\pst}(D): \Rep_G \xrightarrow{D} \Mod_L^{(\varphi, \Gamma)} \xrightarrow{\calD_{\pst}} \Mod_{L\otimes_{\Q_p}F^{\unr}}^{(\varphi, N, \Gal_F)}.
        \]
        We say $D$ is \textbf{potentially semistable} if for any $\sigma\in \Rep_G$, $D(\sigma)$ is potentially semistable, i.e., $\rank_{R_{\rig, L}} D(\sigma) = \rank_{L\otimes_{\Q_p}F^{\unr}} \calD_{\pst}(D)(\sigma)$.
        \item[(v)] Let $A$ be a reduced affinoid algebra over $\Q_p$ (à la Tate) and $D\in G\Mod_{A}^{(\varphi, \Gamma)}$. The \textbf{Sen module} attached to $D$ is the tensor functor \[
            \calD_{\Sen}(D): \Rep_G \xrightarrow{D} \Mod_A^{(\varphi, \Gamma)} \xrightarrow{\calD_{\Sen}} \Mod_{F^{\cyc}\widehat{\otimes}_{\Q_p}A}.
        \]
    \end{enumerate}
\end{Definition}

\begin{Lemma}\label{Lemma: p-adic properties can be checked on faithful representations}
    Let $L$ be a (large enough) finite field extension of $\Q_p$ and $D\in G\Mod_L^{(\varphi, \Gamma)}$. Then, $D$ is de Rham (resp., crystalline; resp., semistable; resp., potentially semistable) if and only if there exists a faithful $\sigma_0\in \Rep_G$ such that $D(\sigma_0)$ is de Rham (resp., crystalline; resp., semistable; resp., potentially semistable). 
\end{Lemma}
\begin{proof}
    All the cases can be proven similarly, so we only provide a proof for the de Rham case and leave the other cases to the reader. 

    Since $D(\sigma_0)$ is de Rham, $D(\sigma_0^{\vee})$ is also de Rham, as is $D(\sigma)$ for any subquoient of $\sigma_0$. Moreover, if $\sigma, \sigma'\in \Rep_G$ such that $D(\sigma)$ and $D(\sigma')$ are both de Rham, $D(\sigma\otimes\sigma')$ is also de Rham. However, by \cite[Proposition 1.21]{Mile-tannakian}, $\sigma_0$ is a tensor generator, this implies that any $D(\sigma)$ is de Rham for any $\sigma\in \Rep_G$. 
\end{proof}

\begin{Proposition}\label{Proposition: potentially semistable equivalent to de Rham}
    Let $L$ be a (large enough) finite field extension of $\Q_p$ and $D\in G\Mod_L^{(\varphi, \Gamma)}$. Then, $D$ is de Rham if and only if $\D_{\rig}^{\dagger}(\rho)$ is potentially semistable. 
\end{Proposition}
\begin{proof}
    This follows immediately from \cite[Proposition 1.2.6]{Benois-LInvariant} (see also \cite[Théorème 0.6]{Berger-pDifferential}).
\end{proof}

We close this subsection by a discussion on the \emph{Hodge--Tate--Sen weight} attached to a $(\varphi, \Gamma)$-module with $G$-structure. Let $A$ be a reduced affinoid algebra over $\Q_p$ (à la Tate) and $D\in G\Mod_A^{(\varphi, \Gamma)}$. Then, by Proposition \ref{Proposition: automorphisms get back to the algebraic group}, there is a canonical isomorphism \[
    \Aut^{\otimes}(\calD_{\mathrm{Sen}}(D)) \cong G(A \widehat{\otimes}_{\Q_p}F^{\cyc}).
\]
Note that, for any $\sigma\in \Rep_G$, the module $\calD_{\mathrm{Sen}}(D)(\sigma)$ admits a semilinear $\Gamma$-action (see, for example, \cite[\S 2.2.7]{BC}). This induces a group homomorphism \[
    \mu_D: \Gamma \rightarrow G(A\widehat{\otimes}_{\Q_p}F^{\cyc}) \rtimes \Gamma.
\] We shall often denote $\mu_D(\sigma) = ([\gamma], \gamma)$ for any $\gamma\in \Gamma$. 

\begin{Lemma}\label{Lemma: HTS weight is continuous}
    Keep the notations as above. The group homomorphism $\mu_D$ is continuous. 
\end{Lemma}
\begin{proof}
    This follows immediately from the fact that the $\Gamma$-action on $\calD_{\Sen}(D)$ is continuous (see the construction in \cite[Sect. 1.3]{Liu-triangulation}).
\end{proof}

\begin{Definition}\label{Definition: HTS cocharacter}
    Keep the notations as above. The  \textbf{Hodge--Tate--Sen cocharacter} attached to $D$ is defined to be the continuous group homomorphism $\mu_D$. 
\end{Definition}

\begin{Remark}\label{Remark: comparison of HTS weight with classical definition}
    \begin{enumerate}
        \item[(i)] The rather tradition notion of \emph{Hodge--Tate--Sen weights} is defined as follows (ss, for example, \cite[Definition 6.2.11]{KPX}). If $D\in \Mod_A^{(\varphi, \Gamma)}$ of rank $n$, then $\calD_{\Sen}(D)$ is a finite free $A \widehat{\otimes}_{\Q_p}F^{\cyc}$-module with a semilinear $\Gamma$-action. Thus, we have a group homomorphism \[
            \Gamma \rightarrow \GL_n(A \widehat{\otimes}_{\Q_p}F^{\cyc}) \rtimes \Gamma. 
        \]
        Here, we fix an isomorphism $\Aut_{A \otimes_{\Q_p}F^{\cyc}}(\calD_{\Sen}(D)) \cong \GL_n(A \widehat{\otimes}_{\Q_p}F^{\cyc})$.
        Given $\gamma\in \Gamma$, we (temporally) use the notation $([\gamma], \gamma)$ for the image of $\gamma$ in $\GL_n(A \widehat{\otimes}_{\Q_p}F^{\cyc}) \rtimes \Gamma$. Suppose $\gamma$ is close to $1$ (under the $p$-adic topology on $\Gamma$), consider the Sen operator \[
            \Theta_{\Sen}(\gamma) \coloneq \frac{\log_p [\gamma]}{\log_p \chi_{\cyc}(\gamma)} \in \End_{A \widehat{\otimes}_{\Q_p}F^{\cyc}}(\calD_{\Sen}(D)).
        \]
        The Sen polynomial $P_{\Sen}(D)$ attached to $D$ is defined to be the characteristic polynomial of $\Theta_{\Sen}(D)$, which turns out to be independent to the choice of $\sigma$. Since $\Theta_{\Sen}(D)$ is $\Gamma$-equivariant, the coefficients of $P_{\Sen}(D)$ lies in $A$. If $D$ is furthermore trianguline, the roots of $P_{\Sen}(D)$ lies in $A$ (\cite[Proposition 2.3.3]{BC}). Then, the Hodge--Tate--Sen weights of $D$ is defined to be these roots. 
        \item[(ii)] We term $\mu_D$ as \emph{Hodge--Tate--Sen cocharacter} to distinguish from the traditional Hodge--Tate--Sen weights. Evidently, our definition is inspired by the construction of the latter and Buzzard's notes on Hodge--Tate theory (\cite{Buzzard-HTNotes}). We sometimes prefer to work with Hodge--Tate--Sen cocharacters since the $G$-structure is encoded (compared with the traditional notion of Hodge--Tate--Sen weights, one needs to first choose an embedding $G \hookrightarrow \GL_n$). 
    \end{enumerate}
\end{Remark}

\begin{Example}\label{Example: HTS weight for a non-HT representation}
    Consider the Galois representation \[
        \rho: \Gal_{\Q_p} \rightarrow \GL_2(\Q_p), \quad \gamma \mapsto \begin{pmatrix} 1 & \log_p\chi_{\cyc}(\gamma) \\ & 1 \end{pmatrix}
    \]
    and let $D = \mathbf{D}_{\rig}^{\dagger}(\rho)$. It is well-known that $\rho$ is non-Hodge--Tate; it is trianguline with Hodge--Tate--Sen weight $(0,0)$ (see, for example, \cite[Sect. II.1.2]{Berger-Intro}). However, one sees easily that, under Definition \ref{Definition: HTS cocharacter}, its Hodge--Tate--Sen cocharacter is \[
        \mu_D: \Gamma \rightarrow \GL_2(\Q_p^{\cyc}) \rtimes \Gamma, \quad \gamma \mapsto \left( \begin{pmatrix} 1 & \log_p\chi_{\cyc}(\gamma) \\ & 1 \end{pmatrix}, \gamma \right)
    \]
\end{Example}

\subsection{\texorpdfstring{$P$}{P}-flags}\label{subsection: P-flag}

In this subsection, we introduce the notion of \emph{triangulation} -- or more generally, \emph{flags} -- for $(\varphi, \Gamma)$-modules with $G$-structure. To this end, we introduce more notations: \begin{itemize}
    \item Given a $(\varphi, \Gamma)$-module $D_{(0, b]}$ over $\calY_{(0, b], A}$, consider the $G$-torsor \[
        \calG_{D_{(0, b]}}: \Rep_{G} \xrightarrow{D_{(0, b]}} \Mod_{\calY_{(0, b],A}}^{(\varphi, \Gamma)} \xrightarrow{\text{forgetful}} \VB_{\calY_{(0, b], A}}.
    \]
    We fix an isomorphism \begin{equation}\label{eq: fixed isomorphism}
        \underline{\Aut}^{\otimes}(\calG_{(0, b]}) \cong G^{\an}_{\calY_{(0, b], A}}.
    \end{equation}
    \item Given a reduced affinoid algebra $A$ over $\Q_p$ (à la Tate), denote by $\Mod_{\calY_{(0, b], A}}^{(\varphi, \Gamma), \Fil_{\bullet}}$ the category of filtred $(\varphi, \Gamma)$-modules over $\calY_{(0, b], A}$ (with increasing filtrations). That is, objects of  $\Mod_{\calY_{(0, b], A}}^{(\varphi, \Gamma), \Fil_{\bullet}}$ are pairs $(D, \Fil_{\bullet} D)$ such that \begin{itemize}
        \item $D\in \Mod_{\calY_{(0, b], A}}^{(\varphi, \Gamma)}$;
        \item $\Fil_{\bullet} D$ is an increasing filtration on $D$ (as a vector bundle) such that each $\Fil_{n} D\in \Mod_{\calY_{(0, b], A}}^{(\varphi, \Gamma)}$; 
        \item for each $n$, the graded piece $\Gr_n D = \Fil_n D/\Fil_{n-1}D \in \Mod_{\calY_{(0, b], A}}^{(\varphi, \Gamma)}$. 
    \end{itemize}
    Obviously, there are forgetful functors \[
        \Mod_{\calY_{(0, b], A}}^{(\varphi, \Gamma), \Fil_{\bullet}} \rightarrow \VB_{\calY_{(0, b], A}}^{\Fil_{\bullet}} \quad \text{ and }\quad \Mod_{\calY_{(0, b], A}}^{(\varphi, \Gamma), \Fil_{\bullet}} \rightarrow \Mod_{\calY_{(0, b], A}}^{(\varphi, \Gamma)}.
    \]
\end{itemize}

\begin{Definition}\label{Definition: P-flag}
    Let $A$ be a reduced affinoid algebra over $\Q_p$ (à la Tate). Fix a Borel subgroup $B \subset G$ and let $P\subset G$ be a parabolic subgroup containing $B$.
    \begin{enumerate}
        \item[(i)] A \textbf{$P$-flagged $(\varphi, \Gamma)$-module} over $\calY_{(0, b], \calX}$ is a $(\varphi, \Gamma)$-module $D_{(0, b]}$ with $G$-structure over $\calY_{(0, b],A}$ together with a factorisation \[
            \begin{tikzcd}
                \Rep_G \arrow[rr, "D_{(0, b]}"]\arrow[rd, "\Fil_{\bullet} D_{(0, b]}"'] && \Mod_{\calY_{(0, b], A}}^{(\varphi, \Gamma)}\\
                & \Mod_{\calY_{(0, b], A}}^{(\varphi, \Gamma), \Fil_{\bullet}}\arrow[ru, "\text{forgetful}"']
            \end{tikzcd}
        \]
        such that the parabolic subgroup $\calP_{\Fil_{\bullet}}^{\calG_{D_{(0, b]}}}$ of the filtred $G$-torsor \[
            \begin{tikzcd}
                \Rep_G \arrow[rr, "D_{(0, b]}"]\arrow[rd, "\Fil_{\bullet} D_{(0, b]}"'] && \Mod_{\calY_{(0, b], A}}^{(\varphi, \Gamma)} \arrow[r] & \VB_{\calY_{(0, b], A}}\\
                & \Mod_{\calY_{(0, b], A}}^{(\varphi, \Gamma), \Fil_{\bullet}}\arrow[ru]\arrow[r] & \VB_{\calY_{(0,b], A}}^{\Fil_{\bullet}}\arrow[ru]
            \end{tikzcd},
        \] denoted by $\Fil_{\bullet}\calG_{D_{(0, b]}}$, is conjugate to $P^{\an}_{\calY_{(0, b], A}}$ under the fixed isomorphism \eqref{eq: fixed isomorphism}. We denote by $G\Mod_{\calY_{(0, b], A}}^{(\varphi, \Gamma), P}$ the category of $P$-flagged $(\varphi, \Gamma)$-modules over $\calY_{(0, b], A}$. 
        \item[(ii)] The \textbf{category of $P$-flagged $(\varphi, \Gamma)$-modules} over $A$ is defined to be \[
            G\Mod_{A}^{(\varphi, \Gamma), P} := 2-\varinjlim_{b\rightarrow 0}G\Mod_{\calY_{(0, b], A}}^{(\varphi, \Gamma), P}.
        \] 
        Consequently, by a \textbf{$P$-flagged $(\varphi, \Gamma)$-module (with $G$-structure)} over $A$, we mean an object $(D, \Fil_{\bullet} D)$ in $G\Mod_{A}^{(\varphi, \Gamma), P}$. 
    \end{enumerate}
\end{Definition}

\begin{Remark}\label{Remark: paraboline}
    In \cite{Bergdall-paraboline}, J. Bergdall termed a similar notion as in Definition \ref{Definition: P-flag} \emph{paraboline}. We chose to use our terminology to emphasise the choice of parabolic subgroup $P$. 
\end{Remark}

\begin{Example}\label{Example: old termimology on triangulation}
    If $G = \GL_n$ (for some $n$), $P=B$ and $\sigma = \mathrm{std}: \GL_n \xrightarrow{\id} \GL_n$ is the standard representation, then a $B$-flag $\Fil_{\bullet} D(\mathrm{std})$ is nothing but a triangulation on the $(\varphi, \Gamma)$-module $D(\mathrm{std})$.
\end{Example}

\begin{Lemma}\label{Lemma: finer filtration can be viewed as a coarser filtration}
    Let $A$ be a reduced affinoid algebra over $\Q_p$ (à la Tate). Let $P\subset Q$ be parabolic subgroups of $G$ containing $B$. There exists a functor \[
        \Phi_P^Q: G\Mod_A^{(\varphi, \Gamma), P} \rightarrow G\Mod_A^{(\varphi, \Gamma), Q}
    \] such that the diagram \[
        \begin{tikzcd}
            G\Mod_A^{(\varphi, \Gamma), P} \arrow[rr, "\Phi_P^Q"]\arrow[rd, "\text{forgetful}"'] && G\Mod_A^{(\varphi, \Gamma), Q}\arrow[ld, "\text{forgetful}"]\\
            & G\Mod_A^{(\varphi, \Gamma)}
        \end{tikzcd}
    \]
    is commutative. 
\end{Lemma}
\begin{proof}
    To prove the lemma, we show that, given $(D, \Fil_{\bullet}D)\in G\Mod_A^{(\varphi, \Gamma), P}$, there exists a unique sub-filtration $\Fil_{\bullet}^QD$ such that $(D, \Fil_{\bullet}^QD)\in G\Mod_A^{(\varphi, \Gamma), Q}$. By definition, we may assume $(D, \Fil_{\bullet}D)\in G\Mod_{\calY_{(0, b], A}}^{(\varphi, \Gamma), P}$ and denote by $\Fil_{\bullet}\calG_{D_{(0, b]}}$ the filtred $G$-torsor as in Definition \ref{Definition: P-flag}.

    Recall from the theory of algebraic groups (\cite[Chapter 21, Sect. i]{Milne-AG}), \[
        P = P(\mu) = \{ g\in G: \lim_{t\mapsto 0} \mu(t)g\mu(t)^{-1}\text{ exists} \}\footnote{ Here, the limit is taking in the sense of \cite[Chapter 13, Sect. b]{Milne-AG}}
    \]
    for some cocharacter $\mu$. Similarly, we have $Q = P(\lambda)$ for some cocharacter $\lambda$. For any representation $\sigma: G \rightarrow \GL_n$, define \begin{align*}
        \mu_{\sigma}: & \bbG_m \xrightarrow{\mu} G \xrightarrow{\sigma}\GL_n,\\
        \lambda_{\sigma}: & \bbG_m \xrightarrow{\lambda} G \xrightarrow{\sigma} \GL_n,
    \end{align*}
    and let $P_{\sigma} = P(\mu_{\sigma})$ and $Q_{\sigma} = Q(\lambda_{\sigma})$ be the corresponding parabolic subgroups in $\GL_n$. Note that $P_{\sigma} \subset Q_{\sigma}$. 

    Consequently, $\Fil_{\bullet}\calG_{D_{(0, b]}}(\sigma)$ is stabilised by $P_{\sigma, \calY_{(0, b], A}}^{\an}$ (up to conjugation). Since parabolic subgroups in $\GL_n$ are in bijection with the flags on which they stabilise and $P_{\sigma} \subset Q_{\sigma}$, there exists a sub-filtration $\Fil_{\bullet}^Q\calG_{D_{(0, b]}}(\sigma)$ of $\Fil_{\bullet} \calG_{D_{(0, b]}}(\sigma)$ that is stabilised by $Q_{\sigma,\calY_{(0, b], A}}^{\an}$ (up to conjugation). As the construction is functorial in $\sigma$, we arrive at a sub-filtration $\Fil^{Q}_{\bullet}\calG_{D_{(0, b]}}$ of $\Fil_{\bullet}\calG_{D_{(0, b]}}$.

    Finally, $\Fil_{\bullet}\calG_{D_{(0, b]}}$ comes from forgetting the $(\varphi, \Gamma)$-structure on $\Fil_{\bullet}D_{(0, b]}$, every step of $\Fil_{\bullet}\calG_{D_{(0, b]}}(\sigma)$ is a saturated $(\varphi, \Gamma)$-submodule over $\calY_{(0, b], A}$. Therefore, every step of $\Fil_{\bullet}^Q\calG_{D_{(0, b]}}(\sigma)$ is a saturated $(\varphi, \Gamma)$-submodule. Consequently, $\Fil_{\bullet}^Q\calG$ defines a $(\varphi, \Gamma)$-stable filtration $\Fil^{Q}_{\bullet}D_{(0, b]}$ and we conclude the proof. 
\end{proof}

\begin{Definition}\label{Definition: P-flagged Galois representation}
    Let $L$ be a (large enough) finite field extension of $\Q_p$ and $\rho$ be a $G$-valued $\Gal_{F}$-representation with coefficients in $L$. Fix a Borel subgroup $B\subset G$ and let $P$ be a parabolic subgroup of $G$ containing $B$. We say $\rho$ is \textbf{$P$-flagged} if there exists a factorisation of functor \[
        \begin{tikzcd}
            \Rep_G \arrow[r, "\rho"]\arrow[rd] & \Rep_{\Gal_{F}}(L)\arrow[r, "\D_{\rig}^{\dagger}"] & \Mod_{L}^{(\varphi, \Gamma)}\\
            & \Mod_{L}^{(\varphi, \Gamma),\Fil_{\bullet}} \arrow[ru, "\text{forgetful}"']
        \end{tikzcd}
    \]
    giving rise to an object $(\D_{\rig}^{\dagger}(\rho), \Fil_{\bullet} \D_{\rig}^{\dagger}(\rho))$ in $G\Mod_{L}^{(\varphi, \Gamma), P}$. Equivalently, $\rho$ is $P$-flagged if $\D_{\rig}^{\dagger}(\rho)$ belongs to the essential image of the forgetful functor $G\Mod_{L}^{(\varphi, \Gamma), P} \rightarrow G\Mod_{L}^{(\varphi, \Gamma)}$.
\end{Definition}

Finally, we introduce the notion of \emph{$P$-non-criticality}, generalising Bergdall's notion (\cite[Definition 3.6]{Bergdall-paraboline}). 

\begin{Definition}\label{Definition: P-non-critical}
    Let $L$ be a large enough finite field extension of $\Q_p$ containing all the embeddings of $F$ in $\overline{\Q}_p$. Let $(D, \Fil_{\bullet} D)\in G\Mod_L^{(\varphi, \Gamma), P}$ be a $P$-flagged $(\varphi, \Gamma)$-module with $G$-structure. We say $(D, \Fil_{\bullet}D)$ is \textbf{$P$-non-critical} if, for any $\sigma\in \Rep_G$, $\Fil_{\bullet}D(\sigma)$ is non-critical (in the sense of \cite[Definition 3.6]{Bergdall-paraboline}). More precisely, for any $\sigma \in \Rep_{G}$, any embedding $\tau: F \hookrightarrow \overline{\Q}_p$, and any $i$, there exists $k_{\tau, i}\in \Z$ such that \[
        \calD_{\dR}(\Fil_i D)(\sigma)_{\tau} \oplus \Fil^{k_{\tau, i}}_{\dR} \calD_{\dR}(D)(\sigma)_{\tau} = \calD_{\dR}(D)(\sigma)_{\tau}.
    \] 
    Here, $\Fil_{\dR}^{\bullet} \calD_{\dR}(D)$ is the de Rham filtration on $\calD_{\dR}(D)$. 
\end{Definition}

\begin{Lemma}\label{Lemma: P-non-criticality can be checked on faithful representations}
    Let $\sigma_0\in \Rep_G$ be a faithful representation of $G$. Then, $(D, \Fil_{\bullet})\in G\Mod_L^{(\varphi, \Gamma), P}$ is $P$-non-critical if and only if $(D(\sigma_0), \Fil_{\bullet}D(\sigma_0))$ is non-critical (in the sense of \cite[Definition 3.6]{Bergdall-paraboline}).
\end{Lemma}
\begin{proof}
    For any $(\sigma: G \rightarrow \GL_n)\in \Rep_G$, let $P_{\sigma}$ be the parabolic subgroup of $\GL_n$ as in Lemma \ref{Lemma: finer filtration can be viewed as a coarser filtration}. Then, the parabolic subgroup of $\GL_n$ fixing the filtration $\Fil_{\bullet}D(\sigma)$ is conjugate to $P_{\sigma}$. Suppose $(D(\sigma), \Fil_{\bullet}D(\sigma))$ is non-critical, we verify the following: 

    \noindent \textbf{Claim 1.} The filtred $(\varphi, \Gamma)$-module $(D(\sigma^{\vee}), \Fil_{\bullet}D(\sigma^{\vee}))$ is also non-critical. 
    
    To show the claim, we first recall the construction of $\Fil_{\bullet}D(\sigma^{\vee})$ (\cite[Sect. 2.1]{SR-Tannakian}): For any $i$, the inclusion $\Fil_i D(\sigma) \subset D(\sigma)$ yields a quotient map $D(\sigma^{\vee}) \rightarrow (\Fil_{i}D(\sigma))^{\vee}$; we then define \[
        \Fil_{1-i}D(\sigma^{\vee}) \coloneq \ker\left( D(\sigma^{\vee}) \rightarrow (\Fil_{i}D(\sigma))^{\vee} \right).
    \] Similarly, the de Rham filtration for $\calD_{\dR}(D)(\sigma^{\vee})$ is given by \[
        \Fil_{\dR}^{1-i} \calD_{\dR}(D)(\sigma^{\vee})_{\tau} = \ker(\calD_{\dR}(D)(\sigma^{\vee})_{\tau} \rightarrow (\calD_{\dR}(D)(\sigma^{\vee}_{\tau}))^{\vee}).
    \] Therefore, if \[
        \calD_{\dR}(\Fil_i D)(\sigma)_{\tau} \oplus \Fil^{k_{\tau, i}}_{\dR} \calD_{\dR}(D)(\sigma)_{\tau} = \calD_{\dR}(D)(\sigma)_{\tau}, 
    \] then \[
        \calD_{\dR}(\Fil_{1-i}D)(\sigma^{\vee})_{\tau} \oplus \Fil_{\dR}^{1-k_{\tau, i}}\calD_{\dR}(D)(\sigma^{\vee})_{\tau} = \calD_{\dR}(D)(\sigma)_{\tau};
    \] so we conclude the claim.

    \noindent \textbf{Claim 2.} If $\sigma'$ is a subrepresentation of $\sigma$, then $(D(\sigma'), \Fil_{\bullet}D(\sigma'))$ is non-critical. 

    If $\sigma'$ is a subrepresentation of $\sigma$, then \[
        \Fil_{\bullet}D(\sigma') = \Fil_{\bullet}D(\sigma) \cap D(\sigma')
    \]
    and \[
        \Fil_{\dR}^{\bullet} \calD_{\dR}(D)(\sigma')_{\tau} = \Fil_{\dR}^{\bullet} \calD_{\dR}(D)(\sigma)_{\tau} \cap \calD_{\dR}(D)(\sigma)_{\tau}.
    \]
    The desired result then follows immediately from the definition. 

    \noindent \textbf{Claim 3.} If $\sigma'$ is a quotient of $\sigma$, then $(D(\sigma'), \Fil_{\bullet}D(\sigma'))$ is non-critical. 

    If $\sigma'$ is a quotient of $\sigma$, then \[
        \Fil_{\bullet}D(\sigma') = \image\left( \Fil_{\bullet}D(\sigma) \rightarrow D(\sigma') \right)
    \] and \[
        \Fil_{\dR}^{\bullet}\calD_{\dR}(D)(\sigma')_{\tau} = \image\left( \Fil_{\dR}^{\bullet}\calD_{\dR}(D)(\sigma)_{\tau} \rightarrow \calD_{\dR}(D)(\sigma')_{\tau} \right).
    \]
    The claim then follows immediately. 

    \noindent \textbf{Claim 4.} If $\sigma,\sigma'\in \Rep_G$ such that $(D(\sigma), \Fil_{\bullet}D(\sigma))$ and $(D(\sigma'), \Fil_{\bullet}D(\sigma'))$ are non-critical, then $(D(\sigma\otimes\sigma'), \Fil_{\bullet}D(\sigma\otimes\sigma'))$ is also non-critical.  

    We have \[
        \Fil_{t}D(\sigma\otimes\sigma') = \sum_{i+j=t} \Fil_{i}D(\sigma) \otimes \Fil_j D(\sigma')
    \]
    and \[
        \Fil_{\dR}^t \calD_{\dR}(D)(\sigma\otimes\sigma)_{\tau} = \sum_{i+j=t} \Fil_{\dR}^i \calD_{\dR}(D)(\sigma)_{\tau} \otimes \Fil_{\dR}^j \calD_{\dR}(D)(\sigma').
    \]
    Suppose for any $i$, we have \[
        \calD_{\dR}(\Fil_i D)(\sigma)_{\tau} \oplus \Fil_{\dR}^{k_{i, \tau}} \calD_{\dR}(D)(\sigma)_{\tau} = \calD_{\dR}(D)(\sigma)_{\tau}
    \] and 
    \[
        \calD_{\dR}(\Fil_j D)(\sigma')_{\tau} \oplus \Fil_{\dR}^{k_{j, \tau}'} \calD_{\dR}(D)(\sigma')_{\tau} = \calD_{\dR}(D)(\sigma')_{\tau}.
    \]
    We have \[
        \cdots \leq k_{i-1, \tau} \leq  k_{i, \tau} \leq k_{i+1, \tau} \leq \cdots \quad \text{ and }\quad \cdots \leq k_{j-1, \tau}' \leq k_{j, \tau}' \leq k_{j+1, \tau}' \leq \cdots . 
    \]
    We can thus define new descending filtrations \[
        \Fil_{\nc}^i \calD_{\dR}(D)(\sigma)_{\tau} \coloneq \Fil_{\dR}^{k_{i, \tau}} \calD_{\dR}(D)(\sigma)_{\tau} \quad \text{ and }\quad \Fil_{\nc}^j\calD_{\dR}(D)(\sigma') \coloneq \Fil_{\dR}^{k_{j, \tau}'}\calD_{\dR}(D)(\sigma')_{\tau}. 
    \] By the construction, the graded pieces of these new filtrations have the following properties \begin{align*}
        \Gr_{\nc}^i \calD_{\dR}(D)(\sigma)_{\tau} & = \frac{\Fil_{\nc}^i \calD_{\dR}(D)(\sigma)_{\tau}}{\Fil_{\nc}^{i+1} \calD_{\dR}(D)(\sigma)_{\tau}} \\
        & = \frac{\Fil_{\nc}^i \calD_{\dR}(D)(\sigma)_{\tau} \oplus \calD_{\dR}(\Fil_i D)(\sigma)_{\tau}}{\Fil_{\nc}^{i+1} \calD_{\dR}(D)(\sigma)_{\tau} \oplus \calD_{\dR}(\Fil_i D)(\sigma)_{\tau}} \\
        & = \frac{\Fil_{\nc}^{i+1} \calD_{\dR}(D)(\sigma)_{\tau} \oplus \calD_{\dR}(\Fil_{i+1} D)(\sigma)_{\tau}}{\Fil_{\nc}^{i+1} \calD_{\dR}(D)(\sigma)_{\tau} \oplus \calD_{\dR}(\Fil_i D)(\sigma)_{\tau}}\\
        & = \frac{\calD_{\dR}(\Fil_{i+1}D)(\sigma)_{\tau}}{\calD_{\dR}(\Fil_i D)(\sigma)_{\tau}}\\
        & = \calD_{\dR}(\Gr_{i+1} D)(\sigma)_{\tau}
    \end{align*}
    and \[
        \Gr_{\nc}^j \calD_{\dR}(D)(\sigma')_{\tau} = \calD_{\dR}(\Gr_{j+1}D)(\sigma')_{\tau}.
    \]
    In other words, as $L$-vector spaces, we have \begin{align*}
        \calD_{\dR}(\Fil_i D)(\sigma)_{\tau} = \bigoplus_{a\leq i-1} \Gr_{\nc}^{a}\calD_{\dR}(D)(\sigma)_{\tau}, \quad & \Fil_{\nc}^i \calD_{\dR}(D)(\sigma)_{\tau} = \bigoplus_{a \geq i} \Gr_{\nc}^{a}\calD_{\dR}(D)(\sigma)_{\tau}, \\
        \calD_{\dR}(\Fil_j D)(\sigma')_{\tau} =\bigoplus_{b\leq j-1} \Gr_{\nc}^{b}\calD_{\dR}(D)(\sigma')_{\tau}, \quad & \Fil_{\nc}^j \calD_{\dR}(D)(\sigma')_{\tau} = \bigoplus_{b \geq j} \Gr_{\nc}^{b}\calD_{\dR}(D)(\sigma')_{\tau}.
    \end{align*}
    Therefore, we can write \[
        \calD_{\dR}(\Fil_s D)(\sigma\otimes\sigma')_{\tau} = \bigoplus_{a+b \leq t-2} \Gr_{\nc}^a \calD_{\dR}(D)(\sigma)_{\tau} \otimes_L \Gr_{\nc}^b\calD_{\dR}(D)(\sigma')_{\tau}
    \]
    and \[
        \Fil_{\nc}^{s} \calD_{\dR}(D)(\sigma \otimes \sigma')_{\tau} = \bigoplus_{a+b\geq s} \Gr_{\nc}^a \calD_{\dR}(D)(\sigma)_{\tau} \otimes_L \Gr_{\nc}^b\calD_{\dR}(D)(\sigma')_{\tau} .
    \]
    Using this explicit description, we see that, for any $s$, \[
        \calD_{\dR}(\Fil_s D)(\sigma \otimes \sigma')_{\tau} \oplus \Fil_{\nc}^{s-1} \calD_{\dR}(D)(\sigma \otimes \sigma')_{\tau} = \calD_{\dR}(D)(\sigma \otimes \sigma')_{\tau}. 
    \]
    Since $\Fil_{\nc}^\bullet \calD_{\dR}(D)(\sigma \otimes \sigma')_{\tau}$ is a sub-filtration of $\Fil_{\dR}^{\bullet} \calD_{\dR}(D)(\sigma \otimes \sigma')_{\tau}$, we conclude the claim.

    Finally, if $\sigma_0\in \Rep_G$ is faithful, hence a tensor generator, our claims above then concludes the proof.  
\end{proof}

\begin{Corollary}\label{Corollary: P-non-critical implies Q-non-critical}
    Let $P \subset Q$ be parabolic subgroups of $G$ containing $B$. Suppose $(D, \Fil_{\bullet}D)\in G\Mod_L^{(\varphi, \Gamma), P}$ is a $P$-non-critical $P$-flagged $(\varphi, \Gamma)$-module with $G$-structure. Then, $(D, \Fil_{\bullet}^Q D)\in G\Mod_{L}^{(\varphi, \Gamma), Q}$ is $Q$-non-critical. 
\end{Corollary}
\begin{proof}
    This follows immediately from Lemma \ref{Lemma: finer filtration can be viewed as a coarser filtration} and Definition \ref{Definition: P-non-critical}. 
\end{proof}

\subsection{\texorpdfstring{$P$}{P}-refinements}\label{subsection: refinements}

In this subsection, we focus on $(\varphi, \Gamma)$-modules that are potentially semistable. To this end, given a $(\varphi, \Gamma)$-module $D\in G\Mod_{L}^{(\varphi, \Gamma)}$, consider the $G$-torsor \[
    \calG_D: \Rep_G \xrightarrow{D} \Mod_{L}^{(\varphi, \Gamma)} \xrightarrow{\calD_{\pst}} \Mod_{L\otimes_{\Q_p}F^{\unr}}^{(\varphi, N, \Gal_F)} \xrightarrow{\text{forgetful}} \Mod_{L\otimes_{\Q_p}K}.
\]
We fix an isomorphism \begin{equation}\label{eq: fixed iso of groups for pst modules}
    \underline{\Aut}^{\otimes}(\calG_D) \cong G_{L\otimes_{\Q_p}F^{\unr}}
\end{equation}

\begin{Definition}\label{Definition: refinement}
    Let $L$ be a (large enough) finite field extension of $\Q_p$. Let $D\in G\Mod_L^{(\varphi, \Gamma)}$ be a potentially semistable $(\varphi, \Gamma)$-module with $G$-structure. Fix a Borel subgroup $B \subset G$ and let $P\subset G$ be a parabolic subgroup containing $B$. A \textbf{$P$-refinement} of $D$ is $(\varphi, N, \Gal_F)$-stable filtration $\Fil_{\bullet} \calD_{\pst}(D)$ on $\calD_{\pst}(D)$ such that \begin{enumerate}
        \item[(i)] for any $\sigma\in \Rep_G$,  the successive quotients of $\Fil_{\bullet}\calD_{\pst}(D)(\sigma)$ are free $L\otimes_{\Q_p}F^{\unr}$-modules; 
        \item[(ii)] the parabolic subgroup for the induced filtred $G$-torsor \[
        \begin{tikzcd}
            \Rep_G \arrow[r, "\calD_{\pst}(D)"]\arrow[rd, "\Fil_{\bullet}\calD_{\pst}(D)"'] & \Mod_{L \otimes_{\Q_p}F^{\unr}}^{(\varphi, N, \Gal_F)} \arrow[r] & \Mod_{L \otimes_{\Q_p}F^{\unr}}\\
            & \Mod_{L \otimes_{\Q_p}F^{\unr}}^{(\varphi, N, \Gal_F),\Fil_{\bullet}}\arrow[ru]
        \end{tikzcd}
    \]
    is conjugate to $P_{L\otimes_{\Q_p}F^{\unr}}$ under the fixed isomorphism \eqref{eq: fixed iso of groups for pst modules}.
    \end{enumerate} 
\end{Definition}

\begin{Remark}\label{Remark: Bergdall's refinements}
    In \cite{Bergdall-paraboline}, Bergdall termed a similar notion as in Definition \ref{Definition: refinement} \emph{partial refinements}. We choose to use our terminology to emphasise the choice of parabolic subgroup $P$.
\end{Remark}

\begin{Proposition}\label{Proposition: P-refinements are equivalent to P-flags}
    Let $L$ be a finite field extension of $\Q_p$ and $D\in G\Mod_L^{(\varphi, \Gamma)}$ a potentially semistable $(\varphi, \Gamma)$-module with $G$-structure. Assume for any $\sigma\in \Rep_G$, the $\varphi^{[F:\Q_p]}$-eigenvalues of $D(\sigma)$ all lies in $L$. Fix a Borel subgroup $B\subset G$ and let $P$ be a parabolic subgroup containing $B$. Then, $D\mapsto \calD_{\pst}(D)$ induces a bijection \[
        \{ \text{$P$-flags of }D \} \leftrightarrow \{\text{$P$-refinements of }D\}.
    \]
\end{Proposition}
\begin{proof}
    This follows immediately from \cite[Théorèm A]{Berger-phiN}.
\end{proof}

\begin{Remark}\label{Remark: (phi, Gamma)-module over local artinian algebra}
    We close this section by remarking that our discussion on $(\varphi, \Gamma)$-modules over $L$ can all be extended to $(\varphi, \Gamma)$-modules over a local artinian $\Q_p$-algebra $A$ (see \cite[Sect. 2.2]{BC}). For example, for $(\varphi, \Gamma)$-modules with $G$-structure over $A$, the notions of \emph{de Rham}, \emph{crystalline}, \emph{semistable}, \emph{potentially semistable} can also be defined (see \cite[Definition 2.2.10]{BC}); consequently, the notion of \emph{$P$-non-criticality} can also be defined. 
\end{Remark}

\section{Families of symplectic Galois representations}\label{section: families of Galois reps}
In this section, we study families of symplectic Galois representations and some deformation problems. To this end, we recall some basic setting of general symplectic groups in Sect. \ref{subsection: GSp}. We then review the theory of \emph{symplectic determinants} by Moakher--Quast in Sect. \ref{subsection: symp. det}. In Sect. \ref{subsection: families of Galois reps}, we introduce the notions of \emph{families of symplectic Galois representations} and study some basic properties. In Sect. \ref{subsection: deformation}, we state the deformation problems that will be in our consideration in later parts of the paper.

\subsection{General symplectic groups}\label{subsection: GSp}
Let $n\in \Z_{>0}$ and $V = V_{\Z}$ be the rank-$2n$ free $\Z$-module equipped with the symplectic pairing \begin{equation}\label{eq: defining symplectic pairing}
    V\times V \rightarrow \Z, \quad (v, v')\mapsto \trans v \begin{pmatrix} & -\oneanti_n\\ \oneanti_n & \end{pmatrix} v'.
\end{equation}
The algebraic group $\GSp_{2n}$ (over $\Z$) is then defined to be the subgroup in $\GL_{2n}$ preserving the symplectic pairing \eqref{eq: defining symplectic pairing} up to a unit. In other words,  \[
    \GSp_{2n} \coloneq \left\{ \bfgamma\in \GL_{2n}: \trans\bfgamma \begin{pmatrix} & -\oneanti_n\\ \oneanti_n & \end{pmatrix} \bfgamma = \mathrm{sim}(\bfgamma) \begin{pmatrix} & -\oneanti_n\\ \oneanti_n & \end{pmatrix} \text{ for some }\mathrm{sim}(\bfgamma)\in \bbG_m \right\}.
\]
We denote by $\Sp_{2n}$ the subgroup of $\GSp_{2n}$, consisting of those matrices $\bfgamma$ with $\mathrm{sim}(\bfgamma)=1$.

The choice of the symplectic pairing \eqref{eq: defining symplectic pairing} allows us to consider the upper triangular Borel subgroup $B \subset \GSp_{2n}$. We denote by $T$ (resp., $N$) the associated maximal torus (resp., unipotent radical). We remark that $T$ is nothing but the diagonal matrices in $\GSp_{2n}$ and it is isomorphic to $\bbG_m^{n+1}$ by the following isomorphism \[
    \bbG_m^{n+1} \xrightarrow{\simeq} T, \quad (\tau_1, ..., \tau_{n}, \tau_0)\mapsto \diag(\tau_1, ..., \tau_n, \tau_0\tau_n^{-1}, ..., \tau_0\tau_1^{-1}).
\]

Consider the character group $\bbX \coloneq \Hom(T, \bbG_m)$. Due to the identification $T \cong \bbG_m^{n+1}$, one sees an isomorphism $\bbX \cong \Z^{n+1}$, where each $(k_1, ..., k_n; k_0)\in \Z^{n+1}$ is viewed as a character \[
    (k_1, ..., k_n; k_0): \diag(\tau_1, ..., \tau_n, \tau_0\tau_n^{-1}, ..., \tau_0\tau_1^{-1}) \mapsto \prod_{i=0}^{n} \tau_i^{k_i}.
\]
Later in the paper, we will consider elements of $\Z^n$ as characters in $\bbX$ via the embedding \begin{equation}\label{eq: embedding, algebriac weights}
    \Z^n \hookrightarrow \Z^{n+1} \cong \bbX, \quad (k_1, ..., k_n) \mapsto (k_1, ..., k_n; 0).
\end{equation}

In what follows, we shall consider the parabolic subgroups of $\GSp_{2n}$ that contain $B$. Since the proper parabolic subgroups are intersections of the maximal ones\footnote{ Here, the adjective `maximal' is in the sense of `maximal ideals', \emph{i.e.}, maximal parabolic subgroups are those parabolic subgroups $P$ such that the only subgroups of $\GSp_{2n}$ containing $P$ are only $\GSp_{2n}$ and $P$ itself.}. We now describe the maximal parabolic subgroups explicitly: for $1 \leq m \leq n$, let \begin{equation}\label{eq: maximal parabolics for GSp}
    P_m \coloneq \begin{pmatrix} M_m & M_{m \times 2(g-m)} & M_m \\ & M_{2(g-m)} & M_{2(g-m)\times m} \\ & &  M_m \end{pmatrix} \cap \GSp_{2n}
\end{equation}
with the Levi decomposition \[
    P_m = M_{P_m} \ltimes N_{P_m},
\]
where \[
    M_{P_m} = \begin{pmatrix} M_m && \\ & M_{2(g-m)} & \\ && M_m\end{pmatrix} \cap \GSp_{2n} \quad \text{ and } \quad N_{P_m} = \begin{pmatrix} \one_m & M_{m\times 2(g-m)} & M_m \\ & \one_{2(g-m)} & M_{2(g-m)\times m} \\ & & \one_m\end{pmatrix} \cap \GSp_{2n}
\]
are the Levi subgroup and the unipotent radical respectively.

\begin{Example}\label{Example: GSp4}
Let us provide explicit descriptions when $n=2$. In this case, there are two maximal parabolic subgroups containing $B$:\begin{align*}
    P_2 = P_{\mathrm{Si}} & = \begin{pmatrix} * & * & * & *\\ * & * & * & * \\ && * & *\\ && * & *\end{pmatrix} \cap \GSp_4: \text{ the Siegel parabolic},\\
    P_1 = P_{\mathrm{Kl}} & = \begin{pmatrix}* & * & * & * \\ & * & * & * \\ & * & * & * \\ &&& *\end{pmatrix} \cap \GSp_4: \text{ the Klingen parabolic}.
\end{align*} 
The Levi decompositions of these two parabolic subgroups are given explicitly as follows \[
    P_{\mathrm{Si}} = M_{\mathrm{Si}} \ltimes N_{\mathrm{Si}} \quad \text{ and }\quad P_{\mathrm{Kl}} = M_{\mathrm{Kl}} \ltimes N_{\mathrm{Kl}},
\]
where \begin{align*}
    M_{\mathrm{Si}} & = \begin{pmatrix} * & * && \\ * & * && \\ && * & * \\ && * & *\end{pmatrix} \cap \GSp_4, \quad N_{\mathrm{Si}} = \begin{pmatrix} 1 & & * & * \\ & 1 & * & * \\ && 1 & \\ &&& 1\end{pmatrix} \cap \GSp_4, \\
    M_{\mathrm{Kl}} & = \begin{pmatrix} * &&& \\ & * & * & \\ & * & * & \\ &&& *\end{pmatrix} \cap \GSp_4, \quad N_{\mathrm{Kl}} = \begin{pmatrix} 1 & * & * & * \\ & 1 && * \\ && 1 & * \\ &&& 1\end{pmatrix} \cap \GSp_4.
\end{align*}
In fact, we have the following identifications \begin{equation}\label{eq: identifications of Levi}
    \begin{array}{ll}
        \GL_2 \times \bbG_m  \xrightarrow{\simeq} M_{\mathrm{Si}}, & (\bfgamma, \varsigma) \mapsto \begin{pmatrix}\bfgamma & \\ & \varsigma \oneanti_2\trans\bfgamma^{-1}\oneanti_2\end{pmatrix},\\ \\
        \bbG_m \times \GL_2 \xrightarrow{\simeq} M_{\mathrm{Kl}}, & (\nu, \bfgamma) \mapsto \begin{pmatrix} \nu && \\ & \bfgamma & \\ && \det(\bfgamma)/\nu\end{pmatrix}.
    \end{array}
\end{equation}
\end{Example}

\begin{Remark}\label{Remark: identifications of Levi subgroups}
    Although both $M_{\mathrm{Si}}$ and $M_{\mathrm{Kl}}$ are isomorphic to $\GL_2 \times \bbG_m$, we presented the isomorphisms differently above. This is because the piece of $\bbG_m$ corresponds to the symplectic similitude in the Siegel case while the symplectic similitude in the Klingen case is given by $\det(\GL_2)$. In other words, if we restrict to $\Sp_4$, the Levi subgroup of the Siegel parabolic of $\Sp_4$ is given by $\GL_2$, whereas the Levi subgroup of the Klingen parabolic of $\Sp_4$ is given by $\bbG_m \times \SL_2$. Such a difference shall be important in Sect. \ref{subsection: p-adic weight spaces}. 
\end{Remark}

\subsection{Symplectic determinants}\label{subsection: symp. det}
In this subsection, we review the theory of Moakher--Quast on \emph{symplectic determinants} (\cite{MQ-SympDet}). Symplectic determinants are analogues of Chenevier's \emph{$d$-dimensional determinants}, but equipped with \emph{symplectic structures}. For the convenience of the reader, we first review Chenevier's theory of $d$-dimensional determinants. We then briefly review the theory of Moakher--Quast and state the key results we shall need in later parts of the paper (Theorem \ref{Theorem: symplectic version of Chenevier's Thm A and B}). We shall also prove an analogue of \cite[Example 2.32]{Chenevier-determinant} in Corollary \ref{Corollary: continuous symplectic determinants and glueing}. 

\paragraph{Symplectic determinants.}

We start by recalling the following definitions from \cite{Chenevier-determinant}.

\begin{Definition}\label{Definition: polynomial law and properties}
    Let $A$ be a commutative ring and $M,N\in\Mod_{A}$. We view them as functors $\Alg_A \rightarrow \Set$ by base change.\footnote{ Here, $\Alg_A$ stands for the category of (commutative) $A$-algebras.} 
    \begin{enumerate}
        \item[(i)] An \textbf{$A$-polynomial law} is a natural transformation $\mathrm{P}:M\rightarrow N$.
        In other words, it is a collection of maps \[
            \mathrm{P}_B:M\otimes_A B\rightarrow N\otimes_A B
        \]
        for any (commutative) $A$-algebra $B$ such that for any $A$-algebra homomorphism $f: B \rightarrow B'$, the diagram 
        \begin{center}
            \begin{tikzcd}
                M\otimes_A B \arrow[d, "\id_M\otimes_A f"'] \arrow[r, "\mathrm{P}_B"] & N\otimes_A B \arrow[d, "\id_N\otimes_A f"] \\
                M\otimes_A B' \arrow[r, "\mathrm{P}_{B'}"] & N\otimes_A B'
            \end{tikzcd}
        \end{center}
        commutative. 
        \item[(ii)] Let $\mathrm{P}: M\rightarrow N$ be a polynomial law. The \textbf{kernel} of $\mathrm{P}$ is the $A$-submodule $\ker(\mathrm{P})\subset M$ defined as \[
            \ker(\mathrm{P}) \coloneq \left\{ m\in M: \mathrm{P}_B(m\otimes b+m') = \mathrm{P}_B(m') \quad \forall B\in \Alg_A, \forall (b, m')\in B\times M\otimes_A B \right\}.
        \]
        \item[(iii)] An $A$-polynomial law $\mathrm{P}: M \rightarrow N$ is \textbf{homogeneous of degree $d$} (for some $d\in\Z_{\geq 0}$ ) if for any $B\in\Alg_{A}$, $b\in B$ and $x\in M\otimes_A B$, we have
        \begin{align*}
            \mathrm{P}_B(bx)=b^d \mathrm{P}_B(x).
        \end{align*}
        \item[(iv)] Let $M, N$ be $A$-algebras (not necessarily commutative). A homogeneous $A$-polynomial law $\mathrm{P}: M \rightarrow N$ of degree $d$ is \textbf{multiplicative} if for any $B\in\CAlg_A$ and $x,y\in M\otimes_A B$, we have
        \begin{align*}
            \mathrm{P}_B(xy)=\mathrm{P}_B(x)\mathrm{P}_B(y) \quad  \text{ and } \quad  \mathrm{P}_B(1)=1.
        \end{align*}
        \item[(v)] Let $M$ be an $A$-algebra (not necessarily commutative). A \textbf{$d$-dimensional determinant} on $M$ is an $A$-polynomial law \[
            \mathrm{D}: M \rightarrow A
        \]
        that is homogeneous of degree $d$ and multiplicative. 
    \end{enumerate}
\end{Definition}

\begin{Example}\label{Example: determinant}
    Let $G$ be a group and $\rho: G \rightarrow \GL_d(A)$ be a $G$-representation over $A$. Then, \[
        \mathrm{D}: A[G] \rightarrow A, \quad G \ni \sigma \mapsto \det \rho(\sigma)
    \]
    defines a $d$-dimensional determinant on $A[G]$.
\end{Example}

Next, we recall the definitions of \emph{characteristic polynomials} and the \emph{Cayley--Hamilton ideal} of a polynomial law. 

\begin{Definition}\label{Definition: characteristic polynomials and CH algebra}
    Let $M$ be an $A$-algebra (not necessarily commutative) and let $\mathrm{P}: M \rightarrow A$ be an $A$-polynomial law of homogeneous degree $d$. \begin{enumerate}
        \item[(i)] Let $B\in \Alg_A$ and $r\in M \otimes_A B$. The \textbf{characteristic polynomial} of $r$ is defined to be \[
            \charpoly_{\mathrm{P}}(r, T) \coloneq \mathrm{P}_{B[T]}(T-r).
        \]
        More explicitly, we have \[
            \charpoly_{\mathrm{P}}(r, T) = \sum_{i=0}^d(-1)^i \Lambda_{i,B}^{\mathrm{P}}(r)T^{d-i},
        \]
        where each $\Lambda_{i}^{\mathrm{P}}: M \rightarrow A$ is the $A$-polynomial law of homogeneous degree $i$, arising from the coefficients. 
        We can understand $\charpoly_{\mathrm{P}}(-, T)$ as an $A$-polynomial law $M \rightarrow A[T]$. 
        \item[(ii)] The \textbf{Cayley--Hamilton ideal} $\CH(\mathrm{P})$ associated with $\mathrm{P}$ is the two-sided ideal in $M$ generated by the set \[
            \left\{\Lambda^{\mathrm{P}}_{\alpha}(r_1,...,r_n): n\in\Z_{\geq 1}, r_1, ..., r_n\in M, \alpha=(\alpha_1, ...,\alpha_n)\in\Z_{\geq 0} \text{ with }\sum_i \alpha_i = d\right\},
        \]
        where $\Lambda_{\alpha}^{\mathrm{P}}: M^n \rightarrow M$ is the function arising from the coefficients \begin{align*}
            \charpoly_{\mathrm{P}}(t_1r_1+\cdots+t_nr_n, t_1r_1+\cdots+t_nr_n) & = \sum_{i=0}^d (-1)^i\Lambda_i^{\mathrm{P}}(r)T^{d-i}\\
            & = \sum_{\alpha} \Lambda_{\alpha}^{\mathrm{P}}(r_1, ..., r_n)\prod_{i=1}^n t_1^{\alpha_i}.
        \end{align*}
    \end{enumerate}
\end{Definition}

\begin{Proposition}\label{Proposition: determinant and trace}
    Suppose $A$ is a $\Q$-algebra and let $M$ be an $A$-algebra (not necessarily commutative) and $\mathrm{D}: M \rightarrow A$ be a determinant of dimension $d$. Denote by $\mathrm{T}_{\mathrm{D}} \coloneq \Lambda_1^{\mathrm{D}}$. \begin{enumerate}
        \item[(i)] The polynomial law $\mathrm{T}_{\mathrm{D}}$ is a pseudocharacter (in the sense of \cite[Sect. 1.2]{BC}).
        \item[(ii)] The assignment $\mathrm{D} \mapsto \mathrm{T}_{\mathrm{D}}$ induces a bijection between $d$-dimensional $A$-valued determinants on $M$ and $d$-dimensional $A$-valued pseudocharacters on $M$.
        \item[(iii)] The following two statements are equivalent: \begin{itemize}
            \item There exists $n$ pseudocharacters $\mathrm{T}_1$, ..., $\mathrm{T}_n$ such that $\mathrm{T}_{\mathrm{D}} = \sum_{i=1}^n \mathrm{T}_i$. 
            \item There exists $n$ determinants $\mathrm{D}_1$, ..., $\mathrm{D}_n$ such that $\mathrm{D} = \prod_{i=1}^n \mathrm{D}_n$.
        \end{itemize}
    \end{enumerate}
\end{Proposition}
\begin{proof}
    The first two assertions are exactly \cite[Proposition 1.27]{Chenevier-determinant}. Note that, for any $B\in \Alg_A$ and any $r\in M \otimes_A B$, the characteristic polynomial of $r$ for $\mathrm{D}$ agrees with the characteristic polynomial of $r$ for $\mathrm{T}_{\mathrm{D}}$ (for the characteristic polynomial of pseudocharacters, see \cite[Sect. 1.2]{BC}). This implies the third assertion (as $\Lambda_{\mathrm{D}}^d = \mathrm{D}$ and $\Lambda_{\mathrm{D}}^1 = \mathrm{T}_{\mathrm{D}}$). 
\end{proof}

To define \emph{symplectic determinants}, we first make the following observation (see also \cite[Sect. 3.2]{MQ-SympDet}). Let $n\in \Z_{>0}$ and consider the set of $(2n\times 2n)$-matrices $M_{2n}(A)$ over $A$. We equip $M_{2n}(A)$ with a symplectic $A$-linear involution $\jmath$ by \[
    \jmath: M_{2n}(A) \rightarrow M_{2n}(A), \quad \bfgamma \mapsto \bfgamma^{\jmath} = \begin{pmatrix} & -\oneanti_n\\ \oneanti_n\end{pmatrix} \trans\bfgamma \begin{pmatrix} & \oneanti_n\\ -\oneanti_n\end{pmatrix}.
\]
To simplify the notation, we write $\mathbb{J} = \begin{pmatrix} & -\oneanti_n\\ \oneanti_n\end{pmatrix}$. Given $\bfgamma \in M_{2n}(A)$ such that $\bfgamma = \bfgamma^{\jmath}$, one observes that \[
    \bfgamma \mathbb{J} = \mathbb{J} \trans \bfgamma = -\trans\mathbb{J} \trans\bfgamma = -\trans(\bfgamma \mathbb{J}).
\]
In particular, $\bfgamma \mathbb{J}$ is alternating and so its Pfaffian $\mathrm{pf}(\bfgamma \mathbb{J})$  is well-defined (see, for example, \cite[Notation (3)]{MQ-SympDet}). In this case, one sees that \[
    \det(\bfgamma) = \det(\bfgamma\mathbb{J}) = \mathrm{pf}(\bfgamma \mathbb{J})^{2}.
\]
Inspired by this, we have the following definition.

\begin{Definition}\label{Definition: symplectic determinants}
    Let $M$ be an $A$-algebra (not necessarily commutative) with an involution $*$ and denote by $M^{+}$ the $(*=1)$-part. \begin{enumerate}
        \item[(i)] A \textbf{weak $2n$-dimensional symplectic determinant} on $M$ is a pair $(\mathrm{D}, \mathrm{P})$, where $\mathrm{D}: M \rightarrow A$ is a $2n$-dimensional deterinant and $\mathrm{P}: M^+\rightarrow A$ is a polynomial law of homogeneous degree $n$, such that  \begin{itemize}
            \item for all $B\in\Alg_A$ we have $\mathrm{D}_B(x^*)=\mathrm{D}_B(x)$ for $x\in M\otimes_A B$; 
            \item $\mathrm{P}^2 = \mathrm{D}|_{M^+}$; and
            \item $\mathrm{P}(1)=1$.
        \end{itemize}
        \item[(ii)] A \textbf{$2n$-dimensional symplectic determinant} is a weak $2n$-dimensional symplectic determinant $(\mathrm{D}, \mathrm{P})$ such that \[
            \CH(\mathrm{P}) \subset \ker(\mathrm{D}).
        \]
    \end{enumerate}
    In both cases, we shall often write $(\mathrm{D}, \mathrm{P}): (M, *) \rightarrow A$.
\end{Definition}

\begin{Remark}\label{Remark: symplectic determinants; conjecture and extension of pfaffians}
    \begin{enumerate}
        \item[(i)] It is conjectured in  \cite[Remark 3.7]{MQ-SympDet} that every weak symplectic determinant is a symplectic determinant. 
        \item[(ii)] Assuming $2\in A^\times$, by using the recursion formula in \cite[Proposition 3.15]{MQ-SympDet} (see also \cite[Example 3.16]{MQ-SympDet}), we can show that $\mathrm{P}$ extends to a $n$-dimensional homogeneous polynomial law $\mathrm{P}^{\mathrm{ext}}$ over $M$. If it is clear from the context, we shall sometimes abuse the notation and write $\mathrm{P}$ for $\mathrm{P}^{\mathrm{ext}}$. 
    \end{enumerate}
\end{Remark}

\begin{Example}\label{Example: symplectic representations and symplectic determinants}
    Let $G$ be a group and let $\rho: G \rightarrow \GSp_{2n}(A)$ be a symplectic representation. We equip with $A[G]$ the involution $*$ defined by \[
        \sigma\mapsto \mathrm{sim}\rho(\sigma) \sigma^{-1}
    \] for all $\sigma\in G$. Then, \[
        (\mathrm{D}, \mathrm{P}) \coloneq (\det\rho, \mathrm{pf}(\rho \mathbb{J})): (A[G], *) \rightarrow A
    \]
    is a $2n$-dimensional symplectic determinant. The fact that $\CH(\mathrm{P}) \subset \ker(\mathrm{D})$ follows from \cite[Lemma 3.10]{MQ-SympDet}.
\end{Example}

Analogues of \cite[Theorem A \& B]{Chenevier-determinant} can also be generalised to the theory of symplectic determinants. Before stating such a generalisation, we introduce the following notation. Given a group $G$ and a character \[
    \varsigma: G \rightarrow A^\times,
\]
we can view $A[G]$ as an involutive $A$-algebra by \[
    \sigma \mapsto \varsigma(\sigma) \sigma
\]
for any $\sigma\in G$. We shall abuse the notation and denote this involutive algebra by $(A[G], \varsigma)$.

\begin{Theorem}\label{Theorem: symplectic version of Chenevier's Thm A and B}
    Let $G$ be a group. 
    \begin{enumerate}
        \item[(i)] Let $k$ be an algebraically closed field and $\varsigma: G \rightarrow k^\times$ be a character. Let $(\mathrm{D}, \mathrm{P}): (A[G], \varsigma) \rightarrow A$ be a $2n$-dimensional symplectic determinant (with respect to the involution defined by $\varsigma$). Then, there exists a unique (up to isomorphism) semisimple symplectic representation \[
            \rho: G \rightarrow \GSp_{2n}(k)
        \]
        such that \[
            \mathrm{D} = \det \rho \quad \text{ and }\quad \mathrm{P} = \mathrm{pf}(\rho \mathbb{J}).
        \]
        In particular, $\mathrm{sim} \rho = \varsigma$.
        \item[(ii)] Let $A$ be a henselian local ring with an algebraically closed residue field $k$ and $\widetilde{\varsigma}: G \rightarrow A^\times$ be a character. Suppose $(\widetilde{\mathrm{D}}, \widetilde{\mathrm{P}}): (A[G], \widetilde{\varsigma}) \rightarrow A$ is a $2n$-dimensional symplectic determinant. Let $\rho$ be the symplectic representation attached to $(\widetilde{\mathrm{D}}, \widetilde{\mathrm{P}})\otimes_A k$ and suppose $\rho$ is irreducible. Then, there exists a unique (up to isomorphism) representation \[
            \widetilde{\rho}: G \rightarrow \GSp_{2n}(A)
        \] such that \[
            \widetilde{\rho}\otimes_A k, \quad \widetilde{\mathrm{D}} = \det\widetilde{\rho}, \quad \text{ and }\quad \widetilde{\mathrm{P}} = \mathrm{pf}(\widetilde{\rho}\mathbb{J}).
        \]
    \end{enumerate}
\end{Theorem}
\begin{proof}
    By applying Lemma \ref{Lemma: equivalence of symplectic representations} below, the first statement follows from \cite[Theorem 3.30]{MQ-SympDet} while the second assertion follows from \cite[Proposition 5.2]{MQ-SympDet}. 
\end{proof}

\begin{Lemma}\label{Lemma: equivalence of symplectic representations}
    Let $G$ be a group. Then, the following two categories are equivalent: \begin{enumerate}
        \item[(i)] The category of symplectic representations $\rho: G \rightarrow \GSp_{2n}(A)$. 
        \item[(ii)] The category of pairs $(\varsigma, \phi)$, where $\varsigma: G \rightarrow A^\times$ is a character and $\phi: (A[G], \varsigma) \rightarrow (M_{2n}(A), \jmath)$ is a morphism of $A$-algebras. 
    \end{enumerate}
\end{Lemma}
\begin{proof}
    Given a symplectic representation $\rho: G \rightarrow \GSp_{2n}(A)$, define \[
        \varsigma_{\rho} \coloneq \mathrm{sim} \rho
    \] 
    and \[
        \phi_{\rho}: \sum_{i=1}^n a_i \sigma_i \mapsto \sum_{i=1}^n a_i \rho(\sigma_i).
    \]
    This gives a functor from (i) to (ii) since \[
        \phi_{\rho}(\varsigma_{\rho}(\sigma) \sigma^{-1}) = \mathrm{sim}\rho(\sigma) \rho(\sigma^{-1}) = \mathbb{J} \trans\rho(\sigma) \mathbb{J}^{-1} = \rho(\sigma)^{\jmath}.
    \] Secondly, given a pair $(\varsigma, \phi)$, define \[
        \rho_{\phi} \coloneq \phi|_{G}: G \rightarrow \GL_{2n}(A).
    \]
    To show that this gives a functor from (ii) to (i), we verify that $\rho_{\phi}$ is valued in $\GSp_{2n}(A)$. Indeed, we have \[
        \trans\phi(\sigma)\mathbb{J}\phi(\sigma) = \mathbb{J} \phi(\sigma)^{\jmath}\phi(\sigma) = \mathbb{J} \phi(\varsigma(\sigma) \sigma^{-1})\phi(\sigma) = \varsigma(\sigma) \mathbb{J}.
    \]
    It is then easy to check that these two functors are inverse to each other. 
\end{proof}

\paragraph{Continuous symplectic determinants}

For our purpose in Sect. \ref{subsection: families of Gal reps over small par. eigenvar}, we will need a notion of continuity for symplectic determinants. In this vein, we generalise Chenevier's definition of continuous determinant to a definition of continuity for homogeneous polynomial laws. We shall then see that this is sufficient to consider continuity for symplectic determinants.

\begin{Definition}\label{Definition: continuous homogeneous polynomial law}
    Let $A$ be a (commutative) topological ring, $G$ a topological group. A homogeneous polynomial law $\mathrm{P}: A[G] \rightarrow A$ is \textbf{continuous} if the coefficients
    \[\Lambda_{i}^{\mathrm{P}}: G\rightarrow A\]
    of the characteristic polynomial are continuous.
\end{Definition}

Note that, if $\mathrm{P}$ is a determinant, this definition is equivalent to Chenevier's definition (cf. \cite[Sect. 2.30]{Chenevier-determinant})) due to Amitsur’s formula as discussed in the beginning of \cite[Sect. 2.30]{Chenevier-determinant}. The following analogue of \cite[Example 2.32]{Chenevier-determinant} for continuous homogeneous polynomial laws.

\begin{Lemma}\label{Lemma: continoues homogeneous polynomial gluing}
    Let $G$ be a compact topological group and $X\subseteq G$ be a dense subset. Let $A$ be a compact topological ring. Let $A$ and $\{A_j\}_{j\in J}$ be topological rings such that $A$ is compact and $\prod_{j\in J}A_j$ is Hausdorff. Assume we are given the following: \begin{itemize}
        \item an injective continuous ring homomorphism $\iota : A\hookrightarrow\prod_{j\in J}A_j$; 
        \item for each $j\in J$, there exists a continuous $d$-dimensional homogeneous polynomial law $\mathrm{P}_j: A_j[G] \rightarrow A_j$ such that for any $x\in X$, \[
            (\Lambda_{i}^{\mathrm{P}_j}(x))_j\in \iota(A).
        \] 
    \end{itemize} 
    Then, there exists a homogeneous polynomial law of dimension $d$
        \[\mathrm{P}:A[G]\rightarrow A\]
        such that, for each $j\in J$, 
        \begin{align*}
            \mathrm{P}\otimes_A A_j = \mathrm{P}_{j}.
        \end{align*}
    If moreover that each $\mathrm{P}_j$ is a determinant, then so is $\mathrm{P}$.
\end{Lemma}

\begin{proof}
    The proof is inspired by \cite[Example 2.32]{Chenevier-determinant}.

    To simplify the notations, we first let $C = \prod_{j\in J} A_j$ and consider the map \[
        \psi: G\rightarrow C[T], \quad  \sigma \mapsto (\charpoly_{\mathrm{P}_j}(\sigma, T)))_{j\in J}.
    \]
    Concretely, we can write \[
        \psi(\sigma) = \sum_{i=0}^d (-1)^i(\Lambda_{i, A_j[T]}^{\mathrm{P_j}}(\sigma))_j T^i.
    \]
    By construction, $\psi$ is a polynomial law. 

    Since $\iota: A \rightarrow C$ is an injective continuous ring homomorphism, $A$ is compact and $C$ is Hausdorff, the image $\iota(A)$ is closed in $C$ and $A$ is homeomorphic to $\iota(A)$. We then identify $A$ with $\iota(A)$. By assumption, we have $\psi(X) \subset A[T]$ and $X$ is dense in $G$, thus $\psi(G) \subset A[T]$. By specialising $\psi$ at $T=0$, we obtain the desired polynomial law $\mathrm{P}$. 

    Finally, it is easy to check that $\mathrm{P}$ is a determinant if all $\mathrm{P}_j$'s are determinants. 
\end{proof}

\begin{Lemma}\label{Lemma: continuous determinant iff continuous pfaffian}
    Let $A$ be a topological ring with $2\in A^{\times}$ and $G$ a topological group. Let $(\mathrm{D}, \mathrm{P}): A[G] \rightarrow A$ be a $2n$-dimensional symplectic determinant and recall the extended polynomial law $\mathrm{P}^{\mathrm{ext}}$ from Remark \ref{Remark: symplectic determinants; conjecture and extension of pfaffians}. Then, $\mathrm{D}$ is continuous if and only if $\mathrm{P}^{\mathrm{ext}}$ is continuous. 
\end{Lemma}
\begin{proof}
    By the recursion formula in \cite[Proposition 3.15]{MQ-SympDet} (see also \cite[Example 3.16]{MQ-SympDet}),  we have $\Lambda_{i}^{\mathrm{D}}\in \Z[\Lambda_{0}^{\mathrm{P}},...,\Lambda_{n}^{\mathrm{P}}]$ and $\Lambda_{i}^{\mathrm{P}}\in \Z[2^{-1},\Lambda_{0}^{\mathrm{D}},..,\Lambda_{2n}^{\mathrm{D}}]$ for each $i$, which shows the assertion. 
\end{proof}

The following corollary is an analogue of Chenevier's \cite[Example 2.32]{Chenevier-determinant} for symplectic determinants. 

\begin{Corollary}\label{Corollary: continuous symplectic determinants and glueing}
    Let $G$, $X$, $A$, $\{A_j\}_{j\in J}$, $\iota$ be as in Lemma \ref{Lemma: continoues homogeneous polynomial gluing}. Suppose $2$ is invertible in $A$ and $A_j$'s. Suppose for each $j\in J$, there exists a $2n$-dimensional weak symplectic determinant \[
        (\mathrm{D}_j, \mathrm{P}_j): A_j[G] \rightarrow A_j
    \] such that \begin{itemize}
        \item for all $j\in J$, $\mathrm{D}_j$ is continuous; 
        \item for all $x\in X$, $(\Lambda_i^{\mathrm{D}_j}(x))_j\in \iota(A)$.
    \end{itemize} 
    Then, there exists a $2n$-dimensional weak symplectic determinant \[
        (\mathrm{D}, \mathrm{P}): A[G] \rightarrow A
    \]
    such that \[
        (\mathrm{D}, \mathrm{P}) \otimes_A A_j = (\mathrm{D}_j, \mathrm{P}_j).
    \]
    Moreover, if each $(\mathrm{D}_j, \mathrm{P}_j)$ is a symplectic determinant, so is $(\mathrm{D}, \mathrm{P})$.
\end{Corollary}
\begin{proof}
    By Lemma \ref{Lemma: continuous determinant iff continuous pfaffian} and Lemma \ref{Lemma: continoues homogeneous polynomial gluing}, we obtain a $n$-dimensional homogeneous polynomial law $\mathrm{P}$ and a $2n$-dimensional determinant $\mathrm{D}$ on $A[G]$. Since each pair $(\mathrm{D}_j, \mathrm{P}_j)$ is a (weak) symplectic determinant, it is easy to verify that $(\mathrm{D}, \mathrm{P})$ is again a (weak) symplectic determinant.
\end{proof}

\subsection{Refined families of symplectic Galois representations}\label{subsection: families of Galois reps}
The purpose of this subsection is to introduce and study \emph{refined families of symplectic Galois representations}. To our knowledge, there are three possible ways to define refined families of Galois representations (or $(\varphi, \Gamma)$-modules). We will prove their relationships with each other (Proposition \ref{Proposition: comparison between different notions of families}). Materials presented in this subsection are inspired by \cite[Sect. 4.2.2]{BC} and \cite[Sect. 5.1 and 6.1]{Bergdall-paraboline}.

We begin with setting up some notations. Throughout the rest of the section, let $G = \GSp_{2n}$. Let $B \subset G$ be the Borel subgroup of upper triangular matrices. Throughout this subsection, we fix a maximal parabolic subgroup $P$ containing $B$; that is $P = P_m$ (see \eqref{eq: maximal parabolics for GSp}). We denote the Levi subgroup of $P$ by $M_P$.

In what follows, $\Gal$ will be of the following Galois groups:\begin{itemize}
    \item $\Gal = \Gal_{\Q}$; 
    \item $\Gal = \Gal_{\Q_p}$; 
    \item let $S$ be a finite set of places of $\Q$ with $p\in S$ and $\Q^S$ be the maximal field of extension of $\Q$ that is unramified outside $S$, $\Gal = \Gal_{\Q, S} \coloneq \Gal(\Q^S/\Q)$. 
\end{itemize}
Let $A$ be a topological commutative ring. Given a continuous character $\varsigma: \Gal \rightarrow A^\times$. Observe that $A[\Gal]$ is then an involutive $A$-algebra with involution \[
    \alpha \mapsto \varsigma(\alpha)\alpha^{-1}. 
\]
We abuse the notation and denote this involutive $A$-algebra by $(A[\Gal], \varsigma)$.

\begin{Definition}[$P$-refined families of symplectic Galois representations]\label{Definition: refined families via determinant}
    A \textbf{$P$-refined family of symplectic $\Gal$-representations} of dimension $2n$ is a datum $(\calX, \calX_{\nc}, \varsigma_{\calX}, \Det, \Pf, [\mu_{\HTS}], F_1, F_2)$, where \begin{itemize}
        \item $\calX$ is a reduced rigid analytic space over $\Spa(\Q_p, \Z_p)$; 
        \item $\calX_{\nc}\subset \calX(\overline{\Q}_p)$ is a Zariski dense set of points (called non-critical classical points); 
        \item $\varsigma_{\calX}: \Gal \rightarrow \scrO_{\calX}(\calX)^\times$ is a continuous character; 
        \item $(\Det, \Pf)$ is a continuous $2n$-dimensional symplectic determinant on the involutive algebra $(\scrO_{\calX}(\calX)[\Gal], \varsigma_{\calX})$;
        \item $[\mu_{\HTS}]$ is a $M_P$-conjugacy class of continuous cocharacters $\mu_{\HTS}: \Gamma \rightarrow G(\scrO_{\calX_{\Q_p^{\cyc}}}(\calX_{\Q_p^{\cyc}}))\rtimes \Gamma$;
        \item $F_i\in \scrO_{\calX}(\calX)$ for $i=1, 2$,
    \end{itemize}
    such that the following conditions hold: \begin{enumerate}[leftmargin=0.64in]
        \item[(RFGal1)] For any $x\in \calX_{\nc}$, let $\rho_x:\Gal \rightarrow G(\overline{\Q}_p)$ be the symplectic Galois representation associated with $(\Det_x, \Pf_x)$, the Hodge--Tate cocharacter for $\D_{\rig}^{\dagger}(\rho_x|_{\Gal_{\Q_p}})$ is $\mu_{\HTS}(x)$.
        \item[(RFGal2)] For any $x\in \calX_{\nc}$, $\D_{\rig}^{\dagger}(\rho_x|_{\Gal_{\Q_p}})$ is potentially semistable. 
        \item[(RFGal3)] For any $x\in \calX_{\nc}$, $\D_{\rig}^{\dagger}(\rho_x|_{\Gal_{\Q_p}})$, up to $M_P$-conjugation, \[
            \mu_{\HTS}(x): \gamma \mapsto ([\diag(\chi_{\cyc}(\gamma)^{\mu_{\HTS, 1}(x)}, ..., \chi_{\cyc}(\gamma)^{\mu_{\HTS, 2n}(x)})], \gamma)
        \] 
        with each $\mu_{\HTS, i}(x)\in \Z$. Moreover, \[
            \mu_{\HTS, 1}(x), ..., \mu_{\HTS, m}(x) < \mu_{\HTS, m+1}(x), ..., \mu_{\HTS, 2n-m}(x)< \mu_{\HTS, 2n-m+1}(x), ..., \mu_{\HTS, 2n}(x).\footnote{ This sequence of inequalities means that $\max\{\mu_{\HTS, 1}(x), ..., \mu_{\HTS, m}(x)\} < \min\{\mu_{\HTS, m+1}(x), ..., \mu_{\HTS, 2n-m}(x)\}$ and $\max\{\mu_{\HTS, m+1}(x), ..., \mu_{\HTS, 2n-m}(x)\} < \min\{\mu_{\HTS, 2n-m+1}(x), ..., \mu_{\HTS, 2n}(x)\}$.}
        \]
        \item[(RFGal4)] For any $x\in \calX_{\nc}$, $\varphi_1 = p^{\sum_{i=1}^m\mu_{\HTS,i}(x)}F_1(x)$ and $\varphi_2 = p^{\sum_{i=2n-m+1}^{2n}\mu_{\HTS,i}(x)}F_2(x)$ are distinct Frobenius eigenvalues on $\bigwedge^m \D_{\pst}(\mathrm{std}\circ \rho_x|_{\Gal_{\Q_p}})$ with multiplicity one. Their corresponding eigenspaces are determinants of two $m$-dimensional $\varphi$-stable isotropic subspaces $M_i$ in $\D_{\pst}(\mathrm{std}\circ \rho_x|_{\Gal_{\Q_p}})$, dual to each other under the symplectic pairing, and $M_1$ is $(\varphi, N, \Gal_{\Q_p})$-stable. 
        \item[(RFGal5)] The subset $\calX_{\nc}$ is an accumulation subset, i.e., for any $x\in \calX_{\nc}$, there exists a basis of affinoid neighbourhoods $\calU$ of $x$ such that $\calU \cap \calX_{\nc}$ is Zariski dense in $\calU$.
    \end{enumerate} 
\end{Definition}

\begin{Definition}[$P$-refined families of $(\varphi, \Gamma)$-modules]\label{Definition: refined families of (phi, Gamma)-modules}
    A \textbf{$P$-refined family of $(\varphi, \Gamma)$-modules} is a datum $(\calX, \calX_{\nc}, D, \delta_1, \delta_2)$, where \begin{itemize}
        \item $\calX = \Spa(A, A^\circ)$ is a reduced affinoid rigid analytic space over $\Spa(\Q_p, \Z_p)$; 
        \item $\calX_{\nc} \subset \calX(\overline{\Q}_p)$ is a Zariski dense set of points; 
        \item $D\in G\Mod_A^{(\varphi, \Gamma)}$ is a $(\varphi, \Gamma)$-module with $G$-structure over $A$; 
        \item $\delta_i: \Q_p^\times \rightarrow A^\times$ is a continuous character for $i=1,2$,
    \end{itemize}
    such that the following conditions hold. \begin{enumerate}[leftmargin=0.44in]
        \item[(RF1)] For any $x\in \calX(\overline{\Q}_p)$, $\wt(\delta_{1, x})$ and $\wt(\delta_{2,x})$\footnote{ For the weight of a character, we refer the readers to \cite[Sect. 2.3.3]{BC}.} are roots of the Sen polynomial of $\bigwedge^m D_x(\mathrm{std})$. 
        \item[(RF2)] For any $x\in \calX_{\nc}$, $D_x$ is potentially semistable. 
        \item[(RF3)] For any $x\in \calX_{\nc}$, $\wt(\delta_{1, x}), \wt(\delta_{2,x})\in \Z$,  $\wt(\delta_{1, x})<\wt(\delta_{2,x})$, and of multiplicity one as roots of the Sen polynomial of $\bigwedge^m D_x(\mathrm{std})$.
        \item[(RF4)] For any $x\in \calX_{\nc}$, $\varphi_1(x) = \delta_{1, x}(p)p^{\wt(\delta_{1, x})}$ and $\varphi_2(x) = \delta_{2, x}(p)p^{\wt(\delta_{2,x})}$ are distinct Frobenius eigenspaces on $\bigwedge^m \calD_{\pst}(D_x(\mathrm{std}))$ with multiplicity one. Their corresponding eigenspaces are determinants of two $m$-dimensional $\varphi$-stable isotropic subspaces $M_i$ in $\calD_{\pst}(\mathrm{std}\circ \rho_x|_{\Gal_{\Q_p}})$, dual to each other under the symplectic pairing, and $M_1$ is $(\varphi, N, \Gal_{\Q_p})$-stable. 
        \item[(RF5)] The subset $\calX_{\nc}$ is an accumulation subset. 
    \end{enumerate} 
\end{Definition}

\begin{Definition}[Pointwise $P$-refined families of $(\varphi, \Gamma)$-modules]\label{Definition: pointwise refined families of (phi, Gamma)-modules}
    A \textbf{pointwise $P$-refined family of $(\varphi, \Gamma)$-modules} is a datum $(\calX, \calX_{\nc}, D, \{\Fil_{\bullet}D_x\}_{x\in \calX_{\nc}}, \delta_1, \delta_2)$, where \begin{itemize}
        \item $\calX = \Spa(A, A^\circ)$ is a reduced affinoid rigid analytic space over $\Spa(\Q_p, \Z_p)$; 
        \item $\calX_{\nc} \subset \calX(\overline{\Q}_p)$ is a Zariski dense set of points; 
        \item $D\in G\Mod_A^{(\varphi, \Gamma)}$ is a $(\varphi, \Gamma)$-module with $G$-structure over $A$; 
        \item for any $x\in \calX_{\nc}$, $\Fil_{\bullet}D_x$ is a $P$-flag on $D_x$;
        \item $\delta_i: \Q_p^\times \rightarrow A^\times$ is a continuous character for $i=1,2$,
    \end{itemize}
    such that the following conditions hold.  \begin{enumerate}[leftmargin=0.57in]
        \item[(RFPt1)] For any $x\in \calX(\overline{\Q}_p)$, $\wt(\delta_{,1, x})$ and $\wt(\delta_{2, x})$ are roots of the Sen polynomial of $\bigwedge^m D_x(\mathrm{std})$.
        \item[(RFPt2)] For any $x\in \calX_{\nc}$, $D_x$ is potentially semistable.
        \item[(RFPt3)] For any $x\in \calX_{\nc}$, $\wt(\delta_{1,x}), \wt(\delta_{2,x})\in \Z$, $\wt(\delta_{1, x})<\wt(\delta_{2,x})$, and of multiplicity one as roots of the Sen polynomial of $\bigwedge^m D_x(\mathrm{std})$.
        \item[(RFPt4)] For any $x\in \calX_{\nc}$, there exists a filtration on $\bigwedge^m D_x(\mathrm{std})$ given by \[
            \calR_{L_x}(\delta_{1, x}) \subset D_x^{(\wedge m)} \subset D_x
        \] such that \[
            D_x/ D_x^{(\wedge m)} \cong \calR_{L_x}(\delta_{2, x}).
        \]  
        Here, $L_x$ is the residue field of $X$. 
        \item[(RFPt5)] For any $x\in \calX_{nc}$, the filtration in (RFPt3) agrees with the induced filtration on $\bigwedge^m D_x(\mathrm{std})$ by $\Fil_{\bullet}D_x$.
        \item[(RFPt6)] The subset $\calX_{\nc}$ is an accumulation subset. 
    \end{enumerate} 
\end{Definition}

\begin{Proposition}\label{Proposition: comparison between different notions of families}
    \begin{enumerate}
        \item[(i)] Let $(\calX, \calX_{\nc}, \lambda_{\calX}, \Det, \Pf, [\mu_{\HTS}], F_1, F_2)$ be a $P$-refined family of symplectic $\Gal$-representations. Then, for any $x\in \calX(\overline{\Q}_p)$ whose $\Gal$-representation $\rho_x: \Gal \rightarrow G(\overline{\Q}_p)$ associated with $(\Det_x, \Pf_x)$ is irreducible, there exists an affinoid neighbourhood $\calU = \Spa(A, A^{\circ})\subset \calX$ of $x$ and a $P$-refined family of $(\varphi, \Gamma)$-modules $(\calU, \calU(\overline{\Q}_p)\cap \calX_{\nc}, D, \delta_1, \delta_2)$ satisfying the following properties:  \begin{itemize}
            \item[(a)] For any $y\in \calU(\overline{\Q}_p)$, $D_y = \D_{\rig}^{\dagger}(\rho_y|_{\Gal_{\Q_p}})$, where $\rho_y: \Gal \rightarrow G(\overline{\Q}_p)$ is the $\Gal$-representation associated with $(\Det_y, \Pf_y)$. 
            \item[(b)] For any $y\in \calU(\overline{\Q}_p) \cap \calX_{\nc}$, \[
                \wt(\delta_{1, y}) = \sum_{j=1}^m \mu_{\HTS, j}(y) \quad \text{ and }\quad \wt(\delta_{2, y}) = \sum_{j=2n-m+1}^{2n} \mu_{\HTS, j}(y)
            \]
            and \[
                \delta_{i,y}(p) = F_i(y)
            \]
            for $i=1, 2$.
        \end{itemize}
        \item[(ii)] Let $(\calX, \calX_{\nc}, D, \delta_1, \delta_2)$ be a $P$-refined family of $(\varphi, \Gamma)$-modules. Then, for any $x\in \calX_{\nc}$, there exists a $P$-flag $\Fil_{\bullet}D_x$ such that $(\calX, \calX_{\nc}, \{\Fil_{\bullet}D_x\}_{x\in \calX_{\nc}}, \delta_1, \delta_2)$ is a pointwise $P$-refined family of $(\varphi, \Gamma)$-modules. 
    \end{enumerate}
\end{Proposition}
\begin{proof}[Proof of Proposition \ref{Proposition: comparison between different notions of families} (i)]
    Let $x\in \calX(\overline{\Q}_p)$ such that $\rho_x: \Gal \rightarrow G(\overline{\Q}_p)$ is irreducible. Then, by Theorem \ref{Theorem: symplectic version of Chenevier's Thm A and B} (ii), there exists a $\Gal$-representation \[
        \widetilde{\rho}_x : \Gal \rightarrow G(\scrO_{\calX, x})
    \] whose associated symplectic determinant is $(\Det, \Pf)\otimes_{\scrO_{\calX}(\calX)} \scrO_{\calX, x}$. Therefore, there exists an affinoid neighbourhood $\calU = \Spa(A, A^{\circ}) \subset \calX$ of $x$ and a $\Gal$-representation \[
        \rho_A: \Gal \rightarrow G(A)
    \] such that $\rho_A \otimes_A \scrO_{\calX, x} = \widetilde{\rho}_x$ and whose associated symplectic determinant agrees with $(\Det, \Pf)\otimes_{\scrO_{\calX}(\calX)}A$. We then define \[
        D \coloneq \D_{\rig}^{\dagger}(\rho_A|_{\Gal_{\Q_p}}).
    \]
    
    The continuous characters $\delta_i$ are defined as follows: \[
        \delta_i(p) \coloneq F_i \quad \text{ and }\quad \delta_i|_{\Z_p^\times} = \left\{ \begin{array}{ll}
           -\dfrac{\log (\det([\gamma]_{st})_{1\leq s,t\leq m})}{\log \chi_{\cyc}(\gamma)},  & i=1  \\ \\
           -\dfrac{\log(\det([\gamma]_{st})_{2n-m+1 \leq s,t\leq 2n})}{\log \chi_{\cyc}(\gamma)},  & i=2 
        \end{array} \right. 
    \]
    for any representative $\mu_{\HTS}$ in the $M_P$-conjugacy class $[\mu_{\HTS}]$. Here, recall that $\mu_{\HTS}(\gamma) = ([\gamma], \gamma)$ with $[\gamma]\in G(\scrO_{\calX_{\Q_p^{\cyc}}}(\calX_{\Q_p^{\cyc}}))$.  Then, by construction, (a) and (b) hold. 

    To conclude the proof, we verify (RF1) -- (RF5). However, these follow easily from the construction and (RFGal1)-(RFGal5).
\end{proof}

\begin{proof}[Proof of Proposition \ref{Proposition: comparison between different notions of families} (ii)]
    We have to construct the $P$-flags $\Fil_{\bullet}D_x$ for each $x\in \calX_{\nc}$. From (RF4), we have a $(\varphi, N, \Gal_{\Q_p})$-equivariant inclusion \[
        M_1 \subset \calD_{\pst}(\mathrm{std}\circ \rho_x|_{\Gal_{\Q_p}}).
    \]
    Let $M^{\perp}_1$ be the subspace in $\calD_{\pst}(\mathrm{std}\circ \rho_x|_{\Gal_{\Q_p}})$ that is perpendicular to $M_1$ with respect to the symplectic pairing \eqref{eq: defining symplectic pairing}. Then, since $M_1$ is isotropic, we have \begin{equation}\label{eq: construction of pointwise filtration}
        M_1 \subset M_1^{\perp} \subset \calD_{\pst}(\mathrm{std}\circ \rho_x|_{\Gal_{\Q_p}})
    \end{equation} with \[
        \calD_{\pst}(\mathrm{std}\circ \rho_x|_{\Gal_{\Q_p}}) / M_1^{\perp} \cong M_2.
    \]
    We claim that \eqref{eq: construction of pointwise filtration} is $(\varphi, N, \Gal_{\Q_p})$-stable. Since both the Frobenius and $\Gal_{\Q_p}$-actions are given by automorphisms, it is easy to see that $M_1^{\perp}$ is $\varphi$- and $\Gal_{\Q_p}$-stable. To check that $M_1^{\perp}$ is monodromy-stable, note that \[
        \langle Nv, w \rangle + \langle v, Nw\rangle = 0
    \]
    for all $v, w\in \calD_{\pst}(\mathrm{std}\circ \rho_x|_{\Gal_{\Q_p}})$ and $\langle \cdot, \cdot\rangle$ is the induced symplectic pairing by \eqref{eq: defining symplectic pairing}. Hence, for any $v\in M_1$ and $w\in M_1^{\perp}$, we have \[
        \langle v, Nw \rangle = -\langle Nv, w\rangle = 0    
    \]
    since $M_1$ is $N$-stable; thus $Nw\in M_1^{\perp}$ and we conclude the claim. 
    
    Then, we let $\Fil_{\bullet}D_x(\mathrm{std})$ be the corresponding flag of \eqref{eq: construction of pointwise filtration} on $D_x(\mathrm{std})$. However, by the proof of Lemma \ref{Lemma: P-non-criticality can be checked on faithful representations}, we have a $P$-flag $\Fil_{\bullet}D_{x}$. Then, by construction, (RFPt4) -- (RFPt6) are satisfied. 
\end{proof}

\begin{Remark}\label{Remark: when the parabolic subgroup is not maximal}
    One can generalise the discussions above to more general parabolic subgroups $P \subset G$. For example, if $P = \bigcap_{j=1}^t P_{m_j}$ for some maximal parabolic subgroups $P_{m_j}$'s, then a $P$-refined family of symplectic $\Gal$-representations of dimension $2n$ is a datum \[
        (\calX, \calX_{nc}, \lambda_{\calX}, \Det, \Pf, [\mu_{\HTS}], \{F_1^{(m_j)}, F_2^{(m_j)}\}_{j=1, ..., t})
    \] where $(\calX, \calX_{nc}, \lambda_{\calX}, \Det, \Pf, [\mu_{\HTS}], F_1^{(m_j)}, F_2^{(m_j)})$ is a $P_{m_j}$-refined family of symplectic $\Gal$-representations. When $P=B$ is the fixed Borel subgroup, this definition is then similar to \cite[Definition 4.2.3]{BC}. One can similarly generalise the notation of $P$-refined families of $(\varphi, \Gamma)$-modules and pointwise $P$-refined families of $(\varphi, \Gamma)$-modules. We leave the details to the interested readers. 
\end{Remark}

\subsection{Deformation problems}\label{subsection: deformation}
In this subsection, we formulate the deformation problems of our interest. To this end, we fix a Galois representation \[
    \rho: \Gal_{\Q} \rightarrow \GSp_{2n}(L),
\]
where $L$ is a (large enough) finite extension of $\Q_p$ such that it contains all eigenvalues of the Frobenius maps, such that \begin{itemize}
    \item $\rho$ is absolutely irreducible;
    \item $\rho$ is unramified at all most all primes $\ell$;
    \item the local representation $\rho_p \coloneq \rho|_{\Gal_{\Q_p}}$ is de Rham and has Hodge--Tate cocharacter $\mu_{\rho, \HT} = (a_1, ..., a_n; a_0)\in \Z^{n+1}$ as \[
        (a_1, ..., a_n; a_0) : \bbG_m \rightarrow \GSp_{2n}, \quad u \mapsto \diag(u^{a_1}, ..., u^{a_n}, u^{a_0-a_n}, ..., u^{a_0-a_1})
    \] 
    such that $a_1< ...< a_m < a_{m+1}< ...< a_n< a_0-a_n< ...< a_0-a_{m+1} < a_0-a_{m}< ...< a_0-a_1$.\footnote{ Note that this is stronger assumption then before. }
\end{itemize}
We fix a finite set of primes $S$ containing $p$ and the primes at which $\rho$ is ramified. 

We again work with the Borel subgroup $B \subset G = \GSp_{2n}$ of upper triangular matrices and fix a maximal parabolic subgroup $P = P_m$ (with $m\leq n$) containing $B$. Let $M_P$ be the Levi decomposition of $P$ and $Z_{M_P}$ be the centre of $M_P$; it is easy to see that \[
    Z_{M_P} \cong \bbG_m^2, \quad \diag(a_1\one_m, a_2\one_{2n-2m}, a_1^{-1}a_2^2\one_m) \mapsfrom (a_1, a_2)
\] We further assume $\rho$ satisfies the following properties: \begin{enumerate}[leftmargin=0.63in]
    \item[($P$-REG)] There exists a $P$-refinement $\Fil_{\bullet}\D_{\pst}(\rho_p)$ on $\D_{\pst}(\rho_p)$ such that, if $M_1$ (resp., $M_2$) is the $m$-dimensional $(\varphi, N, \Gal_{\Q_p})$-stable isotropic subspace (resp., quotient) given by $\Fil_{\bullet} \D_{\pst}(\mathrm{std}\circ \rho_p)$, $\det M_i$ has Frobenius eigenvalue $\varphi_i$ which are distinct with multiplicity one in $\bigwedge^m \D_{\pst}(\mathrm{std}\circ \rho_p)$. 
    \item[($P$-NCR)] Let $\Fil_{\bullet}\D_{\rig}^{\dagger}(\rho_p)$ be the corresponding $P$-flag of $\Fil_{\bullet}\D_{\pst}(\rho_p)$ on $\D_{\rig}^{\dagger}(\rho_p)$ (Proposition \ref{Proposition: P-refinements are equivalent to P-flags}). The $P$-flagged $(\varphi, \Gamma)$-module with $G$-structure $(\D_{\rig}^{\dagger}(\rho), \Fil_{\bullet}\D_{\rig}^{\dagger}(\rho))$ is $P$-non-critical (Definition \ref{Definition: P-non-critical}).
\end{enumerate}

Let $\Ar_L$ be the category of local artinian $\Q_p$-algebras whose local field is isomorphic to $L$. In what follows, we shall consider the following deformation problems: 

\paragraph{Local deformation problems at good primes.} For any $\ell \not \in S$, we consider the unramified deformation problem \[
    \Def_{\rho_{\ell}}^{\unr} : \Ar_L \rightarrow \Set, \quad A \mapsto \left\{ \rho_A: \Gal_{\Q_{\ell}} \rightarrow \GSp_{2n}(A): \begin{array}{l}
        \rho_A \otimes_{A}L \cong \rho_{\ell}  \\
        \rho_A\text{ is unramified} 
    \end{array} \right\}/\cong.
\]
The tangent space of $\Def_{\rho_{\ell}}^{\unr}$ is given by \[
    W_{\ell}^{\unr} \coloneq \Def_{\rho_{\ell}}^{\unr}(L[\epsilon])
\] with $\epsilon^2 = 0$. It is well-known that \[
    W_{\ell}^{\unr} = H^1_{\unr}(\Q_{\ell}, \ad \rho_{\ell}) \coloneq \ker\left( H^1(\Q_{\ell}, \ad \rho_{\ell}) \rightarrow H^1(I_{\ell}, \ad \rho_{\ell}) \right),
\]
where $I_{\ell}$ is the inertia subgroup at $\ell$ and $\ad \rho_{\ell}$ is the adjoint representation.

\paragraph{Local deformation problems at bad primes.} For any $\ell \in S \smallsetminus \{p\}$, we consider the minimal deformation problem \[
    \Def_{\rho_{\ell}}^{\min} : \Ar_L \rightarrow \Set, \quad A \mapsto \left\{ \rho_A: \Gal_{\Q_{\ell}} \rightarrow \GSp_{2n}(A): \begin{array}{l}
        \rho_A \otimes_{A} L \cong \rho_{\ell}  \\
        \rho_A|_{I_{\ell}} \cong \rho_{\ell}|_{I_{\ell}}\otimes_{\Q_p}A 
    \end{array} \right\}/\cong.
\]
The tangent space of $\Def_{\rho_{\ell}}^{\min}$ is given by \[
    W_{\ell}^{\min} \coloneq \Def_{\rho_{\ell}}^{\min}(L[\epsilon]) \subset H^1(\Q_{\ell}, \ad \rho_{\ell})
\]
with $\epsilon^2 = 0$.

\begin{Lemma}\label{Lemma: need not worry about minimal deformation}
    Let $\ell \in S \smallsetminus \{p\}$.  Assume that $\rho_{\ell}|_{I_{\ell}}$ is irreducible and $p \nmid \ell^{\mathrm{lcm}(1, ..., 2n)}-1$. Then, for any $A\in \Ar_L$ and any $\rho_A : \Gal_{\Q_{\ell}} \rightarrow \GSp_{2n}(A)$ such that $\rho_A \otimes_{\Q_p} L \cong \rho_{\ell}$, we have \[
        \rho_A(I_{\ell}) = \rho_{\ell}(I_{\ell}).
    \]
    In this case, $W_{\ell}^{\min} = H^1(\Q_{\ell}, \ad \rho_{\ell})$.
\end{Lemma}
\begin{proof}
    The proof is an easy generalisation of \cite[Lemma 4.3.6]{GT-TWGSp4}.

    Since $\ker(\GSp_{2n}(A) \rightarrow \GSp_{2n}(L))$ is locally pro-$p$, it is enough to show that $\rho_A(I_{\ell})$ is finite of order prime to $p$. Note that $I_{\ell}$ sits inside a split short exact sequence \[
        1 \rightarrow I_{\ell}^{(\ell)} \rightarrow I_{\ell} \rightarrow \prod_{q\neq \ell}\Z_q(1) \rightarrow 1,
    \]
    where $I_{\ell}^{(\ell)}$ is the pro-$\ell$ Sylow subgroup of $I_{\ell}$, which is the Galois group of the maximal tamely ramified extension of $\Q_{\ell}$. Thus, it is enough to show that $\rho_A(\Z_p(1))$ is trivial. 
    
    To show this, we first claim that $\rho_A(\xi)$ is  unipotent for any topological generator $\xi \in \Z_p(1)$. Suppose it is not unipotent, then $\rho_A(\xi)$ has an eigenvalue $\alpha \neq 1$. by conjugating with the Frobenius element $\Frob_{\ell}$, we see that $\rho_A(\xi)$ and $\rho_A(\xi^{\ell})$ have the same eigenvalues. By iterating this process, one sees that $\{\alpha, \alpha^{\ell}, ..., \alpha^{\ell^{2n-1}}\}$ is a subset of eigenvalues of $\rho_{A}(\xi)$ and $\{\alpha^{\ell}, \alpha^{\ell^2}, ..., \alpha^{\ell^{2n}}\}$ is a subset of eigenvalues of $\rho_A(\xi^{\ell})$. Therefore, we learn that \[
        \alpha = \alpha^{\mathrm{lcm}(1, ..., 2n)}.
    \] 
    Since $\xi$ is a topological generator of $\Z_p(1)$, $\alpha$ must be a $p$-power root of unity and so \[
        p | \mathrm{lcm}(1, ..., 2n)-1,
    \]
    which contradicts to our assumption. 

    To conclude the proof, suppose $\rho_A(\xi)$ is not trivial, it would then fix an $A$-submodule of $A^{2n}$, which is stable under $I_{\ell}$. Hence, $\rho_A|_{I_{\ell}}$ is reducible. Since $\rho_A|_{I_{\ell}} \otimes_{A}L \cong \rho_{\ell}|_{I_{\ell}}$, we see that $\rho_{\ell}|_{I_{\ell}}$ is reducible, which contradicts to the assumption.  
\end{proof}

\begin{Remark}\label{Remark: minimal deformation is no problem}
    One easily sees that $\GSp_{2n}$ plays no special role in the proof of Lemma \ref{Lemma: need not worry about minimal deformation}. The statement and the proof generalise verbatim to any $d$-dimensional Galois representations. 
\end{Remark}

\paragraph{Local deformation problems at $p$.}  For $\ell = p$, we consider the following deformation problems:
\begin{enumerate}
    \item[(I)] \emph{The small $P$-flagged deformation} is a functor \[
        \Def_{\rho_p, \Fil_{\bullet}}^{\mathrm{small}}: \Ar_L \rightarrow \Set
    \] 
    sending each $A$ to the isomorphism classes of pairs $(\rho_A, \Fil_{\bullet} \bigwedge^m\D_{\rig}^{\dagger}(\mathrm{std} \circ \rho_A))$ such that \begin{itemize}
        \item $\rho_A: \Gal_{\Q_p} \rightarrow \GSp_{2n}(A)$ and $\rho_A \otimes_{A} L \cong \rho_p$; 
        \item let $[\mu_{\rho_A, \HTS}]$ be the $M_P$-conjugacy class of the Hodge--Tate--Sen cocharacter for $\D_{\rig}^{\dagger}(\rho_A)$, then it has a representative of the form \[
            \gamma \mapsto (\diag(1, \mu_{\rho_A, \HTS, 2}(\gamma), \mu_{\rho_A, \HTS, 3}(\gamma), \mu_{\rho_A, \HTS, 4}(\gamma)), \gamma)
        \]
        where $\mu_{\rho_A, \HTS, i}: \Gamma \rightarrow (A\widehat{\otimes}_{\Q_p}\Q_p^{\cyc})^\times$ is a continuous group homomorphism.
        \item The $M_P$-conjugacy class $[\frac{\mu_{\rho_A, \HTS}}{\mu_{\rho, \HT}}]$ factors through the centre $Z_{M_P}$ of $M_P$;
        \item $\Fil_{\bullet} \bigwedge^m\D_{\rig}^{\dagger}(\mathrm{std} \circ \rho_A)$ is a flag on $\bigwedge^m\D_{\rig}^{\dagger}(\mathrm{std} \circ \rho_A)$ such that  $\Fil_{\bullet} \bigwedge^m\D_{\rig}^{\dagger}(\mathrm{std} \circ \rho_A) \otimes_{\Q}L \cong \Fil_{\bullet} \bigwedge^m \D_{\rig}^{\dagger}(\mathrm{std} \circ\rho_p)$, where $\Fil_{\bullet} \bigwedge^m \D_{\rig}^{\dagger}(\mathrm{std} \circ\rho_p)$ is the flag induced by the $P$-flag $\Fil_{\bullet} \D_{\rig}^{\dagger}(\rho_p)$.
    \end{itemize}
    \item[(II)] \emph{The de Rham deformation} is the functor \[
        \Def_{\rho_p, g} : \Ar_L \rightarrow \Set, \quad A \mapsto \left\{ \rho_A: \Gal_{\Q_p} \rightarrow \GSp_{2n}(A): \begin{array}{l}
            \rho_A \otimes_{A}L \cong \rho_p  \\
            \rho_A \text{ is de Rham}
        \end{array} \right\} /\cong.
    \]
    Then, the tangent space of $\Def_{\rho_p, \dR}$ is given by \[
        W_{p}^{\dR} \coloneq \Def_{\rho_p, \dR}(L[\epsilon])
    \]
    with $\epsilon^2 = 0$. It is well-known that \[
        W_p^{\dR} = H^1_g(\Q_p, \ad \rho_p) = \ker\left( H^1(\Q_p, \ad \rho_p) \rightarrow H^1(\Q_p, \ad \rho_p \otimes_{\Q_p} B_{\dR}) \right),
    \]
    where $B_{\dR}$ is Fontaine's de Rham period ring. 
    \item[(III)] \emph{The small $P$-flagged de Rham deformation} is a functor \[
        \Def_{\rho_p, \dR}^\mathrm{small}: \Ar_L \rightarrow \Set
    \]
    sending each $A$ to the isomorphism classes of Galois representations $\rho_A: \Gal_{\Q_p} \rightarrow \GSp_{2n}(A)$ such that \begin{itemize}
        \item $\rho_A \otimes_{A}L \cong \rho_p$;  
        \item $\rho_A$ is de Rham;
         \item let $[\mu_{\rho_A, \HTS}]$ be the $M_P$-conjugacy class of the Hodge--Tate--Sen cocharacter for $\D_{\rig}^{\dagger}(\rho_A)$, then it has a representative of the form \[
            \gamma \mapsto (\diag(1, \mu_{\rho_A, \HTS, 2}(\gamma), \mu_{\rho_A, \HTS, 3}(\gamma), \mu_{\rho_A, \HTS, 4}(\gamma)), \gamma)
        \]
        where $\mu_{\rho_A, \HTS, i}: \Gamma \rightarrow (A\widehat{\otimes}_{\Q_p}\Q_p^{\cyc})^\times$ is a continuous group homomorphism.
        \item The $M_P$-conjugacy class $[\frac{\mu_{\rho_A, \HTS}}{\mu_{\rho, \HT}}]$ factors through the centre $Z_{M_P}$ of $M_P$;
    \end{itemize}
    Then, the tangent space of $\Def_{\rho_p, \dR}^{\mathrm{small}}$ is given by \[
        W_p^{\dR, \mathrm{small}} \coloneq \Def_{\rho_p, \dR}^{\mathrm{small}}(L[\epsilon])
    \] with $\epsilon^2 = 0$. 
\end{enumerate}

\begin{Lemma}\label{Lemma: P-flagged deformation is a subfunctor of unconditional deformation}
    Let $\Def_{\rho_p}$ be the unconditional deformation problem for $\rho_p$. Then, the forgetful functor \[
        \Def_{\rho_p, \Fil_{\bullet}}^{\mathrm{small}}(A) \rightarrow \Def_{\rho_p}(A), \quad (\rho_A, \Fil_{\bullet}\bigwedge^m\D_{\rig}^{\dagger}(\mathrm{std}\circ \rho_A)) \mapsto \rho_A
    \]
    makes $\Def_{\rho_p, \Fil_{\bullet}}^{\mathrm{small}}$ a subfunctor of $\Def_{\rho_p}$.
\end{Lemma}
\begin{proof}
    We have to show that, if a lift $\Fil_{\bullet}\bigwedge^m\D_{\rig}^{\dagger}(\mathrm{std}\circ \rho_A)$ of $\Fil_{\bullet}\bigwedge^m\D_{\rig}^{\dagger}(\mathrm{std}\circ \rho_p)$ exists, then it is unique. 

    Consider the $(\varphi, \Gamma)$-modules $\D_{\rig}^{\dagger}(\mathrm{std}\circ \rho_p)$ and $\D_{\rig}^{\dagger}(\mathrm{std}\circ \rho_A)$. We denote by $D_1$ (resp., $\widetilde{D}_1$) the rank-$1$ $(\varphi, \Gamma)$-submodules of $\bigwedge^m \D_{\rig}^{\dagger}(\mathrm{std}\circ \rho_p)$ (resp., $\bigwedge^m \D_{\rig}^{\dagger}(\mathrm{std}\circ\rho_A)$) given by the filtration. To show the lemma, we have to show that $\widetilde{D}_1$ is the unique $(\varphi, \Gamma)$-submodule of $\bigwedge^m \D_{\rig}^{\dagger}(\mathrm{std}\circ \rho_A)$ lifting $D_1 \subset \bigwedge^m \D_{\rig}^{\dagger}(\mathrm{std}\circ \rho_p)$.

    By the classification of rank-$1$ $(\varphi, \Gamma)$-modules, we know that $D_1 \cong \calR_L(\delta_1)$ for some continuous character $\delta_1: \Q_p^\times \rightarrow L$. In fact, under the assumptions of ($P$-REG) and ($P$-NCR), we can write down $\delta$ explicitly as following. Recall $\varphi_1$ is the Frobenius eigenvalue of $\det M_1 \subset \bigwedge^m \D_{\pst}(\mathrm{std}\circ \rho_p)$; and note that $\calD_{\pst}(D_1) = \det M_1$ (\cite[Théorèm A]{Berger-phiN}). 
    Therefore, $\delta_1$ is defined by \[
        \delta_1(p) = \varphi_1 \quad \text{ and }\quad \delta_1(\gamma) = \gamma^{-\sum_{i=1}^m a_i} 
    \] for all $\gamma \in \Z_p^\times$.

    We claim that \[
        \Hom(\calR_L(\delta_1), \bigwedge^m \D_{\rig}^{\dagger}(\mathrm{std}\circ \rho_p)/D_1) = 0.
    \]
    Indeed, if there exists a non-zero morphism $\Psi$, define \[
        D'' \coloneq \text{preimage of $\image(\Psi)$ in $\bigwedge^m\D_{\rig}^{\dagger}(\mathrm{std}\circ \rho_p)$}.
    \]
    Note that $\image(\Psi)$ is isomorphic to $\calR_L(\delta_1)$, thus $D''$ sits in the short exact sequence \[
        0 \rightarrow D_1 \cong \calR_L(\delta_1) \rightarrow D'' \rightarrow \calR_L(\delta_1) \cong \image(\Psi) \rightarrow 0.
    \]
    However, consider $\calD_{\pst}(D'')$, the Frobenius has characteristic polynomial $(X-\varphi_1)^2$, which contradict to ($P$-REG). 

    To finish the proof, by applying \cite[Lemma 2.3.7]{BC}, we see that $\widetilde{D}_1$ is the unique $(\varphi, \Gamma)$-submodule of $\bigwedge^m \D_{\rig}^{\dagger}(\mathrm{std}\circ \rho_A)$ lifting $D_1$. 
\end{proof}

\begin{Lemma}\label{Lemma: relationships between de Rham deformations}
    Keep the notations as above. The deformation problem $\Def_{\rho_p, \dR}^{\mathrm{small}}$ is naturally a subfunctor of $\Def_{\rho_p, \Fil_{\bullet}}^{\mathrm{small}}$.
\end{Lemma}
\begin{proof}
    Given $A\in \Ar_L$ with maximal ideal $\frakm_A$ and $\rho_A \in \Def_{\rho_p, \dR}^{\mathrm{small}}(A)$, we claim that there exists a unique $\Fil_{\bullet} \bigwedge^m\D_{\rig}^{\dagger}(\mathrm{std}\circ \rho_A)$ lifting $\Fil_{\bullet} \bigwedge^m\D_{\rig}^{\dagger}(\mathrm{std}\circ \rho_p)$. By Proposition \ref{Proposition: P-refinements are equivalent to P-flags}, it is enough to show there exists a unique flag on $\bigwedge^m\D_{\pst}(\mathrm{std}\circ \rho_A)$ lifting $\Fil_{\bullet}\bigwedge^m\D_{\pst}(\mathrm{std}\circ \rho_p)$.

    Let $M_1 \subset \D_{\pst}(\mathrm{std}\circ \rho_p)$ be the $(\varphi, N, \Gal_{\Q_p})$-stable (isotropic) subspace defined by the refinement $\Fil_{\bullet} \D_{\pst}(\rho_p)$. Let $\charpoly_{\rho_p}^{\varphi}(T)$ (resp., $\charpoly_{M_1}^{\varphi}(T)$) be the characteristic polynomial of $\varphi$ on $\bigwedge^m\D_{\pst}(\mathrm{std}\circ \rho_p)$ (resp., $\det M_1$). Since the Frobenius eigenvalue on $\det M_1$ has multiplicity one among the Frobenius eigenvalues on $\bigwedge^m \D_{\pst}(\mathrm{std}\circ \rho_p)$, we see that the space of $\varphi$-equivariant morphisms \[
        \Hom_{L\otimes_{\Q_p}\Q_p^{\unr}}^{\varphi}\left(\det M_1, \bigwedge^m\D_{\pst}(\mathrm{std}\circ \rho_p)/\det M_1\right) = 0.
    \]
    Therefore, \[
        \charpoly_{\rho_p}^{\varphi}(T) = \charpoly_{M_1}^{\varphi}(T) \charpoly_{\bigwedge^m\D_{\pst}(\mathrm{std}\circ \rho_p)/\det M_1}^{\varphi}(T)
    \]
    such that $\charpoly_{M_1}^{\varphi}(T)$ is coprime to $\charpoly_{\D_{\pst}(\bigwedge^m\mathrm{std}\circ \rho_p)/\det M_1}^{\varphi}(T)$. To simplify the notation, we write $\charpoly_{M_1}^{\varphi}(T) = f_1(T)$ and $\charpoly_{\bigwedge^m\D_{\pst}(\mathrm{std}\circ \rho_p)/\det M_1}^{\varphi}(T) = f_2(T)$

    Consider the Frobenius characteristic polynomial $\charpoly_{\rho_A}^{\varphi}(T)$ on $\bigwedge^m\D_{\pst}(\mathrm{std}\circ \rho_A)$. By Hensel's Lemma, we see that \[
        \charpoly_{\rho_A}^{\varphi}(T) = \widetilde{f}_1(T) \widetilde{f}_2(T)
    \] 
    with $\widetilde{f}_i(T)$'s being coprime to each other and $\widetilde{f}_i(T) \equiv f_i(T) \mod \frakm_A$. We then define \[
        \widetilde{M}_1 \coloneq \ker(\widetilde{f}_1).
    \]
    Note that, since $f_i$'s are coprime to each other, Frobenius-stable lifts of $\det M_1$ in $\bigwedge^m\D_{\pst}(\mathrm{std}\circ \rho_A)$ is uniquely determined. Additionally, since $N(\widetilde{M_1}) \otimes_{A}L \cong N(\det M_1) = 0$, Nakayam's Lemma implies that $N(\widetilde{M_1}) = 0$. Since the Frobenius and monodromy commute with the Galois action, we see that $\widetilde{M}_1$ ise $(\varphi, N, \Gal_{\Q_p})$-stable. As in the proof of Proposition \ref{Proposition: comparison between different notions of families} (ii), we have the filtration\[
        \widetilde{M}_1 \subset \widetilde{M}_1^{\perp} \subset \bigwedge^m\D_{\pst}(\mathrm{std}\circ \rho_A),
    \]
    lifting $\Fil_{\bullet}\bigwedge^m \D_{\pst}(\mathrm{std}\circ \rho_p)$, which concludes the proof. 
\end{proof}

\paragraph{Global deformation problems.} Following our discussions above, we can then consider three global deformation problems: \begin{equation}\label{eq: global deformation problems}
    \begin{split}
        \Def_{\rho, \Fil_{\bullet}}^{\mathrm{small}}  & \coloneq (\{\Def_{\rho_{\ell}}^{\unr}\}_{\ell\not\in S}, \{\Def_{\rho_{\ell}}^{\min}\}_{\ell\in S \smallsetminus \{p\}}, \Def_{\rho_p, \Fil_{\bullet}}^{\mathrm{small}}), \\
        \Def_{\rho, \dR}  & \coloneq (\{\Def_{\rho_{\ell}}^{\unr}\}_{\ell\not\in S}, \{\Def_{\rho_{\ell}}^{\min}\}_{\ell\in S \smallsetminus \{p\}}, \Def_{\rho_p, \dR}), \\
        \Def_{\rho, \dR}^{\mathrm{small}}  & \coloneq (\{\Def_{\rho_{\ell}}^{\unr}\}_{\ell\not\in S}, \{\Def_{\rho_{\ell}}^{\min}\}_{\ell\in S \smallsetminus \{p\}}, \Def_{\rho_p, \dR}^{\mathrm{small}}).
    \end{split}
\end{equation}

\begin{Proposition}\label{Proposition: representibility of deformation problems}
    The deformation problems defined above are all pro-representable. 
\end{Proposition}
\begin{proof}
    The pro-representabilities of the global deformation problems follow from the pro-representability of the local deformation problems. Among them, one can deduce the pro-representabilities of $\Def_{\rho_{\ell}}^{\unr}$, $\Def_{\rho_{\ell}}^{\min}$, and $\Def_{\rho_p, \dR}$ similarly as in \cite[Proposition 3.7]{HT-GLn}. Hence, it remains to show the pro-representabilities of $\Def_{\rho_p, \dR}^{\mathrm{small}}$ and $\Def_{\rho_p, \Fil_{\bullet}}^{\mathrm{small}}$. By Lemma \ref{Lemma: P-flagged deformation is a subfunctor of unconditional deformation} and Lemma \ref{Lemma: relationships between de Rham deformations}, we only need to verify that these deformation functors satisfy the \emph{deformation conditions} in \cite[Sect. 23]{Mazur-deformation}. As the proofs for both functors are similar, we only spell out the one for $\Def_{\rho_p, \Fil_{\bullet}}^{\mathrm{small}}$ in what follows and leave the other case to the reader.

    \noindent \textbf{Condition 1.} For $i=1,2$, let $A_i \in \Ar_L$ and $\rho_{A_i}\in \Def_{\rho_p}(A_i)$. If $f: A_1 \rightarrow A_2$ is a morphism in $\Ar_L$ such that $f$ induces a commutative diagram \[
        \begin{tikzcd}
            \Gal_{\Q_p} \arrow[r]\arrow[rd] & \GSp_{2n}(A_1)\arrow[d]\\
            & \GSp_{2n}(A_2)
        \end{tikzcd}.
    \]
    Suppose $\Fil_{\bullet}\bigwedge^m\D_{\rig}^{\dagger}(\mathrm{std}\circ\rho_{A_1})$ is a $P$-flag on $\bigwedge^m\D_{\rig}^{\dagger}(\mathrm{std}\circ\rho_{A_1})$ making $(\rho_{A_1}, \Fil_{\bullet}\bigwedge^m\D_{\rig}^{\dagger}(\mathrm{std}\circ\rho_{A_1}))\in \Def_{\rho_p, \Fil_{\bullet}}^{\mathrm{small}}(A_1)$. Then, there exists a $P$-flag $\Fil_{\bullet}\bigwedge^m\D_{\rig}^{\dagger}(\mathrm{std}\circ\rho_{A_2})$ on $\bigwedge^m\D_{\rig}^{\dagger}(\mathrm{std}\circ\rho_{A_2})$ making $(\rho_{A_2}, \Fil_{\bullet}\bigwedge^m\D_{\rig}^{\dagger}(\mathrm{std}\circ\rho_{A_2}))\in \Def_{\rho_p, \Fil_{\bullet}}^{\mathrm{small}}(A_2)$.

    By assumption, we have $\D_{\rig}^{\dagger}(\rho_{A_2}) = \D_{\rig}^{\dagger}(\rho_{A_1})\otimes_{A_1, f}A_2$; thus \[
        [\mu_{\rho_{A_2}, \HTS}] = [\mu_{\rho_{A_1}, \HTS}]\otimes_{A_1, f}A_2.
    \] By defining \[
        \Fil_{\bullet} \bigwedge^m\D_{\rig}^{\dagger}(\mathrm{std}\circ\rho_{A_2}) = \Fil_{\bullet} \bigwedge^m\D_{\rig}^{\dagger}(\mathrm{std}\circ\rho_{A_1}) \otimes_{A_1, f}A_2,
    \] one checks easily that $(\rho_{A_2}, \Fil_{\bullet}\bigwedge^m\D_{\rig}^{\dagger}(\mathrm{std}\circ\rho_{A_2}))\in \Def_{\rho_p, \Fil_{\bullet}}^{\mathrm{small}}(A_2)$.

    \noindent \textbf{Condition 2.} Let $A_1 \xrightarrow{f} A \xleftarrow{g} A_2$ be morphisms in $\Ar_L$ and let $A_3 = A_1 \times_A A_2$. Given $\rho_{A_3}\in \Def_{\rho_p}(A_3)$, let $\rho_{A_i}$ be the base change of $\rho_{A_3}$ to $A_i$ for $i=1,2$. Then, there exists a $P$-flag $\Fil_{\bullet}\bigwedge^m\D_{\rig}^{\dagger}(\mathrm{std}\circ\rho_{A_3})$ such that $(\rho_{A_3}, \Fil_{\bullet}\bigwedge^m\D_{\rig}^{\dagger}(\mathrm{std}\circ\rho_{A_3})) \in \Def_{\rho_p, \Fil_{\bullet}}^{\mathrm{small}}(A_3)$ if and only if there exist $P$-flags $\Fil_{\bullet}\bigwedge^m\D_{\rig}^{\dagger}(\mathrm{std}\circ\rho_{A_i})$ such that $(\rho_{A_i}, \Fil_{\bullet}\bigwedge^m\D_{\rig}^{\dagger}(\mathrm{std}\circ\rho_{A_i}))\in \Def_{\rho_p, \Fil_{\bullet}}^{\mathrm{small}}(A_i)$ for $i=1, 2$.

    Consider the projection maps \[
        A_1 \xleftarrow{\mathrm{pr}_1} A_3 \xrightarrow{\mathrm{pr}_2} A_2.
    \] Suppose $(\rho_{A_3}, \Fil_{\bullet}\bigwedge^m\D_{\rig}^{\dagger}(\mathrm{std}\circ\rho_{A_3})) \in \Def_{\rho_p, \Fil_{\bullet}}^{\mathrm{small}}(A_3)$. Then, arguing similarly as in Condition 1, one constructs $\Fil_{\bullet}\bigwedge^m\D_{\rig}^{\dagger}(\mathrm{std}\circ\rho_{A_i})$ such that $(\rho_{A_i}, \Fil_{\bullet}\bigwedge^m\D_{\rig}^{\dagger}(\mathrm{std}\circ\rho_{A_i}))\in \Def_{\rho_p, \Fil_{\bullet}}^{\mathrm{small}}(A_i)$ for $i=1, 2$. Conversely, note that $\D_{\rig}^{\dagger}(\rho_{A_1}) \times \D_{\rig}^{\dagger}(\rho_{A_2})$ can be identified with the $(\varphi, \Gamma)$-module with $\GSp_{2n}$-structure attached to $\rho_{A_1}\times \rho_{A_2}$. By the universal property of fibred product, there exists a unique $h: A_1 \times A_2 \rightarrow A_3$, which is compatible with $\mathrm{pr}_1$ and $\mathrm{pr_2}$. Hence, we have the identification \[
        \D_{\rig}^{\dagger}(\rho_{A_3}) = (\D_{\rig}^{\dagger}(\rho_{A_1})\times \D_{\rig}^{\dagger}(\rho_{A_2})) \otimes_{A_1\times A_2, h} A_3. 
    \]
    Observe that 
    $\Fil_{\bullet}\bigwedge^m\D_{\rig}^{\dagger}(\mathrm{std}\circ\rho_{A_1}) \times \Fil_{\bullet}\bigwedge^m\D_{\rig}^{\dagger}(\mathrm{std}\circ\rho_{A_2})$ defines a $P$-flag on $\bigwedge^m\D_{\rig}^{\dagger}(\mathrm{std}\circ\rho_{A_1})\times \bigwedge^m\D_{\rig}^{\dagger}(\mathrm{std}\circ\rho_{A_2})$. Therefore, we define \[
        \Fil_{\bullet}\bigwedge^m\D_{\rig}^{\dagger}(\mathrm{std}\circ\rho_{A_3}) \coloneq (\Fil_{\bullet}\bigwedge^m\D_{\rig}^{\dagger}(\mathrm{std}\circ\rho_{A_1}) \times \Fil_{\bullet}\bigwedge^m\D_{\rig}^{\dagger}(\mathrm{std}\circ\rho_{A_2})) \otimes_{A_1 \times A_2, h}A_3.
    \]  Again, it is easy to check that $(\rho_{A_3}, \Fil_{\bullet}\bigwedge^m\D_{\rig}^{\dagger}(\mathrm{std}\circ\rho_{A_3}))\in \Def_{\rho_p, \Fil_{\bullet}}^{\mathrm{small}}(A_3)$.

    \noindent \textbf{Condition 3.} For $i=1,2$, let $A_i \in \Ar_L$ and $\rho_{A_i}\in \Def_{\rho_p}(A_i)$. If $f: A_1 \rightarrow A_2$ is a morphism in $\Ar_L$ such that $f$ induces a commutative diagram \[
        \begin{tikzcd}
            \Gal_{\Q_p} \arrow[r]\arrow[rd] & \GSp_{2n}(A_1)\arrow[d]\\
            & \GSp_{2n}(A_2)
        \end{tikzcd}.
    \]
    Suppose $f$ is injective and $\Fil_{\bullet}\bigwedge^m\D_{\rig}^{\dagger}(\mathrm{std}\circ\rho_{A_2})$ is a $P$-flag on $\bigwedge^m\D_{\rig}^{\dagger}(\mathrm{std}\circ\rho_{A_2})$ making $(\rho_{A_2}, \Fil_{\bullet}\bigwedge^m\D_{\rig}^{\dagger}(\mathrm{std}\circ\rho_{A_2}))\in \Def_{\rho_p, \Fil_{\bullet}}^{\mathrm{small}}(A_2)$. Then, there exists a $P$-flag $\Fil_{\bullet}\bigwedge^m\D_{\rig}^{\dagger}(\mathrm{std}\circ\rho_{A_2})$ on $\bigwedge^m\D_{\rig}^{\dagger}(\mathrm{std}\circ\rho_{A_1})$ making $(\rho_{A_1}, \Fil_{\bullet}\bigwedge^m\D_{\rig}^{\dagger}(\mathrm{std}\circ\rho_{A_1}))\in \Def_{\rho_p, \Fil_{\bullet}}^{\mathrm{small}}(A_1)$.

    By assumption, we see that, for any $\sigma \in \Rep_{\GSp_{2n}}$, $\D_{\rig}^{\dagger}(\rho_{A_1})(\sigma) \subset \D_{\rig}^{\dagger}(\rho_{A_2})(\sigma)$ (as $R_{\rig, A_1}$-modules) and $\D_{\rig}^{\dagger}(\rho_{A_1}) \otimes_{A_1,f}A_2 = \D_{\rig}^{\dagger}(\rho_{A_2})$. Hence, the Hodge--Tate--Sen cocharacter $[\mu_{\rho_{A_2}, \HTS}]$ factors through $\GSp_{2n}(A_1\widehat{\otimes}_{\Q_p}\Q_p^{\cyc}) \rtimes \Gamma$. We define $\Fil_{\bullet}\bigwedge^m\D_{\rig}^{\dagger}(\mathrm{std}\circ\rho_{A_1})$ by \[
        \Fil_{\bullet}\bigwedge^m\D_{\rig}^{\dagger}(\mathrm{std}\circ\rho_{A_1}) \coloneq \Fil_{\bullet}\bigwedge^m\D_{\rig}^{\dagger}(\mathrm{std}\circ\rho_{A_2}) \cap \bigwedge^m\D_{\rig}^{\dagger}(\mathrm{std}\circ\rho_{A_1}).
    \]
    One sees that $(\rho_{A_1}, \Fil_{\bullet}\bigwedge^m\D_{\rig}^{\dagger}(\mathrm{std}\circ\rho_{A_1}))\in \Def_{\rho_p, \Fil_{\bullet}}^{\mathrm{small}}(A_1)$.
\end{proof}

\begin{Corollary}\label{Corollary: sort exact sequence of tangent spaces}
    There exists a short exact sequence of $L$-vector spaces \[
        0 \rightarrow \Def_{\rho, \dR}^{\mathrm{small}}(L[\epsilon]) \rightarrow \Def_{\rho, \Fil_{\bullet}}^{\mathrm{small}}(L[\epsilon]) \rightarrow L
    \]
    for $\epsilon^2 = 0$.
\end{Corollary}
\begin{proof}
    For any $A\in \Ar_L$ and any $(\rho_A, \Fil_{\bullet}\bigwedge^m\D_{\rig}^{\dagger}(\mathrm{std}\circ\rho_A))$, we see that $\Fil_{\bullet} \bigwedge^m\D_{\rig}^{\dagger}(\mathrm{std}\circ \rho_A)$ is essentially a filtration \[
        D_{A, 1} \subset D_{A, 1}^{\perp} \subset \D_{\rig}^{\dagger}(\mathrm{std}\circ\rho_A),
    \]
    where $D_{A, 1}$ is a rank-$1$ $(\varphi, \Gamma)$-submodule of $\bigwedge^m\D_{\rig}^{\dagger}(\mathrm{std}\circ\rho_A)$. Let \[
        D_{A, 2} \coloneq \bigwedge^m\D_{\rig}^{\dagger}(\mathrm{std}\circ \rho_A)/D_{A, 1}^{\perp}
    \]
    and so \[
        D_{A, 2} \cong \calR_{A}(\delta_{A, 2})
    \]
    for some continuous character $\delta_{A, 2}: \Q_p^{\times} \rightarrow A^\times$. Let $\delta_2$ be the continuous character constructed similarly out of $\Fil_{\bullet}\bigwedge^m\D_{\rig}^{\dagger}(\mathrm{std}\circ \rho_p)$. By ($P$-REG) and ($P$-NCR), we can write $\delta_2$ explicitly as follows: \[
        \delta_2(p) = \varphi_2 \quad \text{ and }\quad \delta_2(\gamma) = \gamma^{-\sum_{i=1}^{m}(a_0-a_i)}.
    \]

    Let $\calW_{\bbG_m} = \Spa(\Z_p[\![\Z_p^{\times}]\!], \Z_p[\![\Z_p^{\times}]\!])^{\rig}$ be the weight space on $\Z_p^{\times}$; roughly speaking, one thinks of it as the moduli space of continuous characters on $\Z_p^{\times}$ (see Sect. \ref{subsection: p-adic weight spaces} below for more discussion). Then, as functors (of points) on $\Ar_L$, there is a morphism $\Def_{\rho, \Fil_{\bullet}}^{\mathrm{small}} \rightarrow \calW_{\bbG_m}$ by \[
        (\rho_A, \Fil_{\bullet}\bigwedge^m \D_{\rig}^{\dagger}(\mathrm{std}\circ \rho_A) \mapsto \wt(\delta_{A, 2}).
    \]
    By applying the Constant Weight Lemma (Proposition \ref{Proposition: constant weight lemma}) below, we have \begin{equation}\label{eq: morphism from de Rham deformation to fixed-weight P-flag filtration}
        \Def_{\rho, \dR}^{\mathrm{small}} \rightarrow \Def_{\rho, \Fil_{\bullet}}^{\mathrm{small}} \times_{\calW_{\bbG_m}} \{\wt(\delta_2)\}.
    \end{equation}
    We claim that this morphism is an isomorphism. That is, given $(\rho_A, \Fil_{\bullet}\bigwedge^m\D_{\rig}^{\dagger}(\mathrm{std}\circ\rho_A))\in \Def_{\rho, \Fil_{\bullet}}^{\mathrm{small}}(A)$ such that $\wt(\delta_{A, 2}) = \wt(\delta_2)$, we have to show that $\D_{\rig}^{\dagger}(\rho_A)$ is de Rham. 
    
    Since $\wt(\delta_{A, 2}) = \wt(\delta_2)$, we know from the definition of $\Def_{\rho, \Fil_{\bullet}}^{\mathrm{small}}$ that $\rho_A$ has the same Hodge--Tate--Sen weights as $\rho_p$; that is, the Hodge--Tate--Sen weights for $\D_{\rig}^{\dagger}(\mathrm{std}\circ \rho_A)$ are $a_1, ..., a_n, a_0-a_n, ..., a_0-a_1$. To simplify the notation, for $i= 1, ..., 2n$, we write \[
        a_i' = \left\{ \begin{array}{ll}
            a_i, & i \leq n \\
            a_0 - a_n, & i>n
        \end{array} \right. 
    \] By construction, we see that there are elements $\alpha_1, ..., \alpha_{2n}\in \calD_{\mathrm{Sen}}(\D_{\rig}^{\dagger}(\mathrm{std}\circ \rho))$ (resp., $\widetilde{\alpha}_1, ..., \widetilde{\alpha}_{2n}\in \calD_{\Sen}(\D_{\rig}^{\dagger}(\mathrm{std}\circ \rho_A))$) such that \[
        \Theta_{\mathrm{Sen}}(\gamma_0)(\alpha_i) = -a_i'\alpha_i \quad (\text{resp., } \Theta_{\mathrm{Sen}}(\gamma_0)(\widetilde{\alpha}_i) = -a_i'\widetilde{\alpha}_i),
    \] where $\Theta_{\mathrm{Sen}}(\gamma_0)$ is the Sen operator associated with a topological generator $\gamma_0\in \Gamma$. Moreover, $\widetilde{\alpha}_i$'s are lifts of $\alpha_i$'s. As the Hodge--Tate--Sen weights are independent to the choice of the topological generator, we conclude that \[
        \gamma \alpha_i = \chi_{\cyc}^{-a_i'}(\gamma)\alpha_i \quad (\text{resp., } \gamma \widetilde{\alpha_i} = \chi_{\cyc}^{-a_i'}(\gamma) \widetilde{\alpha_i})
    \] for all $\gamma \in \Gamma$. 

    In order to proceed, we recall the constructions of $\calD_{\mathrm{Sen}}$ and $\calD_{\dR}$ in \cite[Sect. 2.2]{BC}. Recall, for any $(\varphi, \Gamma)$-module $D$ (over $B\in \Ar_L$), there exists $b\in \Q_{>0}$ (small enough) and a $(\varphi, \Gamma)$-module $D_{(0, b]}$ over $R_{(0, b], B}$ such that $D = D_{(0, b]}\otimes_{R_{(0, b], B}}R_{\rig, B}$. By \cite[(29)]{BC}, we see there exists large enough $n\in \Z_{>0}$ and a natural surjective morphism \[
        R_{(0, b]} \rightarrow \Q_{p, n}\footnote{ Recall $\Q_{p, n} = \Q_p(\zeta_{p^n})$, where $\zeta_{p^n}$ is the fixed primitive $p^n$-th roots of unity. }
    \] and a natural injective morphism \[
        R_{(0, b], B} \rightarrow \Q_{p, n}[\![t]\!],
    \]
    where $t = \log [\bfepsilon]\in R_{\rig}$. Then, \[
        \calD_{\mathrm{Sen}}(D) \coloneq (D_{(0, b]} \otimes_{R_{(0, b], B}} \Q_{p, n} ) \otimes_{\Q_{p, n}} \Q_{p, \infty}
    \]
    and \begin{align*}
        \calD_{\dR}(D) & \coloneq (D_{(0, b]} \otimes_{R_{(0, b], B}}\Q_{p, n}(\!(t)\!) \otimes_{\Q_{p,n}} \Q_{p, \infty})^{\Gamma}\\
        \Fil_{\dR}^i\calD_{\dR}(D) & \coloneq (D_{(0, b]} \otimes_{R_{(0, b], B}} t^i\Q_{p, n}[\![t]\!] \otimes_{\Q_{p,n}} \Q_{p, \infty})^{\Gamma}
    \end{align*}

    From the construction, we see we may choose $b$ and $n$ so that there exists $D_{(0, b]}$ (resp., $D_{(0, b], A}$) such that $\D_{\rig}^{\dagger}(\mathrm{std}\circ \rho) = D_{(0, b]} \otimes_{R_{(0, b], L}} R_{\rig, L}$ (resp., $\D_{\rig}^{\dagger}(\mathrm{std}\circ \rho_A) = D_{(0, b], A}\otimes_{R_{(0, b], A}} R_{\rig, A}$ and $D_{(0, b], A}\otimes_A L = D_{(0, b]}$) and $\alpha_i \in D_{(0, b]}\otimes_{R_{(0, b], L}} \Q_{p, n}$ (resp., $\widetilde{\alpha}_i \in D_{(0, b], A}\otimes_{R_{(0, b], A}} \Q_{p, n}$).

    Now, for any $i = 1, ..., 2n$, we see from our discussion above that $\Fil_{\dR}^{a_i'}\calD_{\dR}(\D_{\rig}^{\dagger}(\mathrm{std}\circ \rho))$ is generated by $\alpha_j \otimes t^{a_j'}$ for $j\geq i$. Moreover, we also see that $\widetilde{\alpha}_j \otimes t^{a_j'}\in \Fil_{\dR}^{a_i'} \calD_{\dR}(\D_{\rig}^{\dagger}(\mathrm{std}\circ \rho_A))$. Since $a_1'< \cdots < a_j' < \cdots a_{2n}'$ (by assumption), each $\widetilde{\alpha}_j \otimes t^{a_j'}$ is a lift of $\alpha_j \otimes t^{a_j'}$ and they are distinct. Thus,  \[
        2n \leq \rank \calD_{\dR}(\D_{\rig}^{\dagger}(\mathrm{std}\circ \rho_A)).
    \]
    However, $\rank \calD_{\dR}(\D_{\rig}^{\dagger}(\mathrm{std}\circ \rho_A))$ is always bounded above by $2n$, we thus show that $\calD_{\dR}(\D_{\rig}^{\dagger}(\mathrm{std}\circ \rho_A))$ is de Rham. Since $\mathrm{std}$ is a $\otimes$-generator, we conclude the claim.
    
    The desired result then follows immediately from the isomorphism \eqref{eq: morphism from de Rham deformation to fixed-weight P-flag filtration}. 
\end{proof}

\begin{Proposition}[Constant Weight Lemma]\label{Proposition: constant weight lemma}
    Let $\delta: \Q_p^{\times} \rightarrow L^{\times}$ be a continuous character such that $\calR_{L}(\delta)$ is de Rham. Let $A\in \Ar_L$ and $\delta_A: \Q_p^{\times} \rightarrow A^{\times}$ be a continuous character lifting $\delta$. Then, $\calR_{A}(\delta_A)$ is de Rham if and only if $\wt(\delta_A) = \wt(\delta)$. 
\end{Proposition}
\begin{proof}
    Assuming $\wt(\delta_A) = \wt(\delta)$, then $\calR_A(\delta_A)$ is de Rham by \cite[Proposition 2.3.4]{BC}. 

    Conversely, since $\calR_L(\delta)$ is de Rham and of rank $1$, it is potentially crystalline. Hence, after pre-composing with the norm map $\mathrm{Nm}: F^\times \rightarrow \Q_p^{\times}$ if necessary (where $F$ is some finite field extension of $\Q_p$), we may apply \cite[Lemma 2.5.2 (iii)]{BC} (see also \cite[Proposition 2.5.4]{BC}) and see that the Hodge--Tate--Sen weight for $\calR_A(\delta_A)$ is exactly $\wt(\delta)$. 
\end{proof}

\begin{Remark}\label{Remark: deformation problems for non-maximal parabolics}
    Similarly as in Remark \ref{Remark: when the parabolic subgroup is not maximal}, we remark that one can also modify the deformation problems and discussions above when $P$ is not a maximal parabolic. In Sect. \ref{section: R=T}, we will be focusing on the case of $\GSp_4$, where the only proper parabolic subgroups are the Siegel parabolic, Klingen parabolic, and the Borel subgroups. As the Borel case can be well-studied by adapting the techniques in \cite{BC}, the interesting cases for us are the maximal parabolic subgroups in $\GSp_4$. Hence, we leave the general cases to the interested reader. 
\end{Remark}

\section{Small parabolic eigenvarieties for Siegel cuspforms of genus 2}\label{section: small par. eigenvar. for H0}
Starting from this section, we focus on the algebraic group $\GSp_4$. The goal of this section is to construct small parabolic eigenvarieties for Siegel cuspforms of genus $2$. We remark that the case when the parabolic is the Borel subgroup, the construction is already dealt in \cite{AIP-2015}. Hence, we shall focus more on the cases when the parabolic subgroup is either the Siegel parabolic or the Klingen parabolic subgroups. We also remark that a more general theory in the case of Klingen parabolic subgroup has been developed in \cite{Pilloni-higherHidaColemanGSp4} (except that Pilloni didn't construct the small parabolic eigenvariety), which highly inspired our discussion in that situation.

\subsection{Siegel threefolds and algebraic Siegel modular forms}\label{subsection: classical Siegel modular forms}
Let $\A$ be the ring of adèles of $\Q$. For a finite set $S$ of places of $\Q$, we denote the ring of adèles away from $S$ by $A^S$. Let $K = \prod_{\ell} K_{\ell}\subset \GSp_{4}(\A^{\infty})$ be a neat open compact subgroup with \[
    K_{\ell} \subset \left\{ \begin{array}{ll}
        \GSp_4(\Q_{\ell}), & \ell \neq p   \\
        \GSp_4(\Z_p), & \ell = p  
    \end{array}\right.
\]
open compact and $K_{\ell} = \GSp_4(\Z_{\ell})$ for almost all $\ell$. We write $K^p = \prod_{\ell \neq p}K_{\ell}$ so that $K = K^p K_p$. We assume $K^p$ is neat.\footnote{In fact, we may lossen this assumption. See Remark \ref{Remark: at non-neat level, classical forms} below. } We shall fix a choice of $K^p$ and choose a specific $K_p$ later in the paper. For later use, let $S_{\mathrm{bad}} = \{\ell: K_{\ell} \neq  \GSp_4(\Z_{\ell})\} \cup \{2, p\}$.

The (complex) Siegel threefold of level $K$ is the locally symmetric space \[
    X_{K}(\C) := \GSp_4(\Q) \backslash \GSp_4(\A^{\infty}) \times \bbH_2/ K,
\]
where $\bbH_2$ is the (disjoint) union of the Siegel upper- and lower-half plane. Since we fixed $K^p$, we may sometimes denote $X_{K}(\C)$ by $X_{K_p}(\C)$ to stress the choice of $K_p$. It is a well-known fact that $\dim_{\C} X_{K}(\C) = 3$.

It is well-known that $X_K(\C)$ admits a structure of algebraic variety $X_K$ over $\Q$, which can be interpreted as a moduli space parametrising principally polarised abelian surfaces together with a level-$K$ structure (\cite[Definition 1.4.1.4]{Lan-PhD}). Moreover, by choosing an auxiliary cone decomposition $\Sigma$, the algebraic variety $X_K$ admits a \emph{toroidal compactification} $X_K^{\tor}$ (depending on $\Sigma$) with the following properties (\cite[Chapter IV Theorem 6.7]{Faltings-Chai}): \begin{itemize}
    \item The exists an injective morphism $X_K \hookrightarrow X_K^{\tor}$ with Zariski open image. 
    \item The boundary $\partial X^{\tor}_K = X_K^{\tor} \smallsetminus X_K$ is a normal crossing divisor. By equipping $X_K^{\tor}$ the log structure given by $\partial X_K^{\tor}$, we may view $X_K^{\tor}$ as an fs log scheme. 
    \item The universal abelian variety $A^{\univ} \rightarrow X_K$ extends to a tautological semiabelian scheme $\pi: G^{\univ} \rightarrow X_{K}^{\tor}$. We denote the identity section by $e$.
\end{itemize}
In fact, we may fix a choice of $\Sigma$ for $X_{\GSp_4(\Z_p)}$, applying \cite[Theorem 7.6]{Illusie}, and get a finite Kummer étale morphism of fs log schemes \[
    X_{K_p}^{\tor} \rightarrow X_{\GSp_4(\Z_p)}^{\tor}
\]
for any $K_p \subset \GSp_4(\Z_p)$. In what follow, we base change $X_K^{\tor}$ and $X_K$ to $\Q_p$, but still use the same notations to denote them.\footnote{ Many of the discussions below do not require this base change. We do it here just to simplify the future notations.} 

Define \[
    \underline{\omega} \coloneq e^* \Omega_{G^{\univ}/X_K^{\tor}}^1.
\]
For any $k = (k_1, k_2)\in \Z^2$ with $k_1 \geq k_2$, the automorphic sheaf of weight $k$ is defined to be \[
    \underline{\omega}^k \coloneq \mathrm{Sym}^{k_1-k_2}\underline{\omega} \otimes (\det \underline{\omega})^{k_2}.
\] 
We also define the cuspidal automorphic sheaf to be \[
    \underline{\omega}_{\cusp}^{k} \coloneq \underline{\omega}^{k}(-\partial X_K^{\tor})
\]
Then, classical genus-$2$ Siegel modular forms of weight $k$ level $K$ are elements of the space \[
    M_k(K) \coloneq H^0(X_{K}^{\tor}, \underline{\omega}^k). 
\] 
We also define the space of Siegel cuspforms to be \[
    S_k(K) \coloneq H^0(X_K^{\tor}, \underline{\omega}^{k}_{\cusp}).
\]
Similarly as before, since we shall fix a level $K^p$ away from $p$, we may sometimes denote $M_k(K)$ (resp., $S_k(K)$) by $M_k(K_p)$ (resp., $S_k(K_p)$).

We now discuss the Hecke action on $M_k(K)$ and $S_k(K)$. Given any prime number $\ell \neq p$\footnote{ We shall discuss the situation when $\ell = p$ in Sect. \ref{subsection: finite-slope parts, H0}} and any double coset $[K_{\ell} \bfdelta K_{\ell}] \in K_{\ell} \backslash \GSp_4(\Q_{\ell})/K_{\ell}$ with representative $\bfdelta \in \GSp_4(\Q_{\ell})$, we have the correspondence \begin{equation}\label{eq: Hecke correspondence away from p}
    \begin{tikzcd}
        & X_{K \cap \bfdelta^{-1} K \bfdelta}   \arrow[ld, "\pr_1"']\arrow[rd, "\pr_2"]\\
        X_{K} && X_{K}
    \end{tikzcd}, 
\end{equation}
where $\pr_2$ is the composition of $\bfdelta^{-1}: X_{K \cap \bfdelta^{-1} K \bfdelta} \rightarrow X_{\bfdelta K \bfdelta^{-1} \cap K}$ and the natural projection $X_{\bfdelta K \bfdelta^{-1} \cap K} \rightarrow X_K$. This diagram extends to the toroidal compactifications after a suitable choices of cone decompositions in the construction of the toroidal compactifications. That is, we have \begin{equation}\label{eq: Hecke correspondence away from p, toroidal}
    \begin{tikzcd}
        & X_{K \cap \bfdelta^{-1} K \bfdelta}^{\tor, \Sigma''}   \arrow[ld, "\pr_1"']\arrow[rd, "\pr_2"]\\
        X_{K}^{\tor, \Sigma} && X_{K}^{\tor, \Sigma'}
    \end{tikzcd}, 
\end{equation}
where $\Sigma, \Sigma', \Sigma''$ are the corresponding (suitable) cone decompositions. Away from the boundary, there is an isomorphism $\pr_2^* \underline{\omega}^k|_{X_{K \cap \bfdelta^{-1} K \bfdelta}} \rightarrow \pr_1^* \underline{\omega}^k|_{X_{K \cap \bfdelta^{-1} K \bfdelta}}$ (induced by the action of $\bfdelta$). This isomorphism extends to a morphism $\pr_2^{*}\underline{\omega}^k \rightarrow \pr_1^* \underline{\omega}^k$ (\cite[Sect. 4.2.1]{BP-HigherColeman}). Consequently, we define the operator $T_{\bfdelta}$ by the composition \[
    \begin{tikzcd}
        T_{\bfdelta}: R\Gamma(X_{K}^{\tor, \Sigma'}, \underline{\omega}^k) \arrow[r] & R\Gamma(X_{K \cap \bfdelta^{-1} K \bfdelta}^{\tor, \Sigma''}, \pr_2^* \underline{\omega}^k) \arrow[r] & R\Gamma(X_{K \cap \bfdelta^{-1} K \bfdelta}^{\tor, \Sigma''}, \pr_1^* \underline{\omega}^k)\arrow[dl, out = -10, in = 170]\\
        & R\Gamma(X_K^{\tor, \Sigma}, R\pr_{1, *}\pr_1^* \underline{\omega}^k) \arrow[r, "\mathrm{tr}"] & R\Gamma(X_{K}^{\tor, \Sigma}, \underline{\omega}^k)
    \end{tikzcd}.
\] 
Moreover, we can choose a cone decomposition $\Sigma'''$ finer than $\Sigma$ and $\Sigma'$ and so we have morphisms \[
    X_{K}^{\tor, \Sigma'} \xleftarrow{\pi^{\Sigma'''}_{\Sigma}} X_K^{\tor, \Sigma'''} \xrightarrow{\pi^{\Sigma'''}_{\Sigma}} X_K^{\tor, \Sigma}.
\]
By the discussions in \cite[Sect. 2]{Harris-partial}, we see that \begin{itemize}
    \item $\pi_{\Sigma}^{\Sigma''', *}\underline{\omega}^k = \underline{\omega}$ (on $X_K^{\tor, \Sigma'''}$); 
    \item $R^i\pi^{\Sigma'''}_{\Sigma, *} \underline{\omega}^k = \left\{\begin{array}{ll}
        \underline{\omega}^k, & i=0 \\
        0, & i>0 
    \end{array} \right.$ (on $X_K^{\tor, \Sigma}$);
    \item same for $\pi_{\Sigma'}^{\Sigma'''}$.
\end{itemize}
Hence, we have quasi-isomorphisms \[
    R\Gamma(X_{K}^{\tor, \Sigma}, \underline{\omega}^k) \rightarrow R\Gamma(X_K^{\tor, \Sigma'''}, \underline{\omega}^k) \rightarrow R\Gamma(X_K^{\tor, \Sigma'}, R\pi^{\Sigma'''}_{\Sigma', *}\pi_{\Sigma'}^{\Sigma''', *}\underline{\omega}^k) \rightarrow R\Gamma(X_{K}^{\tor, \Sigma'}, \underline{\omega}^k). 
\]
By composing $T_{\bfdelta}$ with these quasi-isomorphisms, we obtain an endomorphism, which we abuse the notation and still call it $T_{\bfdelta}$, \[
    T_{\bfdelta}: R\Gamma(X_{K}^{\tor, \Sigma}, \underline{\omega}^k) \rightarrow R\Gamma(X_{K}^{\tor, \Sigma}, \underline{\omega}^k). 
\]
As a result, we shall later drop the the $\Sigma$ in the notation.  In particular, we have the induced operator, still denoted by $T_{\bfdelta}$, on $M_k(K)$. The same construction works also for $\underline{\omega}^k_{\cusp}$ and $S_k(K)$.

Suppose $\ell$ is a prime such that $K_{\ell} = \GSp_4(\Z_{\ell})$. In this situation, we consider the spherical Hecke algebra (over $\Z_p$) \[
    \bbT^{\mathrm{sph}}_{\ell} \coloneq \Z_p[\GSp_4(\Z_{\ell}) \backslash \GSp_4(\Q_{\ell}) / \GSp_4(\Z_{\ell})].
\]
Recall the Satake isomorphism \[
    \mathrm{Sat}: \bbT_{\ell}^{\mathrm{sph}} \xrightarrow{\cong} \Z_p[T(\Z_{\ell}) \backslash T(\Q_{\ell}) / T(\Z_{\ell})]^{W_{\GSp_4}},
\]
where $W_{\GSp_4}$ is the Weyl group of $\GSp_4$. For any $\bftau \in T(\Q_{\ell})$, denote by $[\bftau]$ the corresponding element in $\Z_p[T(\Z_{\ell}) \backslash T(\Q_{\ell}) / T(\Z_{\ell})]$.  We consider \[
    \begin{array}{l}
        \bft_{\ell, 0} = \diag(1, 1, \ell, \ell)\\
        \bft_{\ell, 1} = \diag(1, \ell, 1, \ell)\\
        \bft_{\ell, 2} = \diag( \ell,1,\ell,1)\\
        \bft_{\ell, 3} = \diag(\ell, \ell, 1, 1)
    \end{array} \in T(\Q_{\ell}).
\]
Following \cite[\S 3]{GT-TWGSp4}, we define the \textbf{\textit{Hecke polynomial}} at $\ell$ to be \begin{equation}\label{eq: Hecke polynomial}
    P_{\mathrm{Hecke}, \ell} := \prod_{i=0}^3 (X - [\bft_{\ell, i}]) \in \Z_p[T(\Z_{\ell}) \backslash T(\Q_{\ell}) /T(\Z_{\ell})]^{W_{\GSp_4}}[X] \cong \bbT_{\ell}^{\mathrm{sph}}[X].
\end{equation}

\begin{Remark}\label{Remark: Hecke polynomial}
    The definition of Hecke polynomials in \cite{GT-TWGSp4} is more conceptual, whose definition makes use of the action of the Weyl group on the torus $T$. After an easy computation, one sees that our definition of Hecke polynomials is nothing but an explicit expression of theirs.  
\end{Remark}

\begin{Remark}\label{Remark: at non-neat level, classical forms}
    The discussion above also works when $K^p \subset \GSp_4(\A^{p, \infty})$ is not neat but contains a neat normal subgroup. Indeed, in this situation, let $\widetilde{K}^p$ be a neat normal subgroup in $K^p$ and write $\widetilde{K} = \widetilde{K}^p K_p$; hence $K/\widetilde{K}$ is a finite subgroup. Then, we define \[
        M_k(K) \coloneq M_k(\widetilde{K})^{K/\widetilde{K}}\quad \text{ and }\quad S_k(K) \coloneq S_k(\widetilde{K})^{K /\widetilde{K}}.
    \]
    One observes that the spherical Hecke actions commute with the action by $K/\widetilde{K}$. Note that the quotient $X_K^{\tor} \coloneq X_{\widetilde{K}}^{\tor}/(K/\widetilde{K})$ exists as a scheme by \cite[Exposé V, Proposition 1.8]{SGA1}, but we do not know if the tautological semiabelian scheme $G^{\univ}$ descends to $X_K^{\tor}$. 
\end{Remark}

\subsection{Finite-slope parts}\label{subsection: finite-slope parts, H0}
Choose an open compact subgroup $K_p \subset \GSp_4(\Z_p)$. Given $k = (k_1, k_2)\in \Z^2$ with $k_1 \geq k_2$, consider $M_k(K_p)$ and $S_k(K_p)$ as before. The purpose of this subsection is to study the these spaces when $K_p$ varies. In particular, we prove Proposition \ref{Proposition: finite-slope part it independent to the level at p, H0}. Our discussion in this subsection is highly inspired by \cite[Sect. 4.2]{BP-HigherColeman}.

To this end, let $\bfitu \in \GSp_4(\Q_p)$ and let $K_{p,1}$, $K_{p,2}$ be two open compact subgroups of $\GSp_4(\Z_p)$. Consider the free group $\Z[K_{p,1}\backslash \GSp_4(\Q_p) /K_{p, 2}]$ with basis elements $[K_{p, 1} \bfitu K_{p,2}]\in K_{p,1}\backslash \GSp_4(\Q_p) /K_{p, 2}$. We embed $\Z[K_{p,1}\backslash \GSp_4(\Q_p) /K_{p, 2}]$ into the convolution algebra $C_c^{\infty}(\GSp_4(\Q_p), \R)$ via \[
    \iota_{K_{p,1}, K_{p,2}}: \Z[K_{p,1}\backslash \GSp_4(\Q_p) /K_{p, 2}] \rightarrow C_c^{\infty}(\GSp_4(\Q_p), \R), \quad [K_{p,1}\bfitu K_{p, 2}] \mapsto \frac{1}{\sqrt{\mathrm{vol}(K_{p, 1})\mathrm{vol}(K_{p,2})}} \boldsymbol{1}_{K_{p,1}\bfitu K_{p, 2}},
\]
where $\boldsymbol{1}_{K_{p,1}\bfitu K_{p, 2}}$ denotes the characteristic function of the double coset $K_{p,1}\bfitu K_{p, 2}$. 

Given open compact subgroups $K_{p,1}, K_{p,2}, K_{p,3}\subset \GSp_4(\Z_p)$, there is a multiplication map \begin{align*}
    \Z[K_{p, 1}\backslash \GSp_4(\Q_p) /K_{p, 2}] \times \Z[K_{p, 2}\backslash \GSp_4(\Q_p) /K_{p, 3}] & \rightarrow \Z[K_{p, 1} \backslash \GSp_4(\Q_p) /K_{p, 3}]\\
    ([K_{p, 1}\bfitu_1 K_{p, 2}], [K_{p, 2}\bfitu_2 K_{p, 3}]) & \mapsto [K_{p, 1}\bfitu_1K_{p, 2}][K_{p, 2}\bfitu_2 K_{p, 3}],
\end{align*}
which can be defined by using the embeddings $\iota_{K_{p, i}, K_{p, j}}$ as follows \[
    \iota_{K_{p, 1}, K_{p, 3}}([K_{p, 1}\bfitu_1K_{p, 2}][K_{p, 2}\bfitu_2 K_{p, 3}]) \coloneq \iota_{K_{p, 1}, K_{p, 2}}([K_{p, 1}\bfitu_1K_{p, 2}]) \star \iota_{K_{p, 2}, K_{p, 3}}([K_{p, 2}\bfitu_2 K_{p, 3}]),
\]
where the operator $\star$ is the (usual) convolution on $C_c^{\infty}(\GSp_4(\Q_p), \R)$. To check that the right-hand side lies in the image of $\iota_{K_{p, 1}, K_{p, 3}}$ and for more detailed discussions, we refer the readers to \cite[Sect. 4.2.2]{BP-HigherColeman}.

The following lemma is inspired by \cite[Lemma 4.2.13]{BP-HigherColeman}. 

\begin{Lemma}\label{Lemma: condition for composing Hecke operators naively}
    Let $K_{p, 1}$, $K_{p, 2}$, $K_{p, 3}$ be subgroups of $\GSp_4(\Z_p)$ having parahoric decompositions \[
        K_{p, i} = N_i^{\opp} M_i N_i.
    \]
    Let $\bfitu_1, \bfitu_2\in T_{\GSp_4}(\Q_p)$ such that \begin{itemize}
        \item $\bfitu_1^{-1} N_1^{\opp} \bfitu_1 \cap \bfitu_2 N_3^{\opp} \bfitu_2^{-1} \subset N_2^{\opp} \subset \bfitu_1^{-1} N_1^{\opp} \bfitu_1$;
        \item $\bfitu_1^{-1} N_1 \bfitu_1 \cap \bfitu_2 N_3 \bfitu_2^{-1} \subset N_2 \subset \bfitu_2 N_3 \bfitu_2^{-1}$;
        \item for $i=1,2,3$, $j=1,2$, $\bfitu_j M_i \bfitu_j^{-1} = M_i$ and $\bfitu_j^{-1}M_i \bfitu_j = M_i$;
        \item $M_1 \cap M_3 \subset M_2 \subset M_1 M_3$.
    \end{itemize}
    Then \[
        [K_{p,1} \bfitu_1 K_{p, 2}][K_{p, 2} \bfitu_2 K_{p, 3}] = [K_{p, 1} \bfitu_1\bfitu_2 K_{p, 3}].
    \]
\end{Lemma}
\begin{proof}
    The conditions imply that \[
        K_{p,2} = (K_{p,2} \cap \bfitu_1^{-1} K_{p, 1} \bfitu_1)(K_{p,2} \cap \bfitu_2 K_{p,3} \bfitu_2^{-1})
    \] 
    and \[
        \bfitu_1^{-1} K_{p, 1} \bfitu_1 \cap \bfitu_2 K_{p,3} \bfitu_2^{-1} \subset K_{p,2}.
    \]
    Therefore, the double quotient \[
        (K_{p,2} \cap \bfitu_1^{-1} K_{p, 1} \bfitu_1) \backslash K_{p,2} / (K_{p,2} \cap \bfitu_2 K_{p,3} \bfitu_2^{-1})
    \] is trivial and the index \[
        [\bfitu_1^{-1} K_{p, 1} \bfitu_1 \cap \bfitu_2 K_{p,3} \bfitu_2^{-1} : K_{p,2} \cap \bfitu_1^{-1} K_{p, 1} \bfitu_1 \cap \bfitu_2 K_{p,3} \bfitu_2^{-1}] =1.
    \]
    We then conclude the result by applying \cite[Proposition 4.2.3]{BP-HigherColeman}. 
\end{proof}

\begin{Remark}\label{Remark: condition for composing Hecke operators is independent to GSp}
    One sees that the group $\GSp_4$ plays no special role in Lemma \ref{Lemma: condition for composing Hecke operators naively}. Indeed, as in \cite[Sect. 4.2]{BP-HigherColeman}, one can generalise this lemma to other algebraic groups. We leave such generalisation to the readers. 
\end{Remark}

For $\bfitu\in \GSp_4(\Q_p)$, we again choose suitable cone decompositions $\Sigma$, $\Sigma'$, $\Sigma''$ and consider the correspondence \[
    \begin{tikzcd}
        & X_{K_{p, 1} \cap \bfitu^{-1} K_{p,2} \bfitu}^{\tor, \Sigma''}\arrow[ld, "\pr_1"'] \arrow[rd, "\pr_2"]\\
        X_{K_{p,1}}^{\tor, \Sigma} && X_{K_{p,2}}^{\tor, \Sigma'}
    \end{tikzcd}
\] 
similarly as before. 
Then, the basis element $[K_{p,1}\bfitu K_{p,2}]$ defines the Hecke operator \[ 
    \begin{tikzcd}
        R\Gamma(X_{K_{p,2}}^{\tor, \Sigma'}, \underline{\omega}^k) \arrow[r, "\pr_2^*"]\arrow[rrdd, bend right = 20, "\text{$[K_{p,1}\bfitu K_{p,2}]$}"'] & R\Gamma(X_{K_{p, 1} \cap \bfitu^{-1} K_{p,2} \bfitu}^{\tor, \Sigma''}, \pr_2^* \underline{\omega}^k) \arrow[r] & R\Gamma(X_{K_{p, 1} \cap \bfitu^{-1} K_{p,2} \bfitu}^{\tor, \Sigma''}, \pr_1^* \underline{\omega}^k) \arrow[d]\\
        && R\Gamma(X_{K_{p,1}}^{\tor, \Sigma}, R\pr_{1, *}\pr_1^* \underline{\omega}^k)\arrow[d, "\mathrm{tr}"]\\
        && R\Gamma(X_{K_{p,1}}^{\tor, \Sigma},  \underline{\omega}^k)
    \end{tikzcd}
\] 
Again, using the same trick, one obtains 
\begin{equation}\label{eq: Hecke operator at p; changing level}
    [K_{p,1}\bfitu K_{p,2}]: R\Gamma(X_{K_{p,2}}^{\tor, \Sigma}, \underline{\omega}^k) \rightarrow R\Gamma(X_{K_{p,1}}^{\tor, \Sigma}, \underline{\omega}^k)
\end{equation}
and so we shall drop the $\Sigma$ from the notation. We shall again abuse the notation and still use $[K_{p,1}\bfitu K_{p,2}]$ for the operator defined similarly for $\underline{\omega}^k_{\cusp}$ and on $H^0$'s.  
Note that if $\bfitu$ is the identity and $K_{p,2}\subset K_{p,1}$, the map $[K_{p,2} \bfitu K_{p,1}]$ is nothing but the trace map.

In what follows, we will focus on the following matrices in $T_{\GSp_4}(\Q_p) \subset \GSp_4(\Q_p)$:  \[
    \bfitu_{P_{\mathrm{Si}}} = \bfitu_{\mathrm{Si}} \coloneq \diag(1, 1, p, p), \quad \bfitu_{P_{\mathrm{Kl}}} = \bfitu_{\mathrm{Kl}} \coloneq \diag(1, p, p, p^2), \quad \text{ and }\quad \bfitu_B \coloneq \diag(1, p, p^2, p^3).
\]
For $P \in \{B, P_{\mathrm{Si}}, P_{\mathrm{Kl}}\}$, if $K_{p,1} = K_{p, 2}$ and when the context is clear, we will simplify the notation and denote the (normalised) operator $\frac{1}{p^\nu}[K_{p, 1}\bfitu_P K_{p, 1}]$ by $U_P$, where \[
    \nu = \left\{ \begin{array}{ll}
        3, & P\neq B\\
        6, & P=B
    \end{array} \right. .
\]

\begin{Definition}\label{Definition: finite slope part of the Siegel modular forms}
    Keep the notations as above. Fix $P\in \{B, P_{\mathrm{Si}}, P_{\mathrm{Kl}}\}$. The \textbf{$P$-finite-slope part} of $M_k(K_p)$ is defined to be \[
        M_k(K_p)^{P-\fs} \coloneq M_k(K_p)\otimes_{\Q_p[U_P]}\Q_p[U_P, U_P^{-1}].
    \]
    Similar definition applies to obtain the $P$-finite-slope part $S_k(K_p)^{P-\fs}$ of $S_k(K_p)$.
\end{Definition}

Fix $P\in \{B, P_{\mathrm{Si}}, P_{\mathrm{Kl}}\}$. We begin by defining the parahoric subgroup at $p$, namely \begin{equation}\label{eq: parahoric decomposition}
    \Iw_P := \left\{ \bfgamma \in \GSp_4(\Z_p) : (\bfgamma \mod p)\in P(\Z/p\Z) \right\}.
\end{equation}
When $P = B$, this $p$-adic group is nothing but the usual Iwahori subgroup associated with $\GSp_4$. The parahoric subgroup $\Iw_P$ admits an Iwahori (or parahoric) decomposition \[
    \Iw_P = N^{\opp}_{P, 1} M_P(\Z_p) N_{P, 0}, 
\]
where, for $?\in \{\emptyset, \opp\}$, \[
    N_{P, m}^{?} = \left\{ \begin{array}{ll}
        \ker\left(N^{?}_{P}(\Z_p) \xrightarrow{\mod p^m} N^{?}_{P}(\Z/p^m\Z)\right), & \text{ if }m\in \Z_{>0}  \\
        N^?_{P}(\Z_p), & \text{ if }m=0
    \end{array} \right. .
\] 
One can easily check that \[
    \bfitu_P N_{P, m}^{\opp} \bfitu_P^{-1} \subset N_{P, m+1}^{\opp} \subset N_{P, m}^{\opp}
\]
and \[
    \bfitu_P^{-1}N_{P, m}\bfitu_P \subset N_{P, m+1}\subset N_{P, m}.
\]

The following proposition is inspired by \cite[Corollary 4.2.16]{BP-HigherColeman}.

\begin{Proposition}\label{Proposition: finite-slope part it independent to the level at p, H0}
    Fix $k = (k_1, k_2)\in \Z^2$ with $k_1\geq k_2$ and $P\in \{B, P_{\mathrm{Si}}, P_{\mathrm{Kl}}\}$. Let $K_{p} \subset \Iw_P$ be a compact open subgroup with the induced parahoric decomposition \[
        K_p = N_{K_p}^{\opp} M_{K_p} N_{K_p}, 
    \]
    where $M_{K_p} = K_p \cap M_P(\Z_p)$, $N_{K_p}  = K_p \cap N_{P, 0}$, and $N_{K_p}^{\opp} = K_p \cap N_{P, 1}^{\opp}$. Suppose the following conditions hold: \begin{itemize}
        \item[(i)] there exists $m\in \Z_{>0}$ such that $\bfitu_P N_{P, m}^{\opp} \bfitu_P^{-1} \subset N_{K_p}^{\opp} \subset N_{P, m}^{\opp}$; 
        \item[(ii)] $\bfitu_P^{-1} N_{K_p} \bfitu_P \subset N_{P,0} \subset N_{K_p} \subset \bfitu_P N_{P,0} \bfitu_P^{-1}$; and 
        \item[(iii)] $M_{K_p} = M(\Z_p)$.
    \end{itemize} Then, there are natural isomorphisms \[
        M_k(\Iw_P)^{P-\fs} \cong M_k(K_p)^{P-\fs} \quad \text{ and }\quad S_k(\Iw_P)^{P-\fs} \cong S_k(K_p)^{P-\fs}.
    \]
\end{Proposition}
\begin{proof}
    Let \[
        \Iw_{P, m} \coloneq \{\bfgamma \in \GSp_{4}(\Z_p): (\bfgamma \mod p^m) \in P(\Z/p^m\Z)\}
    \] 
    with the induced parahoric decomposition \[
        \Iw_{P, m} = N_{P,m}^{\opp} M_p(\Z_p) N_{P,0}.
    \]
    The desired isomorphism is proven via the following sequence of isomorphisms \begin{equation}\label{eq: isomorphisms proving finite-slope part independent to the level at p}
        M_k(\Iw_P)^{P-\fs} \cong M_k(\Iw_{P,m})^{P-\fs} \cong M_k(K_p)^{P-\fs}.
    \end{equation}

    We first demonstrate the second isomorphism. We prove in three steps:

    \noindent \textbf{Step 1.}$[\Iw_{P, m} \bfitu_P \Iw_{P,m}] = [\Iw_{P, m} \one_4 K_p][K_p \bfitu_P \Iw_{P, m}]$.
    
    To show this decomposition, we apply Lemma \ref{Lemma: condition for composing Hecke operators naively} by taking $K_{p, 1} = K_{p, 3} = \Iw_{P, m}$, $K_{p,2} = K_p$, $\bfitu_1 = \one_4$, and $\bfitu_2 = \bfitu_P$. We then check the three conditions:\begin{itemize}
        \item $N_{P, m}^{\opp} \cap \bfitu_P N_{P, m}^{\opp} \bfitu_P \subset N_{K_p}^{\opp} \subset N_{P, m}^{\opp}$; this is verified by (i). 
        \item $N_{P, 0} \cap \bfitu_P N_{P, 0}\bfitu_P^{-1} \subset N_{K_p} \subset \bfitu_P N_{P, 0}\bfitu_P^{-1}$; note that $\bfitu_P^{-1}N_{P, 0}\bfitu_P\subset  N_{P, 0}$, thus $N_{P, 0} \cap \bfitu_P N_{P, 0}\bfitu_P^{-1}  = N_{P, 0}$ and this is verified by (ii). 
        \item $M_P(\Z_p)\cap M_P(\Z_p) \subset M_{K_p}\subset M_P(\Z_p)M_P(\Z_p)$; this is verified by (iii). 
    \end{itemize}

    \noindent \textbf{Step 2.} $[K_p \bfitu_P K_p] = [K_p, \bfitu_P \Iw_{P, m}][\Iw_{P, m}\one_4 K_p]$.

    To show this decomposition, we apply Lemma \ref{Lemma: condition for composing Hecke operators naively} by taking $K_{p, 1} = K_{p, 3} = K_p$, $K_{p, 2} = \Iw_P$, $\bfitu_1 = \bfitu_P$, and $\bfitu_2 = \one_4$. We then check the three conditions:\begin{itemize}
        \item $\bfitu_P^{-1} N_{K_p}^{\opp} \bfitu_P \cap N_{K_p} \subset N_{P, m}^{\opp} \subset \bfitu_P^{-1} N_{K_p}^{\opp} \bfitu_P$; this is verified by (i). 
        \item $\bfitu_P^{-1} N_{K_p} \bfitu_P \cap N_{K_p} \subset N_{P, 0} \subset N_{K_p}$; note that $\bfitu_P^{-1} N_{K_p}\bfitu_P \cap N_{K_p} = \bfitu_P^{-1} N_{K_p}\bfitu_P$ and so this is verified by (ii). 
        \item $M_{K_p}\cap M_{K_p} \subset M(\Z_p) \subset M_{K_p}M_{K_p}$; this is verified by (iii).
    \end{itemize} 

    \noindent \textbf{Step 3. }Combining the two steps above, we conclude a commutative diagram \[
        \begin{tikzcd}
            M_k(K_p) \arrow[r, "U_P"]\arrow[d, "\mathrm{tr}"'] & M_k(K_p)\arrow[d, "\mathrm{tr}"]\\
            M_k(\Iw_{P, m}) \arrow[r, "U_P"]\arrow[ru] & M_k(\Iw_{P,m})
        \end{tikzcd},
    \]
    where the diagonal map is given by $[K_p \bfitu_P \Iw_{P,m}]$. Then, by definition, after taking the $P$-finite-slope part, the horizontal maps in this diagram are isomorphisms, which implies that every map in this diagram must be an isomorphism. 

    Finally to show the first isomorphism in \eqref{eq: isomorphisms proving finite-slope part independent to the level at p}, one argues similarly as above to show \[
        M_k(\Iw_P)^{P-\fs} \cong M_k(\Iw_{P,2})^{P-\fs}.
    \]
    We then conclude by induction. 
\end{proof}

\subsection{Weight spaces}\label{subsection: p-adic weight spaces}

First of all, we consider the $p$-adic weight space associated with the Borel subgroup $B$.\footnote{ Or, rather, the weight space associated with the upper-triangular Borel for $\Sp_4$. See Remark \ref{Remark: weight space mod out the wild twist} below. } Consider $\Alg_{(\Z_p, \Z_p)}$ the category of sheafy $(\Z_p, \Z_p)$-algebras. It is well-known that the functor \begin{equation}\label{eq: defining functor for the full weight space}
    \Alg_{(\Z_p, \Z_p)} \rightarrow \Set, \quad (R, R^+)\mapsto \Hom_{\Group}^{\mathrm{cts}}(T_{\GL_2}(\Z_p), R^\times)
\end{equation}
is represented by the Iwasawa algebra $(\Z_p[\![ T_{\GL_2}(\Z_p)]\!], \Z_p[\![ T_{\GL_2}(\Z_p)]\!])$. The $p$-adic weight space associated with $B$, also known as the full $p$-adic weight space for $p$-adic automorphic forms for $\GSp_4$, is defined as \[
    \calW = \Spa(\Z_p[\![ T_{\GL_2}(\Z_p)]\!], \Z_p[\![ T_{\GL_2}(\Z_p)]\!])^{\mathrm{rig}},
\]
where the superscript `$\bullet^{\mathrm{rig}}$' stands for the associated rigid analytic space over $\Spa(\Q_p, \Z_p)$. It is easy to see that $\calW$ is a finite disjoint union of two-dimensional open unit balls. 

Note that for any $\kappa \in \calW(R, R^+)$, \[
    \kappa(\diag(\tau_1, \tau_2, \tau_0\tau_2^{-1}, \tau_0\tau_1^{-1})) = \kappa_1(\tau_1)\kappa_2(\tau_2),
\]
where each $\kappa_i: \Z_p^\times \rightarrow R^\times$ is a continuous group homomorphism. Applying the proof of \cite[Lemma 3.2.1]{LW-BianchiAdjoint}, one sees that the image of $\kappa$ lies in $R^{\circ, \times}$.

Following \cite[Definition 3.3]{BSW-ParabolicEigen}, we make the following definition. 

\begin{Definition}\label{Definition: weight spaces for parabolic subgroups}
    \begin{enumerate}
        \item[(i)] The \textbf{Siegel $p$-adic weight space} $\calW_{\mathrm{Si}}$ is the rigid analytic space associated with the representable functor \[
            \Alg_{(\Z_p, \Z_p)} \rightarrow \Set, \quad (R, R^+)\mapsto \Hom_{\Group}^{\mathrm{cts}}(\GL_2(\Z_p), R^\times).
        \]
        In other words, \[
            \calW_{\mathrm{Si}} = \Spa(\Z_p[\![\GL_2(\Z_p)^{\mathrm{ab}}]\!], \Z_p[\![\GL_2(\Z_p)^{\mathrm{ab}}]\!])^{\rig}.
        \]
        \item[(ii)] The \textbf{Klingen $p$-adic weight space} $\calW_{\mathrm{Kl}}$ is the rigid analytic space associated with the representable functor \[
            \Alg_{(\Z_p, \Z_p)} \rightarrow \Set, \quad (R, R^+)\mapsto \Hom_{\Group}^{\mathrm{cts}}(\Z_p^\times \times \SL_2(\Z_p), R^\times).
        \]
        In other words, \[
            \calW_{\mathrm{Kl}} = \Spa(\Z_p[\![(\Z_p^\times \times \SL_2(\Z_p))^{\mathrm{ab}}]\!], \Z_p[\![(\Z_p^\times \times \SL_2(\Z_p))^{\mathrm{ab}}]\!])^{\rig}.
        \]
        \item[(iii)] The \textbf{$p$-adic weight space $\calW_{\GSp_4}$ for $\GSp_4$}  is the rigid analytic space associated with the representable functor \[
            \Alg_{(\Z_p, \Z_p)} \rightarrow \Set, \quad (R, R^+) \mapsto \Hom_{\Group}^{\Set}(\Sp_4(\Z_p), R^\times).
         \] 
         In other words, \[
            \calW_{\GSp_4} = \Spa(\Z_p[\![\Sp_4(\Z_p)^{\mathrm{ab}}]\!], \Z_p[\![\Sp_4(\Z_p)^{\mathrm{ab}}]\!])^{\rig}.
         \]
    \end{enumerate}
\end{Definition}

To better understand the $p$-adic weight spaces defined above, we need the following lemmas. 

\begin{Lemma}\label{Lemma: abelianisation of the Levi}
    The following algebraic groups are considered over $\Z_p$. 
    \begin{enumerate}
        \item[(i)] The abelianisation of $\GL_2$ is given by the determinant map \[
            \det: \GL_2 \rightarrow \bbG_m \cong \GL_2^{\mathrm{ab}}.
        \]
        \item[(ii)] The abelianisation of $\bbG_m \times \SL_2$ is given by the first projection map \[
            \pr: \bbG_m \times \SL_2 \rightarrow \bbG_m\cong (\bbG_m \times \SL_2)^{\mathrm{ab}}.
        \]
        \item[(iii)] The abelianisation of $\Sp_4$ is the trivial group. 
    \end{enumerate}
\end{Lemma}
\begin{proof}
    It is a well-known fact in the theory of algebraic groups that the commutator group for $\GL_2$ is $\SL_2$, which is the kernel of the determinant map. Hence, (i) follows. Moreover, $(\bbG_m\times \SL_2)^{\mathrm{ab}} = \bbG_m \times \SL_2^{\mathrm{ab}}$, but $\SL_2^{\mathrm{ab}} = \{1\}$, so (ii) also follows. The third assertion is \cite[Proposition 1.a]{LSTX}. Note here that to establish the commutator for $\GL_2$ and $\Sp^{\mathrm{ab}} = \{1\}$, we used the assumption $p\neq 2$.
\end{proof}

\begin{Lemma}\label{Lemma: special one-dimensional subspaces in the full weight space}
    \begin{enumerate}
        \item[(i)] The subfunctor of \eqref{eq: defining functor for the full weight space} \[
            \Alg_{(\Z_p, \Z_p)} \rightarrow \Set, \quad (R, R^+) \mapsto \left\{ \kappa = (\kappa_1, \kappa_2)\in \Hom_{\Group}^{\mathrm{cts}}(T_{\GL_2}(\Z_p), R^\times)  : \kappa_1 = \kappa_2 \right\}
        \] is representable by $(\Z_p[\![\Z_p^\times]\!], \Z_p[\![\Z_p^\times]\!])$. The associated rigid analytic space is regarded as a closed subspace in $\calW$, denoted by $\calW_{\mathrm{par}}$.
        \item[(ii)] The subfunctor of \eqref{eq: defining functor for the full weight space} \[
            \Alg_{(\Z_p, \Z_p)} \rightarrow \Set, \quad (R, R^+) \mapsto \left\{ \kappa = (\kappa_1, \kappa_2)\in \Hom_{\Group}^{\mathrm{cts}}(T_{\GL_2}(\Z_p), R^\times)  : \kappa_2 \text{ is trivial} \right\}
        \] is representable by $(\Z_p[\![\Z_p^\times]\!], \Z_p[\![\Z_p^\times]\!])$. The associated rigid analytic space is regarded as a closed subspace in $\calW$, denoted by $\calW_{1}$.
    \end{enumerate}
\end{Lemma}
\begin{proof}
    To show (i), for any $(R, R^+)\in \Alg_{(\Z_p, \Z_p)}$, consider the map \[
        \Hom_{\Group}^{\mathrm{cts}}(\Z_p^\times, R^\times) \rightarrow \Hom_{\Group}^{\mathrm{cts}}(T_{\GL_2}(\Z_p), R^\times), \quad \kappa' \mapsto \left( \diag(\tau_1, \tau_2) \mapsto \kappa'(\tau_1\tau_2) = \kappa'(\tau_1)\kappa'(\tau_2)\right).
    \]
    We claim that this map induces an identification \[
        \Hom_{\Group}^{\mathrm{cts}}(\Z_p^\times, R^\times) = \left\{ \kappa = (\kappa_1, \kappa_2)\in \Hom_{\Group}^{\mathrm{cts}}(T_{\GL_2}(\Z_p), R^\times)  : \kappa_1 = \kappa_2 \right\}.
    \] 
    For the injectivity, if $\kappa'(\tau_1\tau_2) = 1$ for all $\diag(\tau_1, \tau_2)\in T_{\GL_2}(\Z_p)$, then $\kappa' = 1$ since $T_{\GL_2}(\Z_p) \xrightarrow{\diag(\tau_1, \tau_2)\mapsto \tau_1\tau_2} \Z_p^\times$ is surjective; on the other hand, if $\kappa = (\kappa_1, \kappa_2) \in \Hom_{\Group}^{\mathrm{cts}}(T_{\GL_2}(\Z_p), R^\times)$ with $\kappa_1 = \kappa_2$, then it is obviously the image of $\kappa_1 \in \Hom_{\Group}^{\mathrm{cts}}(\Z_p^\times, R^\times)$. Therefore, one concludes that the subfunctor in consideration is represented by the quotient map \begin{equation}\label{eq: parallel embedding for weights}
        (\Z_p[\![T_{\GL_2}(\Z_p)]\!], \Z_p[\![T_{\GL_2}(\Z_p)]\!]) \rightarrow (\Z_p[\![\Z_p^\times]\!], \Z_p[\![\Z_p^\times]\!]), \quad \diag(\tau_1, \tau_2)\mapsto \tau_1\tau_2.
    \end{equation}
    
    The proof for (ii) is similar. We only point out that the subfunctor in (ii) is represented by the quotient \begin{equation}\label{eq: x-axis embedding for weights}
        (\Z_p[\![T_{\GL_2}(\Z_p)]\!], \Z_p[\![T_{\GL_2}(\Z_p)]\!]) \rightarrow (\Z_p[\![\Z_p^\times]\!], \Z_p[\![\Z_p^\times]\!]), \quad \diag(\tau_1, \tau_2)\mapsto \tau_1
    \end{equation}
    and leave the detailed arguments to the readers. 
\end{proof}

\begin{Corollary}\label{Corollary: explicit description of weight spaces}
    We have the following identifications of rigid analytic spaces: \begin{enumerate}
        \item[(i)]  $\calW_{\mathrm{Si}} = \calW_{\mathrm{par}}$;
        \item[(ii)] $\calW_{\mathrm{Kl}} = \calW_{1}$;
        \item[(iii)] $\calW_{\GSp_4} = \Spa(\Q_p, \Z_p)$. 
    \end{enumerate}
\end{Corollary}
\begin{proof}
    The third assertion is obvious from the definition and Lemma \ref{Lemma: abelianisation of the Levi} (iii). In what follows, we focus on the proof for (i) and (ii).

    For (i), given any $(R, R^+)\in \Alg_{(\Z_p, \Z_p)}$, we have to show the existence of the identification\[
        \Hom_{\Group}^{\mathrm{cts}}(\GL_2(\Z_p), R^\times)  = \left\{ \kappa = (\kappa_1, \kappa_2)\in \Hom_{\Group}^{\mathrm{cts}}(T_{\GL_2}(\Z_p), R^\times)  : \kappa_1 = \kappa_2 \right\}
    \]
    Given any $\kappa \in \Hom_{\Group}^{\mathrm{cts}}(\GL_2(\Z_p), R^\times)$, we view it as a continuous group homomorphism on $T_{\GL_2}(\Z_p)$ by \[
        T_{\GL_2}(\Z_p) \xrightarrow{\text{natural embedding}} \GL_2(\Z_p) \xrightarrow{\kappa}R^\times.
    \]
    Obviously, this provides an injection $\Hom_{\Group}^{\mathrm{cts}}(\GL_2(\Z_p), R^\times) \hookrightarrow \Hom_{\Group}^{\mathrm{cts}}(T_{\GL_2}(\Z_p), R^\times)$.
    On the other hand, since $\kappa$ factors as \[
        \begin{tikzcd}
            \GL_2(\Z_p) \arrow[rr, "\kappa"]\arrow[dr, "\det"'] && R^\times\\
            & \GL_2(\Z_p)^{\mathrm{ab}} = \Z_p^\times\arrow[ru]
        \end{tikzcd},
    \]
    and by the proof of Lemma \ref{Lemma: special one-dimensional subspaces in the full weight space}, we see that the map $\Hom_{\Group}^{\mathrm{cts}}(\GL_2(\Z_p), R^\times) \hookrightarrow \Hom_{\Group}^{\mathrm{cts}}(T_{\GL_2}(\Z_p), R^\times)$ has image equal to $\left\{ \kappa = (\kappa_1, \kappa_2)\in \Hom_{\Group}^{\mathrm{cts}}(T_{\GL_2}(\Z_p), R^\times)  : \kappa_1 = \kappa_2 \right\}$ as desired.

    For (ii), given any $(R, R^+)\in \Alg_{(\Z_p, \Z_p)}$, we have to show the existence of the identification \[
        \Hom_{\Group}^{\mathrm{cts}}(\Z_p^\times\times \SL_2(\Z_p), R^\times)  = \left\{ \kappa = (\kappa_1, \kappa_2)\in \Hom_{\Group}^{\mathrm{cts}}(T_{\GL_2}(\Z_p), R^\times)  : \kappa_2 \text{ is trivial} \right\}.
    \]
    Given any $\kappa\in \Hom_{\Group}^{\mathrm{cts}}(\Z_p^\times \times \SL_2(\Z_p), R^\times)$, we can again view it as a continuous group homomorphism on $T_{\GL_2}(\Z_p)$ by \[
        T_{\GL_2}(\Z_p) \xrightarrow{\diag(\tau_1, \tau_2) \mapsto (\tau_1, \diag(\tau_2, \tau_2^{-1}))} \Z_p^\times \times \SL_2(\Z_p) \xrightarrow{\kappa} R^\times. 
    \]
    Similar arguments as above then allows one to conclude the desired identification. 
\end{proof}

\begin{Remark}\label{Remark: p-adic weight space for GSp4}
    As shown in Corollary \ref{Corollary: explicit description of weight spaces}, $\calW_{\GSp_4}$ is nothing but the point $\Spa(\Q_p ,\Z_p)$, it is out of our interests since we aim to study \emph{$p$-adic families} of automorphic forms. Therefore, in what follows, when referring to `parabolic subgroups' without any specification, we shall mean the parabolic subgroups $P_{\mathrm{Si}}$, $P_{\mathrm{Kl}}$, and $B$.
\end{Remark}

\begin{Remark}\label{Remark: weight space mod out the wild twist}
    Careful readers may find our definitions of weight spaces are slightly different from the one in \cite{BSW-ParabolicEigen}, where the authors of \emph{loc. cit.} considered directly the continuous characters on the Levi subgroups of the parabolic subgroups. This is because we have factored out the `wild twists' in the present paper; whence, we are looking at the continuous group homomorphisms on the intersection of the Levi subgroups with $\Sp_4(\Z_p)$. Note that a similar discussion applies to the case of $\GSp_2 = \GL_2$: before factoring out the `wild twist', one gets the eigensurface, but one obtains the Coleman--Mazur eigencurve after factorising out the `wild twists' (see \cite[Sect. 4.6]{Hansen-PhD}). We chose the current presentation since the weights in our consideration resonate with the weights for (complex) classical Siegel modular forms of genus $2$. 
\end{Remark}

\begin{Remark}\label{Remark: simplifying notations of Levi}
    Following Remark \ref{Remark: weight space mod out the wild twist}, we make the following simplification of notations: \[
        M^* = M_B^* := T_{\GL_2}, \quad M_{\mathrm{Si}}^* = M_{P_{\mathrm{Si}}}^* := M_{\mathrm{Si}} \cap \Sp_4 = \GL_2, \quad \text{ and }\quad M^*_{\mathrm{Kl}} = M_{P_{\mathrm{Kl}}}^*  := M_{\mathrm{Kl}} \cap \Sp_4 = \bbG_m \times \SL_2.
    \]
    The two embeddings $M^*(\Z_p) \hookrightarrow M_{\mathrm{Si}}^*(\Z_p)$, $M^*(\Z_p) \hookrightarrow M_{\mathrm{Kl}}^*(\Z_p)$ discussed above are compatible with the embedding \[
        T_{\GL_2} \hookrightarrow \Sp_4, \quad \bftau \mapsto \begin{pmatrix}
            \bftau & \\ & \oneanti_2 \trans\bftau^{-1}\oneanti_2
        \end{pmatrix}.
    \]
    Moreover, observe that if $P \in \{B, P_{\mathrm{Si}}, P_{\mathrm{Kl}}\}$, we have $M_P^{\mathrm{ab}} \cong M_P^{*, \mathrm{ab}}$. In particular, any (continuous) character $\kappa$ on $M_P^*(\Z_p)$ can be viewed as a (continuous) character on $M_P$ by $\kappa((\bfgamma, \varsigma)) = \kappa(\bfgamma)$.
\end{Remark}

Finally, as in \cite{BSW-ParabolicEigen}, we may translate the weight spaces. 

\begin{Definition}\label{Definition: translate of weight spaces}
    Let $F$ be a complete field extension of $\Q_p$ in $\C_p$. Let $\kappa = (\kappa_1, \kappa_2) \in \calW(F)$. The \textbf{Siegel $p$-adic weight space} $\calW_{\mathrm{Si}, \kappa}$ (resp., \textbf{Klingen $p$-adic weight space} $\calW_{\mathrm{Kl}, \kappa}$) is defined to be the rigid analytic space associated with the representable subfunctor of \eqref{eq: defining functor for the full weight space} \begin{align*}
        \Alg_{(\Z_p, \Z_p)} \rightarrow \Set, \quad & (R, R^+)\mapsto \left\{ \kappa'= (\kappa_1', \kappa_2')\in \Hom_{\Group}^{\mathrm{cts}}(T_{\GL_2}(\Z_p), (R\widehat{\otimes}_{\Q_p}F)^\times): \kappa_1'\kappa_1^{-1} = \kappa_2'\kappa_2^{-1} \right\}\\
        \left( \text{resp., }\Alg_{(\Z_p, \Z_p)} \rightarrow \Set,\right. \quad & \left.(R, R^+)\mapsto \left\{ \kappa'= (\kappa_1', \kappa_2')\in \Hom_{\Group}^{\mathrm{cts}}(T_{\GL_2}(\Z_p), (R\widehat{\otimes}_{\Q_p}F)^\times):  \kappa_2'\kappa_2^{-1}=1 \right\} \right).
    \end{align*}
\end{Definition}

\subsection{Families of automorphic sheaves}\label{subsection: families of automorphic sheaves}
The goal of this subsection is to define and study \emph{small parabolic} $p$-adic families of (holomorphic) Siegel modular forms. Since the case where $P=B$ is studied in \cite{AIP-2015} and the case where $P= \GSp_4$ is not interesting (Corollary \ref{Corollary: explicit description of weight spaces}), we shall focus on the cases where $P=P_{\mathrm{Si}}$ or $P=P_{\mathrm{Kl}}$ in what follows.

Starting from this subsection, we move to the world of $p$-adic geometry. We denote by $\calX_{K_p}$ and $\calX_{K_p}^{\tor}$ the rigid analytic spaces over $\Q_p$ associated with $X_{K_p}$ and $X_{K_p}^{\tor}$ respectively. Moreover, note that $X_{\GSp_{4}(\Z_p)}^{\tor}$ has a smooth integral model $\frakX_{\GSp_{4}(\Z_p)}^{\tor}$;\footnote{ One can consider the moduli interpretation of $X_{\GSp_{4}(\Z_p)}$ over $\Z_p$, which yields a smooth scheme $X_{\GSp_{4}(\Z_p), \Z_p}$ over $\Z_p$. By fixing a cone decomposition, Faltings--Chai constructed the toroidal compactification $X_{\GSp_{4}(\Z_p), \Z_p}^{\tor}$. Then, $\frakX_{\GSp_{4}(\Z_p)}^{\tor}$ is nothing but the completion of $X_{\GSp_{4}(\Z_p), \Z_p}^{\tor}$ along the special fibre.} we let $\frakX_{K_p}^{\tor}$ be the normalisation of $\frakX_{\GSp_{4}(\Z_p)}^{\tor}$ in $\calX_{K_p}^{\tor}$. In the later discussion, we will consider the following levels at $p$: \begin{align*}
    K(p^n) & \coloneq \ker(\GSp_{4}(\Z_p) \rightarrow \GSp_4(\Z/p^n\Z)), \\
    \Iw_{\mathrm{Kl}, n} & \coloneq \left\{ \bfgamma\in \GSp_4(\Z_p): (\bfgamma \mod p^n)\in P_{\mathrm{Kl}}(\Z/p^n\Z) \right\},\\
    \Iw_{\mathrm{Si}, n} & \coloneq \left\{ \bfgamma\in \GSp_4(\Z_p): (\bfgamma \mod p^n)\in P_{\mathrm{Si}}(\Z/p^n\Z) \right\},\\
    \Iw_{B, n} & \coloneq \left\{ \bfgamma\in \GSp_4(\Z_p): (\bfgamma \mod p^n)\in B(\Z/p^n\Z) \right\}
\end{align*}
for any $n\in \Z_{>0}$.

In what follows, we will choose (auxiliary) non-negative rational numbers $w, w', v$ and consider $p^{w}, p^{w'}, p^v$. We shall then (implicitly) base change the rigid analytic spaces (resp., formal schemes) above to $F$ (resp., $\calO_F$), where $F$ is a finite field extension of $\Q_p$ containing $p^{w}, p^{w'}, p^v$. However, we shall abuse the notation and still use the same symbols to denote them.

\paragraph{The Klingen case.} We start with the Klingen case by summarising the construction in \cite{Pilloni-higherHidaColemanGSp4}; we refer the readers to \cite[Sect. 12]{Pilloni-higherHidaColemanGSp4} for more detailed discussions.  

Denote by $\frakG^{\univ} \rightarrow \frakX_{\GSp_{4}(\Z_p)}^{\tor}$ the tautological semiabelian variety and let $e$ be the identity section. We abuse the notation and again denote by $\underline{\omega}$ the pullback sheaf $e^* \Omega_{\frakG^{\univ}/\frakX_{\GSp_4(\Z_p)}^{\tor}}^1$. We shall keep abusing the notation and still denote by $\underline{\omega}$ the pullback of $\underline{\omega}$ to $\frakX_{K_p}^{\tor}$ for any open compact subgroup $K_p \subset \GSp_4(\Z_p)$.

According to \cite[Proposition 1.2]{Pilloni-Stroh}, there is a morphism of fppf sheaves \[
    \mathrm{HT}: (\Z/p^n\Z)^4 \rightarrow \underline{\omega}/p^n\underline{\omega}
\]
over $\frakX_{K(p^n)}^{\tor}$. It follows from \cite[Proposition 1.10]{Pilloni-Stroh} that there exists a formal scheme $\frakX_{K(p^n)}^{\tor, \mathrm{mod}} \rightarrow \frakX_{K(p^n)}^{\tor}$, which is a normalisation of an admissible formal blowup and carries a locally-free-rank-$2$ modification $\underline{\omega}^{\mathrm{mod}} \hookrightarrow \underline{\omega}$ with the following properties: \begin{itemize}
    \item $p^{\frac{1}{p-1}}\underline{\omega} \subset \underline{\omega}^{\mathrm{mod}} \subset \underline{\omega}$; 
    \item the map $\mathrm{HT}$ factors through $\underline{\omega}^{\mathrm{mod}}/p^n\underline{\omega}^{\mathrm{mod}}$, inducing a surjection \[
        (\Z/p^n\Z)^4\otimes_{\Z} \scrO_{\frakX_{K(p^n)}^{\tor, \mathrm{mod}}} \rightarrow \underline{\omega}^{\mathrm{mod}}/p^{n-\frac{1}{p-1}}\underline{\omega}^{\mathrm{mod}}.
    \]
\end{itemize}
Note that the rigid generic fibre of $\frakX_{K(p^n)}^{\tor, \mathrm{mod}}$ agrees with $\calX_{K(p^n)}^{\tor}$.

Let $\{e_1, e_2, e_3, e_4\}$ be the standard basis on $(\Z/p^n\Z)^4$. For any $v\in [0, n-\frac{1}{p-1}]\cap \Q$, let $\frakX_{K(p^n)}^{\tor}(v)$ be the open formal subscheme of the admissible blowup of $\frakX_{K(p^n)}^{\tor, \mathrm{mod}}$ along the ideal $(\mathrm{HT}(e_1), p^v)$ on which $(\mathrm{HT}(e_1), p^v)$ is generated by $\mathrm{HT}(e_1)$. In particular, we have a natural morphism $\frakX_{K(p^n)}^{\tor}(v) \rightarrow \frakX_{K(p^n)}^{\tor, \mathrm{mod}}$ and $\mathrm{HT}(e_1)=0$ in $\underline{\omega}^{\mathrm{mod}}/p^v\underline{\omega}^{\mathrm{mod}}$ over $\frakX_{K(p^n)}^{\tor}(v)$. We then let $\scrF^{\mathrm{can}}_{v}$ be the coherent subsheaf of $\underline{\omega}^{\mathrm{mod}}/p^v\underline{\omega}^{\mathrm{mod}}$ generated by $\mathrm{HT}(e_2)$ and $\mathrm{HT}(e_3)$. By \cite[Lemma 12.2.2.1]{Pilloni-higherHidaColemanGSp4}, $\scrF_v^{\mathrm{can}}$ is locally free of rank $1$.

Over $\frakX_{K(p^n)}^{\tor, \mathrm{mod}}$, let $\mathfrak{FL}_n$ be the flag variety, parametrising locally-free-rank-$1$ direct summands $\Fil \underline{\omega}^{\mathrm{mod}} \subset \underline{\omega}^{\mathrm{mod}}$. For any rational number $0 \leq w \leq v$, we then define the formal scheme \[
    \mathfrak{FL}_{n, w, v} \rightarrow \mathfrak{FL}_n \times_{\frakX_{K(p^n)}^{\tor, \mathrm{mod}}} \frakX_{K(p^n)}^{\tor}(v) \rightarrow \frakX_{K(p^n)}^{\tor}(v)
\]
parametrising $\Fil \underline{\omega}^{\mathrm{mod}} \subset \underline{\omega}^{\mathrm{mod}}$ such that \[
    \Fil \underline{\omega}^{\mathrm{mod}}/p^w\Fil \underline{\omega}^{\mathrm{mod}} \cong \scrF_v^{\mathrm{can}}/p^w\scrF_v^{\mathrm{can}}.
\]
Moreover, for any rational number $0 \leq w'\leq w$, we define \[
    \mathfrak{FL}_{n, w', w, v}^+ \rightarrow \mathfrak{FL}_{n, w, v}
\]
be the normal admissible formal scheme parametrising the trivialisations $\rho: \scrO_{\frakX_{K(p^n)}^{\tor}(v)} \cong \underline{\omega}^{\mathrm{mod}}/\Fil \underline{\omega}^{\mathrm{mod}}$ such that \[
    \rho \equiv \left( (\Z/p^n\Z)^4/\langle e_1, e_2, e_3\rangle \otimes (\scrO_{\frakX_{K(p^n)}^{\tor}(v)}/p^v) \xrightarrow{\mathrm{HT}_4} \underline{\omega}^{\mathrm{mod}}/\scrF_v^{\mathrm{can}} \right) \mod p^{w'}.
\]

We now take a closer look at the relationships between these formal schemes. For any rational number $w'\geq 0$, let $\frakT_{w'}$ and $\frakT_{w'}^0$ be the formal group schemes such that for any admissible $\calO_{F}$-algebra $A$ (in the sense of \cite[Sect. 4]{AIP-2015}) such that $p^{w'}\in A$, \begin{align*}
    \frakT_{w'}(A) & = \Z_p^\times(1+p^{w'}A),\\
    \frakT_{w'}^0(A) & = 1+p^{w'}A.
\end{align*}
Observe that $\frakT_{w'}$ acts on $\mathfrak{FL}_{n, w', w, v}^+$ as it acts on $\rho$; in fact, $\mathfrak{FL}_{n, w', w, v}^+ \rightarrow \mathfrak{FL}_{n, w, v}$ is a $\frakT_{w'}^0$-torsor. 

For any integer $n\geq w'$, observe that there are morphisms \[
    P_{\mathrm{Kl}}(\Z/p^n\Z) \rightarrow (\Z/p^n\Z)^\times \rightarrow \frakT_{w'}/\frakT_{w'}^0,
\]
where the first map is given by the projection to the last diagonal entry and the second map is the natural projection (since $n\geq w'$). We then define \[
    \frakT_{w',n} \coloneq \frakT_{w'} \times_{\frakT_{w'}/\frakT_{w'}^0} P_{\mathrm{Kl}}(\Z/p^n\Z).
\]
One sees that the $\frakT_{w'}^0$-action on $\mathfrak{FL}_{n, w', w,v}^+$ can be extended to an action of $\frakT_{w', n}$, which induces the $P_{\mathrm{Kl}}(\Z/p^n\Z)$-action on $\frakX_{K(p^n)}^{\tor}(v)$. We denote by $\frakX_{\Iw_{\mathrm{Kl},n}}^{\tor}(v)$ the quotient $\frakX_{K(p^n)}^{\tor}(v)/P_{\mathrm{Kl}}(\Z/p^n\Z)$. Note that this quotient exists as a formal scheme (\cite[Theorem 3.3.4]{Zavyalov-quotient}).

To construct the families of automorphic sheaves, consider the rigid generic fibres of the formal schemes discussed above. We shall use the calligraphic font to denote these rigid analytic spaces. In particular, we have $\mathcal{FL}_{n, w',w,v}^+$, $\mathcal{FL}_{n, w, v}$, $\calX_{K(p^n)}^{\tor}(v)$, $\calX_{\Iw_{\mathrm{Kl}, n}}^{\tor}(v)$ as well as $\calT_{w'}$, $\calT_{w'}^0$, and $\calT_{w', n}$. Summing up, we have a chain of morphisms \[
    \pi: \mathcal{FL}_{n, w',w,v}^+ \rightarrow \mathcal{FL}_{n, w, v} \rightarrow \calX_{K(p^n)}^{\tor}(v) \rightarrow \calX_{\Iw_{\mathrm{Kl}, n}}^{\tor}(v),
\]
where the first morphism is a $\calT_{w'}^0$-torsor and the last map is a $P_{\mathrm{Kl}}(\Z/p^n/Z)$-torsor. Moreover, $\calX_{K(p^n)}^{\tor}(v)$ and $\calX_{\Iw_{\mathrm{Kl}, n}}^{\tor}(v)$ are open subspaces of $\calX_{\Iw_{\mathrm{Kl}, n}}^{\tor}$ and $\calX_{\Iw_{\mathrm{Kl}, n}}^{\tor}$ respectively.

Recall the weight space $\calW_{\mathrm{Kl}}$ (Corollary \ref{Corollary: explicit description of weight spaces}). For any affinoid open subspace $\calU = \Spa(R_{\calU}, R_{\calU}^{\circ}) \subset \calW_{\mathrm{Kl}}$, we view its universal weight as a continuous group homomorphism $\kappa_{\calU}: \Z_p^\times \rightarrow R_{\calU}^{\times}$. By \cite[Lemma 3.4.6]{Urban-2011}, there exists $w_{\calU}\in \Q_{>0}$ such that $\kappa_{\calU}$ can be extended to a continuous group homomorphism \[
    \kappa_{\calU}: \Z_p^{\times}(1+p^{w_{\calU}}\calO_{\C_p}) \rightarrow \left(R_{\calU}^0 \widehat{\otimes}_{\Z_p}\calO_{\C_p}\right)^\times. 
\]

\begin{Definition}\label{Definition: families of automorphic sheaves, Klingen}
    Keep the notations as above. Choose $w', w, v, n$ such that $w_{\calU} \leq w'\leq w\leq v$ and $v\in [0, n-\frac{1}{p-1}]$. Fix $k_2 \in \Z$. \begin{enumerate}
        \item[(i)] The \textbf{$w$-analytic automorphic sheaf in Klingen family of weight $(\kappa_{\calU}, k_2)$} on $\calX_{\Iw_{\mathrm{Kl}, n}}^{\tor}(v)$ is defined to be \[
            \underline{\omega}_{w, v}^{(\kappa_{\calU}, k_2)} \coloneq \left( \pi_* \scrO_{\mathcal{FL}_{n, w', w, v}^+} \widehat{\otimes}_{\Q_p} R_{\calU} \right)^{\calT_{w',n}} \otimes_{\scrO_{\calX_{\Iw_{\mathrm{Kl}, n}}^{\tor}(v)}} (\det \underline{\omega})^{\otimes k_2},
        \]
        where $\calT_{w', n}$ acts diagonally and its action on $R_{\calU}$ is given by $\kappa_{\calU}-k_2$. We write \[
            M_{(\kappa_{\calU}, k_2)}^{w, v}(\Iw_{\mathrm{Kl}, n}) \coloneq H^0(\calX_{\Iw_{\mathrm{Kl}, n}}^{\tor}(v), \underline{\omega}_{w, v}^{(\kappa_{\calU}, k_2)}).
        \]
        \item[(ii)] The \textbf{cuspidal $w$-analytic automorphic sheaf in Klingen family of weight $(\kappa_{\calU}, k_2)$} on $\calX_{\Iw_{\mathrm{Kl}, n}}^{\tor}(v)$ is defined to be \[
            \underline{\omega}_{w, v, \mathrm{cusp}}^{(\kappa_{\calU}, k_2)} \coloneq \underline{\omega}_{w, v}^{(\kappa_{\calU}, k_2)} (-\partial \calX_{\Iw_{\mathrm{Kl},n}}^{\tor}(v)),
        \]
        where $\partial \calX_{\Iw_{\mathrm{Kl},n}}^{\tor}(v) = \partial \calX_{\Iw_{\mathrm{Kl}, n}}^{\tor} \cap \calX_{\Iw_{\mathrm{Kl}, n}}^{\tor}(v)$. We write \[
            S_{(\kappa_{\calU}, k_2)}^{w, v}(\Iw_{\mathrm{Kl}, n}) \coloneq H^0(\calX_{\Iw_{\mathrm{Kl}, n}}^{\tor}(v), \underline{\omega}_{w, v, \mathrm{cusp}}^{(\kappa_{\calU}, k_2)}).
        \]
    \end{enumerate} 
\end{Definition}

\begin{Remark}\label{Remark: families of automorphic sheaves, Klingen}
    \begin{enumerate}
        \item[(i)] Compared with the convention in \cite{Pilloni-higherHidaColemanGSp4}, we choose to let $\calT_{w', n}$ act on $R_{\calU}$ via $\kappa_{\calU}-k_2$ (instead of $\kappa_{\calU}$). This is because the line bundle $(\det\underline{\omega})^{\otimes k_2}$ has weight $(k_2, k_2)$ in our convention but $(0, k_2)$ in Pilloni's convention. 
        \item[(ii)] By \cite[Remark 12.6.2.1]{Pilloni-higherHidaColemanGSp4}, this sheaf is independent to the choice of $w'$. 
    \end{enumerate}
\end{Remark}

The following theorem summarises some nice properties of the automorphic sheaves in Klingen families defined above. 

\begin{Theorem}[Pilloni]\label{Theorem: nice properties of automorphic sheaves in Klingen families}
    Keep the notations as in Definition \ref{Definition: families of automorphic sheaves, Klingen}. \begin{enumerate}
        \item[(i)] The complexes $R\Gamma(\calX_{\Iw_{\mathrm{Kl}, n}}^{\tor}(v), \underline{\omega}_{w, v}^{(\kappa_{\calU}, k_2)})$ and  $R\Gamma(\calX_{\Iw_{\mathrm{Kl}, n}}^{\tor}(v), \underline{\omega}_{w, v, \mathrm{cusp}}^{(\kappa_{\calU}, k_2)})$ are represented by bounded complexes of Banach $R_{\calU}$-modules that has {\normalfont (Pr)} in the sense of \cite{Buzzard_2007}. 
        \item[(ii)] The Hecke operator $U_{\mathrm{Kl}}$ acts compactly on both complexes $R\Gamma(\calX_{\Iw_{\mathrm{Kl}, n}}^{\tor}(v), \underline{\omega}_{w, v}^{(\kappa_{\calU}, k_2)})$ and  $R\Gamma(\calX_{\Iw_{\mathrm{Kl}, n}}^{\tor}(v), \underline{\omega}_{w, v, \mathrm{cusp}}^{(\kappa_{\calU}, k_2)})$. 
        \item[(iii)] Let $h\in \Q_{\geq 0}$ such that the slope-$\leq h$ parts \[
            R\Gamma(\calX_{\Iw_{\mathrm{Kl}, n}}^{\tor}(v), \underline{\omega}_{w, v}^{(\kappa_{\calU}, k_2)})^{\leq h} \quad \text{ and }\quad R\Gamma(\calX_{\Iw_{\mathrm{Kl}, n}}^{\tor}(v), \underline{\omega}_{w, v, \mathrm{cusp}}^{(\kappa_{\calU}, k_2)})^{\leq h},
        \]
        with respect to $U_{\mathrm{Kl}}$, is well-defined (\cite[Sect. 13.1.2]{Pilloni-higherHidaColemanGSp4}). Then, the slope-$\leq h$ parts are independent to the choices of $w, v, n$. 
    \end{enumerate}
\end{Theorem}
\begin{proof}
    We refer the readers to Proposition 12.8.2.1 (for (i)), Corollary 13.2.4.1 (for (ii)), and Corollary 13.2.4.2 (for (iii)) in \cite{Pilloni-higherHidaColemanGSp4}. 
\end{proof}

\begin{Remark}\label{Remark: Pilloni proved more for cuspidal Klingen theory}
    We remark that Pilloni proved more: Given $v$ such that $v \geq w$, there exists $n_v\in \Z_{\geq 0}$ such that for any $n\geq n_{v}$, the complex $R\Gamma(\calX_{\Iw_{\mathrm{Kl}, n}}^{\tor}(v), \underline{\omega}_{w, v, \mathrm{cusp}}^{(\kappa_{\calU}, k_2)})$ is concentrated in degree $[0,1]$ (\cite[Corollary 12.2.9.1]{Pilloni-higherHidaColemanGSp4}). However, we don't need this stronger result in this paper.  
\end{Remark}

The construction above also apply to single weights $\kappa: \Z_p^\times \rightarrow F^\times$. In particular, we have complexes  $R\Gamma(\calX_{\Iw_{\mathrm{Kl}, n}}^{\tor}(v), \underline{\omega}_{w, v, \mathrm{cusp}}^{(k_1, k_2)})$ and $R\Gamma(\calX_{\Iw_{\mathrm{Kl}, n}}^{\tor}(v), \underline{\omega}_{w, v, \mathrm{cusp}}^{(k_1, k_2)})$ for any $k_1 \geq k_2$. A natural question is then to compare with the classical complexes $R\Gamma(X_{\Iw_{\mathrm{Kl}, n}}^{\tor}, \underline{\omega}^{(k_1, k_2)})$ and $R\Gamma(X_{\Iw_{\mathrm{Kl}, n}}^{\tor}, \underline{\omega}^{(k_1, k_2)}_{\mathrm{cusp}})$. One observes from the construction that there are natural maps \begin{equation}\label{eq: comparison map of classical sheaf and overconvergent sheaf, Klingen}
    R\Gamma(X_{\Iw_{\mathrm{Kl}, n}}^{\tor}, \underline{\omega}^{(k_1, k_2)}) \rightarrow  R\Gamma(\calX_{\Iw_{\mathrm{Kl}, n}}^{\tor}(v), \underline{\omega}^{(k_1, k_2)}) \rightarrow R\Gamma(\calX_{\Iw_{\mathrm{Kl}, n}}^{\tor}(v), \underline{\omega}_{w, v}^{(k_1, k_2)}),
\end{equation}
where the first map is the restriction and the second map is constructed in \cite[Lemma 13.3.2.1]{Pilloni-higherHidaColemanGSp4}. Note that this map is Hecke-equivariant. Similar morphisms also apply to the cuspidal cohomology.

\begin{Theorem}[Classicality]\label{Theorem: classicality theorem, Klingen}
    Let $(k_1, k_2)\in \Z^2$ with $k_1\geq k_2>2$. Then, \eqref{eq: comparison map of classical sheaf and overconvergent sheaf, Klingen} induces an isomorphism \[
        M_{(k_1, k_2)}(\Iw_{\mathrm{Kl}, n})^{< k_1-k_2+1} \cong M_{(k_1, k_2)}^{w, v}(\Iw_{\mathrm{Kl}, n})^{< k_1-k_2+1}
    \]
    A similar result holds for the space of cuspforms. 
\end{Theorem}
\begin{proof}
    This is \cite[Corollary 14.7.1 (2)]{Pilloni-higherHidaColemanGSp4}. We simply make the following remarks:\begin{itemize}
        \item Translate from Pilloni's convention to our, the slope bound in the original statement should be $\max\{k_1+k_2-3, k_1-k_2+1\}$. However, since we assume $k_2>2$, $k_2>2$, $k_1+k_2-3 > k_1-k_2+1$.
        \item Pilloni actually proved a stronger version. He did not assume $k_2>2$ and he also showed results for higher cohomology groups. For our later purpose, we are only interested in $H^0$ and the case when $k_2>2$. 
    \end{itemize} 
\end{proof}

\begin{Corollary}\label{Corollary: finite projective of cuspidal Siegel forms in Klingen family}
    Keep the notations as in Definition \ref{Definition: families of automorphic sheaves, Klingen}. Suppose $k_2>2$ and let $h\in \Q_{\geq 0}$ such that the slope-$\leq h$ part of $R\Gamma(\calX_{\Iw_{\mathrm{Kl}, n}}^{\tor}(v), \underline{\omega}_{w, v}^{(\kappa_{\calU}, k_2)})$ is well-defined. Then, \[
        M_{(\kappa_{\calU}, k_2)}^{w, v}(\Iw_{\mathrm{Kl}, n})^{\leq h} = H^0(R\Gamma(\calX_{\Iw_{\mathrm{Kl}, n}}^{\tor}(v), \underline{\omega}_{w, v}^{(\kappa_{\calU}, k_2)})^{\leq h})
    \] 
    is a non-trivial finite projective $R_{\calU}$-module. A similar result also holds for the cuspidal case. 
\end{Corollary}
\begin{proof}
    In what follows, we prove the case for $\underline{\omega}_{w, v}^{(\kappa_{\calU}, k_2)}$. The cuspidal case can be proven similarly. 
    
    By Theorem \ref{Theorem: nice properties of automorphic sheaves in Klingen families}, we may represent $R\Gamma(\calX_{\Iw_{\mathrm{Kl}, n}}^{\tor}(v), \underline{\omega}_{w, v}^{(\kappa_{\calU}, k_2)})$ by a complex \[
        C^0 \rightarrow C^1 \rightarrow \cdots \rightarrow C^m
    \]
    of Banach $R_{\calU}$-modules that has (Pr). By the definition of slope-$\leq h$ part, we have a Hecke-equivariant commutative diagram \[
        \begin{tikzcd}
            C^0 \arrow[r] \arrow[d] & C^1\arrow[d] \arrow[r] & \cdots \arrow[r] & C^m\arrow[d]\\
            C^{0, \leq h}\arrow[r] \arrow[u, bend right = 20] & C^{1, \leq h}\arrow[u, bend right = 20]\arrow[r] & \cdots \arrow[r] & C^{m, \leq h}\arrow[u, bend right = 20] 
        \end{tikzcd},
    \]
    where the vertical maps mapping from the top to the bottom are surjections and the maps mapping from the bottom to the top are sections of the surjections. This shows that \[
        M_{(\kappa_{\calU}, k_2)}^{w, v}(\Iw_{\mathrm{Kl}, n})^{\leq h} = H^0(R\Gamma(\calX_{\Iw_{\mathrm{Kl}, n}}^{\tor}(v), \underline{\omega}_{w, v}^{(\kappa_{\calU}, k_2)})^{\leq h}). 
    \]

    To show $M_{(\kappa_{\calU}, k_2)}^{w, v}(\Iw_{\mathrm{Kl}, n})^{\leq h}$ is non-trivial, note that the intersection $\calU_{\mathrm{reg},h-\mathrm{small}}$ of $\{k\in \Z_{\geq k}: h < k+1\}$ and $\calU = \Spa(R_{\calU}, R_{\calU}^{\circ})$ is a Zariski dense subset in $\calU$. Hence, for any $k_1\in \Z_{\geq k_2}$ such that $k_1-k_2$ defines a classical weight in $\calU_{\mathrm{reg},h-\mathrm{small}}$, we have \[
        M_{(k_1, k_2)}^{w, v}(\Iw_{\mathrm{Kl}, n})^{\leq h} \cong M_{(k_1, k_2)}(\Iw_{\mathrm{Kl}, n})^{\leq h}
    \]
    by Theorem \ref{Theorem: classicality theorem, Klingen} and the latter space is non-trivial (\cite[Example 4.17]{Lan-vanishing}). In particular, we see that the support of $M_{(\kappa_{\calU}, k_2)}^{w, v}(\Iw_{\mathrm{Kl}, n})^{\leq h}$ contains a Zariski dense subset and so its support must be the entire $\calU$, which concludes the claim.

    We see that $M_{(\kappa_{\calU}, k_2)}^{w, v}(\Iw_{\mathrm{Kl}, n})^{\leq h}$ is a non-trivial submodule of $C^{0, \leq h}$, which is a finite projective $R_{\calU}$-module (by construction). Thus, $M_{(\kappa_{\calU}, k_2)}^{w, v}(\Iw_{\mathrm{Kl}, n})^{\leq h}$ is a finite torsion-free $R_{\calU}$-module. However, $R_{\calU}$ is a one-dimensional noetherian normal domain, hence a Dedekind domain. Then, by \cite[\href{https://stacks.math.columbia.edu/tag/0AUW}{Tag 0AUW}]{stacks-project}, we conclude that $M_{(\kappa_{\calU}, k_2)}^{w, v}(\Iw_{\mathrm{Kl}, n})^{\leq h}$ is a finite projective $R_{\calU}$-module. 
\end{proof}

\paragraph{The Siegel case.} Now we turn our attention to the Siegel case. Our construction is inspired by the Klingen one as well as \cite{AIP-2015}.

Recall the Hodge--Tate map \[
    \mathrm{HT}: (\Z/p^n\Z)^4 \rightarrow \underline{\omega}/p^n\underline{\omega}
\]
over $\frakX_{K(p^n)}^{\tor}$ and the induced surjection \[
    (\Z/p^n\Z)^4 \otimes_{\Z} \scrO_{\frakX_{K(p^n)}^{\tor, \mathrm{mod}}} \rightarrow \underline{\omega}^{\mathrm{mod}}/p^{n-\frac{1}{p-1}}\underline{\omega}^{\mathrm{mod}}
\]
over $\frakX_{K(p^n)}^{\tor, \mathrm{mod}}$. For any $v\in [0, n-\frac{1}{p-1}]\cap \Q$, let $\frakX_{K(p^n)}^{\tor}(v,v)$ be the open formal subscheme of the admissible blowup of $\frakX_{K(p^n)}^{\tor}(v)$ along the ideal $(\mathrm{HT}(e_2), p^v)$ on which $(\mathrm{HT}(e_2), p^v)$ is generated by $\mathrm{HT}(e_2)$. We thus has a natural morphism $\frakX_{K(p^n)}^{\tor}(v,v) \rightarrow \frakX_{K(p^n)}^{\tor, \mathrm{mod}}$ and $\mathrm{HT}(e_i)=0$ ($i=1, 2$) in $\underline{\omega}^{\mathrm{mod}}/p^v\underline{\omega}^{\mathrm{mod}}$ over $\frakX_{K(p^n)}^{\tor}(v,v)$. In particular, over $\frakX_{K(p^n)}^{\tor}(v,v)$, the image of $\mathrm{HT}(e_3)\wedge \mathrm{HT}(e_4)$ generates $\det(\underline{\omega}^{\mathrm{mod}}/p^v\underline{\omega}^{\mathrm{mod}})$. Finally, we let $\frakX_{\Iw_{\mathrm{Si}, n}}^{\tor}(v,v)$ be the quotient $\frakX_{K(p^n)}^{\tor}(v,v)/P_{\mathrm{Si}}(\Z/p^n\Z)$.

Over $\frakX_{K(p^n)}^{\tor, \mathrm{mod}}$, consider the $\bbG_m$-torsor $\frakL_n = \mathrm{Isom}(\scrO_{\frakX_{K(p^n)}^{\tor, \mathrm{mod}}}, \det\underline{\omega}^{\mathrm{mod}})$. For any rational number $0\leq w \leq v$, we define the formal scheme \[
    \frakL_{n,w,v} \rightarrow \frakL_n\times_{\frakX_{K(p^n)}^{\tor, \mathrm{mod}}} \frakX_{K(p^n)}^{\tor}(v, v) \rightarrow \frakX_{K(p^n)}^{\tor}(v, v)
\]
parametrising trivialisations $\psi$ such that $\psi(1)$ generates $\det(\underline{\omega}^{\mathrm{mod}}/p^w\underline{\omega}^{\mathrm{mod}})$. Just for the sake of generalisation, for any $0\leq w'\leq w$, we can further define \[
    \frakL_{n, w', w,v} \rightarrow \frakL_{n, w, v}
\]
be the normal admissible formal scheme parametrising those trivialisations $\psi$ such that \[
    \psi(1) \equiv \mathrm{HT}(e_3)\wedge \mathrm{HT}(e_4) \mod p^{w'}.
\]
From the construction, we see that \[
    \frakL_{n, w, v} \rightarrow \frakX_{K(p^n)}^{\tor}(v, v)
\]
is a $\frakT_{w}$-torsor and \[
    \frakL_{n, w', w, v} \rightarrow \frakX_{K(p^n)}^{\tor}(v, v)
\]
is a $\frakT_{w'}^0$-torsor.

Again, we use the calligraphic font to denote the rigid generic fibres of these formal schemes. In particular, we have \[
    \calL_{n,w,v} \xrightarrow{\pi_{n,w,v}} \calX_{K(p^n)}^{\tor}(v,v) \xrightarrow{\pi_n^{\mathrm{Si}}} \calX_{\Iw_{\mathrm{Si}, n}}^{\tor}(v,v).
\]
Note that $\calX_{K(p^n)}^{\tor}(v,v)$ and $\calX_{\Iw_{\mathrm{Si}, n}}^{\tor}(v,v)$ are open subspaces of $\calX_{K(p^n)}^{\tor}$ and $\calX_{\Iw_{\mathrm{Pi}, n}}^{\tor}$ respectively.

Recall the weight space $\calW_{\mathrm{Si}}$ (Corollary \ref{Corollary: explicit description of weight spaces}). For any affinoid open subspace $\calU = \Spa(R_{\calU}, R_{\calU}^{\circ})$, we again view its universal weight sa a continuous group hommorphism $\kappa_{\calU}: \Z_p^\times \rightarrow R_{\calU}^{\times}$.

\begin{Definition}\label{Definition: families of automorphic sheaves, Siegel}
    Keep the notations as above. Choose $w, v, n$ such that $w_{\calU}\leq w'\leq w\leq v$ and $v\in [0,n-\frac{1}{p-1}]$. Fix $(k_1, k_2)\in \Z^2$ with $k_1 \geq k_2$. 
    \begin{enumerate}
        \item[(i)] The \textbf{$w$-analytic automorphic sheaf in Siegel family of weight $(\kappa_{\calU}, \kappa_{\calU})$ (resp., $(k_1+\kappa_{\calU}, k_2+\kappa_{\calU})$)} on $\calX_{\Iw_{\mathrm{Si}, n}}^{\tor}(v,v)$ is defined to be \begin{align*}
            \underline{\omega}_{w, v}^{(\kappa_{\calU}, \kappa_{\calU})} & \coloneq \left(\pi_{n,*}^{\mathrm{Si}} \left( \pi_{n,w, v, *}\scrO_{\calL_{n,w,v}}\widehat{\otimes}_{\Q_p}R_{\calU}\right)^{\calT_{w}} \right)^{P_{\mathrm{Si}}(\Z/p^n\Z)} \\ \\(\text{resp., } \underline{\omega}_{w, v}^{(k_1+\kappa_{\calU}, k_2+\kappa_{\calU})} & \coloneq \underline{\omega}_{w, v}^{(\kappa_{\calU}, \kappa_{\calU})} \otimes \underline{\omega}^{(k_1, k_2)}),
        \end{align*}
        where we let $\calT_{w}$ act on $R_{\calU}$ via $\kappa_{\calU}$. We write \[
            M_{(k_1+\kappa_{\calU}, k_2+\kappa_{\calU})}^{w, v}(\Iw_{\mathrm{Si}, n}) \coloneq H^0(\calX_{\Iw_{\mathrm{Si}, n}}^{\tor}(v,v), \underline{\omega}_{w, v}^{(k_1+\kappa_{\calU}, k_2+\kappa_{\calU})}).
        \]
        \item[(ii)]  The \textbf{ cuspidal $w$-analytic automorphic sheaf in Siegel family of weight  $(k_1+\kappa_{\calU}, k_2+\kappa_{\calU})$} on $\calX_{\Iw_{\mathrm{Si}, n}}^{\tor}(v,v)$ is defined to be \[
            \underline{\omega}_{w, v, \cusp}^{(k_1+\kappa_{\calU}, k_2+\kappa_{\calU})} \coloneq \underline{\omega}_{w, v}^{(k_1+\kappa_{\calU}, k_2+\kappa_{\calU})}(-\partial \calX_{\Iw_{\mathrm{Si}, n}}^{\tor}(v,v)),
        \]
        where $\partial \calX_{\Iw_{\mathrm{Si}, n}}^{\tor}(v,v) = \partial \calX_{\Iw_{\mathrm{Si}, n}}^{\tor} \cap \calX_{\Iw_{\mathrm{Si}, n}}^{\tor}(v,v)$. We write \[
            S_{(k_1+\kappa_{\calU}, k_2+\kappa_{\calU})}^{w, v}(\Iw_{\mathrm{Si}, n}) \coloneq H^0(\calX_{\Iw_{\mathrm{Si}, n}}^{\tor}(v,v), \underline{\omega}_{w, v, \cusp}^{(k_1+\kappa_{\calU}, k_2+\kappa_{\calU})}).
        \]
    \end{enumerate}
\end{Definition}

\begin{Remark}\label{Remark: construction of Siegel families of automorphic sheaf also works for single weight}
    \begin{enumerate}
        \item[(i)] From the construction, the sheaves $\underline{\omega}_{w, v}^{(\kappa_{\calU}, \kappa_{\calU})}$ and $\underline{\omega}_{w, v, \cusp}^{(\kappa_{\calU}, \kappa_{\calU})}$ are locally free of rank one over $\scrO_{\calX_{\Iw_{\mathrm{Si}, n}}^{\tor}(v,v)}\widehat{\otimes}R_{\calU}$. See also \cite[Lemma 2.2]{Pilloni-GL2}. 
        \item[(ii)] Given a single weight $\kappa: \Z_p^\times \rightarrow F^\times$, one can run through the construction above and obtain the sheaves $\underline{\omega}_{w, v}^{(\kappa, \kappa)}$, $\underline{\omega}_{w, v, \cusp}^{(\kappa, \kappa)}$.
    \end{enumerate}
\end{Remark}

\begin{Lemma}\label{Lemma: complex represented by Banach modules with (Pr), Siegel}
    Keep the notations as in Definition \ref{Definition: families of automorphic sheaves, Siegel}. The complexes $R\Gamma(\calX_{\Iw_{\mathrm{Si}, n}}^{\tor}(v,v), \underline{\omega}_{w, v}^{(k_1+\kappa_{\calU}, k_2+\kappa_{\calU})})$ and $R\Gamma(\calX_{\Iw_{\mathrm{Si}, n}}^{\tor}(v,v), \underline{\omega}_{w, v, \cusp}^{(k_1+\kappa_{\calU}, k_2+\kappa_{\calU})})$ are represented by a bounded complex of Banach $R_{\calU}$-modules that have {\normalfont (Pr)}. 
\end{Lemma}
\begin{proof}
    Since the sheaves $\underline{\omega}_{w, v}^{(\kappa_{\calU}, \kappa_{\calU})}$ and $\underline{\omega}_{w, v, \cusp}^{(\kappa_{\calU}, \kappa_{\calU})}$ are locally free of rank one over $\scrO_{\calX_{\Iw_{\mathrm{Si}, n}}^{\tor}(v,v)}\widehat{\otimes}R_{\calU}$, the sheaves $\underline{\omega}_{w, v}^{(k_1+\kappa_{\calU}, k_2+\kappa_{\calU})}$ and $\underline{\omega}_{w, v, \cusp}^{(k_1+\kappa_{\calU}, k_2+\kappa_{\calU})}$ are Banach sheaves in the sense of \cite[Definition A.2.1.2]{AIP-2015} (in fact, they are \emph{rigid generic fibres} of small formal Banach sheaves in the sense of \cite[Definition A.1.2.1]{AIP-2015}). By the construction, we may choose a finite affinoid open covering $\{\calU_i\}_{i\in I}$ for $\calX_{\Iw_{\mathrm{Si}, n}}^{\tor}(v,v)$ that trivialises both $\underline{\omega}_{w, v}^{(\kappa_{\calU}, \kappa_{\calU})}$ and $\underline{\omega}^{(k_1+k_2)}$, so $\underline{\omega}_{w, v}^{(k_1+\kappa_{\calU}, k_2+\kappa_{\calU})}$ and $\underline{\omega}_{w, v, \cusp}^{(k_1+\kappa_{\calU}, k_2+\kappa_{\calU})}$ are isomorphic to $(\scrO_{\calX_{\Iw_{\mathrm{Si}, n}}^{\tor}(v,v)}\widehat{\otimes}R_{\calU})^m$ (for some $m$) over each $\calU_i$. The Čech complexes associated with $\{\calU_i\}_{i\in I}$ compute $R\Gamma(\calX_{\Iw_{\mathrm{Si}, n}}^{\tor}(v,v), \underline{\omega}_{w, v}^{(k_1+\kappa_{\calU}, k_2+\kappa_{\calU})})$ and $R\Gamma(\calX_{\Iw_{\mathrm{Si}, n}}^{\tor}(v,v), \underline{\omega}_{w, v, \cusp}^{(k_1+\kappa_{\calU}, k_2+\kappa_{\calU})})$ (\cite[Theorem A.1.2.2]{AIP-2015}) and are bounded complexes of Banach $R_{\calU}$-modules having (Pr). 
\end{proof}

\begin{Lemma}\label{Lemma: specialisation of the Siegel families of automorphic sheaves are classical}
    Keep the notations as before. For any $k\in \Z$, there is an isomorphism of sheaves \[
        \underline{\omega}_{w, v}^{(k,k)} \cong \underline{\omega}^{(k,k)}|_{\calX_{\Iw_{\mathrm{Si}, n}}^{\tor}(v,v)}.
    \]
\end{Lemma}
\begin{proof}
    Consider \begin{align*}
        V_k^{\mathrm{alg}} & \coloneq \{\phi: \bbG_m \rightarrow \A^1: \phi(ab) = k(b)\phi(a) \,\,\forall a, b\in \bbG_m\},\\
        V_k^{w} & \coloneq \left\{\phi: \Z_p^\times \rightarrow \Q_p: \begin{array}{l}
            \phi(ab) = k(b)\phi(a) \,\,\forall a,b\in \Z_p^\times  \\
            \phi \text{ extends to $\Z_p^\times(1+p^w\calO_{\C_p})$} 
        \end{array} \right\}.
    \end{align*} By definition, we have a natural embedding \begin{equation}\label{eq: algebraic 1-dim representation embeds into w-ananlytic 1-dim representation}
        V_k^{\mathrm{alg}} \hookrightarrow V_k^{w},
    \end{equation}
    which is also compatible with the action of $\Z_p^\times$. 

    Let $\{\calV_i\}_{i\in I}$ be an open affinoid covering of $\calX_{\Iw_{\mathrm{Si}, n}}^{\tor}(v,v)$ on which $\underline{\omega}$ is trivialised. One sees from the construction that \[
        \underline{\omega}^{(k,k)}|_{\calV_i} \cong \scrO_{\calV_i} \otimes_{\Q_p} V_k^{\mathrm{alg}} \quad \text{ and } \underline{\omega}_{w, v}^{(k,k)}|_{\calV_i} \cong \scrO_{\calV_i} \otimes_{\Q_p} V_k^{w}.
    \] 
    for all $i$. Moreover, the natural inclusion \[
        \underline{\omega}^{(k,k)}|_{\calV_i} \hookrightarrow \underline{\omega}_{w, v}^{(k,k)}|_{\calV_i}
    \]
    glues to an inclusion \[
        \underline{\omega}^{(k,k)} \hookrightarrow \underline{\omega}_{w, v}^{(k,k)}.
    \]
    However, by \cite[Lemma 2.2]{Pilloni-GL2}, \eqref{eq: algebraic 1-dim representation embeds into w-ananlytic 1-dim representation} is an isomorphism, we thus finish the proof. 
\end{proof}

We now discuss the Hecke action. Note that the Hecke operators away from $p$ can be defined similarly as before, we thus only focus on the Hecke operator at $p$. Observe first that \[
    \Iw_{\mathrm{Si}, n}^+ \coloneq \Iw_{\mathrm{Si}, n} \cap \bfitu_{\mathrm{Si}}^{-1} \Iw_{\mathrm{Si}, n} \bfitu_{\mathrm{Si}} = \left\{ \begin{pmatrix} \bfgamma_a & \bfgamma_b\\ \bfgamma_c & \bfgamma_d \end{pmatrix}\in \GSp_4 (\Z_p) : \begin{array}{l}
        \bfgamma_c \equiv 0 \mod p^n  \\
        \bfgamma_b \equiv 0 \mod p 
    \end{array}\right\}.
\] Thus, $X_{\Iw_{\mathrm{Si}, n} \cap \bfitu_{\mathrm{Si}}^{-1} \Iw_{\mathrm{Si}, n} \bfitu_{\mathrm{Si}}}$ is the moduli space of the data $(A, \lambda, \psi^p, C_n, C)$, where $(A, \lambda)$ is a principally polarised abelian surface, $\psi^p$ is a level structure away from $p$ (given by $K^p$), $C_n \subset A[p^n]$ is an isotropic subgroup of rank $p^n$, and $C\subset A[p]$ is an isotropic subgroup of rank $p$ such that $C \cap C_n = 0$. Recall that $X_{\Iw_{\mathrm{Si}, n}}$ is the moduli space of the data $(A, \lambda, \psi^p, C_n)$. Then the correspondence \begin{equation}\label{eq: Sigel correspondence}
    \begin{tikzcd}
        & X_{\Iw_{\mathrm{Si}, n}^+} \arrow[ld, "\pr_1"']\arrow[rd, "\pr_2"]\\
        X_{\Iw_{\mathrm{Si}, n}} && X_{\Iw_{\mathrm{Si}, n}}
    \end{tikzcd}
\end{equation}
can be explicitly described as follows: \begin{itemize}
    \item the map $\pr_1$ sends $(A, \lambda, \psi^p, C_n, C)$ to $(A, \lambda, \psi^p, C_n)$; 
    \item the map $\pr_2$ sends $(A, \lambda, \psi^p, C_n, C)$ to $(A/C, \lambda^p, \psi^p_{A/C}, \image(C_n \rightarrow A/C))$, where $\lambda^p$ is the induced principal polarisation, $\psi^p_{A/C}$ is the induced level structure away from $p$. 
\end{itemize}
Again, \eqref{eq: Sigel correspondence} extends to the toroidal compactifications \begin{equation}\label{eq: Sigel correspondence, toroidal}
    \begin{tikzcd}
        & X_{\Iw_{\mathrm{Si}, n}^+}^{\tor, \Sigma''} \arrow[ld, "\pr_1"']\arrow[rd, "\pr_2"]\\
        X_{\Iw_{\mathrm{Si}, n}}^{\tor, \Sigma} && X_{\Iw_{\mathrm{Si}, n}}^{\tor, \Sigma'}
    \end{tikzcd},
\end{equation}
where $\Sigma, \Sigma', \Sigma''$ are suitable choices of cone decompositions.

\begin{Lemma}\label{Lemma: Siegel Hecke correspondence improves overconvergence}
    Let $\calX_{\Iw_{\mathrm{Si}, n}^+}^{\tor, \Sigma''}(v,v) \coloneq \calX_{\Iw_{\mathrm{Si}, n}^+}^{\tor, \Sigma''} \times_{\calX_{\Iw_{\mathrm{Si}, n}}^{\tor, \Sigma}} \calX_{\Iw_{\mathrm{Si}, n}}^{\tor, \Sigma}(v,v)$. Then, \eqref{eq: Sigel correspondence, toroidal} induces a correspondence \[
        \begin{tikzcd}
            & \calX_{\Iw_{\mathrm{Si}, n}^+}^{\tor, \Sigma''}(v,v) \arrow[ld, "\pr_1"'] \arrow[rd, "\pr_2"]\\
            \calX_{\Iw_{\mathrm{Si}, n}}^{\tor, \Sigma}(v,v) && \calX_{\Iw_{\mathrm{Si}, n}}^{\tor, \Sigma'}(v+1,v+1)
        \end{tikzcd}.
    \]
\end{Lemma}
\begin{proof}
    The proof is similar to \cite[Lemma 13.2.1.1]{Pilloni-higherHidaColemanGSp4}.
    
    As all adic spaces in the picture are topologically of finite type, it is enough to check on the rank-$1$ points. Let $x = \Spa(F, \calO_F)$ be a rank-$1$ point of $\calX_{\Iw_{\mathrm{Si}, n}^+}^{\tor, \Sigma''}(v,v)$, which, by construction, corresponds to an isogeny $\alpha: G \rightarrow G'$ of semiabelian schemes (if $x$ does not belong to the boundary, this is nothing but $A \rightarrow A/C$). Over $\calO_{\C_p}$, we have a commutative diagram \[
        \begin{tikzcd}
            T_p G \arrow[r, "\alpha"]\arrow[d, "\mathrm{HT}"'] & T_pG'\arrow[d, "\mathrm{HT}"]\\
            \underline{\omega}_G^{\mathrm{mod}, +} \arrow[r, "\alpha^{\vee}"] & \underline{\omega}_{G'}^{\mathrm{mod}, +}
        \end{tikzcd},
    \]
    where $T_pG$ and $T_pG'$ are the Tate modules of the corresponding $1$-motives and $\mathrm{HT}$ is the Hodge--Tate map.

    For any rank-$1$ point $x' = \Spa(F, \calO_F)$ of $\calX_{\Iw_{\mathrm{Si}, n}^+\cap K(p^n)}^{\tor, \Sigma''}$ lifting $x$, choose bases for $T_pG \cong \Z_p^4$ and $T_pG'\cong \Z_p^4$ lifting the bases for $G[p^n]$ and $G'[p^n]$ provided by the moduli problem. For a suitable choice of bases for $\underline{\omega}_G^{\mathrm{mod}, +}$ and $\underline{\omega}_{G'}^{\mathrm{mod}, +}$, the diagram above is isomorphic to the commutative diagram \begin{equation}\label{eq: Siegel correspondence diagram}
        \begin{tikzcd}
            \Z_p^4 \arrow[rr, "\text{$\diag(\one_2, p\one_2)$}"]\arrow[d, "\pr_1"'] && \Z_p^4\arrow[d, "\pr_2"]\\
            \calO_{\C_p}^2 \arrow[rr, "p\one_2"] && \calO_{\C_p}^2
        \end{tikzcd}.
    \end{equation}
    By definition, we see that $\pr_1(e_i)\in p^v\calO_{\C_p}^2$ for $i=1, 2$ and $\pr_2(\diag(\one_2, p\one_2)e_i)\in p^{v+1}\calO_{\C_p}^2$ for $i=1,2$. Therefore, at the level of points, we see that $\pr_2(\calX_{\Iw_{\mathrm{Si}, n}^+}^{\tor, \Sigma''}(v,v))$ factors through $\calX_{\Iw_{\mathrm{Si}, n}}^{\tor, \Sigma'}(v+1,v+1)$.
\end{proof}

\begin{Lemma}\label{Lemma: morphism between sheaves for Siegel families}
    Keep the notations as in Definition \ref{Definition: families of automorphic sheaves, Siegel} and Lemma \ref{Lemma: Siegel Hecke correspondence improves overconvergence}. There exists a $\scrO_{\calX_{\Iw_{\mathrm{Si}, n}^+}^{\tor, \Sigma''}(v,v)}\widehat{\otimes} R_{\calU}$-morphism of sheaves \[ 
        \pr_{2}^* \underline{\omega}_{w, v+1}^{(\kappa_{\calU}, \kappa_{\calU})} \rightarrow \pr_{1}^* \underline{\omega}_{w, v}^{(\kappa_{\calU}, \kappa_{\calU})}.
    \]
\end{Lemma}
\begin{proof}
    Consider the correspondence \[
        \begin{tikzcd}
            & \calX_{K(p^n) \cap \bfitu_{\mathrm{Si}}^{-1}K(p^n) \bfitu_{\mathrm{Si}}}^{\tor, \Sigma''}(v,v) \arrow[ld, "\pr_1"']\arrow[rd, "\pr_2"]\\
            \calX_{K(p^n)}^{\tor, \Sigma}(v,v) && \calX_{K(p^n)}^{\tor, \Sigma'}(v+1,v+1)
        \end{tikzcd},    
    \]
    where $\calX_{K(p^n) \cap \bfitu_{\mathrm{Si}}^{-1}K(p^n) \bfitu_{\mathrm{Si}}}^{\tor, \Sigma''}(v,v) = \calX_{K(p^n) \cap \bfitu_{\mathrm{Si}}^{-1}K(p^n) \bfitu_{\mathrm{Si}}}^{\tor, \Sigma''} \times_{\calX_{K(p^n)}^{\tor, \Sigma}} \calX_{K(p^n)}^{\tor, \Sigma}(v,v)$. We claim that there is a natural morphism \[
        \pr_1^* \calL_{n, w, v} \rightarrow \pr_2^* \calL_{n, w, v+1}
    \]
    of line bundles.

    Indeed, for any point $\Spa(R, R^+) \rightarrow \calX_{K(p^n) \cap \bfitu_{\mathrm{Si}}^{-1}K(p^n) \bfitu_{\mathrm{Si}}}^{\tor, \Sigma''}(v,v)$ corresponding to the isogeny $\alpha: G \rightarrow G'$, we have a commutative diagram \[
        \begin{tikzcd}
            \Z_p^4 \arrow[r]\arrow[d, "\mathrm{HT}"'] & \Z_p^4\arrow[d, "\mathrm{HT}"]\\
            \underline{\omega}_{G}^{\mathrm{mod}, +} \arrow[r, "\alpha^{\vee}"] & \underline{\omega}_{G'}^{\mathrm{mod}, +}
        \end{tikzcd}.
    \]
    Given a trivialisation $(\psi: R \xrightarrow{\cong} \det\underline{\omega}_G) \in \calL_{n, w, v}(R, R^+)$, we have a trivialisation \[
        \alpha^{\vee, *}\psi: R \xrightarrow{\psi} \det\underline{\omega}_G \xrightarrow{\det\alpha^{\vee}} \det\underline{\omega}_{G'};
    \]
    note that $\det \alpha^{\vee}$ is an isomorphism after inverting $p$. As $\psi \in \calL_{n, w, v}(R, R^+)$, one sees that $\alpha^{\vee, *}\psi \in \calL_{n, w, v+1}(R, R^+)$. That is, we have the claimed map of line bundles. 
\end{proof}

An immediate consequence of Lemma \ref{Lemma: morphism between sheaves for Siegel families} is that we have the composition of morphisms \[
    \begin{tikzcd}
        R\Gamma(\calX_{\Iw_{\mathrm{Si}, n}}^{\tor, \Sigma'}(v+1, v+1), \underline{\omega}_{w, v}^{(k_1+\kappa_{\calU}, k_2+\kappa_{\calU})}) \arrow[r, "\pr_2^*"]\arrow[rddd, bend right = 20, "U_{\mathrm{Si}}^{\text{naïve}}"'] & R\Gamma(\calX_{\Iw_{\mathrm{Si}, n}^+}^{\tor, \Sigma''}(v,v), \pr_2^*\underline{\omega}_{w, v+1}^{(k_1+\kappa_{\calU}, k_2+\kappa_{\calU})}) \arrow[d, "\text{Lemma \ref{Lemma: morphism between sheaves for Siegel families}}"]\\
        & R\Gamma(\calX_{\Iw_{\mathrm{Si}, n}^+}^{\tor, \Sigma''}(v,v), \pr_1^*\underline{\omega}_{w, v}^{(k_1+\kappa_{\calU}, k_2+\kappa_{\calU})})\arrow[d]\\
        & R\Gamma(\calX_{\Iw_{\mathrm{Si}, n}}^{\tor, \Sigma}(v, v), R\pr_{1,*}\pr_1^*\underline{\omega}_{w, v}^{(k_1+\kappa_{\calU}, k_2+\kappa_{\calU})})\arrow[d, "\mathrm{tr}"]\\
        & R\Gamma(\calX_{\Iw_{\mathrm{Si}, n}}^{\tor, \Sigma}(v, v), \underline{\omega}_{w, v}^{(k_1+\kappa_{\calU}, k_2+\kappa_{\calU})})
    \end{tikzcd}.
\]
By pre-composing $\frac{1}{p^3}U_{\mathrm{Si}}^{\text{naïve}}$ with the restriction map \begin{equation}\label{eq: restriction map for Siegel Up}
    R\Gamma(\calX_{\Iw_{\mathrm{Si}, n}}^{\tor, \Sigma'}(v, v), \underline{\omega}_{w, v}^{(k_1+\kappa_{\calU}, k_2+\kappa_{\calU})}) \rightarrow  R\Gamma(\calX_{\Iw_{\mathrm{Si}, n}}^{\tor, \Sigma'}(v+1, v+1), \underline{\omega}_{w, v}^{(k_1+\kappa_{\calU}, k_2+\kappa_{\calU})})
\end{equation}
and using the same trick as before, we then define an operator \[
    U_{\mathrm{Si}}: R\Gamma(\calX_{\Iw_{\mathrm{Si}, n}}^{\tor, \Sigma}(v, v), \underline{\omega}_{w, v}^{(k_1+\kappa_{\calU}, k_2+\kappa_{\calU})}) \rightarrow R\Gamma(\calX_{\Iw_{\mathrm{Si}, n}}^{\tor, \Sigma}(v, v), \underline{\omega}_{w, v}^{(k_1+\kappa_{\calU}, k_2+\kappa_{\calU})}). 
\]
We shall then drop the $\Sigma$ in the notation. A similar construction also applies to the cuspidal version.

\begin{Corollary}\label{Corollary: compact operators, Siegel case}
    On the complexes $R\Gamma(\calX_{\Iw_{\mathrm{Si}, n}}^{\tor}(v,v), \underline{\omega}_{w, v}^{(k_1+\kappa_{\calU}, k_2+\kappa_{\calU})})$ and $R\Gamma(\calX_{\Iw_{\mathrm{Si}, n}}^{\tor}(v,v), \underline{\omega}_{w, v, \cusp}^{(k_1+\kappa_{\calU}, k_2+\kappa_{\calU})})$, the operator $U_{\mathrm{Si}}$ acts compactly.
\end{Corollary}
\begin{proof}
    By the construction, $U_{\mathrm{Si}}$ factors through the restriction map \eqref{eq: restriction map for Siegel Up}. Note that we can compute these complexes using Čech complexes and the restriction map on global sections is an injection, hence compact. Thus \eqref{eq: restriction map for Siegel Up} is compact, and so is $U_{\mathrm{Si}}$.
\end{proof}

\begin{Theorem}[Classicality]\label{Theorem: classicality theorem, Siegel}
    Let $(k_1, k_2)\in \Z^2$ with $k_1\geq k_2>2$. Then, the restriction map $M_{(k_1, k_2)}(\Iw_{\mathrm{Si}, n}) \rightarrow M_{(k_1, k_2)}^{w, v}(\Iw_{\mathrm{Si}, n})$ induces an isomorphism \[
        M_{(k_1, k_2)}(\Iw_{\mathrm{Si}, n})^{< k_2-3} \cong M_{(k_1, k_2)}^{w, v}(\Iw_{\mathrm{Si}, n})^{<k_2-3}.
    \]
\end{Theorem}
\begin{proof}
    First of all, by Lemma \ref{Lemma: specialisation of the Siegel families of automorphic sheaves are classical}, $\underline{\omega}_{w,v}^{(k_1+k, k_2+k)} \cong \underline{\omega}^{(k_1+k, k_2+k)}$ over $\calX_{\Iw_{\mathrm{Si}, n}}^{\tor}(v,v)$. Therefore, \[
        M_{(k_1+k, k_2+k)}^{w,v}(\Iw_{\mathrm{Si}, n}) = H^0(\calX_{\Iw_{\mathrm{Si}, n}}^{\tor}(v,v), \underline{\omega}^{(k_1+k, k_2+k)}).
    \]
    We thus only focus on the algebraic sheaf $\underline{\omega}^{(k_1, k_2)}$. 

    Consider the (toroidally compactified) Siegel threefold $\calX_{\Iw_{B, n}}^{\tor}$ of level $\Iw_{B, n}$. We let $\calX_{\Iw_{B, n}}^{\tor}(v,v) \coloneq \calX_{\Iw_{B, n}}^{\tor} \times_{\calX_{\Iw_{\mathrm{Si}, n}}} \calX_{\Iw_{\mathrm{Si}, n}}^{\tor}(v,v)$. By \cite[Théorèm 1.2]{Pilloni-prolongementSiegel}, we have \[
        M_{(k_1, k_2)}(\Iw_{B, n})^{<k_2-3} \cong H^0(\calX_{\Iw_{B, n}}^{\tor}(v,v), \underline{\omega}^{(k_1, k_2)})^{<k_2-3}.
    \]
    Now, by applying Lemma \ref{Lemma: condition for composing Hecke operators naively}, we have a commutative diagram \[
        \begin{tikzcd}
            M_{(k_1, k_2)}(\Iw_{B, n}) \arrow[r, "U_{\mathrm{Si}}"]\arrow[d, hook]\arrow[ddd, bend right = 90, "\mathrm{tr}"'] & M_{(k_1, k_2)}(\Iw_{B, n})\arrow[d, hook]\arrow[d, hook]\arrow[ddd, bend left = 90, "\mathrm{tr}"]\\
            H^0(\calX_{\Iw_{B, n}}^{\tor}(v,v), \underline{\omega}^{(k_1, k_2)}) \arrow[r, "U_{\mathrm{Si}}"]\arrow[d, "\mathrm{tr}"'] & H^0(\calX_{\Iw_{B, n}}^{\tor}(v,v), \underline{\omega}^{(k_1, k_2)})\arrow[d, "\mathrm{tr}"]\\
            H^0(\calX_{\Iw_{\mathrm{Si}, n}}^{\tor}(v,v), \underline{\omega}^{(k_1, k_2)})\arrow[r, "U_{\mathrm{Si}}"] & H^0(\calX_{\Iw_{\mathrm{Si}, n}}^{\tor}(v,v), \underline{\omega}^{(k_1, k_2)})\\
            M_{(k_1, k_2)}(\Iw_{\mathrm{Si}, n})\arrow[u, hook] \arrow[r, "U_{\mathrm{Si}}"] & M_{(k_1, k_2)}(\Iw_{\mathrm{Si}, n})\arrow[u, hook]
        \end{tikzcd}.
    \]
    Note that, by definition, the trance maps $\mathrm{tr}$ are surjective. Thus, by taking the slope-$<k_2-3$ parts, we deduce that the natural map \[
        M_{(k_1, k_2)}(\Iw_{\mathrm{Si}, n})^{< k_2-3} \hookrightarrow H^0(\calX_{\Iw_{\mathrm{Si}, n}}^{\tor}(v,v), \underline{\omega}^{(k_1, k_2)})^{<k_2-3}.
    \]
    is a surjection, which concludes the result.
\end{proof}

\begin{Corollary}\label{Corollary: finite projective of cuspidal Siegel forms in Siegel family}
    Keep the notations as above. Let $(k_1, k_2)\in \Z^2$ with $k_1 \geq k_2$ and $h\in \Q_{\geq 0}$ such that the slope-$\leq h$ part of $R\Gamma(\calX_{\Iw_{\mathrm{Si}, n}}^{\tor}(v,v), \underline{\omega}_{w, v}^{(k_1+\kappa_{\calU}, k_2+\kappa_{\calU})})$ is well-defined. Then, \[
        M_{(k_1+\kappa_{\calU}, k_2+\kappa_{\calU})}^{w, v}(\Iw_{\mathrm{Si}, n})^{\leq h} =  H^0(R\Gamma(\calX_{\Iw_{\mathrm{Si}, n}}^{\tor}(v,v), \underline{\omega}_{w, v}^{(k_1+\kappa_{\calU}, k_2+\kappa_{\calU})})^{\leq h})
    \]
    is a non-trivial finite projective $R_{\calU}$-module. A similar result also holds for the cuspidal version. 
\end{Corollary}
\begin{proof}
    The proof is similar to Corollary \ref{Corollary: finite projective of cuspidal Siegel forms in Klingen family}, providing Corollary \ref{Corollary: compact operators, Siegel case} and Theorem \ref{Theorem: classicality theorem, Siegel}. We leave the details to the readers. 
\end{proof}

\subsection{Small parabolic eigenvarieties}\label{subsection: small par. eigenvar.}

In this subsection, we construct small parabolic eigenvarieties for (holomorphic) Siegel modular forms. To this end, we fix $P \in \{P_{\mathrm{Kl}}, P_{\mathrm{Si}}\}$\footnote{ Again, the case when $P = B$ is studied in \cite{AIP-2015}. We thus skip our discussion in this paper. } and $k_{\heartsuit} = (k_{\heartsuit, 1}, k_{\heartsuit, 2})\in \Z^2_{\geq 0}$ such that $k_{\heartsuit, 1}\geq k_{\heartsuit, 2}$; if $P = P_{\mathrm{Kl}}$, we further assume $k_{\heartsuit, 2}>2$. We consider the weight space $\calW_{P, k_{\heartsuit}}$ (Definition \ref{Definition: translate of weight spaces}).

For any affinoid open $\calU = \Spa(R_{\calU}, R_{\calU}^{\circ})\subset \calW_P$, denote by $\calU_{k_{\heartsuit}}$ its translate to $\calW_{P, k_{\heartsuit}}$; note that $\calU$ is isomorphic to $\calU_{k_{\heartsuit}}$. For suitable choices of $w, v , n$ as in Sect. \ref{subsection: families of automorphic sheaves}, let $C_{\calU}^{\bullet}(w, v, n)$ be the bounded complex of Banach $R_{\calU}$-modules having (Pr) that computes $R\Gamma(\calX_{\Iw_{\mathrm{Kl}, n}}^{\tor}(v), \underline{\omega}_{w,v}^{(k_{\heartsuit, 1}+\kappa_{\calU}, k_{\heartsuit, 2})})$ (if $P = P_{\mathrm{Kl}}$) or $R\Gamma(\calX_{\Iw_{\mathrm{Si}, n}}^{\tor}(v,v), \underline{\omega}_{w,v}^{(k_{\heartsuit,1}+\kappa_{\calU}, k_{\heartsuit, 2}+\kappa_{\calU})})$ (if $P = P_{\mathrm{Si}}$). By Theorem \ref{Theorem: nice properties of automorphic sheaves in Klingen families} (ii) and Corollary \ref{Corollary: compact operators, Siegel case}, the operator $U_P$ acts compactly on $C_{\calU}^{\bullet}(w,v,n)$. Therefore, we may define the Fredholm power series \[
    \mathrm{FP}_{\calU}(T) \coloneq \det(1-TU_P|\bigoplus C_{\calU}^i(w,v,n)). 
\]
Observe that $\mathrm{FP}_{\calU}(T)$ is independent to the choice of $w, v, n$; indeed, the operator $U_P$ improves convergence and a similar proof as Proposition \ref{Proposition: finite-slope part it independent to the level at p, H0} shows that the finite-slope part of $C_{\calU}^{\bullet}(w,v,n)$ is quasi-isomorphic to the finite-slope part of $C_{\calU}^{\bullet}(w,v,n+1)$.\footnote{ Here, the finite-slope part for $C_{\calU}^{\bullet}(w, v, n)$ is defined to be $C_{\calU}^{\bullet}(w, v, n)\otimes^{L}_{\Q_p[U_P]} \Q_p[U_P, U_P^{-1}]$.} Moreover, one sees from the construction of the family of automorphic sheaves that, for any affinoid open $\calV = \Spa(R_{\calV}, R_{\calV}^{\circ}) \subset \calU$, we have \[
    \underline{\omega}_{w, v}^{k_{\heartsuit, 1}+\kappa_{\calU}, k_{\heartsuit, 2}}\widehat{\otimes}_{R_{\calU}}R_{\calV} = \underline{\omega}_{w, v}^{k_{\heartsuit, 1}+\kappa_{\calV}, k_{\heartsuit, 2}} \quad \text{ and }\quad \underline{\omega}_{w, v}^{k_{\heartsuit, 1}+\kappa_{\calU}, k_{\heartsuit, 2}+\kappa_{\calU}}\widehat{\otimes}_{R_{\calU}}R_{\calV} = \underline{\omega}_{w, v}^{k_{\heartsuit, 1}+\kappa_{\calV}, k_{\heartsuit, 2}+\kappa_{\calV}}.
\]
Therefore, $\mathrm{FP}_{\calU}(T)$ glues to an entire Fredholm determinant $P_{k_{\heartsuit}}(T)$ over $\calW_{P, k_{\heartsuit}}$. We thus have the spectral variety \[
    \calZ_{P, k_{\heartsuit}} \coloneq \text{ the zero locus of $\mathrm{FP}_{k_{\heartsuit}}(T)$ in $\calW_{P, k_{\heartsuit}}\times_{\Spa(\Q_p, \Z_p)}\bbA^{1, \mathrm{rig}}_{\Q_p}$.}
\]

By the construction of $\calZ_{P, k_{\heartsuit}}$, it admits an affinoid open cover $\{\calZ_{\calU, h}\}_{(\calU, h)}$ indexed by pairs $(\calU, h)$, where $\calU$ is an open affinoid of $\calW_P$ and $h\in \Q_{\geq 0}$ such that the slope-$\leq h$ part of $C_{\calU}^{\bullet}(w, v, n)$ is well-defined. Thanks to Corollary \ref{Corollary: finite projective of cuspidal Siegel forms in Klingen family} and Corollary \ref{Corollary: finite projective of cuspidal Siegel forms in Siegel family}, we can define a coherent sheaf $\scrM_{P, k_{\heartsuit}}$ on $\calZ_{P, k_{\heartsuit}}$ by \[
    \scrM_{P, k_{\heartsuit}}(\calZ_{\calU, h}) = \left\{ \begin{array}{ll}
        M_{(k_{\heartsuit, 1}+\kappa_{\calU}, k_{\heartsuit, 2})}^{w, v}(\Iw_{\mathrm{Kl}, n})^{\leq h}, & \text{ if }P = P_{\mathrm{Kl}}  \\ \\
        M_{(k_{\heartsuit, 1}+\kappa_{\calU}, k_{\heartsuit, 2}+\kappa_{\calU})}^{w, v}(\Iw_{\mathrm{Si}, n})^{\leq h}, & \text{ if }P = P_{\mathrm{Si}}
    \end{array} \right. .
\]
Denote by \[
    \bbT \coloneq \left(\bigotimes_{\ell\not\in S_{\mathrm{bad}}}  \Z_p[\GSp_4(\Z_{\ell}) \backslash \GSp_4(\Q_{\ell})/\GSp_4(\Z_{\ell})]\right) \otimes_{\Z_p} \Z_p[U_P]
\]
the abstract Hecke algebra. Then, the Hecke action on $C_{\calU}^{\bullet}(w, v, n)$ induces a morphism \[
    \Psi_{P, k_{\heartsuit}} : \bbT[U_P^{-1}] \rightarrow \End(\scrM_{P, k_{\heartsuit}}).
\]
As a result, the quadruple \[
    (\calZ_{P, k_{\heartsuit}}, \bbT[U_P^{-1}], \scrM_{P, k_{\heartsuit}}, \Psi_{P, k_{\heartsuit}})
\]
is an \emph{eigenvariety datum} in the sense of \cite[Definition 3.1.1]{Johansson-Newton-Irreducible}. We then denote by \begin{equation}\label{eq: weight map of eigenvariety}
    \wt: \calE_{P, k_{\heartsuit}} \rightarrow \calW_{P, k_{\heartsuit}}
\end{equation}
the associated reduced eigenvariety.

Remark that, by considering the coherent sheaf $\scrS_{P, k_{\heartsuit}}$ defined by using (holomorphic) Siegel cuspforms, we again have an eigenvariety datum \[
    (\calZ_{P, k_{\heartsuit}}, \bbT[U_P^{-1}], \scrS_{P, k_{\heartsuit}}, \Psi_{P, k_{\heartsuit}}^{\cusp}: \bbT[U_P^{-1}] \rightarrow \End(\scrS_{P, k_{\heartsuit}})).
\]
We then denote the resulting reduced eigenvariety by \[
    \wt: \calE_{P, k_{\heartsuit}}^{\cusp} \rightarrow \calW_{P, k_{\heartsuit}}.
\]
From the construction, one sees that there is a closed embedding \[
    \calE_{P, k_{\heartsuit}}^{\cusp} \hookrightarrow \calE_{P, k_{\heartsuit}}. 
\]

\begin{Proposition}\label{Proposition: small parabolic eigenvarieties are equidimensional of dimension 1}
    Keep the notations as above. The eigenvarieties $\calE_{P, k_{\heartsuit}}$ and $\calE_{P, k_{\heartsuit}}^{\cusp}$ are equidimensional of dimension $1$. 
\end{Proposition}
\begin{proof}
    We show the assertion for $\calE_{P, k_{\heartsuit}}$; the proof for $\calE_{P, k_{\heartsuit}}^{\cusp}$ is similar. 
    
    For $(\calU, h)$ giving $\calZ_{\calU, h} \subset \calZ_{P, k_{\heartsuit}}$, let \[
        \bbT_{\calU}^{\leq h} \coloneq \text{ reduced $R_{\calU}$-algebra generated by the image of $\bbT[U_P^{-1}] \rightarrow \End_{R_{\calU}}(\scrM(\calZ_{\calU, h}))$}.
    \] 
    Thus, $\bbT_{\calU}^{\leq h}$ is a finite $R_{\calU}$-algebra. Since $\scrM(\calZ_{\calU, h})$ is a finite projective $R_{\calU}$-module (Corollary \ref{Corollary: finite projective of cuspidal Siegel forms in Klingen family} and Corollary \ref{Corollary: finite projective of cuspidal Siegel forms in Siegel family}), the structure morphism $R_{\calU} \rightarrow \bbT_{\calU}^{\leq h}$ is injective. Consequently, $\bbT_{\calU}^{\leq h}$ has Krull dimension $1$. Note that $\calE_{P, k_{\heartsuit}}$ is constructed by glueing $\Spa(\bbT_{\calU}^{\leq h}, \bbT_{\calU}^{\leq h, \circ})$ over $\calZ_{P, k_{\heartsuit}}$, we concludes the result. 
\end{proof}

\begin{Remark}\label{Remark: comparison with small parabolic eigenvarieties of BSW}
    By the Eichler--Shimura decomposition (\cite[Chapter VI, Theorem 6.2]{Faltings-Chai}) and \cite[Theorem 3.2.1]{Johansson-Newton-Irreducible}, one can show that $\calE_{P, k_{\heartsuit}}$ embeds into the small parabolic eigenvariety constructed by Barrera Salazar--Williams in \cite{BSW-ParabolicEigen}.\footnote{ To be more precise, we note that the construction in \cite{BSW-ParabolicEigen} considers every Hecke operator at $p$ (here, we only consider the one corresponds to $P$) and does not invert the controlling operator. So, what we mean is, if one instead works with overconvergent cohomology groups but use the similar eigenvariety datum as we do, one would obtain an embedding.  } Proposition \ref{Proposition: small parabolic eigenvarieties are equidimensional of dimension 1} implies that their small parabolic eigenvariety admits irreducible components of dimension 1.  When $P = P_{\mathrm{Si}}$ and $k_{\heartsuit, 1}= k_{\heartsuit, 2}$, this is a new result since it cannot be deduced from \cite[Corollary 5.16]{BSW-ParabolicEigen} (because $k_{\heartsuit} = (k_{\heartsuit, 1}, k_{\heartsuit, 1})$ is not a regular weight). 
\end{Remark}

\begin{Remark}\label{Remark: non-neat subgroup away from p, eigenvarieties}
    Similarly as in Remark \ref{Remark: at non-neat level, classical forms}, we may loosen the assumption that $K^p$ is a neat subgroup to a subgroup containing a normal neat subgroup. One uses the same trick as in Remark \ref{Remark: at non-neat level, classical forms} to construct the corresponding families of (cuspidal) Siegel modular forms. One can then again run through the construction above to obtain the corresponding (cuspidal) small parabolic eigenvarieties, which are again equidimensional of dimension $1$ over $\calW_{P, k_{\heartsuit}}$. 
\end{Remark}
\section{Refined families of symplectic Galois representations over small parabolic eigenvarieties}\label{section: R=T}
The goal of this section is to construct and study families of Galois representations attached to the (cuspidal) small parabolic eigenvarieties for $\GSp_4$. Throughout this section, we (again) fix a compact open $K^p\subset \GSp_4(\A^{\infty, p})$ containing a normal neat subgroup and, for $P \in \{B, P_{\mathrm{Si}}, P_{\mathrm{Kl}}\}$, choose $K_p \subset \GSp_4(\Z_p)$ as in Proposition \ref{Proposition: finite-slope part it independent to the level at p, H0}; let $K = K^p K_p$ and consider the Siegel modular threefold $\calX_{K}^{\tor} = \calX_{K_p}^{\tor}$ and its subspaces as in the previous section. Note that, as before, we shall mainly focus on the case when $B \subsetneq P$.

\subsection{Galois representations attached to Siegel cuspforms of genus 2}\label{subsection: Gal reps for GSp4}

Thanks to the work of many mathematicians (see, for example, \cite{Taylor-Siegel, Laumon, Weissauer, Urban-GSp4, Sorensen-HilbertSiegel, Jorza-GSp, Mok-GL2CM}), we have the following theorem. 

\begin{Theorem}\label{Theorem: Galois representation attached to Siegel forms}
    Let $(k_1, k_2)\in \Z^2$ with $k_1 \geq k_2 >2$. Let $f\in S_{(k_1, k_2)}(K) = S_{(k_1, k_2)}(K_p)$ be a cuspidal Siegel eigenform corresponding to an automorphic representation $\pi = \bigotimes_v'\pi_v$ with central character $\omega_{\pi}$ and Hecke eigensystem $\lambda_f = \lambda_{\pi}: \bbT \rightarrow \overline{\Q}_p$. There exists a semisimple Galois representation \[
        \rho_f = \rho_{\pi} : \Gal_{\Q} \rightarrow \GSp_4(\overline{\Q}_p)
    \]
    such that the following properties are satisfied: \begin{itemize}
        \item[(i)] We have $\mathrm{sim} \rho_{\pi} = \omega_{\pi} \chi_{\cyc}^{3-k_1-k_2}$.
        \item[(ii)] Let $\mathtt{S}_{\mathrm{bad}} \coloneq \{\ell : K_{\ell}\neq \GSp_4(\Z_{\ell})\} \cup \{2, p\}$. The Galois representation $\rho_{\pi}$ is unramified outside $\mathtt{S}_{\mathrm{bad}}$.
        \item[(iii)] The local Galois representation $\rho_{\pi, p} := \rho_{\pi}|_{\Gal_{\Q_p}}$ is de Rham with Hodge--Tate cocharacter $(0, k_2-2; k_1+k_2-3)$ (so the Hodge--Tate weight for $\mathrm{std}\circ \rho_{\pi, p}$ is $(0, k_2-2, k_1-1, k_1+k_2-3)$).
        \item[(iv)] (Local-global compatibility at good primes) For any $\ell\not\in \mathtt{S}_{\mathrm{bad}}$, \[
            \det(1-\rho_{\pi}(\Frob_{\ell})X) = \lambda_{\pi}(P_{\mathrm{Hecke}, \ell}(X))\in \overline{\Q}_p[X]. 
        \]
    \end{itemize}
\end{Theorem}

For our later purpose, we further assume the following (mild) hypothesis on the explicit local-global compatibility at $p$. 

\begin{Hypothesis}[Explicit local-global compatibility at $p$]\label{Hypothesis: explicit local-global compatibility at p}
    Keep the notations as in Theorem \ref{Theorem: Galois representation attached to Siegel forms}. Suppose $K_p \subset \Iw_{P, n}$ (for some $n\in \Z_{>0}$) and let $f\in S_{(k_1, k_2)}(K_p)$. Then, $f\in S_{(k_1, k_2)}(K_p)^{P-\fs}$ if and only if the following properties hold: \begin{itemize}
            \item[(i)] If $P = P_{\mathrm{Si}}$, then $\varphi_1 = \lambda_{\pi}(U_{\mathrm{Si}})$ and $\varphi_2 = p^{k_1+k_2-3}\lambda_{\pi}(U_{\mathrm{Si}})^{-1}$ are two Frobenius eigenvalues on $\D_{\pst}(\mathrm{std}\circ \rho_{\pi, p})$, their corresponding eigenspaces $M_i$ are dual to each other under the symplectic pairing, and $M_1$ is $(\varphi, N, \Gal_{\Q_p})$-stable.
            \item[(ii)] If $P = P_{\mathrm{Kl}}$, then $\varphi_1 = p^{k_2-2}\lambda_{\pi}(U_{\mathrm{Kl}})$ and $\varphi_2 = p^{2k_1+k_2-4}\lambda_{\pi}(U_{\mathrm{Kl}})^{-1}$ are Frobenius eigenvalues on $\bigwedge^2 \D_{\pst}(\mathrm{std}\circ \rho_{\pi,p})$ and their corresponding eigenspaces are the determinants of two $2$-dimensional $\varphi$-stable isotropic subspaces $M_i$ in $\D_{\pst}(\mathrm{std}\circ \rho_{\pi,p})$, dual to each other under the symplectic pairing, and $M_1$ is $(\varphi, N, \Gal_{\Q_p})$-stable.
            \item[(iii)] If $P = B$, (i) and (ii) both hold. 
        \end{itemize} 
\end{Hypothesis}

\begin{Remark}\label{Remark: hypothesis is mild}
    Hypothesis \ref{Hypothesis: explicit local-global compatibility at p} is expected and inspired by \cite[Coloraire 1]{Urban-GSp4}, where the case when $f$ is ordinary and crystalline at $p$ is proven. 
\end{Remark}

\begin{Remark}\label{Remark: local-global comp. at p}
    The statement for the local-global compatibility at $p$ might seem ad hoc. A rather standard statement would be \begin{equation}\label{eq: standard lpcal-global compatibility}
        \mathrm{WD}(\D_{\pst}(\mathrm{std}\circ \rho_{\pi, p}))^{F-\mathrm{ss}} \cong \mathrm{rec}_p(\mathrm{BC}_{\GSp_4}^{\GL_4}(\pi_p)\otimes |\mathrm{sim}|^{-3/2})
    \end{equation}
    where $\mathrm{WD}(\D_{\pst}(\mathrm{std}\circ \rho_{\pi, p}))^{F-\mathrm{ss}}$ is the $\varphi$-semisimplification of the Weil--Deligne representation associated with $\D_{\pst}(\mathrm{std}\circ \rho_{\pi, p})$, $\mathrm{BC}_{\GSp_4}^{\GL_4}(\pi_p)$ is the functorial lift of $\pi_p$ to $\GL_4$, and $\mathrm{rec}_p(\mathrm{BC}_{\GSp_4}^{\GL_4}(\pi_p)\otimes |\mathrm{sim}|^{-3/2})$ is the local Langlands correspondence of $\mathrm{BC}_{\GSp_4}^{\GL_4}(\pi_p)\otimes |\mathrm{sim}|^{-3/2}$ (see, for example, \cite{GT-LLGSp4} and \cite[Theorem A]{BLGGT-localglobal2}). However, it is important for us to work with an explicit description between the Frobenius eigenvalues on $\D_{\pst}(\mathrm{std}\circ \rho_{\pi, p})$ and the Hecke operators at $p$. At present, it is not clear to us how to deduce Hypothesis \ref{Hypothesis: explicit local-global compatibility at p} directly from \eqref{eq: standard lpcal-global compatibility}. 
\end{Remark}

\begin{Remark}\label{Remark: hypothesis on Gal reps without a choice of faithful representation}
    One sees that the standard representation $\mathrm{std}: \GSp_4 \rightarrow \GL_4$ plays an important role in the formulation of Hypothesis \ref{Hypothesis: explicit local-global compatibility at p}. While it would be elegant to express the hypothesis independently of a specific faithful representation, our current understanding relies on this natural choice.
\end{Remark}

\begin{Definition}\label{Definition: small-slope forms}
    Fix $P\in \{P_{\mathrm{Kl}}, P_{\mathrm{Si}}, B\}$. Let $(k_1, k_2)\in \Z^2$ with $k_1\geq k_2 >2$ (we assume $k_2>3$ if $P = P_{\mathrm{Si}}$ or $B$). A cuspidal Siegel eigenform $f\in S_{(k_1, k_2)}(K_p)$ is said to be of \textbf{$P$-small slope} if \[
        \begin{array}{ll}
            v_p(\lambda_f(U_{\mathrm{Kl}}))<k_1-k_2+1, & \text{ if }P = P_{\mathrm{Kl}}  \\
            v_p(\lambda_f(U_{\mathrm{Si}}))<k_2-3, & \text{ if }P = P_{\mathrm{Si}} \\ 
            \text{above two are satisfied}, & \text{ if }P=B
        \end{array} .
    \]
\end{Definition}

\begin{Lemma}\label{Lemma: small slope implies Galois non-critical}
    Fix $P\in \{P_{\mathrm{Kl}}, P_{\mathrm{Si}}, B\}$. Let $(k_1, k_2)\in \Z^2$ with $k_1\geq k_2>2$ (we assume $k_2>3$ if $P = P_{\mathrm{Si}}$ or $B$). Suppose $f\in S_{(k_1, k_2)}(K_p)$ is of $P$-small slope. Assuming Hypothesis \ref{Hypothesis: explicit local-global compatibility at p}, the $(\varphi, \Gamma)$-module with $\GSp_4$-structure $\D_{\rig}^{\dagger}(\rho_{\pi, p})$ admits a $P^{\vee}$-non-critical $P^{\vee}$-flag (in the sense of Definition \ref{Definition: P-non-critical}), where \[
    P^{\vee} = \left\{ \begin{array}{ll}
        P_{\mathrm{Kl}}, & \text{ if }P = P_{\mathrm{Si}} \\
        P_{\mathrm{Si}}, & \text{ if }P = P_{\mathrm{Kl}} \\
        B, & \text{ if }P=B
    \end{array} \right..
    \]
\end{Lemma}
\begin{proof}
    We show the case when $P = P_{\mathrm{Si}}$, the other cases are similar and are left to the readers. 

    Under Hypothesis \ref{Hypothesis: explicit local-global compatibility at p}, $\D_{\pst}(\mathrm{std}\circ \rho_{\pi, p})$ admits a $P_{\mathrm{Kl}}$-refinement \[
        M_1 \subset M_1^{\perp} \subset \D_{\pst}(\mathrm{std}\circ \rho_{\pi, p}), 
    \]
    which gives rise to a $P_{\mathrm{Kl}}$-refinement on $\D_{\pst}(\rho_{\pi, p})$ (again, since $\mathrm{std}$ is faithful). We then let $\Fil_{\bullet}\D_{\rig}^{\dagger}(\rho_{\pi, p})$ be the corresponding $P_{\mathrm{Kl}}$-flag (Proposition \ref{Proposition: P-refinements are equivalent to P-flags}). By Lemma \ref{Lemma: P-non-criticality can be checked on faithful representations}, to check $\Fil_{\bullet}\D_{\rig}^{\dagger}(\rho_{\pi, p})$ is $P_{\mathrm{Kl}}$-non-critical, it is enough to check it on $\Fil_{\bullet}\D_{\rig}^{\dagger}(\mathrm{std}\circ \rho_{\pi, p})$. To simplify the notation, write \[
        D_1 = \Fil_1\D_{\rig}^{\dagger}(\mathrm{std}\circ \rho_{\pi, p}) \quad \text{ and }\quad D_1^{\perp} = \Fil_2\D_{\rig}^{\dagger}(\mathrm{std}\circ \rho_{\pi ,p}).
    \]
    By construction, we have that $\calD_{\pst}(D_1) = M_1$ and $\calD_{\pst}(D_1^{\perp}) = M_1^{\perp}$ (hence the notation). 

    By Hypothesis \ref{Hypothesis: explicit local-global compatibility at p} and the $P$-small slope, we have \[
        v_p(\varphi_1) < k_2-3  \quad \text{ and }\quad v_p(\varphi_2) >k_1.
    \]
    By weak admissibility of $\D_{\pst}(\mathrm{std}\circ \rho_{\pi, p})$, we see that $\calD_{\pst}(D_1) = M_1$ (resp., $\calD_{\pst}(D_1^{\perp}) = M_1^{\perp}$) can only have Hodge--Tate weight $0$ (resp., $(0, k_2-2, k_1-1)$). This concludes the proof. 
\end{proof}

\subsection{Refined families of symplectic Galois representations over small parabolic eigenvarieties}\label{subsection: families of Gal reps over small par. eigenvar}

Fix $P\in \{P_{\mathrm{Si}}, P_{\mathrm{Kl}}\}$\footnote{As we have seen in many cases (\emph{e.g.}, Remarks \ref{Remark: when the parabolic subgroup is not maximal} and \ref{Remark: deformation problems for non-maximal parabolics}), the case for $P=B$ can be generalised by combining the situation for $P_{\mathrm{Si}}$ and $P_{\mathrm{Kl}}$; and it can be studied by adapting methods in \cite{BC} directly. Hence, starting from this subsection, we shall only focus on the situation when $B \subsetneq P$. } and $k_{\heartsuit} = (k_{\heartsuit, 1}, k_{\heartsuit, 2})\in \Z_{\geq 0}^2$ with $k_{\heartsuit, 1} \geq k_{\heartsuit, 2}$ (when $P_{\mathrm{Kl}}$, we assume $k_{\heartsuit, 2}>2$). We consider the cuspidal small parabolic eigenvariety (and the weight map) \[
    \wt: \calE_{P, k_{\heartsuit}}^{\cusp} \rightarrow \calW_{P, k_{\heartsuit}}
\] constructed in Sect. \ref{subsection: small par. eigenvar.}. By Proposition \ref{Proposition: small parabolic eigenvarieties are equidimensional of dimension 1}, we know that $\calE_{P, k_{\heartsuit}}^{\cusp}$ is equidimensional of dimension $1$.

We denote by $\calE_{P, k_{\heartsuit}}^{\cusp,\cl}$ the set of classical points of $\calE_{P, k_{\heartsuit}}^{\cusp}(\overline{\Q}_p)$. Then, for any $x\in \calE_{P, k_{\heartsuit}}^{\cusp, \cl}$ with corresponding Siegel cuspform $f$, we write $\rho_x = \rho_{f}$ under Theorem \ref{Theorem: Galois representation attached to Siegel forms}. We further define $\calE_{P, k_{\heartsuit}}^{\cusp, \nc} \subset \calE_{P, k_{\heartsuit}}^{\cusp, \cl}$ to be the subset of classical points having small slopes. Immediately from the proofs of Corollary \ref{Corollary: finite projective of cuspidal Siegel forms in Klingen family} and Corollary \ref{Corollary: finite projective of cuspidal Siegel forms in Siegel family}, one sees that $\calE_{P, k_{\heartsuit}}^{\cusp, \nc}$ is also Zariski dense in $\calE_{P, k_{\heartsuit}}^{\cusp}$ and it is an accumulation subset.

By viewing $\calW_{P, k_{\heartsuit}}$ as a subspace in $\calW$, one obtains universal continuous character \[
    \kappa_{\calW_{P, k_{\heartsuit}}} = (\kappa_{\calW_{P, k_{\heartsuit}}, 1}, \kappa_{\calW_{P, k_{\heartsuit}}, 2}): T_{\GL_2}(\Z_p) \rightarrow \scrO_{\calW_{P, k_{\heartsuit}}}(\calW_{P, k_{\heartsuit}})^{\times}.
\] 
We shall abuse the notation and still denote by $\kappa_{\calW_{P, k_{\heartsuit}}}$ the composition \[
    T_{\GL_2}(\Z_p) \xrightarrow{\kappa_{\calW_{P, k_{\heartsuit}}}} \scrO_{\calW_{P, k_{\heartsuit}}}(\calW_{P, k_{\heartsuit}})^{\times} \rightarrow \scrO_{\calE_{P, k_{\heartsuit}}}(\calE_{P, k_{\heartsuit}})^{\times}. 
\]
We shall also simplify the notation and denote $\kappa_{\calW_{P, k_{\heartsuit}}}$ by \[
    \kappa_{\calW_{P, k_{\heartsuit}}} = \kappa_{P} = (\kappa_{P, 1}, \kappa_{P, 2})
\]
when $k_{\heartsuit}$ is clear in the context.

\begin{Proposition}\label{Proposition: global symplectic determinant over small parabolic eigenvarieties}
    Keep the notations as above. There exists a symplectic Galois determinant \[
        (\Det, \Pf): (\scrO_{\calE_{P, k_{\heartsuit}}^{\cusp}}(\calE_{P, k_{\heartsuit}}^{\cusp})[\Gal_{\Q}], \varsigma^{\univ}) \rightarrow \scrO_{\calE_{P, k_{\heartsuit}}^{\cusp}}(\calE_{P, k_{\heartsuit}}^{\cusp}) ,
    \]
    where $\varsigma^{\univ}$ is the composition \[
        \varsigma^{\univ}: \Gal_{\Q} \xrightarrow{\chi_{\cyc}} \Z_p^{\times} \xrightarrow{\kappa_{P, 1}+\kappa_{P, 2}+3} \scrO_{\calE_{P, k_{\heartsuit}}^{\cusp}}^{\times}
    \] such that \[
        \Det(1-\Frob_{\ell} X) = P_{\mathrm{Hecke}, \ell}(X)
    \] for all $\ell \not\in \mathtt{S}_{\mathrm{bad}}$.
\end{Proposition}
\begin{proof}
    For any $x\in \calE_{P, k_{\heartsuit}}^{\cusp, \cl}$, let $L_x$ be its residue field. We consider the morphism \[
        \iota: \scrO_{\calE_{P, k_{\heartsuit}}^{\cusp}}^+(\calE_{P, k_{\heartsuit}}^{\cusp}) \rightarrow \prod_{x\in \calE_{P, k_{\heartsuit}}^{\cusp, \cl}} L_x, \quad f \mapsto (f(x))_x.
    \]
    We equipped with $\prod_{x\in \calE_{P, k_{\heartsuit}}^{\cusp, \cl}} L_x$ the product topology and so one sees that $\iota$ is continuous. Moreover, since $\calE_{P, k_{\heartsuit}}^{\cusp, \cl}$ is Zariski dense in the reduced rigid variety $\calE_{P, k_{\heartsuit}}^{\cusp}$, $\iota$ is injective. By applying \cite[Corollary 5.4.4]{Johansson-Newton}, we know $\scrO_{\calE_{P, k_{\heartsuit}}^{\cusp}}^+(\calE_{P, k_{\heartsuit}}^{\cusp})$ is compact and so the image of $\iota$ is closed in $\prod_{x\in \calE_{P, k_{\heartsuit}}^{\cusp, \cl}} L_x$.
    
    By Theorem \ref{Theorem: Galois representation attached to Siegel forms}, there exists a symplectic determinant \[
        \left(\prod_{x\in \calE_{P, k_{\heartsuit}}^{\cusp, \cl}} L_x\right)[\Gal_{\Q}] \rightarrow \prod_{x\in \calE_{P, k_{\heartsuit}}^{\cl}} L_x
    \]
    with respect to the character $(\mathrm{sim} \rho_{x})_x$. By applying Lemma \ref{Lemma: continoues homogeneous polynomial gluing}, one sees that $\varsigma^{\univ}$ is a polynomial law on $\scrO_{\calE_{P, k_{\heartsuit}}^{\cusp}}^+(\calE_{P, k_{\heartsuit}}^{\cusp})$ whose specialisation to each $x\in \calE_{P, k_{\heartsuit}}^{\cusp, \cl}$ agrees with $\mathrm{sim}\rho_x$. We then consider the involutive algebra structure on $\scrO_{\calE_{P, k_{\heartsuit}}^{\cusp}}^+(\calE_{P, k_{\heartsuit}}^{\cusp})[\Gal_{\Q}]$ by \[
        \sigma \mapsto \varsigma^{\univ}(\sigma)\sigma^{-1}.
    \] 
    One obtains the desired symplectic determinant $(\Det, \Pf)$ by applying Corollary \ref{Corollary: continuous symplectic determinants and glueing}.
\end{proof}

\begin{Corollary}\label{Corollary: refined families over small parabolic eigenvarieties}
    Keep the notations as above. Under Hypothesis \ref{Hypothesis: explicit local-global compatibility at p}, the datum \[
        (\calE_{P, k_{\heartsuit}}^{\cusp}, \calE_{P, k_{\heartsuit}}^{\cusp, \nc}, \varsigma^{\univ}, \Det, \Pf, [\mu_{\HTS}], U_P, U_P^{-1}),
    \] where $[\mu_{\HTS}]$ is the $M_P$-conjugacy class of the cocharacter \begin{align*}
        \mu_{\HTS}: \Gamma & \rightarrow \GSp_4(\scrO_{\calE_{P, k_{\heartsuit}, \Q_p^{\cyc}}^{\cusp}}(\calE_{P, k_{\heartsuit}, \Q_p^{\cyc}}^{\cusp})) \rtimes \Gamma, \\
        \gamma & \mapsto (\diag(1, \kappa_{P,2}(\chi_{\cyc}(\gamma))\chi_{\cyc}^{-2}(\gamma), \kappa_{P,1}(\chi_{\cyc}(\gamma))\chi_{\cyc}^{-1}(\gamma), \kappa_{P,1}(\chi_{\cyc}(\gamma))\kappa_{P,2}(\chi_{\cyc}(\gamma))\chi_{\cyc}^{-3}(\gamma), \gamma),
    \end{align*}
    is a $P^{\vee}$-refined family of symplectic $\Gal_{\Q}$-representations of dimension $4$. 
\end{Corollary}
\begin{proof}
    One observes that (RFGal1)--(RFGal4) are verified by Theorem \ref{Theorem: Galois representation attached to Siegel forms}, Hypothesis \ref{Hypothesis: explicit local-global compatibility at p} and Lemma \ref{Lemma: small slope implies Galois non-critical}. The condition (RFGal5) follows from Corollary \ref{Corollary: finite projective of cuspidal Siegel forms in Klingen family} and Corollary \ref{Corollary: finite projective of cuspidal Siegel forms in Siegel family} as discussed above. 
\end{proof}

\begin{Corollary}\label{Corollary: deformation is refined}
    Keep the notations as above. Let $x\in \calE_{P, k_{\heartsuit}}^{\cusp, \nc}$. Suppose $\rho_x$ is absolutely irreducible and satisfies ($P^{\vee}$-REG) (see Sect. \ref{subsection: deformation}). Let $\widehat{\bbT}_x$ be the completed local ring at $x$. Then, there exists a unique Galois representation \[
        \rho_{\widehat{\bbT}_x}: \Gal_{\Q} \rightarrow \GSp_4(\widehat{\bbT}_x)
    \]
    lifting $\rho_x$ such that, for any ideal $\frakI\subset \widehat{\bbT}_x$ of cofinite length and let $\rho_{\frakI} = \rho_{\widehat{\bbT}_x} \otimes_{\widehat{\bbT}_x}\widehat{\bbT}_x/\frakI$, \[
         \rho_{\frakI, p} = \rho_{\frakI}|_{\Gal_{\Q_p}}\in \Def_{\rho_{x, p}, \Fil_{\bullet}}^{\mathrm{small}}(\widehat{\bbT}_x/\frakI),
    \] where $\rho_{x, p} = \rho_x|_{\Gal_{\Q_p}}$. 
\end{Corollary}
\begin{proof}
    Since $\rho_x$ is absolutely irreducible, the existence and the uniqueness of $\rho_{\widehat{\bbT}_x}$ follows from Theorem \ref{Theorem: symplectic version of Chenevier's Thm A and B} (ii). 
    
    Fix $\frakI$ as in the statement. Let $m = 1$ if $P=P_{\mathrm{Si}}$ and $m=2$ if $P=P_{\mathrm{Kl}}$. Also, let $\kappa = 0$ if $P= P_{\mathrm{Si}}$ and $\kappa = \kappa_{P,2}-2$ if $P = P_{\mathrm{Kl}}$. To show $\rho_{\frakI}|_{\Gal_{\Q_p}}\in \Def_{\rho_{x, p}, \Fil_{\bullet}}^{\mathrm{small}}(\widehat{\bbT}_x/\frakI)$, we first observe that \[
        (\calE_{P, k_{\heartsuit}}^{\cusp}, \calE_{P, k_{\heartsuit}}^{\cusp, \nc}, \bigwedge^m \Det, U_P)
    \]
    is a tuple satisfying the conditions in Sect. \ref{subsection: Setup, App A}. Thus, Corollary \ref{Corollary: generalisation of BC Lemma 4.3.10} implies that \[
        \left(\bigwedge^m \D_{\pst}(\mathrm{std}\circ\rho_{\frakI}(\kappa))\right)^{\varphi = U_P} 
    \]
    is free of rank $1$. 

    Now, let $\widetilde{M_1} \subset \bigwedge^m \D_{\pst}(\mathrm{std}\circ \rho_{\frakI})$ be defined by \[
        \widetilde{M_1}(\kappa) = \left(\bigwedge^m \D_{\pst}(\mathrm{std}\circ\rho_{\frakI}(\kappa))\right)^{\varphi = U_P}.
    \]
    By construction, $\widetilde{M_1}$ lifts $\det M_{x,1}$, where $M_{x,1}\subset \D_{\pst}(\mathrm{std} \circ \rho_{x,p})$ as in Hypothesis \ref{Hypothesis: explicit local-global compatibility at p}. Similarly as in the proof of Lemma \ref{Lemma: relationships between de Rham deformations}, we see that $\widetilde{M_1}$ is $(\varphi, N, \Gal_{\Q_p})$-stable. Therefore, it defines a rank-$1$ submodule $\widetilde{D_1}\subset \bigwedge^m \D_{\rig}^{\dagger}(\mathrm{std}\circ \rho_{\frakI, p})$. We then define $\Fil_{\bullet}\bigwedge^m \D_{\rig}^{\dagger}(\mathrm{std}\circ \rho_{\frakI,p})$ by \[
        \widetilde{D_1}\subset \widetilde{D_1}^{\perp} \subset \bigwedge^m \D_{\rig}^{\dagger}(\mathrm{std}\circ \rho_{\frakI,p}). 
    \]
    From the construction, one easily sees that this filtration lifts $\Fil_{\bullet}\bigwedge^m\D_{\rig}^{\dagger}(\mathrm{std}\circ \rho_{x, p})$, which concludes the result. 
\end{proof}

\begin{Corollary}[Critical points do not accumulate]\label{Corollary: critical points cannot accumulate}
    Keep the notations as above. Let $x\in \calE_{P, k_{\heartsuit}}^{\cusp, \cl}$ with Galois representation $\rho_x$. Suppose the $P^{\vee}$-flag on $\D_{\rig}^{\dagger}(\rho_{x, p})$ given by Hypothesis \ref{Hypothesis: explicit local-global compatibility at p} is $P^{\vee}$-critical. Then, there exists an affinoid open neighbourhood $\calV$ of $x$ such that for any other $y\in \calV(\overline{\Q}_p) \cap \calE_{P, k_{\heartsuit}}^{\cusp, \cl}$, the $P^{\vee}$-flag on $\D_{\rig}^{\dagger}(\rho_{y, p})$ is $P^{\vee}$-non-critical. 
\end{Corollary}
\begin{proof}
    From slope decomposition (see, for example, \cite[Sect. 4]{Ash-Stevens}), one observes that the slope is locally a constant on $\calE_{P, k_{\heartsuit}}^{\cusp}$. Hence, if $x$ has slope $h$, then there exists an affinoid open neighbourhood $\calV$ of $x$ such that every $y\in \calV(\overline{\Q}_p) \cap \calE_{P, k_{\heartsuit}}^{\cusp, \cl}$ has slope $h$. Since small-slope classical points admit $P^{\vee}$-non-critical $P^{\vee}$-flags (Lemma \ref{Lemma: small slope implies Galois non-critical}), for a classical point $y$ such that $\D_{\rig}^{\dagger}(\rho_{y,p})$ has a $P^{\vee}$-critical $P^{\vee}$-flag, it can only happen when the weight $(k_{y, 1}, k_{y, 2})$ of $y$ satisfies \[
        \begin{array}{ll}
            h \geq k_{y, 1}-k_{y, 2} +1, & \text{ if }P = P_{\mathrm{Kl}} \\
            h \geq k_{y, 2}-3, & \text{ if } P = P_{\mathrm{Si}}
        \end{array} .
    \]
    Note that, by construction, $k_{y, 1}$ is fixed if $P = P_{\mathrm{Kl}}$ while $k_{y, 1} = k_{y, 2}$ if $P  = P_{\mathrm{Si}}$. Therefore, there are only finitely many choices of $(k_{y, 1}, k_{y, 2}) \in \Z_{\geq 0}^2$ so that the condition above is satisfied. This implies that the set \[
        \left\{ y\in \calV(\overline{\Q}_p) \cap \calE_{P, k_{\heartsuit}}^{\cusp, \cl}: \D_{\rig}^{\dagger}(\rho_{y, p}) \text{ is $P^{\vee}$-critically flagged} \right\}
    \]
    is a finite set. Hence, by shrinking $\calV$ if necessary, we conclude the desired result. 
\end{proof}

\subsection{Adjoint Bloch--Kato Selmer groups and an infinitesimal \texorpdfstring{$R = \bbT$}{R=T} theorem}\label{subsection: adj Selmer}

In this subsection, we fix a cuspidal Siegel eigenform $f$ of weight $(k_1, k_2)$ with $k_1\geq k_2>3$ such that $f$ appears in $\calE_{P, k_{\heartsuit}}^{\cusp}$; that is, $f$ corresponds to a point $x = x_{f}\in \calE_{P, k_{\heartsuit}}^{\cusp}(L_x)$, where $L_x$ is the residue field of $x$. Consider the Galois representation $\rho_x = \rho_f$. Note that Hypothesis \ref{Hypothesis: explicit local-global compatibility at p} yields a $P^{\vee}$-refinement on $\D_{\pst}(\rho_x)$ and so $\D_{\rig}^{\dagger}(\rho_x)$ admits a $P^{\vee}$-flag $\Fil_{\bullet}\D_{\rig}^{\dagger}(\rho_x)$ (by Proposition \ref{Proposition: P-refinements are equivalent to P-flags}). We further assume the following in this subsection: \begin{itemize}
    \item The Galois representation $\rho_x$ is absolutely irreducible. 
    \item The $P^{\vee}$-flag $\Fil_{\bullet}\D_{\rig}^{\dagger}(\rho_x)$ satisfies ($P^{\vee}$-REG) and ($P^{\vee}$-NCR)  (see Sect. \ref{subsection: deformation}). 
\end{itemize}

Recall the global deformation functors \eqref{eq: global deformation problems} \begin{align*}
    \Def_{\rho_x, \Fil_{\bullet}}^{\mathrm{small}}  & \coloneq (\{\Def_{\rho_{x,\ell}}^{\unr}\}_{\ell\not\in S_{\mathrm{bad}}}, \{\Def_{\rho_{x,\ell}}^{\min}\}_{\ell\in S_{\mathrm{bad}} \smallsetminus \{p\}}, \Def_{\rho_{x,p}, \Fil_{\bullet}}^{\mathrm{small}}), \\
    \Def_{\rho_x, \dR}  & \coloneq (\{\Def_{\rho_{x,\ell}}^{\unr}\}_{\ell\not\in S_{\mathrm{bad}}}, \{\Def_{\rho_{x,\ell}}^{\min}\}_{\ell\in S_{\mathrm{bad}} \smallsetminus \{p\}}, \Def_{\rho_{x,p}, \dR}), \\
    \Def_{\rho_x, \dR}^{\mathrm{small}}  & \coloneq (\{\Def_{\rho_{x,\ell}}^{\unr}\}_{\ell\not\in S_{\mathrm{bad}}}, \{\Def_{\rho_{x,\ell}}^{\min}\}_{\ell\in S_{\mathrm{bad}} \smallsetminus \{p\}}, \Def_{\rho_{x,p}, \dR}^{\mathrm{small}}),
\end{align*}
and their tangent spaces $\Def_{\rho_x, \Fil_{\bullet}}^{\mathrm{small}}(L[\epsilon])$, $\Def_{\rho_x, \dR}(L_x[\epsilon])$, $\Def_{\rho_x, \dR}^{\mathrm{small}}(L_x[\epsilon])$ (for $\epsilon^2 = 0$). By construction, we see that \begin{align*}
    \Def_{\rho_x, \dR}(L_x[\epsilon]) & = H^1_{g, \min}(\Q, \ad\rho_x) \\ 
    & = \ker\left( H^1(\Q, \ad\rho_x) \rightarrow \frac{H^1(\Q_p, \ad\rho_x)}{W_p^{\dR}} \prod_{\ell\not\in \mathtt{S}_{\mathrm{bad}}} \frac{H^1(\Q_{\ell}, \ad\rho_x)}{W_{\ell}^{\unr}} \prod_{\ell\in \mathtt{S}_{\mathrm{bad}}\smallsetminus \{p\}} \frac{H^1(\Q_{\ell}, \ad\rho_x)}{W_{\ell}^{\min}} \right).
\end{align*} 
On the other hand, the geometric adjoint Bloch--Kato Selmer group is defined as \[
    H^1_g(\Q, \ad\rho_x) = \ker\left( H^1(\Q, \ad\rho_x) \rightarrow \frac{H^1(\Q_p, \ad\rho_x)}{W_p^{\dR}} \prod_{\ell\not\in \mathtt{S}_{\mathrm{bad}}} \frac{H^1(\Q_{\ell}, \ad\rho_x)}{W_{\ell}^{\unr}}\right).
\]
Hence, by construction, we have \begin{equation}\label{eq: injections of adjoint BK Selmer groups}
    \Def_{\rho_x, \dR}^{\mathrm{small}}(L_x[\epsilon]) \hookrightarrow H^1_{g, \min}(\Q, \ad\rho_x)\hookrightarrow H^1_{g}(\Q, \ad\rho_x).
\end{equation}
Note that, if the conditions in Lemma \ref{Lemma: need not worry about minimal deformation} hold, then the last inclusion is an isomorphism.

The following is the (geometric) Bloch--Kato conjecture of $\ad \rho_x$ (see also \cite[Conjecture 5.1.3 \& Remark 5.2.4]{BC}).

\begin{Conjecture}[Geometric Bloch--Kato conjecture for $\ad \rho_x$]\label{Conjecture: geometric adjoint BK conjecture}
    The geometric Bloch--Kato Selmer group $H^1_g(\Q, \ad\rho_x)$ vanishes. 
\end{Conjecture}

\begin{Remark}\label{Remark: justification of the geometric adjoint BK conjecture}
    Compared with \cite[Conjecture 5.1.3 \& Remark 5.2.4]{BC}, one sees that we stated the adjoint Bloch--Kato conjecture using the geometric Bloch--Kato Selmer group instead of the crystalline Bloch--Kato Selmer group $H^1_f$. A priori, our version is stronger than the one in \cite{BC}. However, \cite[Prediction 4.2]{Bellaiche-BK} and the following proposition provides evidence for Conjecture \ref{Conjecture: geometric adjoint BK conjecture}.
\end{Remark}

\begin{Proposition}\label{Proposition: geometric adjoint  BK conjecture holds}
    Keep the notations and assumptions as above. Suppose the weak functorial transfer from $\GSp_4$ to $\GL_4$ exists for $f$ and the restriction $\rho_x|_{\Gal_{F(\zeta_p^{\infty})}}$ is irreducible.
    Then, $H^1_g(\Q, \ad \rho_x) = 0$.
\end{Proposition}
\begin{proof}
    Note that the Galois representation for the weak transfer is $\mathrm{std}\circ \rho_f$ and we have a natural Galois-equivariant inclusion \[
        \ad \rho_x \hookrightarrow \ad (\mathrm{std}\circ \rho_x),
    \]
    which induces the inclusion \[
        H^1_g(\Q, \ad \rho_x) \hookrightarrow H^1_g(\Q, \ad(\mathrm{std}\circ \rho_x)).
    \]
    However, under the assumptions above, \cite[Theorem 6.3]{Thorne-nonvanishing} (see also \cite[Theorem B]{Allen-dukepaper}) implies that \[
        H^1_g(\Q, \ad (\mathrm{std}\circ \rho_x)) = 0.
    \]
    We then conclude the result. 
\end{proof}

\begin{Remark}\label{Remark: evidence of transfer from GSp4 to GL4 over totally real}
    In \cite{PSS-transfer}, Pitale--Saha--Schmidt showed the existence of the weak functorial transfer if $f$ has full level. Also note that, Asgari--Shahidi \cite{AS-GSp4toGL4} established such a transfer for generic automorphic representations for $\GSp_4$. However, automorphic representations associated with holomorphic Siegel modular forms are not generic.    
\end{Remark}

\begin{Theorem}\label{Theorem: infinitesimal R=T}
    Keep the notations as above. Denote by $\widehat{\bbT}_x = \widehat{\scrO}_{\calE_{P, k_{\heartsuit}}^{\cusp}, x}$ the completed local ring at $x$ and $R_{\rho_x, \Fil_{\bullet}}^{\mathrm{small}, \univ}$ the universal deformation ring for $\Def_{\rho_x, \Fil_{\bullet}}^{\mathrm{small}}$. Suppose the following hold:  \begin{itemize}
        \item Conditions in Lemma \ref{Lemma: need not worry about minimal deformation} hold, i.e., for any $\ell \in S_{\mathrm{bad}} \smallsetminus \{p\}$, $\rho_x|_{I_{\ell}}$ is irreducible and $p\nmid \ell^{12}-1$.
        \item Conjecture \ref{Conjecture: geometric adjoint BK conjecture} is true.
    \end{itemize} 
    Then, there is a canonical isomorphism \[
        R_{\rho_x, \Fil_{\bullet}}^{\mathrm{small}, \univ} \cong \widehat{\bbT}_x.
    \] 
    Moreover, $\widehat{\bbT}_x$ is smooth. 
\end{Theorem}
\begin{proof}
    We prove the theorem by the following steps: 

    \noindent \textbf{Step 1.} There exists a canonical ring morphism $R_{\rho_x, \Fil_{\bullet}}^{\mathrm{small}, \univ}  \rightarrow \widehat{\bbT}_x$.

    To construct this canonical ring homomorphism, it is enough to show that for any ideal $\frakI \subset \widehat{\bbT}_x$ of cofinite length, there exists a Galois representation \[
        \rho_{\frakI}: \Gal_{\Q} \rightarrow \GSp_4(\widehat{\bbT}_x/\frakI)
    \]
    and a $P^{\vee}$-flag $\Fil_{\bullet}\bigwedge^m\D_{\rig}^{\dagger}(\mathrm{std}\circ\rho_{\frakI,p})$ ($m=1$ if $P=P_{\mathrm{Si}}$, $m=2$ if $P=P_{\mathrm{Kl}}$) such that the pair $(\rho_{\frakI}, \Fil_{\bullet}\bigwedge^m\D_{\rig}^{\dagger}(\mathrm{std}\circ \rho_{\frakI,p}))\in \Def_{\rho_x, \Fil_{\bullet}}^{\mathrm{small}}(\widehat{\bbT}_x/\frakI)$. Indeed, if such a pair exists, then there is a natural map \[
        R_{\rho_x, \Fil_{\bullet}}^{\mathrm{small}, \univ} \rightarrow \widehat{\bbT}_x/\frakI;
    \]
    since $\widehat{\bbT}_x = \varprojlim_{\frakI \text{ cofinite ideal}} \widehat{\bbT}_x/\frakI$, one obtains a natural map \[
        R_{\rho_x, \Fil_{\bullet}}^{\mathrm{small}, \univ} \rightarrow \widehat{\bbT}_x.
    \]
    However, Corollary \ref{Corollary: deformation is refined} and Lemma \ref{Lemma: need not worry about minimal deformation} yields the desired pair. 


    \noindent \textbf{Step 2.} The ring homomorphism $R_{\rho_x, \Fil_{\bullet}}^{\mathrm{small}, \univ} \rightarrow \widehat{\bbT}_x$ is surjective. 

    Since $\widehat{\bbT}_x$ is generated by the Hecke operators, it is enough to show that the Hecke operators are mapped to. Indeed, it follows from the local-global compatibility at good primes (Theorem \ref{Theorem: Galois representation attached to Siegel forms} (iv)) that the Hecke operators away from $S_{\mathrm{bad}}$ are mapped to. On the other hand, the explicit local-global compatibility at $p$ (Hypothesis \ref{Hypothesis: explicit local-global compatibility at p}) states an explicit relationship between the operator $U_P$ and the potentially semistable Frobenius. As $p$ is invertible in both $R_{\rho_x, \Fil_{\bullet}}^{\mathrm{small}, \univ}$ and $\widehat{\bbT}_x$, we see that the operator $U_P$ is also mapped to.

    \noindent \textbf{Step 3.} The ring homomorphism $R_{\rho_x, \Fil_{\bullet}}^{\mathrm{small}, \univ} \rightarrow \widehat{\bbT}_x$ is an isomorphism.

    By Corollary \ref{Corollary: sort exact sequence of tangent spaces} and \eqref{eq: injections of adjoint BK Selmer groups}, we have \[
        1 \geq \dim_{L_x} \Def_{\rho_x, \Fil_{\bullet}}^{\mathrm{small}}(L_x[\epsilon]).
    \]
    On the other hand, since $R_{\rho_x, \Fil_{\bullet}}^{\mathrm{small}, \univ}$ is a local noetherian ring, its Krull dimension is bounded by the dimension of its tangent space (\cite[\href{https://stacks.math.columbia.edu/tag/00KD}{Tag 00KD}]{stacks-project}). In other words, we have \[
        1 \geq \dim_{\mathrm{Krull}} R_{\rho_x, \Fil_{\bullet}}^{\mathrm{small}, \univ}.
    \]
    Note that the equality holds if and only if $R_{\rho_x, \Fil_{\bullet}}^{\mathrm{small}, \univ}$ is regular (and so smooth by \cite[\href{https://stacks.math.columbia.edu/tag/00TV}{Tag 00TV}]{stacks-project}). 

    On the other hand, since $R_{\rho_x, \Fil_{\bullet}}^{\mathrm{small}, \univ} \rightarrow \widehat{\bbT}_x$ is a surjection, we have \[
        1 \geq \dim_{\mathrm{Krull}} R_{\rho_x, \Fil_{\bullet}}^{\mathrm{small}, \univ} \geq \dim_{\mathrm{Krull}} \widehat{\bbT}_x = 1,
    \]
    where the last equality follows from that $\calE_{P, k_{\heartsuit}}^{\cusp}$ is equidimensional of dimension $1$. We conclude that \[
        \dim_{\mathrm{Krull}} R_{\rho_x, \Fil_{\bullet}}^{\mathrm{small}, \univ} = 1 
    \]
    and $R_{\rho_x, \Fil_{\bullet}}^{\mathrm{small}, \univ}$ is regular. To conclude the proof, suppose \[
        \frakk = \ker(R_{\rho_x, \Fil_{\bullet}}^{\mathrm{small}, \univ} \rightarrow \widehat{\bbT}_x)
    \]
    is non-zero. Since $R_{\rho_x, \Fil_{\bullet}}^{\mathrm{small}, \univ}$ is a regular local ring, it is a domain (\cite[\href{https://stacks.math.columbia.edu/tag/00NP}{Tag 00NP}]{stacks-project}) and so \[
        1 = \dim_{\mathrm{Krull}} R_{\rho_x, \Fil_{\bullet}}^{\mathrm{small}, \univ} > \dim_{\mathrm{Krull}} R_{\rho_x, \Fil_{\bullet}}^{\mathrm{small}, \univ}/\frakk = \dim_{\mathrm{Krull}} \widehat{\bbT}_x = 1,
    \]
    which is a contradiction. 
\end{proof}

\section{Saito--Kurokawa points I: finite slope}\label{section: SK points; finite-slope}

The purposes of this section and the next section are to apply the knowledge on families of symplectic Galois representations over small parabolic eigenvarieties for $\GSp_4$ developed above to study the geometry and arithmetic of points given by Saito--Kurokawa lifts. To this end, we recall the definition of paramodular groups as follows. For any prime number $\ell$ and any positive integer $n\in \Z_{>0}$, define \[
    K^{\mathrm{par}}_{\ell^n} \coloneq \begin{pmatrix}
        \Z_{\ell} & \Z_{\ell} & \Z_{\ell} & \frac{1}{\ell^n}\Z_{\ell}\\
        \ell^n\Z_{\ell} & \Z_{\ell} & \Z_{\ell} & \Z_{\ell}\\
        \ell^n\Z_{\ell} & \Z_{\ell} & \Z_{\ell} & \Z_{\ell}\\
        \ell^n\Z_{\ell} & \ell^n\Z_{\ell} & \ell^n\Z_{\ell} & \Z_{\ell}
    \end{pmatrix} \cap \GSp_4(\Q_{\ell}).
\]
For any positive integer $N\in \Z_{>0}$, we define \[
    K^{\mathrm{par}}(N) \coloneq \begin{pmatrix}
        \widehat{\Z} & \widehat{\Z} & \widehat{\Z} & \frac{1}{N}\widehat{\Z}\\
        N\widehat{\Z} & \widehat{\Z} & \widehat{\Z} & \widehat{\Z}\\
        N\widehat{\Z} & \widehat{\Z} & \widehat{\Z} & \widehat{\Z}\\
        N\widehat{\Z} & N\widehat{\Z} & N\widehat{\Z} & \widehat{\Z}
    \end{pmatrix} \cap \GSp_4(\A^{\infty}).
\]
Thus, if we write $N = \prod_{\ell \mid N} \ell^{n_{\ell}}$, then \[
    K^{\mathrm{par}}(N) = \prod_{\ell\mid N} K^{\mathrm{par}}_{\ell^{n_{\ell}}} \times \prod_{\ell \nmid N} \GSp_4(\Z_{\ell}).
\]
Note that $K^{\mathrm{par}}(N)$ is not neat.

\subsection{Saito--Kurokawa lifts}\label{subsection: SK lifts}

Let $N$ and $k$ be positive integers and $k>2$. Let $f$ be a cuspidal newform of level $\Gamma_0(N)$ and weight $2k-2$.\footnote{ Here, by a newform, we always normalise the $q$-expansion so that the first coefficient in the $q$-expansion is $1$.} Recall the functional equation for the $L$-function attached to $f$ is given by \[
    L(f, s) = \epsilon_f L(f, 2k-2-s),
\]
where $\epsilon_f = (-1)^{\mathrm{ord}_{s=k-1} L(f, s)}$ is the sign of $f$. 

If $\epsilon_f = -1$, $f$ admits a lift to a holomorphic cuspidal Siegel eigenform, \emph{i.e.}, the so-called Saito--Kurokawa lift, which we denote by $\SK(f)$. The following theorem summarises its properties.

\begin{Theorem}\label{Theorem: SK lifts}
    Let $f$ be a cuspidal newform as above and suppose $\epsilon_f = -1$. Then, there exists a holomorphic cuspidal Siegel eigenform $\SK(f)$ of level $K^{\mathrm{par}}(N)$ and weight $(k, k)$ whose Galois representation, after choosing a suitable basis, is of the form \[
        \mathrm{std}\circ \rho_{\SK(f)}  \sim \begin{pmatrix}
            \chi_{\cyc}^{2-k} & & \\ & \rho_f & \\ & & \chi_{\cyc}^{1-k}
        \end{pmatrix},
    \]
    where \[
        \rho_f: \Gal_{\Q} \rightarrow \GL_2(\overline{\Q}_p)
    \]
    is the Galois representation attached to $f$ and $\chi_{\cyc}$ is the $p$-adic cyclotomic character.
\end{Theorem}
\begin{proof}
    The case where $N = 1$ is due to a series of papers by N. Kurokawa (\cite{Kurokawa-SiegelExample}), H. Maa{\ss} (\cite{Maass}), A. Andrianov (\cite{Andrianov-SKConjecture}), D. Zagier (\cite{Zagier-SKConjecture}). V. Gritsenko (\cite{Gritsenko}) generalised it to any $N$. A more representation-theoretical approach can be found in \cite{Schmidt-functoriality} (see also \cite[Theorem 5.3]{Schmidt-CAPGSp4}). 
\end{proof}

\subsection{Saito--Kurokawa lifts for finite-slope cuspforms}\label{subsection: finite-slope SK lifts}

Let $f$ be a cuspidal newform as before. In the rest of this section, we assume $p \nmid N$ and let $\alpha, \beta$ be the two roots of the Hecke polynomial  \[
    T^2 - a_p(f) T + p^{2k-3}, 
\]
where $a_p(f)$ is the $p$-th coefficient of the $q$-expansion of $f$. In this situation, $f$ admits two $p$-stabilisations $f_{\alpha}$ and $f_{\beta}$, corresponding to the two choices of triangulations on the Galois representation $\rho_f$ at $p$. We further assume the following properties:

\begin{Assumption}\label{Assumption: nice properties for finite-slope cuspidal newform}
    \begin{itemize}
        \item {\normalfont(REG)} We have $\alpha \neq \beta$ and $\alpha, \beta \neq p^{k-1}$.
        \item {\normalfont(ST)} Let $\pi_f = \otimes_{v}' \pi_{f, v}$ be the automorphic representation of $\GL_2$ attached to $f$. If $\ell \mid N$, then $\pi_{f, \ell} \cong \mathrm{St}_{\ell} \otimes \xi$, where $\mathrm{St}_{\ell}$ is the Steinberg representation for $\GL_2(\Q_{\ell})$ and $\xi$ is the unramified character with $\xi(\ell) = -1$. 
    \end{itemize}
\end{Assumption}

Let $S_{(k,k)}$ be the space of cuspidal Siegel modular forms of weight $(k,k)$ and level $K^{\mathrm{par}}(N)$. Note that $\SK(f)\in S_{(k,k)}$. Recall the Iwahori subgroup $\Iw_B$ for the (upper triangular) Borel subgroup $B \subset \GSp_4$ and let $S_{(k,k)}(\Iw_B)$ be the space of cuspidal Siegel modular forms of weight $(k,k)$ and level $\Iw_B \times \prod_{\ell \mid N} K^{\mathrm{par}}_{\ell^{n_{\ell}}} \times \prod_{\ell \nmid pN} \GSp_4(\Z_{\ell})$. Then, we have a natural morphism \[
    S_{(k,k)} \rightarrow S_{(k,k)}(\Iw_B)
\]
given by the pullback map. By \cite[Table 7.1 (IIb)]{MY}, the image of $\SK(f)$ in $S_{(k,k)}(\Iw_B)$ decomposes into a sum of $4$ Hecke-eigenforms. These are the $p$-stabilisations of $\SK(f)$, which, in terms of Galois representations, correspond (again) to the choices of (certain) ordering of crystalline Frobenius eigenvalues; more precisely, the four choices are the following (\cite[Sect. 9.2]{MY}) \[
    (\alpha, p^{k-1}, p^{k-2}, \beta), (p^{k-1}, \alpha, \beta, p^{k-2}), (p^{k-1}, \beta, \alpha, p^{k-2}), (\beta, p^{k-1}, p^{k-2}, \alpha).
\] We denote by $\SK(f)_{1, \alpha}$ the $p$-stabilisation corresponding to the ordering $(p^{k-1}, \alpha, \beta, p^{k-2})$. 

Note that the $p$-stabilisations for $\SK(f)$ are $U_{\mathrm{Si}}$- and $U_{\mathrm{Kl}}$-finite-slope. In particular, we have \begin{align*}
    U_{\mathrm{Si}}(\SK(f)_{1,\alpha}) & = p^{k-1},\\
    p^{k-2}U_{\mathrm{Kl}}U_{\mathrm{Si}}^{-1}(\SK(f)_{1,\alpha}) & = \alpha,\\
    p^{k-1}U_{\mathrm{Kl}}^{-1}U_{\mathrm{Si}}(\SK(f)_{1,\alpha}) & = \beta,\\
    p^{2k-3} U_{\mathrm{Si}}^{-1}(\SK(f)_{1,\alpha}) & = p^{k-2}.
\end{align*}
Therefore, by working with Siegel modular forms of $\Iw_B$-level at $p$ and considering the abstract Hecke algebra \[
    \bbT = \left(\bigotimes_{\ell\not\in S_{\mathrm{bad}}}  \Z_p[\GSp_4(\Z_{\ell}) \backslash \GSp_4(\Q_{\ell})/\GSp_4(\Z_{\ell})]\right) \otimes_{\Z_p} \Z_p[U_{\mathrm{Si}}, U_{\mathrm{Kl}}], 
\]
the machinery in Sect. \ref{subsection: small par. eigenvar.} produces a small Klingen parabolic eigenvariety $\calE_{\mathrm{Kl}, (k,k)}^{\cusp}$.\footnote{To be more precise, one can pullback the torsors constructed in Sect. \ref{subsection: families of automorphic sheaves} to the variety of Iwahori level. } Moreover, $\SK(f)_{1, \alpha}$ defines a point $x = x_{1, \alpha}\in \calE_{\mathrm{Kl}, (k,k)}^{\cusp}(\overline{\Q}_p)$.\footnote{ The reader will find that our later discussions also apply to the $p$-stabilisation that corresponds to the ordering $(p^{k-1}, \beta, \alpha, p^{k-2})$.} 

\begin{Remark}\label{Remark: slightly different eigenvariety}
    Recall, in the construction for the small Klingen parabolic eigenvariety, we also invert $U_{\mathrm{Kl}}$. This small Klingen parabolic eigenvariety is slightly different from the one in Sect. \ref{subsection: small par. eigenvar.} (although denoted by the same notation) as we added $U_{\mathrm{Si}}$ in the construction.\footnote{ Note that the construction in \cite{BSW-ParabolicEigen} also considers every Hecke operator at $p$.} We thus do not know if the infinitesimal $R=\bbT$ (Sect. \ref{section: R=T}) holds in this situation. However, it turns out that, by adding $U_{\mathrm{Si}}$ to the construction, it is more convenient for our later purpose (see, Propostion \ref{Proposition: reducibility idea = maximal ideal; fs}). 
\end{Remark}

By the construction of the family of symplectic Galois determinants in Proposition \ref{Proposition: global symplectic determinant over small parabolic eigenvarieties}, there is a symplectic Galois determinant \[
    (\Det, \Pf): (\scrO_{\calE_{\mathrm{Kl}, (k,k)}^{\cusp}}(\calE_{\mathrm{Kl}, (k,k)}^{\cusp})[\Gal_{\Q}], \varsigma^{\univ}) \rightarrow \scrO_{\calE_{\mathrm{Kl}, (k,k)}^{\cusp}}(\calE_{\mathrm{Kl}, (k,k)}^{\cusp}).
\] 
Let $\bbT_x$ be the local ring at $x$ with maximal ideal $\frakm_x$. Then, the symplectic Galois determinant above gives rise to a symplectic determinant \[
    (\Det_x, \Pf_x): (\bbT_x[\Gal_{\Q}], \varsigma^{\univ}) \rightarrow \bbT_x.
\]
Let $A = \bbT_x[\Gal_{\Q}]/\ker \Det_x$.

To simplify the notation, we set \[
    \rho_0 = \rho_f, \quad \rho_1 = \chi_{\cyc}^{1-k}, \quad \text{ and }\quad \rho_2 = \chi_{\cyc}^{2-k}.
\]
On the set $\{\rho_0, \rho_1, \rho_2\}$, consider the partition $\Xi = \{I_0 = \{\rho_0\}, I_1 = \{\rho_1\}, I_2 = \{\rho_2\}\}$ and $\sigma = (12)$ be the permutation on $(0,1,2)$, swapping $1$ and $2$. Then, by \cite[Theorem 5.12]{MQ-SympDet}, $A$ has a symplectic GMA structure of type \[
    \Delta = ((I_0, I_1, I_2), \sigma, (1, 2, 1)).
\]
We refer the reader to \cite[Definition 5.4]{MQ-SympDet} for the definition of types of symplectic GMA structure. In particular, \[
    A \cong \begin{pmatrix}
        A_{11} & M_{12}(A_{10}) & A_{12}\\
        M_{21}(A_{01}) & M_{22}(A_{00}) & M_{21}(A_{02})\\
        A_{22} & M_{12}(A_{20}) & A_{22}
    \end{pmatrix},
\]
where $A_{ab}$ are some $\bbT_x$-submodules in $\mathrm{Frac} \bbT_x$ such that $A_{aa} \cong A$, $A_{ab}A_{bc} \subset A_{ac}$, and $A_{ab}A_{ba} \subset \frakm_x$ (\cite[Theorem 1.4.4]{BC}).

Thanks to Proposition \ref{Proposition: determinant and trace}, we define the \emph{reducibility ideal} of $\Det_x$ as the reducibility ideal for the associated pseudocharacter as in \cite[Definition 1.5.2]{BC}. In particular, in our situation, the reducibility ideal is given by \[
    \frakI^{\mathrm{red}} = A_{10}A_{01} + A_{20}A_{02} + A_{12}A_{21}.
\] 
By construction, we have $\frakI^{\mathrm{red}} \subset \frakm_x$.

\begin{Proposition}\label{Proposition: reducibility idea = maximal ideal; fs}
    Keep the notations as above. The inclusion $\frakI^{\mathrm{red}} \subset \frakm_x$ is in fact an equality. 
\end{Proposition}
\begin{proof}
    Since $\SK(f)_{1,\alpha}$ is both $U_{\mathrm{Si}}$- and $U_{\mathrm{Kl}}$-finite-slope, there exists an affinoid neighbourhood $\calV$ of $x$ in $\calE_{\mathrm{Kl}, (k,k)}^{\cusp}$ such that $U_{\mathrm{Kl}}$ is invertible on $\calV$. Therefore, by applying \cite[Theorem 3.2.1]{Johansson-Newton-Irreducible}, one sees that $\calV$ can be embedded into the eigenvariety constructed by Andreatta--Iovita--Pilloni (\cite{AIP-2015}). Then, by applying the argument as in \cite[Lemma B.3]{BB22}, we see that classical points that correspond to $p$-stabilisations in $\calV$ are very Zariski dense. By \cite[Lemma 2.5]{Schmidt-Packet} and \cite[Proposition 5.2]{Schmidt-CAPGSp4}, we see that those classical points can be Saito--Kurokawa lifts or genuine Siegel cuspforms (type \textbf{(P)} and \textbf{(G)} respectively in the sense of \cite{Schmidt-Packet}, see also the summary in Sect. \ref{subsection: heuristic of geometry}). However, since Saito--Kurokawa lifts must have parallel weight and $(k,k)$ is the only parallel weight on the Klingen weight space $\calW_{\mathrm{Kl}, (k,k)}$, we conclude that $x$ is the only classical point of type \textbf{(P)}. This implies the desired result. 
\end{proof}

\subsection{Bloch--Kato Selmer groups}\label{subsection: BK Selmer; fs}

We keep the notations as in the previous subsection. For $a, b, c\in \{0, 1, 2\}$, if $\rho_a$, $\rho_b$, $\rho_c$ are pairwise distinct, \cite[Theorem 1.5.5]{BC} implies that \begin{equation}\label{eq: mysterious isomorphism; fs}
    \Hom_{L_x}(A_{ab}/A_{ac}A_{cb}, L_x) \cong \Ext^1_{A}(\rho_b, \rho_a).
\end{equation}
The following propositions show the relationship between $\Ext_A^1(\rho_b, \rho_a)$ and Bloch--Kato Selmer groups.

\begin{Proposition}\label{Proposition: i0 and BK Selmer; fs}
    \begin{enumerate}
        \item[(i)] There is a natural embedding \[
            \Ext_A^1(\rho_1, \rho_0) \hookrightarrow H^1_f(\Q, \rho_f(k-1)). 
        \]
        \item[(ii)] There is a natural embedding \[
            \Ext_A^1(\rho_2, \rho_0) \hookrightarrow H^1_{f'}(\Q, \rho_f(k-2)),
        \]
        where $H^1_{f'}$ stands for the Selmer group with the unramified condition away from $p$ and no conditions at $p$.
    \end{enumerate}
\end{Proposition}
\begin{proof}
    First of all, let $S = \{\ell: \ell\mid N\}\cup\{p, \infty\}$ and recall the Galois group $\Gal_{\Q, S} = \Gal(\Q^{S}/\Q)$. By \cite[Theorem 1.5.10]{BC} and twisting, we have \[
        \Ext_A^1(\rho_i, \rho_0) \hookrightarrow H^1(\Gal_{\Q, S}, \rho_f(k-i)). 
    \]
    We thus need to show that any extension \begin{equation}\label{eq: extension in GMA; fs}
        0 \rightarrow \rho_f \rightarrow V \rightarrow \chi_{\cyc}^{k-i} \rightarrow 0
    \end{equation} is unramified at primes dividing $N$ and is crystalline at $p$ when $i=1$.

    \noindent \textbf{Step 0.} Preparation.

    As in the proof of Proposition \ref{Proposition: reducibility idea = maximal ideal; fs}, we may view a neighbourhood $\calV$ of $x$ as a subspace in $\calE^{\mathrm{AIP}}$. We denote by $\widetilde{\bbT}_x$ the local ring $\scrO_{\calE^{\mathrm{AIP}}, x}$ and consider the Galois determinant $\widetilde{\Det}_x: \widetilde{\bbT}_x[\Gal_{\Q}] \rightarrow \widetilde{\bbT}_x$ (\cite{BB22}). Let $\widetilde{A} = \widetilde{\bbT}_x[\Gal_{\Q}]/\ker\widetilde{\Det}_x$. Then, by construction, we have $A = \widetilde{A} \otimes_{\widetilde{\bbT}_x}\bbT_x$. In particular, we have a natural inclusion $\Ext_A^1(\rho_1, \rho_0) \hookrightarrow \Ext_{\widetilde{A}}^1(\rho_i, \rho_0)$. In what follows, we may thus view our extension $V$ as an extension in $\Ext_{\widetilde{A}}^1(\rho_i, \rho_0)$. The advantage of this point of view is that we may now consider the weakly refined family \begin{equation}\label{eq: AIP weakly refined family}
        (\calE^{\mathrm{AIP}}, \calE^{\mathrm{AIP}}_{\nc}, \widetilde{\Det}, \{0, \kappa_{B, 2}-2, \kappa_{B, 1}-1, \kappa_{B, 1}+\kappa_{B,2}-3\}, U_{\mathrm{Si}})
    \end{equation}
    in the sense of \cite[Definition 4.2.7]{BC}, where $\calE^{\mathrm{AIP}}_{\nc}$ is the set of small-slope point and $\kappa_B = (\kappa_{B, 1}, \kappa_{B, 2}): T_{\GL_2}(\Z_p) \rightarrow \scrO_{\calW}(\calW)^{\times}$ is the universal weight map.

    \noindent \textbf{Step 1.} The extension $V$ is crystalline at $p$ when $i=1$. 

    When $i=1$, we may apply \cite[Theorem 4.3.6]{BC} to \eqref{eq: extension in GMA; fs} and see that \[
        \dim \D_{\cris}(V|_{\Gal_{\Q_p}})^{\varphi = p^{k-1}} =  1.
    \] 
    On the other hand, by applying the left-exact functor $\D_{\cris}$ to \eqref{eq: extension in GMA; fs}, we have \[
        \D_{\cris}(\rho_{f, p}) \hookrightarrow \D_{\cris}(V|_{\Gal_{\Q_p}}). 
    \]
    Since $\alpha, \beta \neq p^{k-1}$ and $\rho_{f, p}$ is crystalline, we obtain three different $\varphi$-eigenvalues on $\D_{\cris}(V|_{\Gal_{\Q_p}})$ and so it is crystalline.

    \noindent \textbf{Step 2.} The extension $V$ is unramified at primes dividing $N$. 
    
    Since $\pi_{f, \ell}$ is a twisted Steinberg representation (Assumption \ref{Assumption: nice properties for finite-slope cuspidal newform} (ST)), this has been shown by Skinner--Urban in \cite[Lemme 4.1.3]{Skinner--Urban}.
\end{proof}

\begin{Proposition}\label{Proposition: 12 and BK Selmer; fs}
    There is a natural embedding \[
        \Ext_A^1(\rho_1, \rho_2) \hookrightarrow H^1_f(\Q, \chi_{\cyc}). 
    \]
\end{Proposition}
\begin{proof}
    Let \[
        0 \rightarrow \chi_{\cyc}^{2-k} \rightarrow V \rightarrow \chi_{\cyc}^{1-k}\rightarrow 0
    \] 
    be an $A$-extension in $\Ext^1_A(\rho_1, \rho_2)$, which embeds into $H^1(\Gal_{\Q, S}, \chi_{\cyc})$ by \cite[Theorem 1.5.10]{BC}. We have to show that such an extension is crystalline at $p$ and unramified at primes $\ell$ dividing $N$.  

    \noindent \textbf{Step 1.} The extension $V$ is crystalline at $p$. 

    Thanks to Step 0 in the proof of Proposition \ref{Proposition: i0 and BK Selmer; fs}, we may again apply \cite[Theorem 4.3.6]{BC} and see that \[
        \dim \D_{\cris}(V|_{\Gal_{\Q_p}})^{\varphi = p^{k-1}} = 1.
    \]
    However, since $\D_{\cris}(\chi_{\cyc}^{2-k}) \hookrightarrow \D_{\cris}(V|_{\Gal_{\Q_p}})$, we obtain two different $\varphi$-eigenvalues on $\D_{\cris}(V|_{\Gal_{\Q_p}})$ and so it is crystalline.

    \noindent \textbf{Step 2.} The extension $V$ is unramified at primes $\ell$ dividing $N$. 
    
    We argue as in the proof of \cite[Theorem 4.3]{BB22}. We sketch the proof for the convenience of the reader. 
    
    To simplify the notation, we write $Q = \mathrm{Frac} \bbT_x$. Since $\bbT_x$ is reduced, \cite[Theorem 1.4.4 (ii)]{BC} implies a Galois representation \[
        \rho_Q:\Gal_{\Q} \rightarrow \GL_4(Q).
    \]
    Let $N_{Q} = N_{Q, \ell}$ be the monodromy operator on the Weil--Deligne representation $\mathrm{WD}(\rho_{Q, \ell})$. By \cite[Theorem 3.5]{Mok-GL2CM}, \cite[Corollary 1]{Sorensen-HilbertSiegel}, and \cite[Proposition 7.8.19 (ii)]{BC}, the rank of $N_{Q}$ is bounded by $1$.

    Let $A_{\ell}$ be the image of $\bbT_x[\Gal_{\Q_{\ell}}]$ in $A$. Then, \cite[Proposition 2.3]{BB22} and \cite[Lemma 8.2.11]{BC}, there exists idempotents $\widetilde{e}_1$, $\widetilde{e}_2$, $\widetilde{e}_0$ of $A$, lifting the idempotents corresponding to $\chi_{\cyc}^{1-k}$, $\chi_{\cyc}^{2-k}$, $\rho_f$ respectively such that $\widetilde{e} = \widetilde{e}_1 + \widetilde{e}_2$ lands in the centre of $A_{\ell}$ (\cite[Lemma 8.2.12]{BC}). As a result, we see that \[
        A_{\ell} \cong \begin{pmatrix}
            * & * & * \\ * & * & * \\ & & M_2(\bbT_x)
        \end{pmatrix}.
    \] 
    By \cite[Lemma 7.8.14]{BC}, we can view $N_Q$ as an element in $A_{\ell}$ and so $N_Q$ has the form \[
        N_Q = \begin{pmatrix}
            N_{Q, a} & N_{Q, b}\\ & N_{Q, d}
        \end{pmatrix},
    \]
    where each $N_{Q, ?}$ is $2$-by-$2$ matrices. 

    To show the extensions $V$ are unramified at $\ell$, it is enough to show that $\widetilde{e}N_{Q} \in \widetilde{e}A_{\ell}$ is trivial. To show that $\widetilde{e}N_Q$ is trivial, first note that the rank of the monodromy operator on $\mathrm{WD}(\rho_{f, \ell})$ is $1$ since $f$ satisfies (ST). Moreover, by \cite[Sect. 1.5.2]{BC}, there is a surjection $(1-\widetilde{e})(A/\frakm_x A) (1-\widetilde{e}) \twoheadrightarrow \rho_f$. Hence, by applying \cite[Proposition 7.8.8]{BC} to $(1-\widetilde{e})A_{\ell} (1-\widetilde{e})$, we see that $(1-\widetilde{e})N_Q$ has rank $1$. Since $N_Q$ has rank at most $1$, this implies that $N_Q$ has rank $1$ and \[
        (1-\widetilde{e})N_Q = N_Q.
    \]
    However, by the definition of the idempotent $\widetilde{e}$, we see that this implies $N_{Q, a} = 0 = N_{Q, b}$. We thus conclude that $\widetilde{e}N_Q = 0$. 
\end{proof}

\begin{Corollary}\label{Corollary: explicity reducibility ideal; fs}
    There exists $g\in \mathrm{Frac}\bbT_x$ such that the following equalities of fractional ideals of $\mathrm{Frac} \bbT_x$ hold \[
        \frakm_x = \frakI^{\red} = A_{10}A_{01} 
    \]
\end{Corollary}
\begin{proof}
    The first equality is exactly Proposition \ref{Proposition: reducibility idea = maximal ideal; fs}. To show the second equality, we have to show \[
        A_{20}A_{02} \subset A_{10}A_{01}  \quad \text{ and }\quad A_{12}A_{21} \subset A_{10}A_{01}.
    \] 
    Recall the permutation $\sigma = (12)$ on $(0, 1, 2)$. By \cite[Remark 5.5]{MQ-SympDet} (see also \cite[Lemma 8.2.16]{BC}), we see that \[
        A_{20}A_{02} = A_{10}A_{01}.
    \]
    On the other hand, \eqref{eq: mysterious isomorphism; fs} yields \[
        \Hom_{L_x}(A_{21}/A_{20}A_{01}, L_x) \cong \Ext_A^1(\chi_{\cyc}^{1-k}, \chi_{\cyc}^{2-k}) \hookrightarrow H^1_f(\Q, \chi_{\cyc}),
    \]
    where the inclusion follows from Proposition \ref{Proposition: 12 and BK Selmer; fs}. However, $H^1_f(\Q, \chi_{\cyc}) = 0$ and so $A_{21} = A_{20}A_{01}$. Thus, \[
        A_{12}A_{21} = A_{12}A_{20}A_{01} \subset A_{10}A_{01} 
    \]
    and we conclude the result. 
\end{proof}

\begin{Theorem}\label{Theorem: bound of tangent space; fs}
    Let $d$ be the dimension of $H^1_f(\Q,\rho_f(k-1))$ and $t$ be the dimension of the tangent space of $\calE_{\mathrm{Kl}, (k,k)}^{\cusp}$ at $x$. Then, \[
        t \leq d(d+1).
    \]
\end{Theorem}
\begin{proof}
    The proof is inspired by the proof of \cite[Theorem 6.8 \& Corollary 6.10]{BB22}.
    
    First of all, by Proposition \ref{Proposition: i0 and BK Selmer; fs} with $i=1$ and \eqref{eq: mysterious isomorphism; fs}, we see that \[
        \Hom_{L_x}(A_{01}/A_{02}A_{21}) \cong \Ext_A^1(\rho_1, \rho_0) \hookrightarrow H^1_f(\Q, \rho_f(k-1)).
    \]
    Hence, by Nakayama's Lemma, $A_{01}$ is generated by at most $d$ elements.

    Secondly, by Proposition \ref{Proposition: i0 and BK Selmer; fs} with $i=2$ and \eqref{eq: mysterious isomorphism; fs}, we see that \[
        \Hom_{L_x}(A_{10}/A_{12}A_{20}) \cong \Ext_A^1(\rho_0, \rho_a) = \Ext_A^1(\rho_f, \chi_{\cyc}^{1-k}) \cong \Ext_A^1(\chi_{\cyc}^{2-k}, \rho_f) \hookrightarrow H^1_{f'}(\Q, \rho_f(k-2)).
    \] 
    By the definition of $H^1_{f'}$, we have an exact sequence \[
        0 \rightarrow H^1_f(\Q, \rho_f(k-2)) \rightarrow H^1_{f'}(\Q, \rho_f(k-2)) \rightarrow H^1(\Q_p, \rho_f(k-2))/H^1_f(\Q_p, \rho_f(k-2)).
    \]
    By the proof of \cite[Proposition 5.2.7]{BC}, \[
        \dim \left(\frac{H^1(\Q_p, \rho_f(k-2))}{H^1_f(\Q_p, \rho_f(k-2))}\right) = \dim \rho_f(k-2) + \dim H^0(\Q_p, \rho_f^{\vee}(3-k)) - \dim\left(\frac{\D_{\dR}(\rho_f(k-2))}{\Fil^0\D_{\dR}(\rho_f(k-2))}\right)=1.
    \]
    Moreover, it follows from  \cite[Theorem 14.2]{Kato-2004} that $H^1_f(\Q, \rho_f(k-2)) = 0$. Hence, \[
        \dim H^1_{f'}(\Q, \rho_f(k-2)) \leq 1.
    \] So, Nakayama's Lemma yields  \[
        A_{10} = A_{12}A_{20} +g\bbT_x
    \]
    for some $g\in A_{10}$.

    Now, \eqref{eq: mysterious isomorphism; fs} also yields \[
        \Hom_{L_x}(A_{12}/A_{10}A_{02}) \cong \Ext_A^1(\rho_2, \rho_1) = \Ext_A^1(\chi_{\cyc}^{2-k}, \chi_{\cyc}^{1-k}) \hookrightarrow H^1(\Q, \chi_{\cyc}^{-1}).
    \]
    Note that $\dim H^1(\Q, \chi_{\cyc}^{-1}) =  1$ (see, for example, \cite[Proposition 6.5]{BB22}) and so \[
        A_{12} = A_{10}A_{02} + g'\bbT_x
    \]
    for some $g'\in A_{12}$. Therefore, \[
        A_{10} = A_{10}A_{02}A_{20} + g'A_{20} + g\bbT_x = A_{10} \frakm_x + g'A_{20} + g\bbT_x,
    \]
    where the second equation follows from the proof of Corollary \ref{Corollary: explicity reducibility ideal; fs}. Denote by $\mathrm{min.gen}(M)$ a minimal generating set of a $\bbT_x$-module $M$, then Nakayama's Lemma yields \[
        \#\mathrm{min.gen}(A_{10}) \leq  1 + \#\mathrm{min.gen}(A_{20}).
    \]  However, \eqref{eq: mysterious isomorphism; fs} also gives \[
        \Hom_{L_x}(A_{20}/A_{21}A_{10}) \cong \Ext_A^1(\rho_0, \rho_2) = \Ext_A^1(\rho_f, \chi_{\cyc}^{2-k}) \cong \Ext_A^1(\chi_{\cyc}^{1-k}, \rho_f) \hookrightarrow H^1_f(\Q, \rho_f(k-1)).
    \]
    In particular, $\#\mathrm{min.gen}(A_{20}) \leq d$.

     Since $\frakm_x = A_{10}A_{01}$ by Corollary \ref{Corollary: explicity reducibility ideal; fs}, we see that $\#\mathrm{min.gen}(\frakm_x) \leq d(d+1)$, which proves the desired result.  
\end{proof}

\begin{Corollary}\label{Corollary: non-vanishing of BK Selmer group; fs}
    Keep the notations as above. We have \[
        \dim H^1_f(\Q, \rho_f(k-1)) \geq  1. 
    \]
\end{Corollary}
\begin{proof}
    Since $\calE_{\mathrm{Kl}, (k,k)}^{\cusp}$ is equidimensional of dimension $1$, the tangent space at $x$ has at least dimension $1$. Theorem \ref{Theorem: bound of tangent space; fs} yields \[
        1 \leq t \leq d(d+1),
    \]
    which concludes the desired result. 
\end{proof}

\section{Saito--Kurokawa points II: infinite slope}\label{section: SK points; infinite slope}

In this section, we continue to assume that $f$ is a cuspidal newform of level $\Gamma_0(N)$, weight $2k-2$, and $\epsilon_f = -1$. However, we assume that $f$ has infinite slope at $p$. Therefore, By \cite[Proposition 2.8]{Loeffler--Weinstein}, we know that $p^2 \mid N$. Let $r$ be the largest integer such that $p^r \mid N$. For simplicity, we will assume $N = p^r$ in this section and so $\SK(f) \in S_{(k,k)}(K^{\mathrm{par}}_{p^r})$.

\subsection{Saito--Kurokawa lifts for infinite-slope cuspforms}\label{subsection: SK lifts for infinite-slope cuspforms}

Recall the (deeper) Iwahori subgroup \[
    \Iw_{B, r}  = \left\{ \bfgamma\in \GSp_4(\Z_p): (\bfgamma \mod p^r)\in B(\Z/p^r/Z) \right\}.
\]
Observe that $\Iw_{B, r}$ is a subgroup of $K^{\mathrm{par}}_{p^r}$.
We thus have the natural injection \[
    \Psi^{\mathrm{par}}_{B}: S_{(k,k)}(K^{\mathrm{par}}_{p^r}) \rightarrow S_{(k,k)}(\Iw_{B, r}).
\]
Note that $\Psi^{\mathrm{par}}_{B}$ is Hecke-equivariant away from $p$. In particular, $\Psi^{\mathrm{par}}_{B}(\SK(f))$ is an eigenform for Hecke operators away from $p$.

On $S_{(k,k)}(\Iw_{B, r})$, we can look at the operators $U_{\mathrm{Si}}$ and $U_{\mathrm{Kl}}$ at $p$. We have the following proposition. 

\begin{Proposition}\label{Proposition: p-stabilisations of infinite-slope SK(f)}
    There exist at most two Hecke-eigenforms $\SK(f)_1$ and $\SK(f)_2$ in $S_{(k,k)}(\Iw_{\mathrm{Si}, r})$ such that \[
        \Psi^{\mathrm{par}}_{B}(\SK(f)) = \SK(f)_1 + \SK(f)_2
    \]
    and \[
        v_p(U_{\mathrm{Si}}(\SK(f)_i)) = k-i, \quad U_{\mathrm{Kl}}(\SK(f)_i) = 0
    \]
    for $i=1, 2$. We call each $\SK(f)_i$ a Siegel stabilisation of $\SK(f)$.
\end{Proposition}
\begin{proof}
    Suppose $g$ is a Hecke-eigen summand of $\Psi^{\mathrm{par}}_{B}(\SK(f))$ and let $\rho_g$ be the associated Galois representation. Since both $\rho_{\SK(f)}$ and $\rho_g$ are semisimple and are isomorphic outside $p$, they are isomorphic to each other by the Chebotarev density theorem. Hence, by Hypothesis \ref{Hypothesis: explicit local-global compatibility at p}, we see that these summands correspond to the choices of $P_{\mathrm{Kl}}$-flags on the potentially semistable module $\D_{\pst}(\rho_{\SK(f), p})$. So, we argue using the Galois information. 
    
    Consider the potentially semistable module \[
        \D_{\pst}(\mathrm{std}\circ \rho_{\SK(f), p}) = \D_{\pst}(\chi_{\cyc}^{2-k}) \oplus \D_{\pst}(\rho_{f, p}) \oplus \D_{\pst}(\chi_{\cyc}^{1-k}).
    \] 
    Since $f$ is an infinite-slope newform at $p$, its corresponding local $\GL_2(\Q_p)$-representation is supercuspidal (\cite[Proposition 2.8]{Loeffler--Weinstein}). Hence, $\D_{\pst}(\rho_{f, p})$ is $(\varphi, N, \Gal_{\Q_p})$-irreducible (\cite[Theorem 1.3]{CDP-plocallanglandsGL2}). Therefore, the only $P_{\mathrm{Kl}}$-refinements on $\D_{\pst}(\mathrm{std}\circ \rho_{\SK(f), p})$ can only be either \[
        \D_{\pst}(\chi_{\cyc}^{2-k}) \subset \D_{\pst}(\chi_{\cyc}^{2-k}) \oplus \D_{\pst}(\rho_{f,p}) \subset \D_{\pst}(\rho_{\SK(f), p}) 
    \]
    or \[
        \D_{\pst}(\chi_{\cyc}^{1-k}) \subset \D_{\pst}(\chi_{\cyc}^{1-k}) \oplus \D_{\pst}(\rho_{f,p}) \subset \D_{\pst}(\rho_{\SK(f), p}).
    \]
    Note that both $\D_{\pst}(\chi_{\cyc}^{i-k})$ are isotropic subspaces in $\D_{\pst}(\rho_{\SK(f), p})$ with respect to the symplectic pairing. 
    Since $\mathrm{std}: \GSp_4 \hookrightarrow \GL_4$ is a tensor generator, we conclude the result. 
\end{proof}

\begin{Remark}\label{Remark: automorphic proof for the p-stabilisation for infty-slope SK lifts}
    We argued in Proposition \ref{Proposition: p-stabilisations of infinite-slope SK(f)} using Galois information. However, it would be an interesting task to prove the statement via a pure automorphic language. Similar results have been done in \cite{Schmidt-classicalSK} for those cuspidal eigenforms with square-free levels. On the other hand, comparing with the $p$-stabilisations in the case when $f$ is of finite-slope, it seems to suggest there could be only one $p$-stabilisations when $f$ is of infinite-slope. In any case, our strategy can be adjusted to work for a different choice of $p$-stabilisation.  
\end{Remark}

\begin{Corollary}\label{Corollary: p-stabilisations for SK lifts of infty-slope forms are critical}
    Keep the notations as in Proposition \ref{Proposition: p-stabilisations of infinite-slope SK(f)}. The $P_{\mathrm{Kl}}$-flags associated with $\SK(f)_1$ and $\SK(f)_2$ are both critical (in the sense of Definition \ref{Definition: P-non-critical}). 
\end{Corollary}
\begin{proof}
    This follows immediately from Hypothesis \ref{Hypothesis: explicit local-global compatibility at p} and the proof of Proposition \ref{Proposition: p-stabilisations of infinite-slope SK(f)}.
\end{proof}

\subsection{Families passing through points given by Saito--Kurokawa lifts of infinite-slope cuspforms}\label{subsection: heuristic of geometry}

Similarly as before, since we are working with Siegel cuspforms of $\Iw_{B, r}$-level at $p$, we consider the abstract Hecke algebra \[
    \bbT = \left(\bigotimes_{\ell\not\in S_{\mathrm{bad}}}  \Z_p[\GSp_4(\Z_{\ell}) \backslash \GSp_4(\Q_{\ell})/\GSp_4(\Z_{\ell})]\right) \otimes_{\Z_p} \Z_p[U_{\mathrm{Si}}, U_{\mathrm{Kl}}].
\] 
By the machinery in Sect. \ref{subsection: small par. eigenvar.}, we produce a small Siegel parabolic eigenvariety $\calE_{\mathrm{Si}, (0,0)}^{\cusp}$ over the parallel weight space $\calW_{\mathrm{Si}} \subset \calW$.

\begin{Remark}\label{Remark: slightly different eigenvariety; infintie slope}
    Similarly as in Remark \ref{Remark: slightly different eigenvariety}, we remind the reader that $U_{\mathrm{Si}}$ is inverted in the construction. We should also mention that this eigenvariety is slightly different from the one in Sect. \ref{subsection: small par. eigenvar.} since we also include $U_{\mathrm{Kl}}$ in the construction. As a result, we do not know if the infinitesimal $R=\bbT$ (Sect. \ref{section: R=T}) holds in this situation.
\end{Remark}

Thanks to Proposition \ref{Proposition: p-stabilisations of infinite-slope SK(f)} and Corollary \ref{Corollary: p-stabilisations for SK lifts of infty-slope forms are critical},  $\SK(f)_1$ and $\SK(f)_2$ define points in $\calE_{\mathrm{Si}, (0,0)}^{\cusp}$. In what follows, let $x\in \calE_{\mathrm{Si}, (0,0)}^{\cusp}(\overline{\Q}_p)$ be either the point correspond to $\SK(f)_1$ or $\SK(f)_2$. We would like to understand how $\calE_{\mathrm{Si}, (0,0)}^{\cusp}$ behaves around $x$. 

To this end, we first prove the following proposition. 

\begin{Proposition}\label{Proposition: dijoint union of Kl-fs and Kl-infs parts}
    Recall from the construction that $U_{\mathrm{Kl}}$ defines a global section on $\calE_{\mathrm{Si}, (0,0)}^{\cusp}$, which we still denote by $U_{\mathrm{Kl}}$. Let $\calE_{\mathrm{Si}, (0,0)}^{\cusp, \mathrm{Kl}=0}$ (resp., $\calE_{\mathrm{Si}, (0,0)}^{\cusp, \mathrm{Kl}\neq 0}$) be the locus on which $U_{\mathrm{Kl}} = 0$ (resp., $U_{\mathrm{Kl}} \neq 0$). Then, $\calE_{\mathrm{Si}, (0,0)}^{\cusp, \mathrm{Kl}=0}$ is a one-dimensional closed subspace in $\calE_{\mathrm{Si}, (0,0)}^{\cusp}$ and there is a decomposition \[
        \calE_{\mathrm{Si}, (0,0)}^{\cusp} = \calE_{\mathrm{Si}, (0,0)}^{\cusp, \mathrm{Kl}=0} \sqcup \calE_{\mathrm{Si}, (0,0)}^{\cusp, \mathrm{Kl}\neq 0}.
    \]
\end{Proposition}
\begin{proof}
    We only need to show that $\calE_{\mathrm{Si}, (0,0)}^{\cusp, \mathrm{Kl}=0}$ is of one-dimensional since the decomposition is obvious as the conditions are disjoint. 
    
    To this end, for any affinoid open neighbourhood $\calU = \Spa(R_{\calU}, R_{\calU}^{\circ}) \subset \calW_{\mathrm{Si}}$, let $\kappa_{\calU}$ be the corresponding weights, choose $w, v$ as in Definition \ref{Definition: families of automorphic sheaves, Siegel}, and consider the space $S_{\kappa_{\calU}}^{w, v}(\Iw_{B, r})$ of families of $w$-analytic Siegel cuspforms of weight $\kappa_{\calU}$. Let $h\in \Q_{\geq 0}$ such that the slope-$\leq h$ part (with respect to $U_{\mathrm{Si}}$) of $S_{\kappa_{\calU}}^{w, v}(\Iw_{B, n})$ is well-defined. We further define \[
        S_{\kappa_{\calU}}^{w, v}(\Iw_{B, r})^{\leq h, \mathrm{Kl}=0} 
        \coloneq \ker U_{\mathrm{Kl}}|_{S_{\kappa_{\calU}}^{w, v}(\Iw_{B, r})^{\leq h}}.
    \]
    We claim that $ S_{\kappa_{\calU}}^{w, v}(\Iw_{B, r})^{\leq h, \mathrm{Kl}=0}$ is a finite projective $R_{\calU}$-module. 

    This is claim can be reinterpreted purely in terms of commutative algebra: Let $R$ be a Dedekind domain, $M$ be a finite projective $R$-module, $\psi$ be an endomorphism on $M$, and $M^{\psi = 0} \coloneq \ker \psi$; we claim that $M^{\psi = 0}$ is a finite projective $R$-module. To show this, choose a finite projective $R$-module $N$ such that $M \oplus N  = R^n$ (for some $n\in \Z_{>0}$). We extend $\psi$ to $M \oplus N$ by putting $\psi(N) = 0$. Then, we have the following commutative diagram \[
        \begin{tikzcd}
            0 \arrow[r] & N \arrow[r]\arrow[d, equal] & \ker\psi \arrow[r]\arrow[d] & M^{\psi = 0}\arrow[d]\\
            0 \arrow[r] & N \arrow[r] &  M \oplus N \arrow[l, bend left = 20] \arrow[r] & M \arrow[r] \arrow[r] & 0
        \end{tikzcd},
    \]
    where the rows are exact sequences and the map $M \oplus N \rightarrow N$ is the splitting. By the snake lemma, we see that the top row is also a short exact sequence. Note that, since $\ker \psi \subset M \oplus N \cong R^n$ and $R$ is a Dedekind domain, $\ker \psi$ is a finite projective $R$-module. Moreover, the splitting on the bottom row yields a splitting on the top row. In other words, we see that $M^{\psi = 0}$ is a direct summand of a finite projective $R$-module, hence it is finite projective.  

    Therefore, by letting \[
        \bbT_{\calU}^{\leq h, \mathrm{Kl}=0} \coloneq \text{reduced $R_{\calU}$-algebra generated by the image of $\bbT[U_{\mathrm{Si}}^{-1}] \rightarrow \End_{R_{\calU}}(S_{\kappa_{\calU}}^{w, v}(\Iw_{B, r})^{\leq h, \mathrm{Kl}=0})$},
    \]
    we see that $\bbT_{\calU}^{\leq h, \mathrm{Kl}=0}$ is a finite $R_{\calU}$-algebra. However, by construction, $\{ \Spa( \bbT_{\calU}^{\leq h, \mathrm{Kl}=0}, \bbT_{\calU}^{\leq h, \mathrm{Kl}=0, \circ})\}_{(\calU, h)}$ form an affinoid open covering for $\calE_{\mathrm{Si}, (0,0)}^{\cusp, \mathrm{Kl}=0}$. We thus conclude the result. 
\end{proof}

From the construction, we see that $x\in \calE_{\mathrm{Si}, (0,0)}^{\cusp, \mathrm{Kl} = 0}$. Let $\calV$ be an open affinoid neighbourhood of $x$ in $\calE_{\mathrm{Si}, (0,0)}^{\cusp, \mathrm{Kl} = 0}$. Our next task is to understand what kind of classical points can appear in $\calV$ (up to shrinking).  To this end, recall the space of classical Siegel cuspforms $S_{(k_1, k_2)}(\Iw_{B, r})$ of weight $(k_1, k_2)$. According to Arthur's classification (see \cite[Sect. 1.1]{Schmidt-Packet}), there are six different types of parameters that can appear in $S_{(k_1, k_2)}(\Iw_{B, r})$: \begin{enumerate}
    \item[\textbf{(G)}] This is the \emph{general type}, whose parameter is given by $\pi = \pi_{\GL_4} \boxtimes 1$, where $\pi_{\GL_4}$ is a self-dual, symmplectic, unitary, cuspidal automorphic representation of $\GL_4$. Their associated Galois representations are stable irreducible $\GSp_4$-valued Galois representations. Here, `stable' means that it remains irreducible after composing with $\mathrm{std}: \GSp_4 \rightarrow \GL_4$.  
    \item[\textbf{(Y)}] This is the \emph{Yoshida type}, whose parameter is given by $\pi = (\pi_1 \boxtimes 1)\boxplus (\pi_2\boxtimes 1)$, where $\pi_1$ and $\pi_2$ are distinct, unitary, cuspidal automorphic representations of $\GL_2$. Their associated Galois representations are of the form \[
        \rho_{\pi} \sim \begin{pmatrix} \rho_{\pi_1, a} & & \rho_{\pi_1, b}\\ & \rho_{\pi_2}(j) & \\ \rho_{\pi_1, c} & & \rho_{\pi_1, d} \end{pmatrix},
    \]
    (for some $j\in \Z$), where $\rho_{\pi_i} \sim \begin{pmatrix}
        \rho_{\pi_i, a} & \rho_{\pi_i, b} \\ \rho_{\pi_i, c} & \rho_{\pi_i, d}
    \end{pmatrix}$ is the Galois representation attached to $\pi_i$. Note that $\rho_{\pi}$ can be irreducible but not stable as $\mathrm{std} \circ \rho_{\pi} \sim \diag(\rho_{\pi_1}, \rho_{\pi_2})$.
    \item[\textbf{(Q)}] This is the \emph{Klingen type}, whose parameters are given by $\pi = \pi_{\GL_2} \boxtimes \nu(2)$, where $\pi_{\GL_2}$ is a self-dual, unitary, cuspidal automorphic representation of $\GL_2$ and $\nu(2)$ is the two-dimensional irreducible representation of $\SL_2(\C)$. Their associated Galois representations are of the form $\rho_{\pi} \sim \diag(\rho_{\pi_{\GL_2}}, \oneanti_2\trans\rho_{\pi_{\GL_2}}^{-1}(j)\oneanti_2)$ (for some $j\in \Z$).
    \item[\textbf{(P)}] This is the \emph{Saito--Kurokawa type}, whose parameter is given by $\pi = (\pi_{\GL_2} \boxtimes 1) \boxplus (\sigma \boxtimes \nu(2))$, where $\pi_{\GL_2}$ is a unitary, cuspidal automorphic representation of $\GL_2$ and $\sigma$ is a quadratic Hecke character. Their associated Galois representations are of the form described in Theorem \ref{Theorem: SK lifts}.
    \item[\textbf{(B)}] This is the \emph{Howe--Piatetski-Shapiro type}, whose parameter is given by $\pi  = (\sigma_1 \boxtimes \nu(2)) \boxplus (\sigma_2 \boxtimes \nu(2))$, where $\sigma_i$'s are distinct quadratic Hecke characters. Their associated Galois representations are of the form $\rho_{\pi} \sim \diag(\rho_{\sigma_1}, \rho_{\sigma_2}, \rho_{\sigma_2}^{-1}(j), \rho_{\sigma_1}^{-1}(j))$ (for some $j\in \Z$), where $\rho_{\sigma_i}$ is the one-dimensional Galois representation attached to $\sigma_i$. 
    \item[\textbf{(F)}] The parameters of this type are given by $\pi = \sigma \boxtimes \nu(4)$, where $\sigma$ is a quadratic Hecke character and $\nu(4)$ is the four-dimensional irreducible representation of $\SL_2(\C)$. Their associated Galois representations are of the form $\rho_{\pi} \sim \diag(\rho_{\sigma}(j), \rho_{\sigma}(j-1), \rho_{\sigma}^{j-2}(1), \rho_{\sigma}^{-1}(j-3))$ (for some $j\in \Z$).
\end{enumerate}

Under Hypothesis \ref{Hypothesis: explicit local-global compatibility at p}, we see that classical Siegel cuspforms in $S_{(k_1, k_2)}(\Iw_{B, r})$ of types \textbf{(Q)}, \textbf{(B)}, \textbf{(F)} are of $U_{\mathrm{Kl}}$-finite slope. Hence, these classical forms do not appear in $\calV$. To address what happens to other type of classical Siegel cuspforms, we have the following proposition. 

\begin{Proposition}\label{Proposition: a generic family passing through SK lifts of infinite-slope forms}
    Keep the notations as above. Assume the elliptic cuspform $f$ satisfies the following properties: \begin{enumerate}
        \item[(i)] The residual representation $\overline{\rho} = \overline{\rho}_f$ is absolutely irreducible. 
        \item[(ii)] The cuspform $f$ is not approximable by finite-slope modular forms. 
    \end{enumerate}
    Then, up to shrinking, classical points in $\calV$ are either $x$ or of type {\normalfont \textbf{(G)}}.
\end{Proposition}
\begin{proof}
    We first rule out the existence of classical points of type \textbf{(P)}. Suppose $g$ is a cuspidal eigenform of sign $-1$ such that $\SK(g)$ defines a point $y$, possibly up to a choice of $p$-stabilisation, in $\calV$. First of all, note that $g$ cannot be of finite-slope, otherwise $\D_{\rig}^{\dagger}(\rho_{\SK(g), p})$ would be triangulisable, which implies that $y \in \calE_{\mathrm{Si}, (0,0)}^{\cusp, \mathrm{Kl} \neq 0}$ by Hypothesis \ref{Hypothesis: explicit local-global compatibility at p}. Hence, $g$ must be of infinite-slope. However, our discussion in Proposition \ref{Proposition: p-stabilisations of infinite-slope SK(f)} shows that both $P_{\mathrm{Kl}}$-flags on $\D_{\rig}^{\dagger}(\rho_{\SK(g), p})$ are $P_{\mathrm{Kl}}$-critical. Hence, by Corollary \ref{Corollary: critical points cannot accumulate}, up to shrinking $\calV$, we may have $y\not\in \calV(\overline{\Q}_p)$.

    Now, we rule out the existence of classical points of type \textbf{(Y)}. Let $g, h$ be two elliptic cuspidal eigenforms. According to \cite[Proposition 3.1]{Saha--Schmidt}, for $g$ and $h$ to form a Yoshida lift of parallel weight $(m, m)$ and possibly appearing in $\calV$, $g$ and $h$ satisfy the following conditions: \begin{itemize}
        \item One of $g, h$ is of weight $2m-2$ and the other is of weight $2$. Without loss of generality, suppose $g$ has weight $2m-2$ with $m>2$ and $h$ has weight $2$.
        \item Since we built our eigenvariety by looking at $S_{\kappa_{\calU}}^{w, v}(\Iw_{B, r})$ with trivial tame level, $g$ and $h$ can only have level at $p$. Moreover, their Atkin--Lehner eigenvalues at $p$ coincide. 
    \end{itemize}  
    We denote by $\Yos(g, h)$ the Yoshida lift of the pair $(g, h)$. Let $\rho_g$ and $\rho_h$ be the $2$-dimensional Galois representation attached to $g$ and $h$ respectively, the Galois representation attached to $\Yos(g, h)$ is then of the form \[
        \rho_{\Yos(g, h)} \sim \begin{pmatrix}
            \rho_{g, a} && \rho_{g, b}\\
            & \rho_h(1-m) & \\
            \rho_{g, c} && \rho_{g, d}
        \end{pmatrix}.
    \]

    Observe that exactly one of $g$ or $h$ is of infinite-slope. Indeed, if both of them are  of infinite-slope, one cannot obtain a $P_{\mathrm{Kl}}$-flag on $\D_{\rig}^{\dagger}(\rho_{\Yos(g, h), p})$; on the other hand, if both of them are of finite-slope, then $\D_{\rig}^{\dagger}(\rho_{\Yos(g, h), p})$ is triangulisable and so the points defined by $\Yos(g, h)$ land in $\calE_{\mathrm{Si}, (0,0)}^{\cusp, \mathrm{Kl}\neq 0}$ by Hypothesis \ref{Hypothesis: explicit local-global compatibility at p}. 

    Suppose $g$ has infinite slope and $h$ has finite slope. Let $\Fil_{\bullet} \D_{\rig}^{\dagger}(\rho_h(1-m)_{p})$ be one of the triangulations on $\D_{\rig}^{\dagger}(\rho_h(1-m)_{p})$. Then, \[
        \Fil_1 \D_{\rig}^{\dagger}(\rho_h(1-m)_{p}) \subset \Fil_1 \D_{\rig}^{\dagger}(\rho_h(1-m)_{p}) \oplus \D_{\rig}^{\dagger}(\rho_{g, p}) \subset \D_{\rig}^{\dagger}(\mathrm{std} \circ \rho_{\Yos(g, h), p})
    \]
    gives a $P_{\mathrm{Kl}}$-flag on $\D_{\rig}^{\dagger}(\mathrm{std} \circ \rho_{\Yos(g, h), p})$. However, note that the Hodge--Tate weights on $\rho_{h}(1-m)$ are $\{m-2, m-1\}$. Hence, this $P_{\mathrm{Kl}}$-flag is $P_{\mathrm{Kl}}$-critical. By Corollary \ref{Corollary: critical points cannot accumulate} (again), we may exclude this point by shrinking $\calV$. 

    We may thus assume $h$ has infinite slope and $g$ has finite slope. We consider the residual representations $\mathrm{std} \circ \overline{\rho}_{\SK(f)}$ and $\mathrm{std} \circ \overline{\rho}_{\Yos(g, h)}$. Since \[
        \mathrm{std} \circ \overline{\rho}_{\SK(f)} = \overline{\chi_{\cyc}^{2-k}} \oplus \overline{\rho} \oplus \overline{\chi_{\cyc}^{1-k}}
    \]
    and $\overline{\rho}_f$ is absolutely irreducible, $\mathrm{std} \circ \overline{\rho}_{\SK(f)}$ is a semisimple $\Gal_{\Q}$-representation. On the other hand, \[
        \mathrm{std}\circ \overline{\rho}_{\Yos(g, h)} = \overline{\rho}_g \oplus \overline{\rho}_h(1-m). 
    \]
    If $\Yos(g, h)$ defines a point in $\calV$, then \[
        \mathrm{std} \circ \overline{\rho}_{\SK(f)} \cong \mathrm{std}\circ \overline{\rho}_{\Yos(g, h)}.  
    \]
    Since $\mathrm{std} \circ \overline{\rho}_{\SK(f)}$ is a semisimple $\Gal_{\Q}$-representation, either \[
        \overline{\rho} \cong \overline{\rho}_g  \quad \text{or} \quad \overline{\rho} \cong \overline{\rho}_h(1-m). 
    \]

    We claim that $\overline{\rho} \cong \overline{\rho}_h(1-m)$ cannot happen. We prove by contradiction and so we assume $\overline{\rho} \cong \overline{\rho}_h$. Note that both $\overline{\rho}$ and $\overline{\rho}_h(1-m)$ are odd representations since they are residual representations of cuspforms of even weights. By \cite[Theorem 1.1]{Khare-SerreConj}, $\overline{\rho}(m-1)$ is the residual representation of the Galois representation of a cuspform of level-1 and of the smallest weight $k(\rho)\geq 2$. On the other hand, we see that $\overline{\rho}(m-1)\cong \overline{\rho}_h$ is the residual representation of weight $2$. Thus, $\overline{\rho}(m-1)$ is the residual representation of a cuspform of level $1$ and weight $2$; however, there exists no such cuspforms. 

    We thus conclude that, if $\Yos(g,h)$ defines a point in a small enough neighbourhood $\calV$ of $x$, then \begin{itemize}
        \item $h$ has infinite slope and $g$ has finite slope; 
        \item $\overline{\rho} \cong \overline{\rho}_g$. 
    \end{itemize}
    However, if any affinoid open neighbourhood of $f$ consists of such points, then $f$ can be approximated by finite-slope modular forms $g$. As we assumed that $f$ is not approximable by finite-slope modular forms, by shrinking $\calV$ if necessary, we exclude these points.  
\end{proof}

\begin{Remark}\label{Remark: infinite-slope forms not approximable by finite-slope forms}
    \begin{enumerate}
        \item[(i)] Let us justify condition (ii) in Proposition \ref{Proposition: a generic family passing through SK lifts of infinite-slope forms}. In \cite{Coleman--Stein}, Coleman--Stein studied infinite-slope modular forms that can be approximated by finite-slope modular forms. In particular, in \cite[Guess 4.3]{Coleman--Stein}, they guessed that, in our situation, if $r >2$, then $f$ cannot be approximated by finite-slope modular forms.  
        \item[(ii)] Proposition \ref{Proposition: a generic family passing through SK lifts of infinite-slope forms} suggests the following interesting open question: 

        {\it For infinite-slope newform $f$ of even weight and sign $-1$ but of more general level, what kind of classical points can appear in an affinoid neighbourhood of $x$ (in the corresponding small Siegel parabolic eigenvariety)? More precisely, let $\calV$ be a (small enough) affinoid neighbourhood of $x$, does there exist an irreducible component $\calV_{\normalfont \textbf{(G)}} \subset \calV$ such that classical points in $\calV_{\normalfont \textbf{(G)}}$ are either $x$ or of type {\normalfont \textbf{(G)}}?}
    \end{enumerate}
\end{Remark}

In what follows, suppose $f$ satisfies the conditions in Proposition \ref{Proposition: a generic family passing through SK lifts of infinite-slope forms}. Recall we have a symplectic determinant \[
    (\mathrm{Det}_x, \mathrm{Pf}_x): \Gal_{\Q} \rightarrow \bbT_x,
\]
where $\bbT_x$ is the local ring $\scrO_{\calV, x}$ with maximal ideal $\frakm_x$. Similarly as before, let $A = \bbT_x[\Gal_{\Q}]/\ker \mathrm{Det}_x$, $\Xi = \{I_0=\{\rho_0 = \rho_f\}, I_1 = \{\rho_1 = \chi_{\cyc}^{1-k}\}, I_2 = \{\rho_2 = \chi_{\cyc}^{2-k}\}\}$, and $\sigma = (12)$ be the permutation on $(0, 1, 2)$. Then, $A$ has a symplectic GMA structure of type \[
    \Delta = ((I_0, I_1, I_2), \sigma, (1, 2, 1))
\]
and \[
    A = \begin{pmatrix}
        A_{11} & M_{12}(A_{10}) & A_{12}\\
        M_{21}(A_{01}) & M_{22}(A_{00}) & M_{21}(A_{02})\\
        A_{22} & M_{12}(A_{20}) & A_{22}
    \end{pmatrix},
\]
where $A_{ab}$ are some $\bbT_x$-submodules in $\mathrm{Frac} \bbT_x$ such that $A_{aa} \cong A$, $A_{ab}A_{bc} \subset A_{ac}$, and $A_{ab}A_{ba} \subset \frakm_x$ (\cite[Theorem 1.4.4]{BC}). We can again define the reducibility ideal of $\Det_x$ as in \cite[Definition 1.5.2]{BC}, which is given by \[
    \frakI^{\red} = A_{10}A_{01} + A_{20}A_{02} + A_{12}A_{21}.
\]
Thanks to Proposition \ref{Proposition: a generic family passing through SK lifts of infinite-slope forms}, we have the following nice result.

\begin{Corollary}\label{Corollary: reducibility idea = maximal ideal; infinite slope}
    Keep the notations as above and assume the conditions in Proposition \ref{Proposition: a generic family passing through SK lifts of infinite-slope forms} are satisfied. The natural inclusion $\frakI^{\red} \subset \frakm_x$ is in fact an equality. 
\end{Corollary}
\begin{proof}
    This follows from Proposition \ref{Proposition: a generic family passing through SK lifts of infinite-slope forms} immediately. 
\end{proof}

\subsection{Bloch--Kato Selmer groups}\label{subsection: BK Selmer groups; infinite slope}
We keep the notations as in the previous subsection. We also assume the conditions in Proposition \ref{Proposition: a generic family passing through SK lifts of infinite-slope forms} holds. From now on, we assume $x$ correspond to $\SK(f)_1$.\footnote{The discussion for $\SK(f)_2$ is similar; we leave it to the interested readers.} Similarly as before, we further assume the following condition:\begin{itemize}
    \item (REG) The Frobenius eigenvalues on $\D_{\pst}(\rho_{f, p})$ are different from $p^{k-1}$. 
\end{itemize}

For $a, b, c\in \{0, 1, 2\}$, if $\rho_a$, $\rho_b$, $\rho_c$ are two-by-two distinct, \cite[Theorem 1.5.5]{BC} implies that \begin{equation}\label{eq: mysterious isomorphism; infinite slope}
    \Hom_{L_x}(A_{ab}/A_{ac}A_{cb}, L_x) \cong \mathrm{Ext}^1_{A}(\rho_b, \rho_a).
\end{equation}
We prove the following proposition showing the relationship between $\Ext^1_A(\rho_b, \rho_a)$ and Bloch--Kato Selmer groups.

\begin{Proposition}\label{Proposition: i0 and BK Selmer; infinite slope}
    \begin{enumerate}
        \item[(i)] There is a natural embedding \[
            \Ext_A^1(\rho_1, \rho_0) \hookrightarrow H^1_f(\Q, \rho_f(k-1)). 
        \] 
        \item[(ii)] There is a natural embedding \[
            \Ext_A^1(\rho_2, \rho_0) \hookrightarrow H^1_{f'}(\Q, \rho_f(k-2)),
        \]
        where $H^1_{f'}$ is as defined in Proposition \ref{Proposition: i0 and BK Selmer; fs}.
    \end{enumerate}
\end{Proposition}
\begin{proof}
    Since $\rho_f$ is unramified outside $p$, it is enough to show that, any $A$-extension \[
        0 \rightarrow \rho_f \rightarrow V \rightarrow \chi_{\cyc}^{1-k} \rightarrow 0
    \]
    is crystalline at $p$ (in the sense that $\dim \D_{\cris}(V|_{\Gal_{\Q_p}})  = \dim \D_{\cris}(\rho_{f, p}) + \dim \D_{\cris}(\chi_{\cyc}^{1-k})$). Since $\D_{\pst}(\rho_{f, p})$ is an irreducible $(\varphi, N, \Gal_{\Q_p})$-module, $\D_{\cris}(\rho_{f, p}) = \D_{\pst}(\rho_{f, p})^{\Gal_{\Q_p}, N=0} = 0$. Hence, it is enough to show that $\dim \D_{\cris}(V|_{\Gal_{\Q_p}}) = 1$.

    By Theorem \ref{Theorem: pst period}, we see that $\dim \D_{\pst}(V)^{\varphi = p^{k-1}} = 1$.
    This implies that $p^{k-1}$ is a Frobenius eigenvalue in $\D_{\pst}(V)$. Since $\D_{\pst}(\rho_{f, p}) \hookrightarrow \D_{\pst}(V)$ and the Frobenius eigenvalues of $\D_{\pst}(\rho_{f, p})$ are different from $p^{k-1}$, we conclude that $V$ is potentially semistable and hence de Rham. This shows that, after restricting at $p$, the image of $\Ext_A^1(\rho_1, \rho_0)$ lands in $H^1_g(\Q_p, \rho_f(k-1))$. However, in our situation, the difference between the dimension of $H^1_g(\Q_p, \rho_f(k-1))$ and the dimension of $H^1_f(\Q_p, \rho_f(k-1))$ is $\dim \D_{\cris}(\rho_f^{\vee}(2-k))^{\varphi= 1} = 0$. We thus conclude the desired result. 
\end{proof}

\begin{Corollary}\label{Corollary: explicity reducibility ideal; infs}
    There exists $g\in \mathrm{Frac}\bbT_x$ such that following equalities of fractional ideals of $\mathrm{Frac} \bbT_x$ hold \[
        \frakm_x = \frakI^{\red} = A_{10}A_{01} + gA_{12}.
    \]
\end{Corollary}
\begin{proof}
    We have the first equality from Corollary \ref{Corollary: reducibility idea = maximal ideal; infinite slope}. To show the second equality, we have to show \[
        A_{20}A_{02} \subset A_{10}A_{01} + gA_{12} \quad \text{ and }\quad A_{12}A_{21} \subset A_{10}A_{01} + gA_{12}
    \] 
    for some $g\in \mathrm{Frac}\bbT_x$.
    Similarly as in the proof of Corollary \ref{Corollary: explicity reducibility ideal; fs}, we have \[
        A_{20}A_{02} = A_{10}A_{01}.
    \] 

    On the other hand, \eqref{eq: mysterious isomorphism; infinite slope} yields \[
        \Hom_{L_x}(A_{21}/A_{20}A_{01}, L_x) \cong \Ext_A^1(\chi_{\cyc}^{1-k}, \chi_{\cyc}^{2-k}).
    \]
    By \cite[Theorem 1.5.10]{BC}, the latter space embeds into $H^1(\Gal_{\Q, \{p, \infty\}}, \chi_{\cyc}) \cong L_x$, where the isomorphism follows from Kummer theory. Hence, $A_{21} = A_{20}A_{01} + g\bbT_x$ for some $g\in A_{21}$. Thus, \[
        A_{12}A_{21} = A_{12}A_{20}A_{01} + gA_{12}\subset A_{10}A_{01} + gA_{12}
    \]
    and we conclude the result. 
\end{proof}

\begin{Theorem}\label{Theorem: lower bound theorem for infinite forms}
    Let $d = \dim_{L_x} H^1_f(\Q, \rho_f(k-1))$ and $t$ be the dimension of the tangent space of $\calE_{\mathrm{Si}, (0,0)}^{\cusp}$ at $x$. Then, \[
        t \leq \frac{(d+1)(3d+2)}{2}+1.
    \]
\end{Theorem}
\begin{proof}
    The proof is similar to Theorem \ref{Theorem: bound of tangent space; fs}. 
    
    First of all, by Proposition \ref{Proposition: i0 and BK Selmer; infinite slope} with $i=1$, we have \[
        \Hom_{L_x}(A_{01}/A_{02}A_{21}) \cong \Ext_A^1(\rho_1, \rho_0) \hookrightarrow H^1_f(\Q, \rho_f(k-1))
    \]
    and so the minimal generating set for $A_{01}$ has at most $d$ element. Secondly, by Proposition \ref{Proposition: i0 and BK Selmer; infinite slope} with $i=2$, we have \[
        \Hom_{L_x}(A_{10}/A_{12}A_{20}) \cong \Ext_A^1(\rho_0, \rho_a) = \Ext_A^1(\rho_f, \chi_{\cyc}^{1-k}) \cong \Ext_A^1(\chi_{\cyc}^{2-k}, \rho_f) \hookrightarrow H^1_{f'}(\Q, \rho_f(k-2)).
    \]
    Similar as in the proof of Theorem \ref{Theorem: bound of tangent space; fs}, $\dim H^1_{f'}(\Q, \rho_f(k-2))\leq 1$ and so $A_{10} = A_{12}A_{20} + g'\bbT_x$ for some $g'\in A_{10}$. 
    
    Next, \eqref{eq: mysterious isomorphism; infinite slope} also yields \[
        \Hom_{L_x}(A_{12}/A_{12}A_{02}) \cong \Ext_A^1(\rho_2, \rho_1) = \Ext_A^1(\chi_{\cyc}^{2-k}, \chi_{\cyc}^{1-k}) \hookrightarrow H^1(\Q, \chi_{\cyc}^{-1}).
    \] 
    Since $\dim H^1(\Q, \chi_{\cyc}^{-1}) =  1$, we again obtain $ A_{12} = A_{10}A_{02} + g''\bbT_x$ for some $g''\in A_{12}$. Therefore, \[
        A_{10} = A_{10}A_{02}A_{20} + g'A_{20} +g'\bbT_x \subset A_{10} \frakm_x + g''A_{20} + g'\bbT_x,
    \]
    where the inclusion follows from the proof of Corollary \ref{Corollary: explicity reducibility ideal; infs}. By Nakayama's Lemma, $\#\mathrm{min.gen}(A_{10}) \leq \#\mathrm{min.gen}(A_{20})+1$. However, \eqref{eq: mysterious isomorphism; infinite slope} also gives \[
        \Hom_{L_x}(A_{20}/A_{21}A_{10}) \cong \Ext_A^1(\rho_0, \rho_2) = \Ext_A^1(\rho_f, \chi_{\cyc}^{2-k}) \cong \Ext_A^1(\chi_{\cyc}^{1-k}, \rho_f) \hookrightarrow H^1_f(\Q, \rho_f(k-1)).
    \]
    In particular, the cardinality of the minimal generating set of $A_{20}$ is at most $d$.

    Together with Corollary \ref{Corollary: explicity reducibility ideal; infs}, we see that \[
        \frakm_x = A_{10}A_{01} + gA_{12} = A_{10}A_{01} + gA_{10}A_{02} + gg'\bbT_x.
     \]
     By \cite[Lemma 1.8.5 (ii)]{BC}, $A_{02} \cong A_{10}$ as $\bbT_x$-modules and so $A_{10}A_{02} \cong A_{10}A_{10}$ as $\bbT_x$-modules. Since $\#\mathrm{min.gen}(A_{10}) \leq d+1$, \[
        \#\mathrm{min.gen}(A_{10}A_{10}) \leq \frac{(d+1)(d+2)}{2}.
     \] Hence, \[
        \#\mathrm{min.gen}(\frakm_x) \leq d(d+1)+\frac{(d+1)(d+2)}{2} +1 = \frac{(d+1)(3d+2)}{2}+1,
     \]
     which proves the result. 
\end{proof}

\begin{Corollary}\label{Corollary: non-vanishing}
    If the tangent space of $\calE_{\mathrm{Si}, (0,0)}^{\cusp}$ at $x$ has dimension strictly bigger than $2$, then the Bloch--Kato Selmer group $H^1_f(\Q, \rho_f(k-1))$ does not vanish. 
\end{Corollary}
\begin{proof}
    This follows immediately from Theorem \ref{Theorem: lower bound theorem for infinite forms} and that $\calE_{\mathrm{Si}, (0,0)}^{\cusp}$ is equidimensional of dimension $1$. 
\end{proof}

\begin{Remark}
    Since $f$ has sign $-1$, the Bloch--Kato conjecture already predicts that $H^1_f(\Q, \rho_f(k-1))$ is non-vanishing. As both $p$-stabilisations of $\SK(f)$ are critical, Corollary \ref{Corollary: non-vanishing} suggests that the local geometry of $\calE_{\mathrm{Si}, (0,0)}^{\cusp}$ at $x$ is an interesting problem to study. 
\end{Remark}

\begin{appendix}
    \section{A note on potentially semistable period}\label{section: pstperiod}

\subsection{Setup and the statement of the main result}\label{subsection: Setup, App A}
Suppose we are given the following data \[
    (\calX, \Sigma, \Det, \kappa, F),
\]
where $\calX$ is a reduced rigid analytic variety over $\Q_p$, $\Sigma \subset \calX(\overline{\Q}_p)$ is an accumulating set of points, $\Det: \scrO_{\calX}(\calX)[\Gal_{\Q_p}] \rightarrow \scrO_{\calX}(\calX)$ is a continuous determinant of dimension $d$ on $\calX$, $\kappa\in \scrO_{\calX}(\calX)$, and $F \in \scrO_{\calX}(\calX)^\times$ such that the following properties hold: \begin{itemize}
    \item for any $x\in \Sigma$, there exists a de Rham Galois representation $\rho_x: \Gal_{\Q_p} \rightarrow \GL_d(L_x)$ such that $\Det(x) = \det \rho_x$; 
    \item for any $x\in \Sigma$, $\kappa(x)\in \Z$ is the minimal Hodge--Tate weight of $\rho_x$;
    \item for any $x\in \Sigma$, the natural morphism $(\Fil_{\dR}^0 \D_{\pst}(\rho_x(\kappa(x))))^{\varphi = F(x)} \rightarrow \D_{\pst}(\rho_x(\kappa(x)))^{\varphi = F(x)}$ is an isomorphism and $\D_{\pst}(\rho_x)^{\varphi = p^{\kappa(x)F(x)}} $ has dimension $1$.\footnote{ Here, we fix an embedding $B_{\st} = B_{\cris}[X] \hookrightarrow B_{\dR}$ by $X \mapsto \log \frac{[p^{\flat}]}{p}$.} 
    \item For any $C\in \Z_{\geq 0}$, the set \[
        \scalemath{0.9}{ \Sigma_C \coloneq \left\{ x\in \Sigma: \text{Hodge--Tate weights of $\rho_x$ other than $\kappa(x)$ are bigger than } \kappa(x)+C \right\} }
    \]
    accumulates at any $x\in \Sigma$. 
\end{itemize}

We fix $x\in \calX(\overline{\Q}_p)$, not necessarily in $\Sigma$, and denote by $\bbT_x = \scrO_{\calX, x}$ the local ring at $x$ and $\frakm_x$ the corresponding maximal ideal. The Galois determinant $\Det$ induced a Galois determinant $\Det_x$ on $\bbT_x$ and we denote by $A = \bbT_x[\Gal_{\Q_p}]/\ker \Det_x$. We further assume the following: \begin{itemize}
    \item[(ACC)] The set $\Sigma$ accumulates at $x$. 
    \item[(MF)] The Galois determinant $\Det(x)$ is multiplicity free and so it corresponds to a Galois representation $\rho_x: \Gal_{\Q_p} \rightarrow \GL_d(L_x)$. 
    \item[(REG)] The representation $\rho_x$ is de Rham and $\D_{\pst}(\rho_x(\kappa(x)))^{\varphi = F(x)}$ has dimension $1$. 
\end{itemize}
From these assumptions, we see that $\rho_x = \bigoplus_{i=1}^r\rho_{x, r}$ such that $\rho_{x, i}$'s are absolutely irreducible Galois representations that are pairwise non-isomorphic. We denote by $d_i = \dim \rho_{x, i}$. Moreover, there exists a unique $j\in \{1, ..., r\}$ such that \[
    \D_{\pst}(\rho_{x, j}(\kappa(x)))^{\varphi = F(x)} \neq 0.
\]

On the other hand, by \cite[Theorem 1.4.4 (ii)]{BC}, there exists a representation \[
    \widetilde{\rho}: \bbT_x[\Gal_{\Q_p}] \rightarrow M_d(\mathrm{Frac} \bbT_x)
\]
whose determinant agrees with the induced determinant and $\ker\widetilde{\rho} = \ker \Det_x$. We fix a GMA structure on $A$ and let $M_j \subset (\mathrm{Frac}\bbT_x)^d$ be the `$j$-th column' $A$-submodule defined \cite[Sect. 1.5.7]{BC}. Suppose $\frakI \supset \frakI^{\red}$ is a cofinite ideal in $\bbT_x$. For any $i = 1, ..., r$, let \[
    \widetilde{\rho}_i: \Gal_{\Q_p} \rightarrow \GL_{d_i}(\bbT_x/\frakI)
\]
be the Galois representation defined in \cite[Definition 1.5.3]{BC}, which lifts $\rho_{x, i}$.

The goal of this section is to prove the following theorem, which is a generalisation of \cite[Theorem 4.3.2 \& Theorem 4.3.6]{BC}:

\begin{Theorem}\label{Theorem: pst period}
    Keep the notations as above. Let $\frakI \supset \frakI^{\red}$ be an ideal of $\bbT_x$ with cofinite length. \begin{enumerate}
        \item[(i)] The modules $\D_{\pst}(\widetilde{\rho}_j(\kappa))^{\varphi = F}$ and $\D_{\pst}(M_j/\frakI M_j(\kappa))^{\varphi = F}$ are free of rank $1$ over $(\bbT_x/\frakI) \otimes_{\Q_p}\Q_p^{\unr}$. 
        \item[(ii)] Let $i \in \{1, ..., r\}$ such that $i\neq j$ and let $\widetilde{\rho}_c$ be an extension of $\widetilde{\rho}_j$ by $\widetilde{\rho}_i$ in $\Ext_{A/\frakI}^1(\widetilde{\rho}_j, \widetilde{\rho}_i)$. Then, $\D_{\pst}(\widetilde{\rho}_c(\kappa))^{\varphi = F}$ is free of rank $1$. 
    \end{enumerate}
\end{Theorem}

\subsection{A generalisation of a result of Kisin and Bella{\"i}che--Chenevier}\label{subsection: generalisation of Kisina nd BC}

The purpose of this subsection is to prepare ourselves for the proof of Theorem \ref{Theorem: pst period}, which slightly generalises \cite[Theorem 3.3.3 \& Theorem 3.4.1]{BC}. We will work with different (but related) assumptions as in Sect. \ref{subsection: Setup, App A}. More precisely, suppose we are given the following data \[
    (\calX, \Sigma, \scrM, \kappa, F),
\]
where $(\calX, \Sigma, \kappa, F)$ are as in Sect. \ref{subsection: Setup, App A} and $\scrM$ is a coherent sheaf on $\calX$ that admits a continuous action by $\Gal_{\Q_p}$ as described in \cite[Sect. 3.2.1]{BC}. We further assume the following conditions hold: \begin{itemize}
    \item For any $x\in \Sigma$, the fibre $\scrM|_x $ is a de Rham representation and $\kappa(x)\in \Z$ is its minimal Hodge--Tate weight. 
    \item For any $x\in \Sigma$, the natural morphism $(\Fil_{\dR}^0 \D_{\pst}(\scrM|_x (\kappa(x))))^{\varphi = F(x)} \rightarrow \D_{\pst}(\scrM|_x(\kappa(x)))^{\varphi = F(x)}$  is an isomorphism and $\D_{\pst}(\scrM|_x(\kappa(x)))^{\varphi = F(x)}$ has dimension $1$. 
    \item For any $C\in \Z_{\geq 0}$, the set $\Sigma_C$ accumulates at any $x\in \Sigma$. 
\end{itemize}

The following theorem is a generalisation of \cite[Theorem 3.3.3]{BC}. 

\begin{Proposition}\label{Proposition: pst period for locally free sheaves}
    Suppose $\scrM$ is locally free. \begin{enumerate}
        \item[(i)] For all $x\in \calX(\overline{\Q}_p)$, $\D_{\pst}(\scrM|_x(\kappa(x)))^{\varphi = F(x)}$ is non-zero. 
        \item[(ii)] Let $x\in \calX(\overline{\Q}_p)$ such that $\D_{\pst}((\scrM|_x)^{\mathrm{ss}}(\kappa(x)))^{\varphi = F(x)}$ is $1$-dimensional. Then, for any ideal $\frakI \subset \scrO_{\calX, x}$ of cofinite length, $\D_{\pst}((\scrM_x/\frakI\scrM_x)(\kappa))^{\varphi = F}$ is free of rank $1$. 
    \end{enumerate}
\end{Proposition}
\begin{proof}
    By replacing $\scrM$ with $\scrM(\kappa)$, we may assume $\kappa = 0$. For any $z\in \calX(\overline{\Q}_p)$, choose an open affinoid $\calU = \Spa(R, R^{\circ}) \subset \calX$ containing $z$ such that $\scrM$ is free over $\calU$, $\Sigma \cap \calU$ is Zariski dense in $\calU$, and $\calU$ is $F$-small in the sense of \cite[Sect. 5.2]{Kisin-FM}. Note that, for any $C\in \Z_{\geq 0}$, $\Sigma_C \cap \calU$ is again Zariski dense in $\calU$. 

    For any $k\in \Z_{>0}$, we set $C_k = k+\sup_{y\in \calU(\overline{\Q}_p)}\{v_p(F(y))\} +1$. Then, by weak admissibility, for any $x\in \Sigma_{C_k}$, any map \[
        \scrM^{\vee}|_x \rightarrow (B_{\dR}^+/\textbf{\textit{t}}^k B_{\dR}^+)\otimes_{\Q_p} L_x
    \]
    factors through $(B_{\st}^+ \otimes_{\Q_p}L_x)^{\varphi = F(x)}$ under our assumption. Here, $\textbf{\textit{t}} = \log[\bfepsilon]\in B_{\dR}^+$. Then, similar as in the proof of \cite[Theorem 3.3.3]{BC}, we may apply \cite[Corollary 5.15 \& Corollary 5.16]{Kisin-FM} and obtain a non-trivial morphism \[
        \scrM^{\vee}(\calU)\rightarrow (B_{\st}^+ \widehat{\otimes}_{\Q_p}R)^{\varphi = F}.
    \] 
    This concludes the first assertion in this case. Note that Kisin worked with $B_{\cris}^+$, however the statements in Proposition 5.14, Corollary 5.15, and Corollary 5.16 in \cite{Kisin-FM} remain true by replacing $B_{\cris}^+$ with $B_{\st}^+$.\footnote{The key observation in Kisin's argument is that there exists $k\in \Z_{\geq 0}$ (large enough) such that the map $(B_{\cris}^+ \widehat{\otimes}_{\Q_p}R)^{\varphi = F} \rightarrow (B_{\dR}^+/\textbf{\textit{t}}^k B_{\dR}^+)\widehat{\otimes}_{\Q_p}R$ is injective and has a closed image (\cite[Corollary 3.7]{Kisin-FM}). This remains true by replacing $B_{\cris}^+$ with $B_{\st}^+$ since $(B_{\st}^+ \widehat{\otimes}_{\Q_p}R)^{\varphi = F}$ is isomorphic to a finite direct sum of $(B_{\cris}^+ \widehat{\otimes}_{\Q_p}R)^{\varphi = F/p^i}$ (for some $i$'s). For each $i$, there exists a large enough $k_i$ such that $(B_{\cris}^+ \widehat{\otimes}_{\Q_p}R)^{\varphi = F/p^i} \rightarrow (B_{\dR}^+/\textbf{\textit{t}}^{k_i} B_{\dR}^+)\widehat{\otimes}_{\Q_p}R $ is injective and has a closed image. We then take $k$ to be the largest $k_i$.  }

    To show (ii), we fix $x\in \calX(\overline{\Q}_p)$ such that $\D_{\pst}((\scrM|_x)^{\mathrm{ss}})^{\varphi = F(x)}$ is $1$-dimensional. Let $\calU = \Spa(R, R^{\circ})$ be an $F$-small affinoid open neighbourhood of $x$ such that $\scrM$ is free over $\calU$. Let $\frakJ \subset \scrO_{\calU}(\calU)$ be the smallest ideal such that \[
        \D_{\pst}(\scrM(\calU))^{\varphi = F} \subset \frakJ (\scrM(\calU) \widehat{\otimes}_{\Q_p} B_{\st}).
    \]
    Let $\calZ \subset \calU$ be the closed subspace defined by $\frakJ$. We then consider the blowup \[
        \pi: \calU' \rightarrow \calU
    \] 
    along the ideal $\frakJ$ and let $\scrM'$ be the pullback of $\scrM$. For any $x'\in \pi^{-1}(x)$, we claim that if $\frakI'$ is an ideal in $\scrO_{\calU', x'}$ with cofinite length, then $\D_{\pst}(\scrM'_{x'}/\frakI'\scrM'_{x'})^{\varphi = F}$ is free of rank $1$ over $(\scrO_{\calU', x'}/\frakI')\otimes_{\Q_p}\Q_p^{\unr}$.

    By (i), we see that $\D_{\pst}(\scrM|_x)^{\varphi = F}$ is $1$-dimensional. Hence $\D_{\pst}(\scrM'|_{x'})^{\varphi = F} = \D_{\pst}(\scrM|_x)^{\varphi = F} \otimes L_{x'}$ is of dimension $1$. On the other hand, the specialisation map \[
        \D_{\pst}(\scrM'_{x'}/\frakI'\scrM'_{x'})^{\varphi = F} \rightarrow \D_{\pst}(\scrM'|_{x'})^{\varphi = F}
    \]  
    is non-zero. Hence, by \cite[Lemma 3.3.9 (i)]{BC}, we conclude the claim. 

    To finish the proof, we can now apply \cite[Proposition 3.2.2]{BC} and get \[
        \mathrm{length}(\D_{\pst}(\scrM_x/\frakI \scrM_x)^{\varphi = F}) = \mathrm{length}(\scrO_{\calX, x}/\frakI).
    \]
    Then, \cite[Lemma 3.3.9 (ii)]{BC} implies the desired result. 
\end{proof}

The following corollary is a generalisation of \cite[Theorem 3.4.1]{BC}. 

\begin{Corollary}\label{Corollary: pst period for non-torsion sheaves}
    Suppose $\scrM$ is torsion-free (not necessarily locally free). Let $x\in \calX(\overline{\Q}_p)$ such that $\D_{\pst}((\scrM|_x)^{\mathrm{ss}}(\kappa(x)))^{\varphi = F(x)}$ is $1$-dimensional. Then, for any ideal $\frakI \subset \scrO_{\calX, x}$ of cofinite length, \[
        \mathrm{length}(\D_{\pst}((\scrM_x/\frakI\scrM_x)(\kappa))^{\varphi = F}) = \mathrm{length}(\scrO_{\calX, x}/\frakI).
    \] 
\end{Corollary}
\begin{proof}
    Providing Proposition \ref{Proposition: pst period for locally free sheaves}, the proof is exactly the same as in \cite[Theorem 3.4.1]{BC}. 
\end{proof}

\subsection{Proof of Theorem \ref{Theorem: pst period}}\label{subsection: proof of pst periods}

We keep the notations and assumptions as in Sect. \ref{subsection: Setup, App A}. We first prove the following key lemma: 

\begin{Lemma}\label{Lemma: key lemma for pst period}
    We have the following equality \[
        \mathrm{length}(\D_{\pst}(M_j/\frakI M_j(\kappa))^{\varphi = F}) = \mathrm{length}(\bbT_x/\frakI). 
    \]
\end{Lemma}
\begin{proof}
    The following proof is similar to \cite[Lemma 4.3.3]{BC}. 
    
    By \cite[Lemma 4.3.9]{BC}, there exists an $A$-submodule $N \subset (\mathrm{Frac}\bbT_x)^d$ such that $(N \oplus M_j) \mathrm{Frac}\bbT_x = (\mathrm{Frac}\bbT_x)^d$ and the semisimplification of $N/\frakm_x$ is isomorphic to a direct sum of copies of $\rho_{x, i}$ with $i \neq j$. Put $M = M_j \oplus N$. Then, by (REG), we see that \[
        \dim \D_{\pst}((M/\frakm_x)^{\mathrm{ss}}(\kappa(x)))^{\varphi = F} = 1. 
    \]
    Since $\D_{\pst}(M/\frakI M ) = \D_{\pst}(M_j/\frakI M_j) \oplus \D_{\pst}(N/\frakI N)$, we see that $\D_{\pst}(N/\frakI N(\kappa))^{\varphi = F} = 0$. 

    By \cite[Lemma 4.3.7]{BC}, there exists an affinoid neighbourhood $\calU \subset \calX$ of $x$ such that $\Sigma \cap \calU$ is Zariski dense and $M$ extends to a torsion-free coherent sheaf on $\calU$ with continuous $\Gal_{\Q_p}$-action. 

    Since $\scrM(\calU)$ is torsion-free of generic rank $d$ with determinant $\Det$, by the generic flatness theorem, there exists a proper Zariski closed subset $\calZ \subset \calU$ such that, for any $y\in (\calU \smallsetminus \calZ)\cap \Sigma$, $(\scrM|_y)^{\mathrm{ss}} = \rho_y$. Since $\scrM(\calU) \otimes_{\scrO_{\calU}(\calU)} \mathrm{Frac}\scrO_{\calU}(\calU)$ is semisimple by \cite[Lemma 4.3.7]{BC}, we may enlarge $\calZ$ so that for every $y\in (\calU \smallsetminus \calZ)\cap \Sigma$, $(\scrM|_y)^{\mathrm{ss}} = \scrM|_y$. We replace $\Sigma$ by $(\calU \cap \Sigma) \smallsetminus (\calU \cap \Sigma \cap \calZ)$, then the property of $\Sigma_C$ still holds by (ACC). Then, by the conditions in Sect. \ref{subsection: Setup, App A}, we see that $(\calU, \Sigma \cap \calU, \scrM, \kappa, F)$ satisfies the conditions in Sect. \ref{subsection: generalisation of Kisina nd BC}. 

    We conclude the proof by applying Corollary \ref{Corollary: pst period for non-torsion sheaves}.  
\end{proof}

\begin{Corollary}\label{Corollary: generalisation of BC Lemma 4.3.10}
    Instead of (MF), we assume \begin{itemize}
        \item[(FM)] There exists a free $\bbT_x$-module $M$ of rank $d$ with a $\bbT_x$-linear action of $\Gal_{\Q_p}$ whose corresponding determinant agrees with $\Det$, and such that $M\otimes \mathrm{Frac}\bbT_x$ is a semisimple $\mathrm{Frac}\bbT_x[\Gal_{\Q_p}]$-module.  
    \end{itemize}
    Then, for any ideal of cofinite length $\frakI \subset\bbT_x$, $\D_{\pst}(M/\frakI M (\kappa))^{\varphi = F}$ is free of rank $1$ over $(\bbT_x/\frakI)\otimes_{\Q_p}\Q_p^{\unr}$. 
\end{Corollary}
\begin{proof}
    Providing Lemma \ref{Lemma: key lemma for pst period}, the proof is exactly the same as \cite[Lemma 4.3.10]{BC}.
\end{proof}

\begin{proof}[Proof of Theorem \ref{Theorem: pst period}]
    By \cite[Theorem 1.5.6 (0)]{BC}, we have a short exact sequence \[
        0 \rightarrow N \rightarrow M_j/\frakI M_j \rightarrow \widetilde{\rho}_j \rightarrow 0,
    \]
    where all simple subquotients of $N$ are isomorphic to $\rho_{x, i}$ for some $i \neq j$. By (REG), we see that $\D_{\pst}(N(\kappa))^{\varphi = F} = 0$ and so \[
        \D_{\pst}(M_j/\frakI M_j(\kappa))^{\varphi = F} \hookrightarrow \D_{\pst}(\widetilde{\rho}_j(\kappa))^{\varphi = F}.
    \]
    Therefore, Lemma \ref{Lemma: key lemma for pst period} implies $\mathrm{length}(\D_{\pst}(\widetilde{\rho}_j(\kappa))^{\varphi = F} )\geq \mathrm{length}(\bbT_x/\frakI)$. On the other hand, by \cite[Lemma 3.2.9 (i)]{BC}, we have $\mathrm{length}(\D_{\pst}(\widetilde{\rho}_j(\kappa))^{\varphi = F} )= \mathrm{length}(\bbT_x/\frakI)$. Hence, we have an isomorphism \[
        \D_{\pst}(M_j/\frakI M_j(\kappa))^{\varphi = F} \cong \D_{\pst}(\widetilde{\rho}_j(\kappa))^{\varphi = F}.
    \]
    Then, by \cite[Lemma 3.3.9]{BC}, we see that $\D_{\pst}(\widetilde{\rho}_j(\kappa))^{\varphi = F}$ is free of rank $1$. This concludes (i). 

    The proof for (ii) is exactly the same except that, instead of applying \cite[Theorem 1.5.6 (0)]{BC} in the beginning, we implement \cite[Theorem 1.5.6 (2)]{BC}. 
\end{proof}

\end{appendix}

\printbibliography[heading=bibintoc]

@article {Allen-dukepaper,
    AUTHOR = {Allen, Patrick B.},
     TITLE = {Deformations of polarized automorphic {G}alois representations
              and adjoint {S}elmer groups},
   JOURNAL = {Duke Math. J.},
  FJOURNAL = {Duke Mathematical Journal},
    VOLUME = {165},
      YEAR = {2016},
    NUMBER = {13},
     PAGES = {2407--2460},
      %ISSN = {0012-7094,1547-7398},
   MRCLASS = {11F80 (11F70 11R34)},
  MRNUMBER = {3546966},
MRREVIEWER = {Ivan\ Mati\'c},
       %DOI = {10.1215/00127094-3477342},
       %URL = {https://doi.org/10.1215/00127094-3477342},
}

@article{AIP-2015,
    %ISSN = {0003486X},
    author = {Fabrizio Andreatta and Adrian Iovita and Vincent Pilloni},
    journal = {Annals of Mathematics},
    number = {2},
    pages = {623--697},
    publisher = {Annals of Mathematics},
    title = {$p$-adic families of Siegel modular cuspforms},
    volume = {181},
    year = {2015}
}

@article {AIS-2014,
    AUTHOR = {Andreatta, Fabrizio and Iovita, Adrian and Stevens, Glenn},
     TITLE = {Overconvergent modular sheaves and modular forms for {${\bf
              GL}_{2/F}$}},
   JOURNAL = {Israel J. Math.},
  FJOURNAL = {Israel Journal of Mathematics},
    VOLUME = {201},
      YEAR = {2014},
    NUMBER = {1},
     PAGES = {299--359},
   MRCLASS = {11F41 (11F85 14F30)},
  MRNUMBER = {3265287},
MRREVIEWER = {G.\ K.\ Sankaran},
}

@article {Andrianov-SKConjecture,
    AUTHOR = {Andrianov, Anatolii N.},
     TITLE = {Modular descent and the {S}aito--{K}urokawa conjecture},
   JOURNAL = {Invent. Math.},
  FJOURNAL = {Inventiones Mathematicae},
    VOLUME = {53},
      YEAR = {1979},
    NUMBER = {3},
     PAGES = {267--280},
   MRCLASS = {10D20 (10D24)},
  MRNUMBER = {549402},
MRREVIEWER = {Hiroshi\ Saito},
}

@article {AS-GSp4toGL4,
    AUTHOR = {Asgari, Mahdi and Shahidi, Freydoon},
     TITLE = {Generic transfer from {$\rm GSp(4)$} to {$\rm GL(4)$}},
   JOURNAL = {Compos. Math.},
  FJOURNAL = {Compositio Mathematica},
    VOLUME = {142},
      YEAR = {2006},
    NUMBER = {3},
     PAGES = {541--550},
   MRCLASS = {11F70 (11R42 22E50 22E55)},
  MRNUMBER = {2231191},
MRREVIEWER = {Wee\ Teck\ Gan},
}

@misc{Ash-Stevens,
    author = {Avner Ash and Glenn Stevens},
    title = {$p$-adic deformations of arithmetic cohomology},
    howpublished = {Preprint. Available at \url{http://math.bu.edu/people/ghs/preprints/Ash-Stevens-02-08.pdf}},
    year = {2008}
}

@article {BSDW26,
    AUTHOR = {Barrera Salazar, Daniel and Dimitrov, Mladen and Williams,
              Chris},
     TITLE = {On {$p$}-adic {$L$}-functions for {${\rm GL}_{2n}$} in finite
              slope {S}halika families},
   JOURNAL = {Adv. Math.},
  FJOURNAL = {Advances in Mathematics},
    VOLUME = {486},
      YEAR = {2026},
     PAGES = {Paper No. 110741, 104},
   MRCLASS = {11F33 (11F67 11R23)},
  MRNUMBER = {5006161},
  SHORTHAND = {BSDW26},
}

@misc{Brasca--Rosso,
      title={Eigenvarieties for non-cuspidal modular forms over certain PEL Shimura varieties}, 
      author={Riccardo Brasca and Giovanni Rosso},
      year={2018},
      howpublished = {Preprint. Available at: \url{url{https://arxiv.org/abs/1605.05065}}},
}

@article {Breuil--Ding-BernsteinEigen,
    AUTHOR = {Breuil, Christophe and Ding, Yiwen},
     TITLE = {Bernstein eigenvarieties},
   JOURNAL = {Peking Math. J.},
  FJOURNAL = {Peking Mathematical Journal},
    VOLUME = {7},
      YEAR = {2024},
    NUMBER = {2},
     PAGES = {471--642},
   MRCLASS = {11S37 (11F80 11S23 22E35 22E50)},
  MRNUMBER = {4792986},    
}

@misc{Bellaiche-BK,
    author = {Joël Bellaïche}, 
    title = {An introduction to the conjecture of Bloch and Kato},
    year = {2009},
    howpublished = {Lecture Notes at  the Clay Mathematical Institute Summer School. Available at: \url{https://virtualmath1.stanford.edu/~conrad/BSDseminar/refs/BKintro.pdf}}
}

@book{BC,
     author = {Joël Bellaïche and Gaëtan Chenevier},
     title = {Families of Galois representations and Selmer groups},
     series = {Ast{\'e}risque},
     publisher = {Soci{\'e}t{\'e} math{\'e}matique de France},
     number = {324},
     year = {2009},
     mrnumber = {2656025},
     zbl = {1192.11035},
     %language = {en},
     %url = {http://www.numdam.org/item/AST_2009__324__R1_0/}
}

@article{Bellovin-GDeform,
     author = {Bellovin, Rebecca},
     title = {Generic smoothness for $G$-valued potentially semi-stable deformation rings},
     journal = {Annales de l'Institut Fourier},
     pages = {2565--2620},
     publisher = {Association des Annales de l{\textquoteright}institut Fourier},
     volume = {66},
     number = {6},
     year = {2016},
     %doi = {10.5802/aif.3072},
     language = {en},
     %url = {http://www.numdam.org/articles/10.5802/aif.3072/}
}

@article {Bergdall-paraboline,
    AUTHOR = {Bergdall, John},
     TITLE = {Paraboline variation over {$p$}-adic families of
              {$(\phi,\Gamma)$}-modules},
   JOURNAL = {Compos. Math.},
  FJOURNAL = {Compositio Mathematica},
    VOLUME = {153},
      YEAR = {2017},
    NUMBER = {1},
     PAGES = {132--174},
      %ISSN = {0010-437X,1570-5846},
   MRCLASS = {11F80 (11F33 11F55 11F85 14F30)},
  MRNUMBER = {3622874},
MRREVIEWER = {Ivan\ Mati\'c},
       %DOI = {10.1112/S0010437X16007831},
       %URL = {https://doi.org/10.1112/S0010437X16007831},
}

@article {Berger-pDifferential,
    AUTHOR = {Berger, Laurent},
     TITLE = {Repr\'{e}sentations {$p$}-adiques et \'{e}quations
              diff\'{e}rentielles},
   JOURNAL = {Invent. Math.},
  FJOURNAL = {Inventiones Mathematicae},
    VOLUME = {148},
      YEAR = {2002},
    NUMBER = {2},
     PAGES = {219--284},
      %ISSN = {0020-9910,1432-1297},
   MRCLASS = {14F30 (11S20 12H25 14F40 14G20)},
  MRNUMBER = {1906150},
MRREVIEWER = {Adolfo\ Quir\'{o}s},
       %DOI = {10.1007/s002220100202},
       %URL = {https://0-doi-org.pugwash.lib.warwick.ac.uk/10.1007/s002220100202},
}

@incollection {Berger-phiN,
    AUTHOR = {Berger, Laurent},
     TITLE = {\'Equations diff\'erentielles {$p$}-adiques et
              {$(\phi,N)$}-modules filtr\'es},
      NOTE = {Repr\'esentations $p$-adiques de groupes $p$-adiques. I.
              Repr\'esentations galoisiennes et $(\phi,\Gamma)$-modules},
   JOURNAL = {Ast\'erisque},
  FJOURNAL = {Ast\'erisque},
    NUMBER = {319},
      YEAR = {2008},
     PAGES = {13--38},
      %ISSN = {0303-1179,2492-5926},
      %ISBN = {978-2-85629-256-3},
   MRCLASS = {11F80 (12H25 14F30)},
  MRNUMBER = {2493215},
MRREVIEWER = {Michael\ M.\ Schein},
}

@incollection {Berger-Intro,
    AUTHOR = {Berger, Laurent},
     TITLE = {An introduction to the theory of {$p$}-adic representations},
 BOOKTITLE = {Geometric aspects of {D}work theory. {V}ol. {I}, {II}},
     PAGES = {255--292},
 PUBLISHER = {Walter de Gruyter, Berlin},
      YEAR = {2004},
      %ISBN = {3-11-017478-2},
   MRCLASS = {11S25 (11F80 12H25 13K05 14F30)},
  MRNUMBER = {2023292},
MRREVIEWER = {Mark\ Kisin},
}

@article {BB22,
    AUTHOR = {Berger, Tobias and Betina, Adel},
     TITLE = {On {S}iegel eigenvarieties at {S}aito--{K}urokawa points},
   JOURNAL = {Ann. Inst. Fourier (Grenoble)},
  FJOURNAL = {Universit{\'e} de Grenoble. Annales de l'Institut Fourier},
    VOLUME = {72},
      YEAR = {2022},
    NUMBER = {3},
     PAGES = {901--961},
      %ISSN = {0373-0956,1777-5310},
   MRCLASS = {11F33 (11F46 11F80 11F85 14G35)},
  MRNUMBER = {4485817},
MRREVIEWER = {Christopher\ Birkbeck},
       %DOI = {10.5802/aif.3482},
       %URL = {https://doi.org/10.5802/aif.3482},
}

@article {Benois-LInvariant,
    AUTHOR = {Benois, Denis},
     TITLE = {A generalization of {G}reenberg's {$\mathcal{L}$}-invariant},
   JOURNAL = {Amer. J. Math.},
  FJOURNAL = {American Journal of Mathematics},
    VOLUME = {133},
      YEAR = {2011},
    NUMBER = {6},
     PAGES = {1573--1632},
      %ISSN = {0002-9327,1080-6377},
   MRCLASS = {11G40 (11F67)},
  MRNUMBER = {2863371},
MRREVIEWER = {Sarah\ Livia\ Zerbes},
}

@article{BLGGT-localglobal2,
     author = {Barnet-Lamb, Thomas and Gee, Toby and Geraghty, David and Taylor, Richard},  
     shorthand = {BLGGT14},
     title = {Local-global compatibility for~$l=p$, {II}},
     journal = {Annales scientifiques de l'\'Ecole Normale Sup\'erieure},
     pages = {165--179},
     publisher = {Soci\'et\'e Math\'ematique de France. Tous droits r\'eserv\'es},
     volume = {Ser. 4, 47},
     number = {1},
     year = {2014},
     %doi = {10.24033/asens.2212},
     mrnumber = {3205603},
     zbl = {1395.11081},
     %url = {https://www.numdam.org/articles/10.24033/asens.2212/}
}

@misc{BP-HigherColeman,
    author = {George Boxer and Vincent Pilloni}, 
    title = {Higher Coleman Theory},
    year = {2020}, 
    howpublished = {Preprint. Available at: \url{https://perso.ens-lyon.fr/vincent.pilloni/HigherColeman.pdf}}
}

@inbook{Buzzard_2007, 
    place={Cambridge}, 
    series={London Mathematical Society Lecture Note Series}, 
    title={Eigenvarieties}, 
    %DOI={10.1017/CBO9780511721267.004}, 
    booktitle={$L$-Functions and Galois Representations}, 
    publisher={Cambridge University Press}, 
    author={Buzzard, Kevin}, 
    editor={Burns, David and Buzzard, Kevin and Nekovář, Jan}, 
    year={2007}, 
    pages={59–120}, 
    collection={London Mathematical Society Lecture Note Series}
}

@misc{Buzzard-HTNotes,
    author = {Kevin Buzzard}, 
    title = {Hodge--Tate theory} ,
    howpublished = {Notes. Available at: \url{https://www.ma.imperial.ac.uk/~buzzard/maths/research/notes/hodge_tate_theory.pdf}}, 
    year = {2012}
}

@article {BSW-ParabolicEigen,
    AUTHOR = {Barrera Salazar, Daniel and Williams, Chris},
 SHORTHAND = {BSW21},
     TITLE = {Parabolic eigenvarieties via overconvergent cohomology},
   JOURNAL = {Math. Z.},
  FJOURNAL = {Mathematische Zeitschrift},
    VOLUME = {299},
      YEAR = {2021},
    NUMBER = {1-2},
     PAGES = {961--995},
      %ISSN = {0025-5874,1432-1823},
   MRCLASS = {11F85 (11F75)},
  MRNUMBER = {4311626},
MRREVIEWER = {Manouchehr\ Misaghian},
       %DOI = {10.1007/s00209-021-02707-9},
       %URL = {https://0-doi-org.pugwash.lib.warwick.ac.uk/10.1007/s00209-021-02707-9},
}

@incollection {Chenevier-determinant,
    AUTHOR = {Chenevier, Ga\"etan},
     TITLE = {The {$p$}-adic analytic space of pseudocharacters of a
              profinite group and pseudorepresentations over arbitrary
              rings},
 BOOKTITLE = {Automorphic forms and {G}alois representations. {V}ol. 1},
    SERIES = {London Math. Soc. Lecture Note Ser.},
    VOLUME = {414},
     PAGES = {221--285},
 PUBLISHER = {Cambridge Univ. Press, Cambridge},
      YEAR = {2014},
      %ISBN = {978-1-107-69192-6},
   MRCLASS = {11F70 (11F80 14G22 20E18)},
  MRNUMBER = {3444227},
MRREVIEWER = {Cameron\ Franc},
}

@article {Coleman-pAdicFamilies,
    AUTHOR = {Coleman, Robert F.},
     TITLE = {{$p$}-adic {B}anach spaces and families of modular forms},
   JOURNAL = {Invent. Math.},
  FJOURNAL = {Inventiones Mathematicae},
    VOLUME = {127},
      YEAR = {1997},
    NUMBER = {3},
     PAGES = {417--479},
      %ISSN = {0020-9910},
   MRCLASS = {11F33 (11F67 11F85 11G07 14E30)},
  MRNUMBER = {1431135},
MRREVIEWER = {Alexey A. Panchishkin},
       %DOI = {10.1007/s002220050127},
       %URL = {https://0-doi-org.pugwash.lib.warwick.ac.uk/10.1007/s002220050127},
}

@inbook{Coleman_Mazur, 
    place={Cambridge}, 
    series={London Mathematical Society Lecture Note Series}, 
    title={The Eigencurve}, 
    %DOI={10.1017/CBO9780511662010.003}, 
    booktitle={Galois Representations in Arithmetic Algebraic Geometry}, 
    publisher={Cambridge University Press}, 
    author={Coleman, Robert F. and Mazur, Barry C.}, 
    editor={Scholl, A. J. and Taylor, R. L.Editors}, 
    year={1998}, 
    pages={1–114}, 
    collection={London Mathematical Society Lecture Note Series}
}

@incollection {Coleman--Stein,
    AUTHOR = {Coleman, Robert F. and Stein, William A.},
     TITLE = {Approximation of eigenforms of infinite slope by eigenforms of
              finite slope},
 BOOKTITLE = {Geometric aspects of {D}work theory. {V}ol. {I}, {II}},
     PAGES = {437--449},
 PUBLISHER = {Walter de Gruyter, Berlin},
      YEAR = {2004},
   MRCLASS = {11F33 (11F11)},
  MRNUMBER = {2023296},
MRREVIEWER = {Mark\ Kisin},
}

@incollection {Colmez-trianguline,
    AUTHOR = {Colmez, Pierre},
     TITLE = {Repr\'esentations triangulines de dimension 2},
      NOTE = {Repr\'esentations $p$-adiques de groupes $p$-adiques. I. Repr\'esentations galoisiennes et $(\phi,\Gamma)$-modules},
   JOURNAL = {Ast\'erisque},
  FJOURNAL = {Ast\'erisque},
    NUMBER = {319},
      YEAR = {2008},
     PAGES = {213--258},
   MRCLASS = {11S37 (11S23)},
  MRNUMBER = {2493219},
MRREVIEWER = {Jan\ Nekov\'a\v r},
}

@article {CDP-plocallanglandsGL2,
    AUTHOR = {Colmez, Pierre and Dospinescu, Gabriel and Pa\v sk\=unas,
              Vytautas},
     TITLE = {The {$p$}-adic local {L}anglands correspondence for {${\rm
              GL}_2(\mathbb{Q}_p)$}},
   JOURNAL = {Camb. J. Math.},
  FJOURNAL = {Cambridge Journal of Mathematics},
    VOLUME = {2},
      YEAR = {2014},
    NUMBER = {1},
     PAGES = {1--47},
   MRCLASS = {11S37 (11F80 22E50)},
  MRNUMBER = {3272011},
MRREVIEWER = {Wen-Wei\ Li},
}

@article {Dimitrov--Hsu,
    AUTHOR = {Dimitrov, Mladen and Hsu, Chi-Yun},
     TITLE = {Eigenvariety for partially classical {H}ilbert modular forms},
   JOURNAL = {Tunis. J. Math.},
  FJOURNAL = {Tunisian Journal of Mathematics},
    VOLUME = {7},
      YEAR = {2025},
    NUMBER = {3-4},
      ISSN = {2576-7658,2576-7666},
   MRCLASS = {11F33},
  MRNUMBER = {4959870},
}

@inproceedings{Deligne-Milne,
    title = {Tannakian Categories}, 
    author = {Deligne, Pierre and Milne, James S},
    booktitle = {Hodge cycles, motives, and Shimura varieties},
    series = {Lecture Notes in Mathematics}, 
    number = {900}, 
    publisher = {Springer-Verlag, Berlin-New York}, 
    year = {1982}
}

@article{DR,
  title={Diagonal cycles and Euler systems I: A $p$-adic Gross--Zagier formula},
  author={Henri Darmon and Victor Rotger},
  journal={Annales Scientifiques de l'École Normale Supérieure},
  volume={47},
  number={4},
  pages={779-832},
  year={2014}
}

@article {Emerton,
    AUTHOR = {Emerton, Matthew},
     TITLE = {On the interpolation of systems of eigenvalues attached to
              automorphic {H}ecke eigenforms},
   JOURNAL = {Invent. Math.},
  FJOURNAL = {Inventiones Mathematicae},
    VOLUME = {164},
      YEAR = {2006},
    NUMBER = {1},
     PAGES = {1--84},
   MRCLASS = {22E55 (11F70 11F75 11F85)},
  MRNUMBER = {2207783},
MRREVIEWER = {Payman\ L.\ Kassaei},
       
}

@book{Faltings-Chai,
    author = {Faltings, Gerd and Chai, Ching-Li}, 
    title = {Degeneration of Abelian Varieties},
    series = {Ergebnisse der Mathematik und ihrer Grenzgebiete},
    %DOI = {10.1007/978-3-662-02632-8},
    publisher = {Springer-Verlag Berlin Heidelberg}, 
    year = {1990}
}

@incollection {Fontaine-pRep1,
    AUTHOR = {Fontaine, Jean-Marc},
     TITLE = {Repr\'{e}sentations {$p$}-adiques des corps locaux. {I}},
 BOOKTITLE = {The {G}rothendieck {F}estschrift, {V}ol. {II}},
    SERIES = {Progr. Math.},
    VOLUME = {87},
     PAGES = {249--309},
 PUBLISHER = {Birkh\"{a}user Boston, Boston, MA},
      YEAR = {1990},
      %ISBN = {0-8176-3428-2},
   MRCLASS = {11S23 (14F30 14L05)},
  MRNUMBER = {1106901},
MRREVIEWER = {Rutger\ Noot},
}

@article {GT-LLGSp4,
    AUTHOR = {Gan, Wee Teck and Takeda, Shuichiro},
     TITLE = {The local {L}anglands conjecture for {${\rm GSp}(4)$}},
   JOURNAL = {Ann. of Math. (2)},
  FJOURNAL = {Annals of Mathematics. Second Series},
    VOLUME = {173},
      YEAR = {2011},
    NUMBER = {3},
     PAGES = {1841--1882},
      %ISSN = {0003-486X,1939-8980},
   MRCLASS = {22E50 (20G25 22E35)},
  MRNUMBER = {2800725},
MRREVIEWER = {B.\ Sury},
       %DOI = {10.4007/annals.2011.173.3.12},
       %URL = {https://doi.org/10.4007/annals.2011.173.3.12},
}

@incollection {Gritsenko,
    AUTHOR = {Gritsenko, Valeri},
     TITLE = {Arithmetical lifting and its applications},
 BOOKTITLE = {Number theory ({P}aris, 1992--1993)},
    SERIES = {London Math. Soc. Lecture Note Ser.},
    VOLUME = {215},
     PAGES = {103--126},
 PUBLISHER = {Cambridge Univ. Press, Cambridge},
      YEAR = {1995},
   MRCLASS = {11F46 (11F55 11F66)},
  MRNUMBER = {1345176},
MRREVIEWER = {Aloys\ Krieg},
}

@book {SGA1,
    AUTHOR = {Grothendieck, Alexander},
     TITLE = {Rev\^etements \'etales et groupe fondamental. {F}asc. {I}:
              {E}xpos\'es 1 \`a{} 5},
      NOTE = {Troisi\`eme \'edition, corrig\'ee,
              S\'eminaire de G\'eom\'etrie Alg\'ebrique, 1960/61},
 PUBLISHER = {Institut des Hautes \'Etudes Scientifiques, Paris},
      YEAR = {1963},
     PAGES = {pp. iv+143},
    SHORTHAND = {SGA1},
   %MRCLASS = {14.55},
  %MRNUMBER = {217087},
}

@incollection{GT-TWGSp4,
     author = {Genestier, Alain and Tilouine, Jacques},
     title = {Syst\`emes de {Taylor-Wiles} pour $GSp_4$},
     booktitle = {Formes automorphes (II) - Le cas du groupe $GSp(4)$},
     editor = {Tilouine, Jacques and Carayol, Henri and Harris, Michael and Vign\'eras, Marie-France},
     series = {Ast\'erisque},
     pages = {177--290},
     publisher = {Soci\'et\'e math\'ematique de France},
     number = {302},
     year = {2005},
     mrnumber = {2234862},
     zbl = {1142.11036},
     %url = {http://www.numdam.org/item/AST_2005__302__177_0/}
}

@article {Harris-partial,
    AUTHOR = {Harris, Michael},
     TITLE = {Automorphic forms of {$\overline\partial$}-cohomology type as
              coherent cohomology classes},
   JOURNAL = {J. Differential Geom.},
  FJOURNAL = {Journal of Differential Geometry},
    VOLUME = {32},
      YEAR = {1990},
    NUMBER = {1},
     PAGES = {1--63},
      %ISSN = {0022-040X,1945-743X},
   MRCLASS = {11G18 (17B56 22E46 32L10 32N10)},
  MRNUMBER = {1064864},

}

@article{Hansen-PhD, 
    author = {David Hansen},
    title = {Universal eigenvarieties, trianguline Galois representations, and $p$-adic Langlands functoriality},
    journal = {Journal für die reine und angewandte Mathematik}, 
    year = {2017},
    %doi = {https://doi.org/10.1515/crelle-2014-0130},
    issue = {730},
    pages = {1--64}
}

@article {HT-GLn,
    AUTHOR = {Hansen, David and Thorne, Jack A.},
     TITLE = {On the {$\mathrm{GL}_n$}-eigenvariety and a conjecture of {V}enkatesh},
   JOURNAL = {Selecta Math. (N.S.)},
  FJOURNAL = {Selecta Mathematica. New Series},
    VOLUME = {23},
      YEAR = {2017},
    NUMBER = {2},
     PAGES = {1205--1234},
      %ISSN = {1022-1824,1420-9020},
   MRCLASS = {11F75 (11F80 11F85 14F30)},
  MRNUMBER = {3624909},
MRREVIEWER = {Anton\ Deitmar},
       %DOI = {10.1007/s00029-017-0303-0},
       %URL = {https://doi.org/10.1007/s00029-017-0303-0},
}

@article {Hida-IwasawaCongruence,
    AUTHOR = {Hida, Haruzo},
     TITLE = {Iwasawa modules attached to congruences of cusp forms},
   JOURNAL = {Ann. Sci. \'Ecole Norm. Sup. (4)},
  FJOURNAL = {Annales Scientifiques de l'\'Ecole Normale Sup\'erieure.
              Quatri\`eme S\'erie},
    VOLUME = {19},
      YEAR = {1986},
    NUMBER = {2},
     PAGES = {231--273},
   MRCLASS = {11F33 (11G18 11R23)},
  MRNUMBER = {868300},
MRREVIEWER = {Yasutaka\ Ihara},
}

@article {Hida-bigGalois,
    AUTHOR = {Hida, Haruzo},
     TITLE = {Galois representations into {${\rm GL}_2({\bf Z}_p[[X]])$}
              attached to ordinary cusp forms},
   JOURNAL = {Invent. Math.},
  FJOURNAL = {Inventiones Mathematicae},
    VOLUME = {85},
      YEAR = {1986},
    NUMBER = {3},
     PAGES = {545--613},
      %ISSN = {0020-9910},
   MRCLASS = {11F11 (11F33 11F85 11R23)},
  MRNUMBER = {848685},
MRREVIEWER = {Ernst-Wilhelm Zink},
       %DOI = {10.1007/BF01390329},
       %URL = {https://0-doi-org.pugwash.lib.warwick.ac.uk/10.1007/BF01390329},
}

@article {Hill-Loeffler,
    AUTHOR = {Hill, Richard and Loeffler, David},
     TITLE = {Emerton's {J}acquet functors for non-{B}orel parabolic
              subgroups},
   JOURNAL = {Doc. Math.},
  FJOURNAL = {Documenta Mathematica},
    VOLUME = {16},
      YEAR = {2011},
     PAGES = {1--31},
   MRCLASS = {11F75 (11F70 22E50)},
  MRNUMBER = {2804506},
MRREVIEWER = {Neven\ Grbac},
}

@article {Huber-1994,
    AUTHOR = {Huber, Roland},
     TITLE = {A generalization of formal schemes and rigid analytic
              varieties},
   JOURNAL = {Math. Z.},
  FJOURNAL = {Mathematische Zeitschrift},
    VOLUME = {217},
      YEAR = {1994},
    NUMBER = {4},
     PAGES = {513--551},
      %ISSN = {0025-5874,1432-1823},
   MRCLASS = {14A20 (13A18 14L15 32P05)},
  MRNUMBER = {1306024},
MRREVIEWER = {W.\ Bartenwerfer},
       %DOI = {10.1007/BF02571959},
       %URL = {https://doi.org/10.1007/BF02571959},
}

@incollection{Illusie,
    author = {Illusie, Luc},
    title = {An overview of the work of K. Fujiwara, K. Kato, and C. Nakayama on logarithmic \'etale cohomology},
    booktitle = {Cohomologies $p$-adiques et applications arithm\'etiques (II)},
    editor = {Berthelot, Pierre and Fontaine, Jean-Marc and Illusie, Luc and Kato, Kazuya and Rapoport, Michael},
    series = {Ast\'erisque},
    publisher = {Soci\'et\'e math\'ematique de France},
    number = {279},
    year = {2002},
    pages = {271-322},
}

@article{Johansson-Newton,
    title = "Extended eigenvarieties for overconvergent cohomology",
    author = "Christian Johansson and James Newton",
    year = "2019",
    month = "2",
    day = "13",
    %doi = "10.2140/ant.2019.13.93",
    %language = "English",
    volume = "13",
    pages = "93--158",
    journal = "Algebra and Number Theory",
    %issn = "1937-0652",
    publisher = "Mathematical Sciences Publishers",
    number = "1",
}

@article {Johansson-Newton-Irreducible,
    AUTHOR = {Johansson, Christian and Newton, James},
     TITLE = {Irreducible components of extended eigenvarieties and
              interpolating {L}anglands functoriality},
   JOURNAL = {Math. Res. Lett.},
  FJOURNAL = {Mathematical Research Letters},
    VOLUME = {26},
      YEAR = {2019},
    NUMBER = {1},
     PAGES = {159--201},
      %ISSN = {1073-2780,1945-001X},
   MRCLASS = {14D24 (11F33)},
  MRNUMBER = {3963980},
MRREVIEWER = {Atsushi\ Yamagami},
       %DOI = {10.4310/MRL.2019.v26.n1.a9},
       %URL = {https://doi.org/10.4310/MRL.2019.v26.n1.a9},
}

@article {Jorza-GSp,
    AUTHOR = {Jorza, Andrei},
     TITLE = {{$p$}-adic families and {G}alois representations for {$\rm
              GS_p(4)$} and {$\rm GL(2)$}},
   JOURNAL = {Math. Res. Lett.},
  FJOURNAL = {Mathematical Research Letters},
    VOLUME = {19},
      YEAR = {2012},
    NUMBER = {5},
     PAGES = {987--996},
   MRCLASS = {11F70 (11F46 11F80)},
  MRNUMBER = {3039824},
MRREVIEWER = {Anton\ Deitmar},
}

@article {Khare-SerreConj,
    AUTHOR = {Khare, Chandrashekhar},
     TITLE = {Serre's modularity conjecture: the level one case},
   JOURNAL = {Duke Math. J.},
  FJOURNAL = {Duke Mathematical Journal},
    VOLUME = {134},
      YEAR = {2006},
    NUMBER = {3},
     PAGES = {557--589},
   MRCLASS = {11F80 (11F11 11R39)},
  MRNUMBER = {2254626},
MRREVIEWER = {Thomas\ A.\ Weston},
}

@incollection {Kato-2004,
    AUTHOR = {Kato, Kazuya},
     TITLE = {{$p$}-adic {H}odge theory and values of zeta functions of
              modular forms},
      NOTE = {Cohomologies $p$-adiques et applications arithm\'etiques. III},
   JOURNAL = {Ast\'erisque},
  FJOURNAL = {Ast\'erisque},
    NUMBER = {295},
      YEAR = {2004},
     PAGES = {ix, 117--290},
   MRCLASS = {11F85 (11F67 11G40 11R33 11S80 14G10 14G35)},
  MRNUMBER = {2104361},
MRREVIEWER = {Fabrizio\ Andreatta},
}

@article {KPX,
    AUTHOR = {Kedlaya, Kiran S. and Pottharst, Jonathan and Xiao, Liang},
     TITLE = {Cohomology of arithmetic families of
              {$(\varphi,\Gamma)$}-modules},
   JOURNAL = {J. Amer. Math. Soc.},
  FJOURNAL = {Journal of the American Mathematical Society},
    VOLUME = {27},
      YEAR = {2014},
    NUMBER = {4},
     PAGES = {1043--1115},
      %ISSN = {0894-0347,1088-6834},
   MRCLASS = {11F33 (11R23 11S25 11S31 13D09)},
  MRNUMBER = {3230818},
MRREVIEWER = {Th\cfac ong\ Nguy\cftil en-Quang-\Dbar\cftil o},
       %DOI = {10.1090/S0894-0347-2014-00794-3},
       %URL = {https://doi.org/10.1090/S0894-0347-2014-00794-3},
}

@article {Kisin-FM,
    AUTHOR = {Kisin, Mark},
     TITLE = {The {F}ontaine-{M}azur conjecture for {${\rm GL}_2$}},
   JOURNAL = {J. Amer. Math. Soc.},
  FJOURNAL = {Journal of the American Mathematical Society},
    VOLUME = {22},
      YEAR = {2009},
    NUMBER = {3},
     PAGES = {641--690},
   MRCLASS = {11F80 (11F33)},
  MRNUMBER = {2505297},
MRREVIEWER = {David\ L.\ Savitt},
}

@article {Kurokawa-SiegelExample,
    AUTHOR = {Kurokawa, Nobushige},
     TITLE = {Examples of eigenvalues of {H}ecke operators on {S}iegel cusp
              forms of degree two},
   JOURNAL = {Invent. Math.},
  FJOURNAL = {Inventiones Mathematicae},
    VOLUME = {49},
      YEAR = {1978},
    NUMBER = {2},
     PAGES = {149--165},
   MRCLASS = {10D20},
  MRNUMBER = {511188},
MRREVIEWER = {Rita\ Hall},
}

@book {Lan-PhD,
    AUTHOR = {Lan, Kai-Wen},
     TITLE = {Arithmetic compactifications of {PEL}-type {S}himura
              varieties},
    SERIES = {London Mathematical Society Monographs Series},
    VOLUME = {36},
 PUBLISHER = {Princeton University Press, Princeton, NJ},
      YEAR = {2013},
     PAGES = {xxvi+561},
      %ISBN = {978-0-691-15654-5},
   MRCLASS = {14G35 (11G18 14D23 14M27)},
  MRNUMBER = {3186092},
MRREVIEWER = {Rolf Berndt},
       %DOI = {10.1515/9781400846016},
       %URL = {https://0-doi-org.pugwash.lib.warwick.ac.uk/10.1515/9781400846016},
}

@article {Lan-vanishing,
    AUTHOR = {Lan, Kai-Wen},
     TITLE = {Vanishing theorems for coherent automorphic cohomology},
   JOURNAL = {Res. Math. Sci.},
  FJOURNAL = {Research in the Mathematical Sciences},
    VOLUME = {3},
      YEAR = {2016},
     PAGES = {Paper No. 39, 43},
      %ISSN = {2522-0144,2197-9847},
   MRCLASS = {11G18 (11F75 14F17 14F30 14G35)},
  MRNUMBER = {3571345},
MRREVIEWER = {Giovanni\ Rosso},
}

@incollection {Laumon,
    AUTHOR = {Laumon, G\'erard},
     TITLE = {Fonctions z\^etas des vari\'et\'es de {S}iegel de dimension
              trois},
      NOTE = {Formes automorphes. II. Le cas du groupe $\rm GSp(4)$},
   JOURNAL = {Ast\'erisque},
  FJOURNAL = {Ast\'erisque},
    NUMBER = {302},
      YEAR = {2005},
     PAGES = {1--66},
      %ISSN = {0303-1179,2492-5926},
   MRCLASS = {22E55 (11F46 14G10 14G35)},
  MRNUMBER = {2234859},
}

@article {Liu-triangulation,
    AUTHOR = {Liu, Ruochuan},
     TITLE = {Triangulation of refined families},
   JOURNAL = {Comment. Math. Helv.},
  FJOURNAL = {Commentarii Mathematici Helvetici. A Journal of the Swiss
              Mathematical Society},
    VOLUME = {90},
      YEAR = {2015},
    NUMBER = {4},
     PAGES = {831--904},
      %ISSN = {0010-2571,1420-8946},
   MRCLASS = {11F80 (11F33)},
  MRNUMBER = {3433281},
MRREVIEWER = {Daniel\ Disegni},
       %DOI = {10.4171/CMH/372},
       %URL = {https://doi.org/10.4171/CMH/372},
}

@article {LSTX,
    AUTHOR = {Landesman, Aaron and Swaminathan, Ashvin A. and Tao, James and
              Xu, Yujie},
     TITLE = {Lifting subgroups of symplectic groups over
              {$\mathbb{Z}/\ell\mathbb{Z}$}},
   JOURNAL = {Res. Number Theory},
  FJOURNAL = {Research in Number Theory},
    VOLUME = {3},
      YEAR = {2017},
     PAGES = {Paper No. 14, 12},
      %ISSN = {2522-0160,2363-9555},
   MRCLASS = {11F55},
  MRNUMBER = {3667841},
MRREVIEWER = {Mehmet\ Haluk\ \c Seng\"un},
       %DOI = {10.1007/s40993-017-0078-6},
       %URL = {https://doi.org/10.1007/s40993-017-0078-6},
}

@article {Loeffler--Weinstein,
    AUTHOR = {Loeffler, David and Weinstein, Jared},
     TITLE = {On the computation of local components of a newform},
   JOURNAL = {Math. Comp.},
  FJOURNAL = {Mathematics of Computation},
    VOLUME = {81},
      YEAR = {2012},
    NUMBER = {278},
     PAGES = {1179--1200},
   MRCLASS = {11F70 (11F11)},
  MRNUMBER = {2869056},
MRREVIEWER = {Nathan\ C.\ Ryan},
}

@article {Loeffler-UnivDeform,
    AUTHOR = {Loeffler, David},
     TITLE = {{$p$}-adic {$L$}-functions in universal deformation families},
   JOURNAL = {Ann. Math. Qu\'{e}.},
  FJOURNAL = {Annales Math\'{e}matiques du Qu\'{e}bec},
    VOLUME = {47},
      YEAR = {2023},
    NUMBER = {1},
     PAGES = {117--137},
      %ISSN = {2195-4755,2195-4763},
   MRCLASS = {11F67 (11F85)},
  MRNUMBER = {4569756},
MRREVIEWER = {Neven\ Grbac},
       %DOI = {10.1007/s40316-021-00187-1},
       %URL = {https://0-doi-org.pugwash.lib.warwick.ac.uk/10.1007/s40316-021-00187-1},
}

@article{LW-BianchiAdjoint,
    AUTHOR = {Lee, Pak-Hin and Wu, Ju-Feng},
     TITLE = {On {$p$}-adic adjoint {$L$}-functions for {B}ianchi cuspforms:
              the {$p$}-split case},
   JOURNAL = {Trans. Amer. Math. Soc.},
  FJOURNAL = {Transactions of the American Mathematical Society},
    VOLUME = {378},
      YEAR = {2025},
    NUMBER = {8},
     PAGES = {5579--5640},
      %ISSN = {0002-9947,1088-6850},
   MRCLASS = {11F33 (11F41 11F67 55N30)},
  MRNUMBER = {4929855},
}

@article {Maass,
    AUTHOR = {Maa{\ss}, Hans},
     TITLE = {\"Uber eine {S}pezialschar von {M}odulformen zweiten {G}rades I, II, III},
   JOURNAL = {Invent. Math.},
  FJOURNAL = {Inventiones Mathematicae},
    VOLUME = {52},
      YEAR = {1979},
    NUMBER = {1},
     PAGES = {95--104},
   MRCLASS = {10D12},
  MRNUMBER = {532746},
MRREVIEWER = {J.\ Spilker},
}

@incollection {Mazur-deformation,
    AUTHOR = {Mazur, Barry},
     TITLE = {An introduction to the deformation theory of {G}alois
              representations},
 BOOKTITLE = {Modular forms and {F}ermat's last theorem ({B}oston, {MA},
              1995)},
     PAGES = {243--311},
 PUBLISHER = {Springer, New York},
      YEAR = {1997},
      %ISBN = {0-387-94609-8},
   MRCLASS = {11F80},
  MRNUMBER = {1638481},
}

@book{Milne-AG,
    AUTHOR = {Milne, James S.},
     TITLE = {Algebraic groups},
    SERIES = {Cambridge Studies in Advanced Mathematics},
    VOLUME = {170},
      NOTE = {The theory of group schemes of finite type over a field},
 PUBLISHER = {Cambridge University Press, Cambridge},
      YEAR = {2017},
     PAGES = {xvi+644},
      %ISBN = {978-1-107-16748-3},
   MRCLASS = {14L15 (14-01 17B45 20-01 20G15)},
  MRNUMBER = {3729270},
MRREVIEWER = {Boris\ \`E.\ Kunyavski\u i},
       %DOI = {10.1017/9781316711736},
       %URL = {https://doi.org/10.1017/9781316711736},
}

@misc{Mile-tannakian,
    author = {Milne, James S.}, 
    title = {Tannakian Categories}, 
    year = {2025},
    howpublished = {\url{https://jmilne.org/math/Books/tcdraft.pdf}}
}

@article {Mok-GL2CM,
    AUTHOR = {Mok, Chung Pang},
     TITLE = {Galois representations attached to automorphic forms on {${\rm
              GL}_2$} over {CM} fields},
   JOURNAL = {Compos. Math.},
  FJOURNAL = {Compositio Mathematica},
    VOLUME = {150},
      YEAR = {2014},
    NUMBER = {4},
     PAGES = {523--567},
      %ISSN = {0010-437X,1570-5846},
   MRCLASS = {11R39 (11F80 22E55)},
  MRNUMBER = {3200667},
}

@article{MY,
     author = {Miyauchi, Michitaka and Yamauchi, Takuya},
     title = {An explicit computation of $p$-stabilized vectors},
     journal = {Journal de th\'eorie des nombres de Bordeaux},
     pages = {531--558},
     year = {2014},
     publisher = {Soci\'et\'e Arithm\'etique de Bordeaux},
     volume = {26},
     number = {2},
     mrnumber = {3320491},
}

@misc{MQ-SympDet,
      title={Symplectic determinant laws and invariant theory}, 
      author={Mohamed Moakher and Julian Quast},
      year={2023},
      howpublished = {Preprint. Available at: \url{https://arxiv.org/abs/2310.15822}}, 
}

@article{Pilloni-Stroh,
    author="Pilloni, Vincent and Stroh, Beno{\^i}t",
    title="Cohomologie coh{\'e}rente et repr{\'e}sentations Galoisiennes",
    journal="Annales math{\'e}matiques du Qu{\'e}bec",
    year="2016",
    volume="40",
    number="1",
    pages="167--202",
    %issn="2195-4763",
    %doi="10.1007/s40316-015-0056-0",
}

@article {Pilloni-prolongementSiegel,
    AUTHOR = {Pilloni, Vincent},
     TITLE = {Prolongement analytique sur les vari\'et\'es de {S}iegel},
   JOURNAL = {Duke Math. J.},
  FJOURNAL = {Duke Mathematical Journal},
    VOLUME = {157},
      YEAR = {2011},
    NUMBER = {1},
     PAGES = {167--222},
      %ISSN = {0012-7094,1547-7398},
   MRCLASS = {11F46 (11G18 14G22 14G35)},
  MRNUMBER = {2783930},
}

@article{Pilloni-GL2,
     author = {Pilloni, Vincent},
     title = {Overconvergent modular forms},
     journal = {Annales de l'Institut Fourier},
     pages = {219--239},
     year = {2013},
     publisher = {Association des Annales de l{\textquoteright}institut Fourier},
     volume = {63},
     number = {1},
     mrnumber = {3097946},
     zbl = {06177080},
}

@article {Pilloni-higherHidaColemanGSp4,
    AUTHOR = {Pilloni, Vincent},
     TITLE = {Higher coherent cohomology and {$p$}-adic modular forms of
              singular weights},
   JOURNAL = {Duke Math. J.},
  FJOURNAL = {Duke Mathematical Journal},
    VOLUME = {169},
      YEAR = {2020},
    NUMBER = {9},
     PAGES = {1647--1807},
      %ISSN = {0012-7094,1547-7398},
   MRCLASS = {11F33 (11G18)},
  MRNUMBER = {4105535},
MRREVIEWER = {Chan-Ho\ Kim},
       %DOI = {10.1215/00127094-2019-0075},
       %URL = {https://doi.org/10.1215/00127094-2019-0075},
}

@article{PSS-transfer,
    AUTHOR = {Pitale, Ameya and Saha, Abhishek and Schmidt, Ralf},
     TITLE = {Transfer of {S}iegel cusp forms of degree 2},
   JOURNAL = {Mem. Amer. Math. Soc.},
  FJOURNAL = {Memoirs of the American Mathematical Society},
    VOLUME = {232},
      YEAR = {2014},
    NUMBER = {1090},
     PAGES = {vi+107},
}

@article {Saha--Schmidt,
    AUTHOR = {Saha, Abhishek and Schmidt, Ralf},
     TITLE = {Yoshida lifts and simultaneous non-vanishing of dihedral
              twists of modular {$L$}-functions},
   JOURNAL = {J. Lond. Math. Soc. (2)},
  FJOURNAL = {Journal of the London Mathematical Society. Second Series},
    VOLUME = {88},
      YEAR = {2013},
    NUMBER = {1},
     PAGES = {251--270},
   MRCLASS = {11F30 (11F46 11F66 11F70)},
  MRNUMBER = {3092267},
MRREVIEWER = {Shuichiro\ Takeda},
}

@article {Schmidt-functoriality,
    AUTHOR = {Schmidt, Ralf},
     TITLE = {The {S}aito--{K}urokawa lifting and functoriality},
   JOURNAL = {Amer. J. Math.},
  FJOURNAL = {American Journal of Mathematics},
    VOLUME = {127},
      YEAR = {2005},
    NUMBER = {1},
     PAGES = {209--240},
   MRCLASS = {11F70 (11F32 11F46 22E55)},
  MRNUMBER = {2115666},
MRREVIEWER = {Mahdi\ Asgari},
}

@article {Schmidt-classicalSK,
    AUTHOR = {Schmidt, Ralf},
     TITLE = {On classical {S}aito--{K}urokawa liftings},
   JOURNAL = {J. Reine Angew. Math.},
  FJOURNAL = {Journal f\"ur die Reine und Angewandte Mathematik. [Crelle's
              Journal]},
    VOLUME = {604},
      YEAR = {2007},
     PAGES = {211--236},
   MRCLASS = {11F46 (11F11 11F70)},
  MRNUMBER = {2320318},
MRREVIEWER = {Rainer\ Schulze-Pillot},
}

@article {Schmidt-Packet,
    AUTHOR = {Schmidt, Ralf},
     TITLE = {Packet structure and paramodular forms},
   JOURNAL = {Trans. Amer. Math. Soc.},
  FJOURNAL = {Transactions of the American Mathematical Society},
    VOLUME = {370},
      YEAR = {2018},
    NUMBER = {5},
     PAGES = {3085--3112},
   MRCLASS = {11F46 (11F70)},
  MRNUMBER = {3766842},
MRREVIEWER = {Min\ Ho\ Lee},
}

@article {Schmidt-CAPGSp4,
    AUTHOR = {Schmidt, Ralf},
     TITLE = {Paramodular forms in {CAP} representations of {${\rm
              GSp}(4)$}},
   JOURNAL = {Acta Arith.},
  FJOURNAL = {Acta Arithmetica},
    VOLUME = {194},
      YEAR = {2020},
    NUMBER = {4},
     PAGES = {319--340},
   MRCLASS = {11F46 (11F70)},
  MRNUMBER = {4103275},
MRREVIEWER = {Tomoyoshi\ Ibukiyama},
}

@inproceedings {Serre-Antwerp,
    AUTHOR = {Serre, Jean-Pierre},
     TITLE = {Formes modulaires et fonctions z\^{e}ta {$p$}-adiques},
 BOOKTITLE = {Modular functions of one variable, {III} ({P}roc. {I}nternat.
              {S}ummer {S}chool, {U}niv. {A}ntwerp, {A}ntwerp, 1972)},
    SERIES = {Lecture Notes in Math., Vol. 350},
     PAGES = {191--268},
 PUBLISHER = {Springer, Berlin},
      YEAR = {1973},
   MRCLASS = {10D05},
  MRNUMBER = {0404145},
MRREVIEWER = {K.-B. Gundlach},
}

@misc{stacks-project,
  author       = {The {Stacks project authors}},
  title        = {The Stacks project},
  howpublished = {\url{https://stacks.math.columbia.edu}},
  year         = {2025},
}

@book{Scholze-Weinstein-Berkeley,
    title = {Berkeley lectures on $p$-adic geometry}, 
    author = {Peter Scholze and Jared Weinstein}, 
    year = {2020}, 
    series = {Annals of Mathematics Studies}, 
    publisher = {Princeton University Press},
}

@article {Skinner--Urban,
    AUTHOR = {Skinner, Christopher and Urban, Eric},
     TITLE = {Sur les d\'eformations {$p$}-adiques de certaines
              repr\'esentations automorphes},
   JOURNAL = {J. Inst. Math. Jussieu},
  FJOURNAL = {Journal of the Institute of Mathematics of Jussieu. JIMJ.
              Journal de l'Institut de Math\'ematiques de Jussieu},
    VOLUME = {5},
      YEAR = {2006},
    NUMBER = {4},
     PAGES = {629--698},
      %ISSN = {1474-7480,1475-3030},
   MRCLASS = {11G40 (11F33 11F46 11F80 11F85)},
  MRNUMBER = {2261226},
MRREVIEWER = {Neil\ P.\ Dummigan},
       %DOI = {10.1017/S147474800600003X},
       %URL = {https://doi.org/10.1017/S147474800600003X},
}

@article {Sorensen-HilbertSiegel,
    AUTHOR = {Sorensen, Claus M.},
     TITLE = {Galois representations attached to {H}ilbert-{S}iegel modular
              forms},
   JOURNAL = {Doc. Math.},
  FJOURNAL = {Documenta Mathematica},
    VOLUME = {15},
      YEAR = {2010},
     PAGES = {623--670},
      %ISSN = {1431-0635,1431-0643},
   MRCLASS = {11F80 (11F33 11F46 11F70 11G18 22E55)},
  MRNUMBER = {2735984},
}

@book {SR-Tannakian,
    AUTHOR = {Saavedra Rivano, Neantro},
     TITLE = {Cat\'egories {T}annakiennes},
    SERIES = {Lecture Notes in Mathematics},
    VOLUME = {Vol. 265},
 PUBLISHER = {Springer-Verlag, Berlin-New York},
      YEAR = {1972},
     PAGES = {ii+418},
   MRCLASS = {14L15 (18D10 20G05)},
  MRNUMBER = {338002},
MRREVIEWER = {P.\ Abellanas},
 SHORTHAND = {SR72},
}

@article{Taylor-Siegel,
    AUTHOR = {Taylor, Richard},
     TITLE = {On the {$l$}-adic cohomology of {S}iegel threefolds},
   JOURNAL = {Invent. Math.},
  FJOURNAL = {Inventiones Mathematicae},
    VOLUME = {114},
      YEAR = {1993},
    NUMBER = {2},
     PAGES = {289--310},
      %ISSN = {0020-9910,1432-1297},
   MRCLASS = {11G18 (11F46 14F30 14G35)},
  MRNUMBER = {1240640},
MRREVIEWER = {Ernst-Wilhelm\ Zink},
       %DOI = {10.1007/BF01232672},
       %URL = {https://doi.org/10.1007/BF01232672},
}

@article{Thorne-nonvanishing,
    AUTHOR = {Thorne, Jack A.},
     TITLE = {On the vanishing of adjoint {B}loch-{K}ato {S}elmer groups of
              irreducible automorphic {G}alois representations},
   JOURNAL = {Pure Appl. Math. Q.},
    VOLUME = {18},
      YEAR = {2022},
    NUMBER = {5},
     PAGES = {2159--2202},
}

@incollection{Urban-GSp4,
     author = {Urban, Eric},
     title = {Sur les repr\'esentations $p$-adiques associ\'ees aux repr\'esentations cuspidales de $\mathrm{GSp}_{4/\mathbb{Q}}$},
     booktitle = {Formes automorphes (II) - Le cas du groupe $\mathrm{GSp}(4)$},
     editor = {Tilouine, Jacques and Carayol, Henri and Harris, Michael and Vign\'eras, Marie-France},
     series = {Ast\'erisque},
     pages = {151--176},
     publisher = {Soci\'et\'e math\'ematique de France},
     number = {302},
     year = {2005},
     mrnumber = {2234861},
     zbl = {1100.11017},
     %url = {https://www.numdam.org/item/AST_2005__302__151_0/}
}

@article{Urban-2011,
    author = {Eric Urban},
    journal = {Annals of Mathematics},
    number = {3},
    pages = {1685--1784},
    publisher = {Annals of Mathematics},
    title = {Eigenvarieties for reductive groups},
    volume = {174},
    year = {2011}
}

@incollection{Weissauer,
     author = {Weissauer, Rainer},
     title = {Four dimensional {Galois} representations},
     booktitle = {Formes automorphes (II) - Le cas du groupe $\mathrm{GSp}(4)$},
     editor = {Tilouine, Jacques and Carayol, Henri and Harris, Michael and Vign\'eras, Marie-France},
     series = {Ast\'erisque},
     pages = {67--150},
     publisher = {Soci\'et\'e math\'ematique de France},
     number = {302},
     year = {2005},
     mrnumber = {2234860},
     zbl = {1097.11027},
     %language = {en},
     %url = {http://www.numdam.org/item/AST_2005__302__67_0/}
}

@incollection {Zagier-SKConjecture,
    AUTHOR = {Zagier, Don},
     TITLE = {Sur la conjecture de {S}aito-{K}urokawa (d'apr\`es {H}.
              {M}aass)},
 BOOKTITLE = {Seminar on {N}umber {T}heory, {P}aris 1979--80},
    SERIES = {Progr. Math.},
    VOLUME = {12},
     PAGES = {371--394},
 PUBLISHER = {Birkh\"auser, Boston, MA},
      YEAR = {1981},
   MRCLASS = {10D24 (10D20)},
  MRNUMBER = {633910},
MRREVIEWER = {A.\ N.\ Andrianov},
}

@article {Zavyalov-quotient,
    AUTHOR = {Zavyalov, Bogdan},
     TITLE = {Quotients of admissible formal schemes and adic spaces by
              finite groups},
   JOURNAL = {Algebra Number Theory},
  FJOURNAL = {Algebra \& Number Theory},
    VOLUME = {18},
      YEAR = {2024},
    NUMBER = {3},
     PAGES = {409--475},
      %ISSN = {1937-0652,1944-7833},
   MRCLASS = {14A99 (14G22 14L30)},
  MRNUMBER = {4705884},
}

\vspace{10mm}

\begin{tabular}[t]{l}
    M.M.\\
    Concordia University   \\
    Department of Mathematics and Statistics\\
    Montr\'{e}al, Qu\'{e}bec, Canada\\
    \textit{E-mail address: }\texttt{muhammad.manji@concordia.ca}
\end{tabular}

\vspace{5mm}

\begin{tabular}[t]{l}
    F.E.T.\\
    School of Mathematical Sciences\\
    University of Nottingham\\
    Nottingham, UK\\
    \textit{E-mail address: }\texttt{Frederick.Thogersen@nottingham.ac.uk}
\end{tabular}

\vspace{5mm}

\begin{tabular}[t]{l}
    J.-F.W.\\
    School of Mathematics and Statistics\\
    University College Dublin\\
    Belfield, Dublin 4, Ireland\\
    \textit{E-mail address: }\texttt{ju-feng.wu@ucd.ie}
\end{tabular}

\end{document}